\documentclass[11pt,twoside,a4paper]{amsart}

\usepackage{amssymb}
\usepackage{vmargin}
\usepackage{amscd}
\usepackage{stmaryrd}
\usepackage{mathrsfs}
\usepackage[all]{xy}
\usepackage{enumitem}
\usepackage{xr-hyper}%%%or {xr}
\usepackage{hyperref}
\hypersetup{colorlinks,
allcolors=black}
\setmargins{32mm}{20mm}{14.6cm}{22cm}{1cm}{1cm}{1cm}{1cm}

\newcommand\da{\!\downarrow\!}

\newcommand\la{\leftarrow}

\newcommand\id{\mathrm{id}}

\newcommand\ten{\otimes}
\newcommand\hten{\hat{\otimes}}

\newcommand\vareps{\varepsilon}
\newcommand\eps{\epsilon}

\renewcommand\H{\mathrm{H}}
\newcommand\z{\mathrm{Z}}
\renewcommand\b{\mathrm{B}}

\newcommand\N{\mathbb{N}}
\newcommand\Z{\mathbb{Z}}
\newcommand\Q{\mathbb{Q}}
\newcommand\Ql{\mathbb{Q}_{\ell}}
\newcommand\Zl{\mathbb{Z}_{\ell}}

\newcommand\R{\mathbb{R}}
\newcommand\Cx{\mathbb{C}}

\newcommand\vv{\mathbb{V}}

\newcommand\bA{\mathbb{A}}

\newcommand\bD{\mathbb{D}}

\newcommand\bG{\mathbb{G}}

\newcommand\bI{\mathbb{I}}

\newcommand\bK{\mathbb{K}}
\newcommand\bL{\mathbb{L}}

\newcommand\bP{\mathbb{P}}

\newcommand\bT{\mathbb{T}}

\newcommand\C{\mathcal{C}}

\newcommand\cA{\mathcal{A}}

\newcommand\cC{\mathcal{C}}
\newcommand\cD{\mathcal{D}}
\newcommand\cE{\mathcal{E}}
\newcommand\cF{\mathcal{F}}

\newcommand\cH{\mathcal{H}}

\newcommand\cN{\mathcal{N}}

\newcommand\cO{\mathcal{O}}
\newcommand\cP{\mathcal{P}}

\newcommand\cU{\mathcal{U}}

\newcommand\sA{\mathscr{A}}
\newcommand\sB{\mathscr{B}}
\newcommand\sC{\mathscr{C}}

\newcommand\sE{\mathscr{E}}
\newcommand\sF{\mathscr{F}}

\newcommand\sM{\mathscr{M}}

\newcommand\sO{\mathscr{O}}

\newcommand\sT{\mathscr{T}}

\newcommand\fY{\mathfrak{Y}}

\newcommand\sHom{\mathscr{H}\!\mathit{om}}
\newcommand\cHom{\mathcal{H}\!\mathit{om}}

\newcommand\cHHom{\underline{\mathcal{H}\!\mathit{om}}}

\newcommand\Ho{\mathrm{Ho}}

\newcommand\Alg{\mathrm{Alg}}
\newcommand\CAlg{\mathrm{CAlg}}

\newcommand\comm{\mathrm{com}}

\newcommand\Mod{\mathrm{Mod}}

\newcommand\Hom{\mathrm{Hom}}
\newcommand\Map{\mathrm{Map}}
\newcommand\map{\mathrm{map}}

\newcommand\cyc{\mathrm{cyc}}
\newcommand\Tw{\mathrm{Tw}}
\newcommand\EFC{\mathrm{EFC}}
\newcommand\FEFC{\mathrm{FEFC}}

\newcommand\HHom{\underline{\mathrm{Hom}}}
\newcommand\EEnd{\underline{\mathrm{End}}}

\newcommand\cone{\mathrm{cone}}

\newcommand\dg{\mathrm{dg}}
\newcommand\per{\mathrm{per}}

\newcommand{\llb}{\llbracket}
\newcommand{\rrb}{\rrbracket}

\newcommand\Gal{\mathrm{Gal}}

\newcommand\coker{\mathrm{coker\,}}

\newcommand\Ab{\mathrm{Ab}}

\newcommand\Shf{\mathrm{Shf}}

\newcommand\loc{\mathrm{loc}}

\newcommand\Spec{\mathrm{Spec}\,}

\newcommand\Set{\mathrm{Set}}

\newcommand\Affd{\mathrm{Affd}}

\newcommand\Sp{\mathrm{Sp}}

\newcommand\PreBiSp{\mathrm{PreBiSp}}
\newcommand\BiSp{\mathrm{BiSp}}

\newcommand\BiLag{\mathrm{BiLag}}

\newcommand\Pol{\mathrm{Pol}}

\newcommand\nondeg{\mathrm{nondeg}}

\newcommand\ad{\mathrm{ad}}

\newcommand\<{\langle}
\renewcommand\>{\rangle}
\newcommand\Lim{\varprojlim}
\newcommand\LLim{\varinjlim}
\DeclareMathOperator*{\holim}{holim}
\newcommand\ho{\mathrm{ho}\!}
\newcommand\into{\hookrightarrow}
\newcommand\onto{\twoheadrightarrow}
\newcommand\abuts{\implies}
\newcommand\xra{\xrightarrow}
\newcommand\xla{\xleftarrow}

\newcommand\pr{\mathrm{pr}}

\newcommand\alg{\mathrm{alg}}

\newcommand\bt{\bullet}
\newcommand\by{\times}

\newcommand\mc{\mathrm{MC}}
\newcommand\mmc{\underline{\mathrm{MC}}}

\newcommand\Perf{\mathrm{Perf}}
\newcommand\Ban{\mathrm{Ban}}

\newcommand\SO{\mathrm{SO}}

\newcommand\Mat{\mathrm{Mat}}
\newcommand\et{\acute{\mathrm{e}}\mathrm{t}}

\newcommand\Et{\acute{\mathrm{E}}\mathrm{t}}

\newcommand\an{\mathrm{an}}
\newcommand\hol{\mathrm{hol}}

\newcommand\Tot{\mathrm{Tot}\,}

\newcommand\tr{\mathrm{tr}}

\newcommand\ind{\mathrm{ind}}
\newcommand\pro{\mathrm{pro}}

\newcommand\pd{\partial}
\newcommand\dc{d^{\mathrm{c}}}

\newcommand\half{\frac{1}{2}}

\newcommand\rig{\mathrm{rig}}

\newcommand\gr{\mathrm{gr}}

\newcommand\cts{\mathrm{cts}}

\newcommand\Zar{\mathrm{Zar}}

\newcommand\Fil{\mathrm{Fil}}

\renewcommand\alg{\mathrm{alg}}

\newcommand\Dol{\mathrm{Dol}}
\newcommand\dR{\mathrm{dR}}
\newcommand\DR{\mathrm{DR}}
\newcommand\Hod{\mathrm{Hod}}

\newcommand\op{\mathrm{opp}}

\newcommand\co{\colon\thinspace}

\newcommand\oB{\mathbf{B}}

\newcommand\oR{\mathbf{R}}

\newcommand\oL{\mathbf{L}}

\newcommand\uleft\underleftarrow
\newcommand\uline\underline
\newcommand\uright\underrightarrow

\newcommand{\tps}{\texorpdfstring}

\DeclareMathOperator{\hatHHom}{\uline{\mathrm{H}\widehat{\mathrm{o}}\mathrm{m}}}

\newtheorem{theorem}{Theorem}[section]
\newtheorem{proposition}[theorem]{Proposition}
\newtheorem{corollary}[theorem]{Corollary}

\newtheorem{lemma}[theorem]{Lemma}
\newtheorem*{theorem*}{Theorem}
\newtheorem*{proposition*}{Proposition}
\newtheorem*{corollary*}{Corollary}
\newtheorem*{lemma*}{Lemma}
\newtheorem*{conjecture*}{Conjecture}

\theoremstyle{definition}
\newtheorem{definition}[theorem]{Definition}
\newtheorem*{definition*}{Definition}

\newtheorem*{notation*}{Notation}

\theoremstyle{remark}
\newtheorem{example}[theorem]{Example}
\newtheorem{examples}[theorem]{Examples}
\newtheorem{remark}[theorem]{Remark}
\newtheorem{remarks}[theorem]{Remarks}

\newtheorem*{example*}{Example}
\newtheorem*{examples*}{Examples}
\newtheorem*{remark*}{Remark}
\newtheorem*{remarks*}{Remarks}
\newtheorem*{exercise*}{Exercise}
\newtheorem*{property*}{Property}
\newtheorem*{properties*}{Properties}

\newcommand\Tc{\mathcal{T}c}

\begin{document}

\begin{abstract}
We develop a formulation for non-commutative derived analytic geometry built from differential graded (dg) algebras equipped with free entire functional calculus (FEFC), relating them to simplicial FEFC algebras and to locally multiplicatively convex  complete topological dg  algebras. The theory is optimally suited for accommodating analytic morphisms between functors of algebraic origin, and we establish forms of Riemann--Hilbert equivalence in this setting. 
We also investigate  classes of topological dg algebras for which  moduli functors of analytic origin tend to behave well, and relate their homotopy theory to that of FEFC algebras. Applications include the construction of derived non-commutative analytic moduli stacks of pro-\'etale local systems and non-commutative derived twistor moduli functors, both equipped with shifted analytic bisymplectic structures, and hence shifted analytic double Poisson structures.

\end{abstract}

\title{Non-commutative derived analytic moduli functors} 

\author{J. P. Pridham}

\maketitle
\section*{Introduction}

There are many natural moduli functors which exist only analytically rather than algebraically, such as twistor functors over $\Cx$ or $\R$  and moduli of \'etale local systems over non-Archimedean fields. One of the simplest ways to enhance these to derived objects is to work with dg algebras with entire functional calculus (EFC), since their homotopy theory is easily controlled. However, most constructions moduli functors rely on algebras carrying a well-behaved topology rather than just EFC, so \cite{PTLag} established dg dagger affinoid algebras as a convenient class of differential graded topological algebras whose homotopy theory is determined solely by the EFC structure, providing a setting which captures the advantages of both approaches.

On the other hand, the natural class of test algebras for moduli functors of linear objects should be non-commutative algebras. Infinitesimally, this manifests itself through the governing Lie algebras of such functors  underlying associative algebras of derived endomorphisms. A global theory built on these lines for algebraic moduli functors was developed in \cite{NCstacks}, with many features of moduli functors of linear objects arising as shadows of the non-commutative moduli functor.

In this paper, we set about developing a similar generalisation of derived analytic geometry  incorporating non-commutativity. To develop the   homotopical theory, this means replacing EFC with the free entire functional calculus (FEFC) of \cite{taylorFnsNC}. In practice, moduli functors are naturally defined on categories of topological, rather than FEFC, algebras. While the rarity of Noetherianity and the faster growth of non-commutative polynomials form an apparent obstacle to forming a class of non-commutative topological algebras satisfying both  the categorical compactness and cocompactness properties of dg dagger affinoid algebras, we identify various subcategories of finitely generated objects which suffice for most purposes (Corollary \ref{afpTcFEFCcor}). Examples of nuclear Fr\'echet algebras whose FEFC structure fully captures the topology include the algebra of free entire functions, finite-dimensional algebras, finitely embeddable Stein algebras, and the algebra of functions on quantum affine space.

Our first main results are Corollaries \ref{keyBanEFCcor} and \ref{keyBanFEFCcor}, which show, in the commutative and non-commutative settings, that to forget the topology on %locally multiplicatively (and differentially) convex 
pro-Banach
dg algebras and just retain the underlying (free) entire functional calculus  is equivalent to formally inverting abstract homology isomorphisms. Thus any functor on such topological algebras which preserves abstract quasi-isomorphisms will give rise to a functor on (F)EFC dg algebras.

Preserving abstract quasi-isomorphisms   is a strong constraint unless the functor is an analytification of an algebraic functor (i.e. does not involve the topology). However, it is a very common property for functors on dg Fr\'echet algebras, and particularly on dg nuclear Fr\'echet algebras. %, to preserve abstract quasi-isomorphisms.
Many such functors can be constructed by first taking the complete tensor product with a fixed (nuclear) Fr\'echet algebra $\C$, then applying an algebraic functor. Our main motivating examples arise as homotopy limits of such constructions, with derived moduli of hyperconnections arising from the nuclear Fr\'echet dg algebra of differential forms, and derived moduli of pro-\'etale local systems arising from the Banach algebra of continuous $\Ql$-valued functions (Examples \ref{CXex} and \ref{CXex2}).     

We study these constructions in detail for  derived moduli functors parametrising various flavours of  hyperconnections. A generalisation of the Poincar\'e lemma to pro-Banach dg algebras gives equivalences between the analytification of the  Betti derived moduli functor parametrising $\infty$-local systems on a manifold on one hand, and the de Rham  derived moduli functors parametrising hyperconnections on either smooth or holomorphic perfect complexes (Corollary \ref{poincarecor}). This can be thought of as a Riemann--Hilbert correspondence, but there is also an algebraic de Rham functor for smooth complex varieties whose analytification has a natural \'etale map to the analytic de Rham functor, giving an \'etale analytic Riemann--Hilbert map from the algebraic de Rham functor to the Betti functor (Corollary \ref{RHcoralg}).

Considering  hyperconnections with respect to more general derivations gives corresponding notions of Dolbeault, Hodge and twistor functors, with both algebraic and analytic variants (Theorem \ref{twistorthm}). The analytic twistor functor is relatively simple to define using smooth forms, but only preserves quasi-isomorphisms when restricted to the subcategory of dg Fr\'echet algebras. This has an open subfunctor, the algebro-analytic twistor functor,  defined by Deligne gluing of analytifications of algebraic functors via the Riemann--Hilbert map; it preserves all quasi-isomorphisms so is naturally defined as a functor on dg FEFC algebras. These twistor  constructions are $\Cx^*$-equivariant and have real forms, living as families over the analytification of $[\bP^1_{\R}/\SO_2]$, i.e. the complex-analytic stack $[\bP^1(\Cx)/\Cx^*]$ equipped with its complex conjugation $z \mapsto \bar{z}^{-1}$.    

We also introduce analytic analogues of the shifted bisymplectic and double Poisson structures from \cite{NCpoisson}. When restricted to commutative inputs, these yield analytic shifted symplectic and Poisson structures, but they also for instance give such structures on functors of representations by taking matrix rings as inputs, together with Lagrangian correspondences for Hall algebras (Remark \ref{bispcommrmk2}). Thanks to the integration property of Example \ref{IntCYex}, shifted analytic double Poisson structures are often simpler to describe on derived NC prestacks than on their commutative counterparts, In \S \ref{bispexsn} we give several examples of derived moduli functors carrying such structures. Examples include the $(2-2m)$-shifted analytic bisymplectic structures on the derived analytic non-commutative moduli functor of constructible $\Ql$-sheaves on an $m$-dimensional variety, and on the non-degenerate locus of our analytic twistor functor for an $m$-dimensional complex manifold.

\medskip
The structure of the paper is as follows.

In Section \ref{EFCsn} we recall the notions of algebras and differential graded algebras with entire functional calculus. The $\infty$-equivalence between EFC-DGAs and pro-Banach CDGAs localised at abstract quasi-isomorphisms appears in \S \ref{commcompsn}. We also establish key properties of nuclear Fr\'echet algebras and spaces (\S \ref{commnuclearFrechetsn}), many of which are used later, and recall how EFC-DGAs are used to form derived dagger analytic geometry (\S \ref{derivedanalyticstacksn}).

Section \ref{FEFCsn} then recalls and develops the first results we need for the non-commutative analogue, algebras with free entire functional calculus. The main new results characterise Beck modules and derivations for this theory. Whereas a module over an EFC-algebra is simply a module over the underlying commutative ring, an FEFC-module over an FEFC-algebra $A$ is a module over a form of completed tensor product $A\hten A^{\op}$. In \S \ref{dgFEFCsn}, we extend this theory to derived algebras. Although an EFC-DGA only has entire functional calculus in degree $0$, the correct notion of an FEFC-DGA has functional calculus involving all degrees, necessitated by non-commutativity, and this notion has the same homotopy theory as simplicial FEFC algebras (Proposition \ref{simpFEFCmodelprop}).

In Section \ref{DGLMDCsn}, we relate the homotopy theory of FEFC-DGAs to that of various classes of topological dg algebras. The most striking result is the $\infty$-equivalence between FEFC-DGAs and pro-Banach DGAAs localised at abstract quasi-isomorphisms, in \S \ref{FEFCcompsn}. We also identify various conditions on subcategories in \S \ref{subcatfunsn} which, though not necessarily $\infty$-subcategories after simplicial localisation, still satisfy various universality properties which suffice for many purposes. %{Dlfpcor}
 Two such subcategories are given by Fr\'echet and nuclear Fr\'echet dg algebras, and the latter in particular include many whose FEFC cotangent complexes can be calculated topologically (Proposition \ref{resnprop} and subsequent examples). %\S \ref{nucfrechet} 
The exactness properties of completed topological tensor products on nuclear Fr\'echet spaces then give rise to many examples of well-behaved non-commutative analytic derived moduli functors, notably those parametrising arithmetic (pro-\'etale) local systems and hyperconnections (Examples \ref{ctbleex}). 
%{CXlemma1}
%{RGammacotprop}

Hyperconnections are an example of a moduli functor associated  via denormalisation to a stacky DGAA (topological or FEFC), i.e. a bidifferential bigraded associative algebra. We develop the relevant theory for stacky DGAAs in Section \ref{stackyFEFCsn}, with Corollary \ref{stackykeyBanFEFCcor} providing a bigraded analogue of our earlier results showing that homotopically, giving an FEFC structure is equivalent to giving a pro-Banach topology but inverting levelwise abstract quasi-isomorphisms. Our key motivation for introducing stacky DGAAs lies in their ability to give formally \'etale atlases for well-behaved moduli functors, allowing us to define structures  purely in terms of formally \'etale data  (Corollary \ref{etsitecor}). In Section \ref{RHsn}, we then study derived moduli of hyperconnections and related structures in detail, as described above.

Section \ref{poissonsn} provides the main motivation for introducing stacky DGAAs, since they allow us to develop a notion of shifted analytic double Poisson structures for non-commutative derived analytic moduli functors. We also introduce shifted analytic bisymplectic structures, which being linear data are easier to define, and show in Proposition \ref{equivpropprestack} that these correspond to non-degenerate shifted analytic double Poisson structures via a homotopical Legendre transformation. 

\subsection*{Notation}

For a chain (resp. cochain) complex $M$, we write $M_{[i]}$ (resp. $M^{[j]}$) for the complex $(M_{[i]})_m= M_{i+m}$ (resp. $(M^{[j]})^m = M^{j+m}$). We often work with double complexes, in the form of cochain chain complexes, in which case $M_{[i]}^{[j]}$ is the double complex $(M_{[i]}^{[j]})^n_m= M_{i+m}^{j+n}$. When we have a single grading and need to compare chain and cochain complexes, we silently  make use of the equivalence  $u$ from chain complexes to cochain complexes given by $(uV)^i := V_{-i}$, and refer to this as rewriting the chain complex as a cochain complex (or vice versa). On suspensions, this has the effect that $u(V_{[n]}) = (uV)^{[-n]}$. We also occasionally write $M[i]:=M^{[i]} =M_{[-i]}$ when there is only one grading. 

For chain complexes, by default we denote differentials by $\delta$. When we work with cochain chain complexes, the cochain differential is usually denoted by $\pd$. We use the subscripts and superscripts $\bt$ to emphasise that chain and cochain complexes incorporate differentials, with $\#$ used instead when we are working  with the underlying graded objects.

Given $A$-modules $M,N$ in chain complexes, we write $\HHom_A(M,N)$ for the cochain complex given by
\[
 \HHom_A(M,N)^i= \Hom_{A_{\#}}(M_{\#},N_{\#[-i]}),
\]
with differential $\delta f= \delta_N \circ f \pm f \circ \delta_M$,
where $V_{\#}$ denotes the graded vector space underlying a chain complex $V$.

We write $s\Set$ for the category of simplicial sets, and $\oR\map$ for derived mapping spaces, i.e. the right-derived functor of simplicial $\Hom$. (For dg categories, $\oR\map$ corresponds via the Dold--Kan correspondence to the truncation of $\oR\HHom$.)

\tableofcontents

\section{Algebras with entire functional calculus}\label{EFCsn}

\subsection{EFC-algebras}

\subsection{Complex algebras}

\begin{definition}\label{EFCdef}
Recall (cf. \cite{pirkovskiiHFG}) that a $\Cx$-algebra $A$ with entire functional calculus (or EFC $\Cx$-algebra for short)  is given by a product-preserving set-valued  functor $\Cx^n \mapsto A^n$  on the category with objects $\{\Cx^n\}_{n \ge 0}$ and morphisms consisting of complex-analytic maps. 

Explicitly, the set $A$ is equipped, for every complex-analytic function $f \in \sO(\Cx^n)$, with an operation $\Phi_f \co A^n \to A$. These operations are required to be compatible in the sense that given functions $g_i \in \sO(\Cx^{m_i})$, we must have 
\[
 \Phi_{f \circ (g_1, \ldots, g_n)}= \Phi_f \circ (\Phi_{g_1}, \ldots, \Phi_{g_n}) \co A^{\sum_{i=1}^n m_i}\to A.  
\]
\end{definition}
Thus an EFC $\Cx$-algebra is a $\Cx$-algebra $A$ with a systematic and consistent way of evaluating expressions of the form $\sum_{m_1, \ldots, m_n=0}^{\infty} \lambda_{m_1, \ldots, m_n} a_1^{m_1}\cdots a_n^{m_n}$ in $A$ whenever the coefficients $\lambda_{m_1, \ldots, m_n} \in \Cx$ satisfy $\lim_{\sum m_i \to \infty } |\lambda_{m_1, \ldots, m_n}|^{1/\sum m_i}= 0$.
In particular, every EFC $\Cx$-algebra is a commutative unital $\Cx$-algebra, with  addition and multiplication coming from the functions $(x,y) \mapsto x+y$ and
  $(x,y) \mapsto xy$ in $\sO(\Cx^2)$.

\begin{examples}\label{EFCexamples}
 For every complex-analytic space $X$, the set $\sO(X)$ of global holomorphic functions on $X$ naturally has the structure of an EFC $\Cx$-algebra, as does any quotient of $\sO(X)$ by an ideal (see e.g.  \cite[Examples \ref{DStein-Fermatexamples}]{DStein}). %\cite{DubucKock}\cite[Definition 2.7]{joyceAGCinfty}).
In particular, Artinian rings with residue field $\Cx$ (for instance  the dual numbers $\Cx[x]/x^2$) can naturally be regarded as EFC $\Cx$-algebras.

\end{examples}

From a categorical perspective, the forgetful functor from EFC $\Cx$-algebras to sets has a left adjoint, 
 which sends a set $S$ to the EFC $\Cx$-algebra
\[
\sO(\Cx^S):=  \LLim_{\substack{T \subset S}\\ \text{finite}} \sO(\Cx^T),
\]
and  EFC $\Cx$-algebras are algebras for the resulting monad structure on the functor $S \mapsto \sO(\Cx^S)$.

\subsection{Other base rings} 
We now fix a complete valued $\Q$-algebra $R$, leading to  the following generalisation:
\begin{definition}\label{EFCnonArchdef}
Define a  $R$-algebra $A$ with entire functional calculus (or EFC $R$-algebra for short)  to be a product-preserving set-valued  functor $\bA_{R}^n \mapsto A^n$  on the full subcategory of rigid analytic varieties with objects the affine spaces $\{\bA_{R}^n\}_{n \ge 0}$. 
\end{definition}
Explicitly, for the rings $ \sO^{\an}(\bA^n_{R})$ of  $R$-analytic functions, i.e.
\[
\{\sum_{m_1, \ldots, m_n=0}^{\infty} \lambda_{m_1, \ldots, m_n} z_1^{m_1}\cdots z_n^{m_n}  ~:~ \lim_{\sum m_i \to \infty } |\lambda_{m_1, \ldots, m_n}|^{1/\sum m_i}= 0\} \subset R\llb z_1, \ldots,z_n\rrb,
\]
the set of analytic morphisms $\bA_{R}^n \to \bA_{R}^m$ is isomorphic to $\sO^{\an}(\bA^n_{R})^m$.
Thus  an EFC $R$-algebra $A$  is a $R$-algebra equipped with a systematic and consistent way of evaluating expressions of the form $\sum_{m_1, \ldots, m_n=0}^{\infty} \lambda_{m_1, \ldots, m_n} a_1^{m_1}\cdots a_n^{m_n}$ in $A$ whenever the coefficients $\lambda_{m_1, \ldots, m_n} \in R$ satisfy $\lim_{\sum m_i \to \infty } |\lambda_{m_1, \ldots, m_n}|^{1/\sum m_i}= 0$. 

\smallskip
For a more familiar interpretation of real  EFC algebras,  consider the natural action of the Galois group $\Gal(\Cx/\R)$ on $\sO(\Cx^S)$ given by letting it act naturally on the coefficients of Taylor series; the complex conjugate $\bar{f}$ of a holomorphic function $f$ is then given by the formula
\[
 \bar{f}(z_1, \ldots, x_n):=\overline{f(\bar{z}_1, \ldots, \bar{z}_n)}. 
\]
We then have  $\sO^{\an}(\bA^n_{\R})\cong  \sO(\Cx^n)^{\Gal(\Cx/\R)}$,  the ring  of holomorphic functions $f$ on $\Cx^S$ satisfying $f=\bar{f}$, or equivalently $f(\bar{z}_1, \ldots, \bar{z}_n)=\overline{f(z_1, \ldots,z_n)}$.

Also note that every EFC $\Cx$-algebra is naturally an EFC $\R$-algebra, and that this forgetful functor has a left adjoint simply given by sending an EFC $\R$-algebra $A$ to the tensor product $A\ten_{\R}\Cx$; this adjunction is in fact comonadic.

\begin{examples}\label{REFCexamples}
For every complex-analytic space $X$ equipped with an anti-holomorphic involution $\sigma$ (e.g. the space of complex points of a real algebraic variety), the set $\sO(X)^{\Gal(\Cx/\R)}$ of global holomorphic functions $f$ on $X$ with $f(\sigma(x))=\overline{f(x)}$  naturally has the structure of an EFC $\R$-algebra, as does any quotient of $\sO(X)^{\Gal(\Cx/\R)}$ by an ideal.
In particular, Artinian rings with residue field $\R$  can naturally be regarded as EFC $\R$-algebras. 
\end{examples}

The category of EFC $R$-algebras has all small limits and colimits.
In particular, there is a coproduct, which we denote by $\odot$, with the property that $\sO(M \by N) \cong \sO(M)\odot\sO(N)$  for  analytic spaces $M,N$ . We also denote pushouts by $A\odot_BC$. Because all quotient rings of an  EFC $R$-algebra are  EFC $R$-algebras   as in \cite{DubucKock} (\cite[Corollary 2.8]{CarchediRoytenberg}), it follows that if we have surjective EFC $\Cx$-algebra homomorphisms $\sO(\Cx^S) \onto A$, $\sO(\Cx^T) \onto C$, then the pushout of $A \la B \to C$ can be expressed as
\[
 A\odot_BC:= \sO(\Cx^{S \sqcup T})\ten_{(\sO(\Cx^S)\ten\sO(\Cx^T))}(A\ten_BC),
\]
with a similar expression for EFC $R$-algebras.

The following terminology is adapted from \cite[Definition \ref{DStein-EFCDGAdef}]{DStein}, restricting to non-negative degrees; see \cite[Definition 4.14]{CarchediRoytenbergHomological} and  \cite[Example 2.16]{CarchediRoytenberg} for earlier instances of the same concept.

\begin{definition}\label{EFCDGAdef}
Define an EFC-differential graded $R$-algebra (EFC-DGA for short) to be a chain complex $A_{\bt}=(A_{\#}, \delta)$ of $R$-vector spaces in non-negative homological degrees ($\ldots \xra{\delta} A_1 \xra{\delta} A_0$), equipped with:
\begin{itemize}
 \item an associative graded $R$-linear multiplication, graded-commutative in the sense that $ab= (-1)^{\deg a\deg b}ba$ for all $a,b \in A$, and 
\item an enhancement of the $R$-algebra structure on $A_0$ to an EFC-algebra structure,
\end{itemize}
such that $\delta$ is a graded derivation in the sense that $\delta(ab)= \delta(a)b + (-1)^{\deg a}a\delta(b)$ for all $a,b \in A$.

We denote the category of such objects by $dg_+\EFC\Alg_{R}$ --- morphisms are ${R}$-CDGA morphisms $A_{\bt} \to B_{\bt}$ for which the map $A_0 \to B_0$ preserves the entire functional calculus.
\end{definition}

As in \cite[Proposition \ref{DStein-stdmodelprop}]{DStein}, results from \cite{CarchediRoytenbergHomological,nuitenThesis} imply the following:
\begin{proposition}\label{stdmodelprop}
There is a cofibrantly generated model structure (which we call the standard model structure) on the category $dg_+\EFC\Alg_{R}$  in which  a morphism $A_{\bt} \to B_{\bt}$ is
\begin{enumerate}
 \item  a weak equivalence if it is a quasi-isomorphism;
\item a fibration if it is surjective in strictly positive chain degrees.
\end{enumerate}

There is also a model structure  for simplicial EFC $R$-algebras in which weak equivalences are $\pi_*$-isomorphisms and fibrations are Kan fibrations, and Dold--Kan normalisation combines with the Eilenberg--Zilber shuffle product \cite[Definitions 8.3.6 and 8.5.4]{W} to give a right Quillen equivalence $N$ from simplicial EFC $R$-algebras to our standard model structure. 
\end{proposition}

\subsection{Complete topological dg  algebras}

\subsubsection{Basic comparisons}\label{commcompsn}

\begin{definition}
Write $\Ban\CAlg_{R}$ for the category of unital commutative Banach algebras over ${R}$, with bounded  morphisms, and $\Ban\Alg_{R}$ for the category of unital associative Banach algebras over ${R}$, with bounded  morphisms. Here, the units are required to have norm $1$ and the norm is submultiplicative. 
\end{definition}

\begin{definition}
 Define a unital associative  differential graded Banach ${R}$-algebra $A_{\bt}$ to be a chain complex of Banach ${R}$-vector spaces, 
 with $\|\delta a\|\le \|a\|$,
 %with the differentials $\delta\co A_n \to A_{n-1}$ being bounded operators, 
 together with a unit $1 \in \z_0A$ of norm $1$ and an %bounded 
 ${R}$-bilinear associative multiplication $A_i \by A_j \to A_{i+j}$ %%%should be submultiplicative
 with $\|ab\|\le \|a\|\cdot \|b\|$ and with respect to which $\delta$ is a derivation.
 
 We say that $A_{\bt}$ is commutative if the multiplication is graded-commutative. 
\end{definition}

\begin{definition}
Write $dg_+\Ban\CAlg_{R}$ (resp.  $dg_+\Ban\Alg_{R}$) for the category of  unital commutative (resp. associative) differential graded Banach ${R}$-algebras concentrated  in non-negative homological degrees, in which morphisms are ${R}$-CDGA morphisms $A_{\bt} \to B_{\bt}$ which are bounded in each degree.
\end{definition}

\begin{definition} %%we should call this locally multiplicatively convex (LMC) as \in \cite[Ch.3]{nachbin}, but beware they work with commutative only
  As in \cite{NachbinCh3}, say that a topological $R$-algebra $A$ is locally  multiplicatively convex (LMC) if its topology is given by a family of seminorms over $R$ which are submultiplicative in the sense that 
 \[
  \|ab\|\le \|a\|.\|b\|
 \]
for all $a,b \in A$ and all seminorms $\|-\|$ in the family. We say that $A$ is moreover complete if it is isomorphic to the limit of the Banach $R$-algebras given by completion with respect to these seminorms.
\end{definition}

Equivalently, this amounts to saying that $A$ is a filtered inverse limit of Banach $R$-algebras for which all the morphisms from $A$ are dense.
%% By \cite[(3.3.2) Lemma [ARE 11]]{nachbin}, the map from the limit is then dense when the system is countable.

\begin{remark}
 Whenever we refer to objects as complete, this is understood to mean that they are also  Hausdorff.
\end{remark}

\begin{definition}\label{Tcalgdef} %%%%%is coho completeness as in \cite{prosmans} the key issue???
Write 
$\hat{\Tc}\CAlg_{R}$
for the category of unital commutative complete  LMC $R$-algebras, with continuous  morphisms, and $\hat{\Tc}\Alg_{R}$ for the category of unital associative complete  LMC   algebras over ${R}$, with continuous  morphisms.
\end{definition}

Note that the system of seminorms on each object, together with standard properties of locally convex topological vector spaces as in e.g. \cite[\S 1]{prosmans} mean that we have full and faithful functors
\[
 \hat{\Tc}\CAlg_{R} \to \pro(\Ban\CAlg_{R}) \quad \quad \hat{\Tc}\Alg_{R} \to \pro(\Ban\Alg_{R})
\]
given by sending a topological algebra to the filtered system of completions with respect to continuous submultiplicative seminorms. The essential image consists of filtered systems in which all the transition morphisms are dense, together with any pro-objects isomorphic to such.

\begin{definition}
 Define a unital associative  differential graded  complete  LDMC  ${R}$-algebra $A_{\bt}$ to be a chain complex of complete locally convex topological ${R}$-vector spaces, with the differentials $\delta\co A_n \to A_{n-1}$ being continuous operators, together with a unit $1 \in \z_0A$ and a continuous ${R}$-bilinear multiplication $A_i \by A_j \to A_{i+j}$, such that the topology is given by a system of submultiplicative  seminorms $\|-\|$ with $\|1\|=1$ and $\|\delta a\|\le \|a\|$. We require that the multiplication be
 associative and that $\delta$ be a derivation with respect to it.  We say that $A_{\bt}$ is commutative if the multiplication is graded-commutative. 
\end{definition}

\begin{definition}
Write $dg_+\hat{\Tc}\CAlg_{R}$ (resp.  $dg_+\hat{\Tc}\Alg_{R}$) for the category of  unital commutative (resp. associative) differential graded complete  LDMC ${R}$-algebras concentrated  in non-negative homological degrees, in which morphisms are dg ${R}$-algebra morphisms $A_{\bt} \to B_{\bt}$ which are continuous in each degree.

Given $A \in dg_+\hat{\Tc}\CAlg_{R}$, we write $dg_+\hat{\Tc}\CAlg_{A}$ for the slice category $A \da dg_+\hat{\Tc}\CAlg_{R}$ of objects over $A$ (i.e. complete  LDMC commutative dg $A$-algebras), and $dg_+\hat{\Tc}\Alg_{A}$ for the category of pairs $(B,f)$ for $B \in dg_+\hat{\Tc}\Alg_{R}$ and $f \co A \to B$ a morphism such $f(A)$ lies in the centre of $B$ (i.e. $B$ is a complete  LDMC associative dg $A$-algebra).
\end{definition}

By completing with respect to each submultiplicative seminorm, we can identify $dg_+\hat{\Tc}\CAlg_{R}$ and  $dg_+\hat{\Tc}\Alg_{R}$ with the full subcategories of the pro-categories $\pro(dg_+\Ban\CAlg_{R}) $ and $\pro(dg_+\Ban\Alg_{R})$ of pro-Banach dg algebras consisting of those filtered systems in which all transition morphisms are dense.

\begin{lemma}\label{EFCforgetlemma}
 There is a natural forgetful functor $U \co dg_+\hat{\Tc}\CAlg_{R} \to  dg_+\EFC\Alg_{R}$, sending a topological dg  algebra $A_{\bt}$ to a dg EFC algebra with the same underlying dg algebra. 
\end{lemma}
\begin{proof}
 Given  $A_{\bt} \in dg_+\hat{\Tc}\CAlg_{R}$, we can consider the underlying CDGA and then it suffices to show that $A_0$ is naturally an EFC algebra. Since $A_0$ is complete and LMC, this is well-known: we can simply evaluate expressions of the form $\sum_{m_1, \ldots, m_n=0}^{\infty} \lambda_{m_1, \ldots, m_n} a_1^{m_1}\cdots a_n^{m_n}$ in $A_0$ by taking limits with respect to the family of submultiplicative seminorms,  with convergence guaranteed whenever  $\lim_{\sum m_i \to \infty } |\lambda_{m_1, \ldots, m_n}|^{1/\sum m_i}= 0$.
\end{proof}

We then have:

\begin{proposition}\label{keyBanEFCprop}
 The forgetful functor $U \co \pro(dg_+\hat{\Tc}\CAlg_{R}) \to  dg_+\EFC\Alg_{R}$ has a left adjoint, and the unit of the adjunction gives isomorphisms on cofibrant objects in the category $dg_+\EFC\Alg_{R}$. %%proof now gives a more general description of $F$, which'll do.
\end{proposition}
\begin{proof}
%The functor $U$ clearly preserves all limits.
The left adjoint functor $F$ is most easily described on objects $C$ which are quasi-free in the sense that the  underlying graded EFC-algebra is freely generated, say by generators $S_n$ in degree $n$. Then any function $f \co S \to \R_{\ge 0}$ gives rise to a submultiplicative seminorm $\|-\|_f$ on $C$ by setting
\[
 \|\sum_{m_1, \ldots, m_r} \lambda_{m_1, \ldots, m_r} s_1^{m_1}\cdots s_r^{m_r}\|_f:= \sum_{m_1, \ldots, m_r}| \lambda_{m_1, \ldots, m_r}| f(s_1)^{m_1}\cdots f(s_r)^{m_r}
\]
for $s_1, \ldots, s_r \in S$, where $m_i \in \N_0$ when $s_i$ has even degree and $m_i \in \{0,1\}$ when $s_i$ has odd degree. 
We can  form a dg Banach algebra $\hat{C}_f$ given by completing $C$ with respect to the seminorm $\|-\|_f$, and then $FC$ is given by the filtered limit $\Lim_f\hat{C}_f$ of dg Banach algebras. 

To see this, observe that if we forget the differentials, then for any $B \in dg_+\Ban\CAlg_{R}$, the set  of graded EFC-algebra morphisms $C_{\#} \to UB_{\#}$ is isomorphic to $\prod_n B_n^{S_n}$, as is the the set of graded continuous algebra morphisms $(\Lim_f\hat{C}_f)_{\#} \to B_{\#}$, by the
universal property of  $\ell^1$ norms. The intertwining relations on the respective differentials then match up to give
 $ %\begin{align*}
 \Hom_{dg_+\hat{\Tc}\CAlg_{R}}(FC, B) \cong \Hom_{ dg_+\EFC\Alg_{R}}(C,UB).
$ %\end{align*}

% It follows from the universal property of  $\ell^1$ norms that for any $B \in dg_+\Ban\CAlg_{R}$, we then have
% \begin{align*}
%  \Hom_{dg_+\hat{\Tc}\CAlg_{R}}(FC, B) &\cong \Hom_{ dg_+\EFC\Alg_{R}}(C,UB)\\
%  &\subset \prod_n B_n^{S_n}, %%need to explain that's map of underlying  graded algebras? 
% \end{align*}
% where the inclusion map is given by evaluation on generators $S \subset C$. 

Now, given an arbitrary dg EFC algebra $A$, we can write it as a quotient $C/I$ of a quasi-free dg EFC algebra $C$  by a dg ideal $I$.  For a  submultiplicative seminorm $\|-\|_f$ on $C$ as above, we then consider the induced seminorm on $A=C/I$, and let $\hat{A}_f$ be given by completing  with respect to it.  Then $FA$ is the filtered limit $\Lim_f\hat{A}_f$ of dg Banach algebras, satisfying
\[
  \Hom_{dg_+\hat{\Tc}\CAlg_{R}}(FA, B) \cong \Hom_{ dg_+\EFC\Alg_{R}}(A,UB)
\]
for all dg Banach algebras $B$, and hence for all $B \in dg_+\hat{\Tc}\CAlg_{R}$, since $U$ preserves limits and every object of $dg_+\hat{\Tc}\CAlg_{R}$ is a filtered limit of dg Banach algebras. Equivalently, $FA$ is the quotient of $FC$ by the closure of $I$.

It remains to show that the unit $C \to UFC$ of the adjunction is an isomorphism on cofibrant objects. Every cofibrant object is a retract of a quasi-free object, so it suffices to show this when $C$ is quasi-free, with generating set $S$, say.

Arguing as in  \cite[Proposition 1.1.4]{prosmans}, it follows that any element $\sum_{\alpha} \lambda_{\alpha}$ of $UFC = \Lim_f \hat{C}_f$ involves only finitely many variables in $S$. Explicitly,  we could otherwise choose a sequence $\{\alpha_n\}_n$ of words with $\lambda_{\alpha_n} \ne 0$ such that $\alpha_n$ involves at least one variable not featuring in $\{\alpha_1, \ldots, \alpha_{n-1}\}$, and we can then construct $f$ inductively so that $f(\alpha_n)= |\lambda_{\alpha_n}|^{-1}$, so $\|\sum_{n}\lambda_{\alpha_n}\alpha_n\|_f=\infty$.
In particular, we have
\begin{align*}
 UFC_0 ~=~ \LLim_{\substack{T\subset S_0 \\ \text{finite}}} (
 %\Lim_{f\co T \to \R_{>0}}\ell^1( \prod_{t \in T} (t/f(t))^{m_t} ~:~ m \in \N_0^T))
\Lim_{r>0}\ell^1( \prod_{t \in T} (t/r)^{m_t} ~:~ m \in \N_0^T))
 ~\cong~ C_0,
\end{align*}
the free EFC algebra generated by $S_0$. Moreover, only finitely many  multiplicities of variables in $S$ of positive degree can feature in an element of $UFC$. Thus $UFC$ is the free graded $(UFC)_0$-algebra generated by $S_{>0}$, 
so the map $C \to UFC$ is indeed an isomorphism.  
\end{proof}

\begin{remarks}
 The monad of Proposition \ref{keyBanEFCprop} is certainly not the identity on arbitrary EFC $R$-algebras, since for instance any ideal in an EFC $R$-algebra is an EFC ideal, but few are closed.

Our reason for working with EFC (and later FEFC) algebras instead of Fr\'echet algebras is that the former are essentially algebraic in nature from a categorical perspective, so are much more amenable to homotopical algebra.
\end{remarks}

\begin{lemma}\label{inftyequivlemma} 
Assume we have categories $\C, \cD$, a functor $U \co \cD \to \C$ with left adjoint $F$,  and a model structure on $\C$ such that the unit $\id \to UF$ of the adjunction is a weak equivalence when evaluated on any object in the image of a functorial cofibrant replacement functor $Q$. Then the functor $U$ induces an equivalence of the $\infty$-categories given by simplicial localisation at weak equivalences in $\C$ and  $U$-weak equivalences in $\cD$. The same is true if we replace $\cD$ with any full subcategory $\cD'$ containing the image of $FQ$.
\end{lemma}
\begin{proof}
% Let $\C' \subset \C$ be the full subcategory on objects in the image of $Q$. The natural weak equivlaence  $\rho \co Q \to \id$ ensures that $Q$ preserves weak equivalences and is $\infty$-quasi-inverse to the inclusion functor $\C' \into \C$, so the simplicial localisations of $\C$ and $\C'$ are equivalent.

It suffices to prove that the functor $FQ \co \C \to \cD'$ gives an $\infty$-quasi-inverse to  $U \co \cD' \to \C$ on simplicial localisation. 
By hypothesis, for any $A \in \C$ %the unit $\eta_{QA} \co QA \to UFQA$ of the adjunction is a weak equivalence, so 
we have a natural zigzag
of weak equivalences
\[
 A \xla{\rho_A} QA \xra{ \eta_{QA}}  UFQA,
\]
so $U \circ (FQ) \simeq \id$. In particular, this means that $FQ$ preserves weak equivalences and is homotopy right inverse to $U$.

For $B \in \cD'$, the co-unit of the adjunction induces a natural map  $\gamma_B$ as the composite
\[
 FQUB \xra{F\rho_{UB}} FUB \xra{\vareps_B} B 
\]
which it then suffices to show is a $U$-weak equivalence (i.e. that $U\gamma_B$ is a weak equivalence). Now by naturality of the unit, we have $\eta_{UB} \circ \rho_{UB} = UF(\rho_{UB}) \circ \eta_{QUB} \co QUB \to UFUB$. Thus 
\[
 U(\gamma_B)\circ \eta_{QUB} = U\vareps_B \circ \eta_{UB} \circ \rho_{UB}= \rho_{UB},
\]
so $U(\gamma_B)$ is a weak equivalence since $\rho_{UB}$ and $\eta_{QUB}$ are so by hypothesis. 
\end{proof}

\begin{corollary}\label{keyBanEFCcor}
The  forgetful functor $U \co dg_+\hat{\Tc}\CAlg_{R} \to  dg_+\EFC\Alg_{R}$ induces an equivalence of the $\infty$-categories given by simplicial localisation at quasi-isomorphisms, i.e.  morphisms $A \to B$ for which the maps  $\H_*A \to \H_*B$ are isomorphisms of abstract vector spaces. The same is true if we replace $dg_+\hat{\Tc}\CAlg_{R}$ with any full subcategory containing all objects of the form $(\sO^{\an}(\bA^{S_0}_{R})[S_+],\delta)$ for all graded (possibly infinite) sets $S$,
where
\[
 \sO^{\an}(\bA^{S_0}_{R})=\LLim_{\substack{T\subset S \\ \text{finite}}} \sO^{\an}(\bA^{T}_{R}).
\]
\end{corollary}
\begin{proof}
The first statement is given  by combining  Proposition \ref{stdmodelprop} with Lemma \ref{inftyequivlemma}, since quasi-isomorphisms are weak equivalences in the model structure of the former. For the second statement, it suffices to show that we have functorial cofibrant replacement by quasi-free objects. This can be constructed in the standard way, choosing generators $S$ for $QA$ recursively by setting $S_n:= A_n\by_{\delta, \z_{n-1}A}\z_{n-1}QA$.
\end{proof}

\begin{remarks}
 From the perspective of moduli theory, one significant consequence of Corollary \ref{keyBanEFCcor} is that the derived analytic stacks of \cite{DStein} are determined by their associated simplicial functors on $dg_+\hat{\Tc}\CAlg_{R}$; 
 note that   \cite[Proposition \ref{DStein-DHomprop2}]{DStein} ensures that these derived analytic stacks agree with those of \cite{lurieDAG9}, Lurie's additional data being redundant. %%%absorb this into next subsection. / For more details, see \S \ref{derivedanalyticstacksn}.
 
The corollary also facilitates comparison with the approach towards derived analytic geometry proposed in papers such as \cite{BambozziBenBassatKremnizer}. From this perspective, the striking aspect of Corollary \ref{keyBanEFCcor} is the invariance under abstract quasi-isomorphisms; informally, this says that the derived analytic stacks of \cite{DStein} are not inaccessibly far from being algebraic. It does mean that our setup is very well suited to describing analytic morphisms between algebraic objects, and also gluing algebraic objects in an analytic way. 

The more usual notion of equivalence for complexes of topological vector spaces is given by the notion of strictly exact complexes as in \cite[Corollary 1.1.8]{prosmans}. However, by \cite[Corollary 3.4.3]{prosmans} a morphism of Fr\'echet spaces is strict if it has closed image, %%that's thanks to Banach--Schnauder
so any complex of Fr\'echet spaces whose abstract vector space homology is zero must also be strictly acyclic. We will exploit this phenomenon in \S\ref{commnuclearFrechetsn} to produce natural subcategories of Corollary \ref{keyBanEFCcor} on which abstract quasi-isomorphism invariance is often more easily satisfied.

Finally, note that we can apply  \cite[Theorem 11.3.2]{Hirschhorn} to induce a model structure on $ dg_+\hat{\Tc}\CAlg_{R}$ in which morphisms are weak equivalences or fibrations whenever their images under $U$ are so. This follows because $U$ preserves filtered colimits and for any  generating trivial cofibration $u$ in $dg_+\EFC\Alg_{R}$, pushouts of $Fu$ are $U$-weak equivalences.
\end{remarks}

\subsubsection{Nuclear Fr\'echet algebras}\label{commnuclearFrechetsn} %%%hqve just copied from FEFC section here 19/12/23. We might want to cut this, though.

Defining quasi-isomorphism-preserving functors on  $dg_+\hat{\Tc}\CAlg_{R}$ tends to be difficult unless they forget the topological structure. However, there are many such functors on the subcategory of nuclear Fr\'echet algebras.
We now investigate a smaller subcategory %of $dg_+\hat{\Tc}\CAlg_{R}$  
whose homotopy theory is governed by the FEFC structure.

% \begin{definition} %%%not sure we want this.
% Define  $dg_+\hat{\Tc}\CAlg_{R}^{\mathrm{afp,cof}} \subset  dg_+\hat{\Tc}\CAlg_{R}$ and $dg_+\EFC\Alg_{R}^{\mathrm{afp,cof}} \subset dg_+\EFC\Alg_{R}$ to be the full subcategories consisting of all retracts of quasi-free objects $(\cO^{\an}(\bA^S),\delta) $ which are levelwise finitely generated in the sense that the graded set $S$ has finitely many elements in each degree.
% 
% Define $dg_+\hat{\Tc}\CAlg_{R}^{\mathrm{fp,cof}}$ and  $dg_+\EFC\Alg_{R}^{\mathrm{fp,cof}}$ similarly, but for finite sets $S$.
% \end{definition}

%%%%{afpTcFEFCcor} mostly unique to nc setting (simplicial model cat), so I've cut it.

The forgetful functor from EFC-DGAs to non-negatively graded sets has left adjoint $S \mapsto (\cO^{\an}(\bA^{S_0 \sqcup dS_1})[S_{\ge 1}, dS_{\ge 2}],\delta)$, where $dS_i$ is an isomorphic copy of $S_i$ a degree lower, and the differential $\delta$ is determined by the property $\delta s = ds \in dS$ for $s\in S$.
Since the resulting monad %$S \mapsto \sO(\Cx^{S})$ 
preserves filtered colimits, we immediately have: 
\begin{lemma}\label{indFPdgEFC}
 The category of EFC-DGAs over $R$ is equivalent to the category of ind-objects of the category of finitely presented EFC-DGAs, i.e. those of the form $(\cO^{\an}(\bA^{S_0}[S_+]), \delta)/I$ for finite graded sets $S$ and finitely generated dg ideals $I$.
 \end{lemma}
 
% In the following lemma, nuclear modules are taken in the sense of \cite[Definition 1.2]{freitag}. 
 
\begin{lemma}\label{nucfrechetcomm} 
Taking a base field $\bK$ to be either  $\R,\Cx$ or a complete valued non-Archimedean field, every levelwise finitely presented EFC-DGA $A$ over $\bK$ has a canonically associated object of $dg_+\hat{\Tc}\CAlg_{\bK}$ giving rise to the EFC structure. The underlying chain complex is a complex of nuclear Fr\'echet spaces whose differentials have closed images.
\end{lemma}
\begin{proof}
The functor is given by the left adjoint $F$ of Proposition \ref{keyBanEFCprop}, which sends an EFC-DGA $A$  of the form $(\cO^{\an}(\bA^S), \delta)/I$ to $(\cO^{\an}(\bA^S), \delta)/\bar{I}$ when $S$ is a levelwise finite graded set, where $\bar{I}$ denotes the closure of $I$. Elements of $\cO^{\an}(\bA^S)$ are generated by words involving only finitely many variables in $S$ (those of the same degree and lower), so  $\cO^{\an}(\bA^S)$ is defined in each degree as a K\"othe sequence space, making it nuclear Fr\'echet.

When the dg  ideal  $I$ is levelwise finitely generated, each module $I_n$ is finitely generated as an $ \cO^{\an}(\bA^{S_0})$-module (by  products of generators of $I$ and monomials in  $S_+$ with the correct total degree). When $\bK=\Cx$,  \cite[V.6, Corollary 2]{GrauertRemmertStein} implies that any finitely generated ideal in $\cO^{\an}(\bA^n)$ is closed, which immediately implies the same result for $\bK=\R$; the same is true for non-Archimedean fields by similar arguments, substituting \cite[\S 5.2.7]{BoschGuentzerRemmertNonArchanalysis} in the proof. Likewise, the submodules $\delta(A_n) \subset A_{n-1}$ are finitely generated $A_0$-submodules, so are also closed.
\end{proof}

The following lemma is key to the good behaviour of many functors on such algebras.

% the quasi-dagger dg affinoids of \S \ref{qdaggersn}, where Noetherianity makes all submodules of  finite modules closed, and objects are filtered colimits of Stein algebras, which are nuclear. In the non-commutative setting, we use it as an intermediate result.

\begin{lemma}\label{nucexactlemma1} %%%%provisionally moved earlier %%%%shd we break out separate statmeent saying algebraically eact sequences of Fr\'echet are exact? Does Prosmans sya that anywhere, and if not, why not?
 Let $\bK$ be a complete valued field, $U$ a Fr\'echet space over  $\bK$, and $V_{\bt}$ a chain complex of  Fr\'echet $\bK$-spaces whose differentials $\delta$ have closed images $\b_nV \subset V_n$. Assume that either $U$ is nuclear of that the spaces $V_n$ are all nuclear. Then the completed (injective or equivalently projective) tensor product satisfies  $ \H_*( U\hten_K V_{\bt}) \cong U\hten \H_*(V)$. 
\end{lemma}
\begin{proof}
 Since the subspaces $\b_nV \subset V_n$ are closed, the exact sequences 
 \begin{align*}
   0 \to \z_{n+1}V \to V_{n+1} \xra{\delta} \b_nV \to 0\\
   0 \to \b_nV \to \z_nV \to \H_nV \to 0
 \end{align*}
consist of maps which are strict (also known as admissible or homomorphisms) in the sense that the bijections from their coimages to their images are homeomorphisms; for the injections this is immediate, while for the surjections it follows from the Banach--Schnauder theorem. %%\cite[Corollary 3.4.3]{prosmans} says morphism of Fr\'echet spaces is strict iff closed image.
When spaces $V_n$ are nuclear, all these objects are  nuclear Fr\'echet spaces since they are constructed as closed subspaces and as quotients by such. As for instance in \cite[Lemma 2.4]{bertelson},
applying $U\hten$ preserves exactness for such sequences, so 
\[
 \H_n(U\hten V) = \frac{\z_n(U\hten V)}{\b_n(U\hten V)} \cong \frac{U\hten \z_nV}{U\hten \b_nV}\cong U\hten \H_nV.
\]
\end{proof}

\begin{remark}\label{ultrametricrmk} %%we might well want this to be a section in its own right, because this category does behave nicely in terms of homogeneity etc. (but then so does nuc, so it's more an issue of how well $\Hom(S,-)$ behaves.
 In the non-Archimedean setting, a similar statement holds for ultrametric Banach spaces and the completed ultrametric tensor product, %%%%%% I've said ultrametric, since I don't think it's projective, having max rather than sum.
 by \cite[Theorem 1]{grusonThFredholmPadic}.  Every nuclear Fr\'echet space can be written as an inverse limit of $\ell^{\infty}$ spaces, which are ultrametric in non-Archimedean settings. Since countable  inverse limits of bounded dense Banach space morphisms   are exact by %Lemma \ref{denselimlemma}, 
 \cite[Proposition 1.2.9]{prosmans}
 that gives an alternative proof of Lemma \ref{nucexactlemma1} in that setting, and allows us to relax the nuclearity condition to a requirement that the topology be defined by a system of ultrametric seminorms.
\end{remark}

\begin{example}\label{dgsteinex} 
As in \cite[V.6]{GrauertRemmertStein}, the vector space $\Gamma(X,\sF)$  of global sections of any coherent sheaf $\sF$ on  a complex  Stein space $X$ is naturally endowed  with a Fr\'echet space structure, and since $\cO^{\an}(\Cx^n)$ is a nuclear space, it follows by construction that $\Gamma(X,\sF)$ is also so. Lemma \ref{nucexactlemma1} thus implies that if $U$ is a nuclear Fr\'echet space and $\sF_{\bt}$ a complex of coherent sheaves on $X$, then we have $\H_*\Gamma(X, U\hten \sF_{\bt}) \cong U\hten \H_*\Gamma(X, \sF_{\bt})$.

Moreover, if $U \Subset X$ (in the sense that for any holomorphic map $f \co X \to \bA^n$, the image of $U$ is contained in some polydisc), then $\Gamma(U, \sF)$ is a finitely generated $\Gamma(U, \sO_X)$-module; this follows because there must be a compact Stein space \cite[Proposition 11.9.2]{TaylorCV} or dagger affinoid $K$ with $U \subset K \subset X$, and then the ring  $\Gamma(K, \sO_X)$ of overconvergent functions is Noetherian. Arguing as in \cite[Remark \ref{DStein-embeddingrmk}]{DStein}, it follows that for any dg Stein space $(X^0, \sO_X)$ (i.e. $X^0$ a Stein space, $\sO_X= \sO_{X, \ge 0}$ a CDGA of coherent $\sO_{X^0}$-modules with $\sO_{X,0}=\sO_{X^0}$) and $U \Subset X^0$, %%should ensure finite embedding dim
the topological CDGA $\Gamma(U, \sO_X)$ corresponds to a levelwise finitely presented EFC-DGA as in Lemma \ref{nucfrechetcomm}; in particular, $X$ admits an open cover by such objects.
\end{example}

The following is an immediate consequence of Lemmas \ref{nucfrechetcomm} and \ref{nucexactlemma1}.
\begin{corollary}\label{exactEFCtencor}
Under the hypotheses of Lemma \ref{nucfrechetcomm} and for any Fr\'echet space $U$, the completed tensor product $U\hten - $ defines a functor on the  category of levelwise finitely presented EFC-DGAs over $\bK$,  satisfying $\H_i(U\hten A) \cong U \hten \H_iA$.
 \end{corollary}

Lemma \ref{nucexactlemma1} leads to the  following lemma without strictness hypotheses, which  is a source of many weak equivalence preserving functors on categories of  dg FEFC algebras of \S \ref{dgFEFCsn}. %%in the next section.

\begin{lemma}\label{nucexactlemma} %%(provisionally?) moved earlier
 If $\bK$ is a complete valued field and $U$ a Fr\'echet space over  $\bK$, then the completed (injective or equivalently projective) tensor product functor $V_{\bt} \mapsto U\hten_{\bK} V_{\bt}$ on the category of chain complexes of nuclear Fr\'echet $\bK$-spaces preserves quasi-isomorphisms. When $U$ is also nuclear, the same is true on the category of chain complexes of Fr\'echet $\bK$-spaces.
\end{lemma}
\begin{proof}
 A morphism $f \co V_{\bt} \to V_{\bt}'$ is a quasi-isomorphism if and only if its cone has trivial homology, so it suffices to prove that the functor preserves the latter property. A complex $V_{\bt}$ has $\H_*V_{\bt} \cong 0$ if and only if the maps $\delta \co V_{n+1} \to \z_nV$ are all surjective. Since $\b_nV = \z_nV \subset V_n$ is  a closed subspace,   the result then follows from Lemma \ref{nucexactlemma1}.
\end{proof}

\begin{remark}
Beware that Lemma \ref{nucexactlemma} does not claim that $\H_*(U\hten V)$ can be deduced from the spaces $U$ and $\H_*V$, which might not be separated in general. Statements of that form only hold for strict chain complexes as in  Lemma \ref{nucexactlemma1}.
\end{remark}

% \subsection{Cotangent complexes} %%%%16/11/23 I feel this beongs here: reduction to algebraic case. could scrap??
 
\subsection{Relation to derived dagger analytic stacks}\label{derivedanalyticstacksn}\label{affinoidsn}

\subsubsection{Quasi-dagger affinoids}\label{qdaggersn}
% 11/11/23: here we should introduce full $\infty$-subcats of dagger affinoids, just by quoting {PTLag}. 
%%%Important to note that $\Lim B(i)$ shouldn't then be identified with the space $\LLim_i \oR \Spec B(i)$, since they only take the same value on strongly quasi-compact drived Artin stacks. Thus functor on the whole caty ncludes compact spaces with generic points

%%%for NC version, can we at least say that dagger affionoids are submersive (cot conc in deg 0)? Comm argt can say that for flat immersion $U \to X$, $U = U\by_XU \simeq U\by^h_XU$, so $\bL^{U/X} \simeq 0$ (pull back along diagonal). Is there a good NC analogue of fltness, so pushouts are homotopuy pushouts?

We briefly recall some theory from \cite{PTLag}. The following %is \cite[Definition \ref{PTLag-genWashdef}]{PTLag}, 
slightly generalises the dagger algebras of \cite[\S 1]{GrosseKloenne}.

\begin{definition}
 Define an  quasi-dagger algebra to be a quotient algebra $A=K\<\frac{x_1}{r_1}, \ldots ,\frac{x_n}{r_n}\>^{\dagger}/I$ for some $r_i \ge 0$, where $K\<\frac{x_1}{r_1}, \ldots ,\frac{x_n}{r_n}\>^{\dagger}$ consists of power series $\sum_{\nu} a_{\nu}x^{\nu}$ for $a_{\nu} \in K$, such that 
\[
 |a_{\nu}|\rho_1^{\nu_1}\ldots \rho_n^{\nu_n} \xra{|\nu| \to \infty} 0
\]
for some $\rho_i>r_i$.

 Define the associated affinoid quasi-dagger space $\Sp(A)$ to be the set of maximal ideals of $A$, equipped with the obvious structure sheaf on open localised affinoid subdomains. 
 
 The category of  quasi-dagger algebras is then defined by letting morphisms be all $K$-algebra homomorphisms between   quasi-dagger algebras, and the category of affinoid quasi-dagger spaces is it opposite.
 \end{definition}
 In other words, affinoid quasi-dagger spaces $(X,\sO_X)$ are ringed spaces of  the form $(\bar{X},i^{-1}\sO_Y)$ for closed immersions $i \co \bar{X} \to Y$  of affinoid dagger spaces.
 
 Note that we are allowing the numbers $r_i$ to be $0$. In particular, we allow the algebra $K\<x_1, \ldots , x_m, \frac{x_{m+1}}{0}, \ldots ,\frac{x_n}{0}\>^{\dagger}$, which can be thought of as $\Gamma(\bD^m, i^{-1}\sO_{\bD^n})$ for the inclusion $i \co \bD^m \into \bD^n$ of polydiscs, so we are looking at germs of overconvergent functions on  the dagger space $\bD^n$ restricted to $\bD^m$.

% The following is \cite[Definition \ref{PTLag-dgqdaggerdef}]{PTLag}: 
\begin{definition} \cite[Definition \ref{PTLag-dgqdaggerdef}]{PTLag}.
 Define an affinoid quasi-dagger dg space $X$ over $K$ to consist of an affinoid quasi-dagger space $X^0$ over $K$  together with an $\sO_{X^0}$-CDGA $\sO_{X,\ge 0}$ in coherent sheaves on $X^0$, with $\sO_{X,0}=\sO_{X^0}$. We then define an  quasi-dagger dg algebra $A$ over $K$ to be a dg $K$-algebra of the form $\Gamma(X^0,\sO_X)$ for an affinoid quasi-dagger dg space $X$; this is equivalent to saying that $A_0$ is a quasi-dagger algebra and the $A_0$-modules $A_m$ are all finite.
 
 We  say that an affinoid quasi-dagger dg space $X$ is a localised affinoid dagger dg space if  
 the vanishing locus $\pi^0X$  of $\delta$ is dagger affinoid and %%I've just added htis
  the closed immersion $i \co \pi^0X \to X^0$ %of the vanishing locus of $\delta$ 
  gives an isomorphism on the underlying sets of points. We then define a localised  dagger dg algebra $A$ over $K$ to be a dg $K$-algebra of the form $\Gamma(X^0,\sO_X)$ for a localised affinoid dagger dg space $X$.
 
A morphism $f \co X \to Y$ of affinoid quasi-dagger dg spaces %or of localised  affinoid dagger dg spaces 
consists of a morphism $f^0 \co X^0 \to Y^0$ of affinoid quasi-dagger spaces, together with a morphism $f^{\sharp} \co (f^0)^*\sO_Y\to \sO_X$ of CDGAs in coherent sheaves on $X^0$. Equivalently a morphism $A \to B$ of  quasi-dagger dg algebras %or of localised  dagger dg algebras 
is just a homomorphism of dg $K$-algebras.

We denote the category of localised dagger dg algebras by $ dg_+\Affd\Alg^{\loc,\dagger}_K$.
 \end{definition}

% The following is \cite[Proposition \ref{PTLag-affdsubEFCprop}]{PTLag}:
\begin{proposition}\cite[Proposition \ref{PTLag-affdsubEFCprop}]{PTLag}
 The functor from localised  dagger dg algebras to EFC-DGAs given by the natural direct limit topology induces a fully faithful functor on simplicial categories after simplicial localisation at quasi-isomorphisms, as does its restriction to quasi-free localised  dagger dg algebras.
\end{proposition}

%%% {analyticRH} content ends here

The following is a source of many functors on that localised $\infty$-category. \cite[Corollary \ref{PTLag-preservehtpycor}]{PTLag} can be interpreted as the case $\Hom_{\cts}(S,K)$ for $S$ pro-finite, giving derived moduli of pro-\'etale local systems by the method of Example \ref{CXex} below.

\begin{lemma}\label{daggertencolimlemma}
 Given a Fr\'echet space $V$ and a localised dg dagger algebra $A$, we have
 \[
  \H^*(A\hten_{\pi}V)\cong \H^*(A)\hten{\pi}V.
 \]
\end{lemma}
\begin{proof}
 We can write $A$ as a filtered colimit $\LLim_{\alpha} A(\alpha)$ of dense  dg Stein subalgebras, and Lemma \ref{nucexactlemma1} then gives $\H^*( A(\alpha)\hten_{\pi}V)\cong \H^*(A(\alpha))\hten_{\pi}V$. Now, $A$ is  complete with respect to the system of norms induced by the norms
 $\{\sum_{\alpha} \lambda_{\alpha}|-|_{\alpha}\}_{\uline{\lambda}\ge 0}$ on $\bigoplus_{\alpha} A(\alpha)$, so $\hten_{\pi}$ commutes with all filtered colimits here.
\end{proof}

\subsubsection{Analytification}\label{analytificnsn} %%11/11/23

In order to induce functors on EFC-DGAs from functors $F$ on CDGAs,  we can simply follow \cite[\S 4.4]{DStein} and define $F^{\EFC}$ to be $F$ composed with  the forgetful functor from EFC-DGAs to CDGAs. When $F$ is a derived Artin stack and we  restrict to localised dg dagger affinoid algebras $B$, the space $F^{\EFC}(B)$ is just the space of maps from the derived dagger affinoid space $\oR\Spec^{\an}B$ to the analytification $F^{\an}$ as defined in approaches such as 
\cite{HolsteinPortaDerivedAnMapping}.

However, if $B$ is a dg Stein algebra, our space $F^{\EFC}(B)$ only agrees with  the space of maps from the derived Stein space $\oR\Spec^{\an}B$ to $F^{\an}$ when $F$ is strongly quasi-compact. In general, if we write $F$ as a filtered colimit $F=\LLim_i F(i)$ of strongly quasi-compact  derived Artin stacks, then
\[
 F^{\EFC}(B) \simeq\LLim_i \oR\map(\oR\Spec^{\an}B, F(i)^{\an}). 
\]
In other words, this means that a dg Stein algebra in the category of EFC-DGAs behaves as a form of compactified Stein space as far as our functor $F^{\EFC}$ is concerned.

Of course, we could recover the usual space $\oR\map(\oR\Spec^{\an}B, F^{\an})$ by writing $B$ as an inverse limit  of dg dagger affinoid spaces $B(n)$ (given by pulling back $B$ to dagger affinoid subspaces of $\oR\Spec^{\an}B_0$), in which case
\[
 \oR\map(\oR\Spec^{\an}B, F^{\an}) \simeq \ho\Lim_n F^{\EFC}(B(n)).
 \]

In other words, in our setting the object associated to a dg Stein space $\oR\Spec^{\an}B $ is not the EFC-DGA underlying the dg Stein algebra $B$, but the pro-object $\{B(n)\}_n$, which better reflects non-compactness of Stein spaces (we might think of $B$ itself as the ring of functions on a compactification of $\oR\Spec^{\an}B $). That pro-object construction does not extend to
all of $\hat{\Tc}\CAlg_{R}$ in a manner which  respects quasi-isomorphisms, so cannot be used to produce an alternative analytification functor on EFC-DGAs for which Stein algebras correspond to Stein spaces. 

The analytification  functor $F^{\an}$ of \cite{HolsteinPortaDerivedAnMapping}  can thus be recovered from the restriction of $F^{\EFC}$ to dg dagger affinoid algebras. By the same reasoning it can be recovered from the subcategories of dg affinoid algebras  and of dg Banach algebras, since the inverse system $\{B(n)\}_n$ of dg dagger affinoids  is isomorphic in the pro-category to
 an inverse system of dg Banach algebras (given by completion with respect to multiplicative seminorms on $B_0$, corresponding to  compact Stein spaces in $\Spec^{\an}B_0$, whereas dagger affinoids correspond to overconvergent functions on such) and to 
an inverse system of dg Stein algebras  (corresponding to Stein subdomains  $U \Subset \Spec^{\an}B_0$  as in Example \ref{dgsteinex}).
In particular,
 \[
 \oR\map(\oR\Spec^{\an}B, F^{\an}) \simeq \holim_{\substack{\longleftarrow \\ U \Subset \oR\Spec^{\an}B_0}} F^{\an}(B\ten_{B_0}\cO^{\an}(U)). 
 \]

\section{Algebras with free entire functional calculus}\label{FEFCsn}

\subsection{FEFC-algebras}\label{FEFCalgsn}

When $R=\Cx$, the following definition is from \cite{taylorFnsNC}:
\begin{definition}\label{FSdef}
 Given a commutative Banach algebra $R$ and a finite set $S$, define the  algebra $\cF_S(R)$ of free entire functions in terms of the set $W(S)$ of words in  variables $S$ to be the subspace of $R^{ W(S)}$ given by
 \[
  \cF_S(R) = \{\sum_{\alpha \in W(S)} \lambda_{\alpha}x^{\alpha}~:~  \sum_{\alpha}\|\lambda_{\alpha}\|\rho^{|\alpha|} <\infty ~\forall \rho \in \R_{>0}\},
 \]
where $|\alpha|$ denotes the length of a word $\alpha$. We write $\cF_n(R):= \cF_{\{1,\ldots,n\}}(R)$.

Multiplication is defined by concatenation of words.

Given a  commutative  LMC Fr\'echet algebra $R$, i.e. a countable inverse limit $R=\Lim_i R_i$ of Banach algebras, set $\cF_n(R)=\Lim_i \cF_n(R_i)$.
\end{definition}

It thus follows that $\cF_n(R)$ has the structure of an LMC Fr\'echet $R$-algebra, defined by the system $\|-\|_{\rho}$ of seminorms $ \|\sum_{\alpha}\lambda_{\alpha}x^{\alpha}\|_{\rho}:= \sum_{\alpha}\|\lambda_{\alpha}\|\rho^{|\alpha|}$. It is the Arens--Michael envelope of the free associative algebra on $n$ variables.

\begin{definition}
Recall (cf. \cite{pirkovskiiHFG} when $R=\Cx$) that an $R$-algebra $B$ with free entire functional calculus (or FEFC $R$-algebra for short)  is a set $B$ is equipped, for every  function $f \in \cF_n(R)$, with an operation $\Phi_f \co B^n \to B$. These operations are required to be compatible in the sense that given functions $g_i \in \cF_{m_i}(R)$, we must have 
\[
 \Phi_{f \circ (g_1, \ldots, g_n)}= \Phi_f \circ (\Phi_{g_1}, \ldots, \Phi_{g_n}) \co B^{\sum_{i=1}^n m_i}\to B.  
\]
\end{definition}
Thus an FEFC $R$-algebra is an $R$-algebra $B$ with a systematic and consistent way of evaluating expressions of the form $\sum_{\alpha \in W(S) } \lambda_{\alpha} b^{\alpha}$ in $B$ whenever the coefficients $ \lambda_{\alpha}\in R$ satisfy $\lim_{|\alpha| \to \infty } |\lambda_{\alpha}|^{1/|\alpha|}= 0$. %% How many terms of length $r$? fewer than $n^r$. In commm case, coeff is (n+r-1,r)= ((n+r-1)\ldots r+1)/(n-1)! . Here, it's at most n! times that. Roughly $r^n$, so polynomial in $r$ 

\begin{remarks}\label{FEFCcatrmks}
 Another perspective is to look at the full subcategory of Fr\'echet algebras with objects $\cF_n(R)$,  so $\Hom( \cF_m(R),\cF_n(R))\cong \cF_n(R)^m$, and then 
%think of $\cF_n(R)$ as functions on  NC affine analytic space $\bA_R^{nc,n}$, so $\Hom(\bA_R^{nc,n},\bA_R^{nc,m}) \cong (\cF_n)^m$ with an obvious composition law, and then 
an FEFC $R$-algebra $B$ is a product-preserving functor $ \cF_n(R) \mapsto B^n$ from the opposite category to sets, since $\cF_{m+n}(R)$ is the coproduct of $\cF_m(R)$ and $\cF_n(R)$ in this category.

In particular, every FEFC $R$-algebra is an associative $R$-algebra, with  addition and multiplication coming from the functions $(x,y) \mapsto x+y$ and
  $(x,y) \mapsto xy$ in $\cF_2$, and the map $R \to B$ given by $\lambda \mapsto \Phi_{\lambda}(*)$ via the isomorphism $\cF_0 \cong R$.

From a categorical perspective, the forgetful functor from FEFC $R$-algebras to sets has a left adjoint, 
 which sends a set $S$ to the FEFC $R$-algebra
\[
\cF_S(R):=  \LLim_{\substack{T \subset S}\\ \text{finite}} \cF_T(R),
\]
and  FEFC $R$-algebras are algebras for the resulting monad structure on the functor $S \mapsto \cF_S(R)$. 

There is also a forgetful functor from 
%pro-Banach
LMC topological  $R$-algebras to FEFC $R$-algebras, sending $B$ to its underlying set 
% an inverse system $\{B_i\}_i$ to the underlying set $\Lim_i B_i$  
with the free entire functional calculus determined by convergence of the infinite sums with respect to the multiplicative norms. This functor has a left adjoint, determined by the property that for finite sets $S$, the FEFC $R$-algebra $\cF_S(R)$ maps to itself endowed with its natural %pro-Banach 
topological
$R$-algebra structure. Explicitly, an FEFC  $R$-algebra $B$ is sent to the system of completions of $B$ with respect to seminorms $\nu$ compatible with the FEFC structure in the sense that $\nu(\Phi_f(b_1, \ldots,b_n) \le \sum {\alpha}\|\lambda_{\alpha}\|\nu(b_{\alpha_1})\cdots \nu(b_{\alpha_m})$ when $f= \sum {\alpha}\lambda_{\alpha}x^{\alpha}$, where $\alpha= \alpha_1\cdots \alpha_m$. 

Since $\cF_n(R)$ is the pro-Banach algebra freely generated by $n$ elements, another way to interpret these adjunctions is to say that 
FEFC is the 
Lawvere theory giving the 
closest algebraic approximation to Banach $R$-algebras.
\end{remarks}

The category of FEFC $R$-algebras has all small limits and colimits. In particular, there is a coproduct $\coprod$ with the property that $\cF_{S \sqcup T} \cong \cF_S\coprod\cF_T$.  

%%will be helpful later if $\C^{\infty}(\R^S)\ten\C^{\infty}(\R^T) \to \C^{\infty}(\R^{S \sqcup T}) $ is flat.
% 

\begin{definition}
% Following \cite{pirkovskii}, we say that a Stein algebra is holomorphically finitely generated (HFG) if it arises as a quotient of $\sO(\Cx^n)$ by a closed ideal,  for some finite $n$. 
We say that an FEFC $R$-algebra $B$ is finitely presented if it arises as a coequaliser  of the form  $\cF_m(R) \implies \cF_n(R) \to B$.
\end{definition}

Since the monad $S \mapsto \cF_S$ preserves filtered colimits, we immediately have: 
\begin{lemma}\label{indFP}
 The category of FEFC $R$-algebras is equivalent to the category of ind-objects of the category of finitely presented FEFC $R$-algebras. In other words, every FEFC $R$-algebra can be written as a filtered colimit $\LLim_{\alpha}A(\alpha)$ of finitely presented FEFC $R$-algebras, and 
 \[
  \Hom(\LLim_{\beta}B(\beta), \LLim_{\alpha}A(\alpha)) \cong \Lim_{\beta} \LLim_{\alpha}\Hom(B(\beta), A(\alpha))
 \]
for filtered systems $\{A(\alpha)\}_{\alpha}, \{B(\beta)\}_{\beta}$ of finitely presented FEFC $R$-algebras. 
\end{lemma}

%%%%%%%%%%%%%%%%%%%%%%%%%%%  %%not sure why that barrier's there? end of analyticrH content??

The following is an almost tautological consequence of the definition of $\cF_n$ as a completion with respect to a system of $\ell^1$ norms:
\begin{lemma}\label{extendmaplemma1}
 An $R$-linear map $\theta$ from the free associative algebra $R\<x_1, \ldots , x_n\>$ to a Fr\'echet $R$-module $M$ extends to an $R$-linear map on $\cF_n$ if and only if for every continuous seminorm $\|-\|$ on $M$, there exist $C, r \in \R_{>0}$ such that 
 \[
  \|\theta(x_{i_1} \cdots x_{i_m})\| \le C r^m 
 \]
for every non-commutative monomial $x_{i_1} \cdots x_{i_m}$ in $R\<x_1, \ldots , x_n\>$.
 \end{lemma}

 \subsection{FEFC ideals, bimodules and derivations}
 
As observed in \cite{beckThesis} and \cite[\S 2]{Q},  for any algebraic theory there is an associated notion of modules  over an algebra $A$, given by abelian group objects in the slice category over $A$, and known as \emph{Beck modules}. For commutative rings (and indeed commutative EFC rings) $A$, these  correspond to $A$-modules $M$ via the functor $M \mapsto A \oplus M\eps$, with $\eps^2=0$. For associative rings $A$, they correspond to bimodules $M$, again via the functor  $M \mapsto A \oplus M\eps$. Beck and Quillen observed that the forgetful functor from that category of group objects to the slice category has a left adjoint $\Ab$, and that $\Ab(A)$ corresponds to the cotangent module of $A$ in the commutative and associative cases.

We now need to carry out this procedure for FEFC %, dg FEFC,  and stacky dg FEFC 
$R$-algebras, where the natural category of modules is somewhat harder to characterise. Any abelian group object over an FEFC-algebra $A$ has a section from $A$ corresponding to the unit, so must   take the form $A \oplus M$, with the group structure $(A\oplus M)\by_A(A\oplus M) \to A \oplus M$ given by $(a,x,y) \mapsto (a,x+y)$, for $a \in A$, $x,y \in M$, and we need to find the constraints on the FEFC operations on $A \oplus M$ which permit this group operation to be an FEFC-homomorphism. Considering the multiplication alone, we deduce that $M$ has an $A$-bimodule structure, but (in contrast to the case of commutative EFC algebras), this alone is not enough.

 \subsubsection{FEFC ideals}

 \begin{definition}
  Say that a subspace $I \subset A$ of an FEFC algebra $A$ is an FEFC ideal if the FEFC operations all descend to $A/I$.
 \end{definition}
 
 \begin{definition}\label{iotadef}
Given $c \in \cF_n(R)$, define $\iota_c$ to be the continuous $R$-linear map
 \begin{align*}
\iota_c \co  \cF_n(R)\hten_{R,\pi} \cF_n(R) &\to \cF_n(R)\\ 
a \ten b &\mapsto acb,
\end{align*}
where $\hten_{R,\pi}$ denotes the $R$-linear completed projective tensor product of Fr\'echet spaces. 
 \end{definition}

 The following can be thought of as  a non-commutative analogue of Hadamard's lemma.
 \begin{lemma}\label{NChadamardlemma}
  For any $f \in \cF_{n+1}(R)$, there exists $\psi(f) \in  \cF_{n+2}(R)\hten_{R,\pi} \cF_{n+2}(R)$ such that 
  \[
   f(x, z_1, \ldots, z_n) - f(y, z_1, \ldots, z_n)= \iota_{(x-y)}\psi(f)(x,y,z_1, \ldots,z_n). 
  \]
 \end{lemma}
\begin{proof}
For $a_i=a_i(z_1, \ldots, z_n) \in \cF_n(R)$ we have
\[
  a_0 x a_1 x\cdots xa_n -a_0ya_1y\cdots ya_n =  \sum_{i=1}^n a_0ya_1y \cdots ya_{i-1}(x-y)a_ixa_{i+1}x\cdots xa_n, 
 \]
so setting $\psi(a_0xa_1x\cdots xa_n):= \sum_{i=1}^n a_0ya_1y \cdots ya_{i-1}\ten a_ixa_{i+1}x\cdots xa_n$, we have $\iota_{(x-y)}\psi(f)(x,y, \uline{z})= f(x, \uline{z}) - f(y, \uline{z})$ for all monomials $f(x, z_1, \ldots, z_n)$. Now, $\|\psi( (x, \uline{z})^{\alpha})\|_r \le |\alpha|r^{|\alpha|-1}$, so for $f(t_0, \ldots, t_n):= \sum_{\alpha} \lambda_{\alpha}  \uline{t}^{\alpha}$ and extending $\psi$ linearly, we have (for any $\eps >0$)
\[
\|\psi(f)\|_r \le \sum_{\alpha} |\lambda_{\alpha}|\cdot |\alpha|r^{|\alpha|-1} \le  \eps^{-1}\sum_{\alpha}|\lambda_{\alpha}|\cdot |\alpha|(r+\eps)^{|\alpha|} = \eps^{-1}\|f\|_{r+\eps} < \infty. \qedhere
\]
%%Think about rephrasing {EFCexNC} --- how does that compare?
\end{proof}

In the commutative setting, all ideals $I$ of an EFC ring $B$ are EFC ideals in the sense that $B/I$ is EFC, but the situation for FEFC rings is less straightforward: 

\begin{lemma}\label{ideallemma}
Every FEFC ideal is a $2$-sided ideal with respect to the associative algebra structure, and in $\cF_n(R)$ a $2$-sided ideal $I$ is an FEFC ideal 
 if and only if for all $c \in I$ the map 
\begin{align*}
\iota_c \co  \cF_n(R)\hten_{R,\pi} \cF_n(R) &\to \cF_n(R) %\\ %%%make that $\iota$
% a \ten b &\mapsto acb
\end{align*}
has image contained in $I$. %, where $\hten_{R,\pi}$ denotes the $R$-linear completed projective tensor product of Fr\'echet spaces. 
\end{lemma}
\begin{proof}
Since associative multiplication is one of the FEFC operations, every FEFC ideal must be a $2$-sided ideal.

To see that the converse condition is necessary, observe that the map %there is a natural embedding 
\begin{align*}
\iota_z \co  \cF_{\{x_1, \ldots,x_n\}}(R)\hten_{R,\pi}\cF_{\{x_1, \ldots,x_n \}}(R) &\into \cF_{\{x_1, \ldots,x_n,z \}}(R)%\\
%  a\ten b &\mapsto azb.
\end{align*}
is injective. 
For any $f \in \cF_{\{x_1, \ldots,x_n,z \}}(R)$, the  map $\Phi_f\co \cF_n^{n+1} \to \cF_n(R)$ must then satisfy $\Phi_f(x_1, \ldots, x_n, c)- \Phi_f(x_1, \ldots, x_n, 0 )\in I$, so letting $f$ range over the image of $\iota_z$ gives $\iota_c(\cF_n(R)\hten_{R,\pi} \cF_n(R)) \subset I$.

To see that the condition is sufficient, we need to show that for all $f \in \cF_m(R)$, we have $\Phi_f(u_1+c_1, \ldots,u_m+c_m)-\Phi_f(u_1, \ldots, u_m) \in I$ whenever $u_i \in \cF_n(R)$ and $c_i \in I$. By altering variables separately and permuting inputs, this reduces  to showing that $\Phi_f(u_1+c,u_2, \ldots,u_m)-\Phi_f(u_1, \ldots, u_m) \in I$ for all $f$. By Lemma \ref{NChadamardlemma}, we have
\[
 \Phi_f(u_1+c,u_2, \ldots,u_m)-\Phi_f(u_1, \ldots, u_m)= \iota_c\Phi_{\psi(f)}(u_1+c,u_1, u_2, \ldots, u_m) \in I.\qedhere
\]
\end{proof}
 
\begin{example}\label{idealctrex}
 Consider the complete topological algebra $A$ given by the quotient of $\cF_{n+1}$ by the closure of the $2$-sided ideal generated by any term featuring the variable $x_{n+1}$ more than twice. Thus
 \[
  A \cong \cF_n \oplus (\cF_n\hten \cF_n) \oplus (\cF_n \hten \cF_n\hten \cF_n),
 \]
with the surjective homomorphism from $\cF_{n+1}$ induced by the map from $R\<x_1, \ldots, x_{n+1}\>$ which replaces any instance of $x_{n+1}$ with $\ten$.  
 
For $\cF_n$ acting on the left and $\cF_n^{\op}$ on the right, any $\cF_n\hten \cF_n^{\op}$-submodule $I$ of $ \cF_n\hten\cF_n\hten \cF_n$ %%in other words, the top level term
is then an FEFC ideal in $A$, with no topological restriction. Each such submodule gives rise to a square-zero extension $A/I$ of $\cF_n \oplus (\cF_n\hten \cF_n) $ in the category of FEFC algebras, but only the closed submodules give rise to natural extensions in the category of Fr\'echet algebras.

In general, for any $\cF_n\hten \cF_n^{\op}$-module $M$ equipped with a morphism $f$ from $\cF_n \hten \cF_n\hten \cF_n $, we have an FEFC algebra structure on  $\cF_n \oplus (\cF_n\hten \cF_n) \oplus M$ as a square-zero extension of $\cF_n \oplus (\cF_n\hten \cF_n)$, in which the product of elements $a\ten b$ and $a'\ten b'$ in $\cF_n\hten \cF_n$ is given by $f(a\ten ba' \ten b') \in M$.
\end{example}

\subsubsection{Bimodules}\label{bimodsn}

\begin{definition}
 We define the involution $A \mapsto A^{\op}$ of the category of  FEFC $R$-algebras as follows. There is a continuous homomorphism $(-)^* \co \cF_S^{\op} \to \cF_S$ given by reversing the order of letters in each word. We then define $A^{\op}$ to have the same elements as $A$, but with operations $\Phi_f(a_1^{\op}, \ldots, a_n^{\op}):= \Phi_{f^*}(a_1, \ldots,a_n)^{\op}$. 
\end{definition}

\begin{definition}\label{anAedef}
 Given  FEFC $R$-algebras $A,B$, we define the  dg FEFC $R$-algebra $A\ten^{\FEFC}B$ to represent the functor
 \[
  C \mapsto \{(\phi, \psi) \in \Hom(A,C)\by \Hom(B,C) ~:~ [\phi(A),\psi(B)]=0\}.
 \]

Explicitly, for finite $S,T$ we have $\cF_S\ten^{\FEFC}\cF_T= \cF_S\hten_{\pi,R} \cF_T$ (the completed projective tensor product), since this is the quotient of $\cF_{S \sqcup T}$ by the FEFC ideal generated by commutators $\{[s,t]~:~ s \in S,~t \in T\}$. We extend this to infinite sets $S,T$ by passing to filtered colimits. In general, for FEFC ideals $I \subset \cF_S$ and $J \subset \cF_T$, we set $(\cF_S/I)\ten^{\FEFC}(\cF_T/J)$ to be the quotient of $\cF_S\ten^{\FEFC}\cF_T$ by the FEFC ideal generated by $I\ten 1$ and $1\ten J$; explicitly, this is given by the quotient of the morphism
\begin{align*}
 \left(I\ten_R(\cF_S\hten_{\pi,R}\cF_S\hten_{\pi,R}\cF_T)\right) \oplus\left( (\cF_S\hten_{\pi,R}\cF_T\hten_{\pi,R}\cF_T)\ten J \right)&\to \cF_S\hten_{\pi,R} \cF_T\\
(x\ten a_1\ten a_1'\ten b_1, a_2\ten b_2\ten b_2'\ten y) &\mapsto a_1xa_1'b_1 + a_2b_2yb_2'. 
  \end{align*}

  We then define $A^{e,\FEFC}:= A\ten^{\FEFC}A^{\op}$.
  \end{definition}

\begin{lemma}\label{bimodlemma}
 For any FEFC $R$-algebra $A$, there is an equivalence of categories between $A^{e,\FEFC}$-modules and Beck modules of $A$. 
%  over  group objects in the slice category of FEFC $R$-algebras %(resp. dg FEFC $R$-algebras, resp. stacky dg FEFC $R$-algebras) over $A$. 
\end{lemma}
\begin{proof}
As discussed at the start of the subsection,
a Beck module is a module  of the form $A \oplus M$ equipped with an FEFC  structure for which the operations $\Phi_f$ have the property that $\Psi_f(\uline{a}, \uline{x}):=\Phi_f(a_1+x_1, a_n+x_n)- \Phi_f(a_1, \ldots, a_n)\in M$ and is additive in $\uline{x}$.  There is a forgetful functor from such pairs $(A, A \oplus M)$ to pairs $(A,M)$ of sets, which has a  left adjoint as follows.

Given finite sets $S,T$ with maps $S \to A$ and $T \to M$, we know that we have an  FEFC $R$-algebra homomorphism $\theta \co \cF_{S \sqcup T}(R) \to A \oplus M$ satisfying
$\theta(x_s) \in A$ and $\theta(x_t) \in M$ for all $s \in S$, $t \in T$. The additivity property above then implies that for an element $f =\sum_{\alpha \in W(S\sqcup T)} \lambda_{\alpha}x^{\alpha}$ of  $\cF_{S \sqcup T}(R)$, its image $\theta(f)$ depends only on the coefficients $\lambda_{\alpha}$ for words $\alpha$ containing at most one letter from $T$. In other words, $\theta$ factors uniquely through the quotient of $\cF_{S \sqcup T}(R)$ spanned by $\lambda_{\alpha}x^{\alpha}$ for $\alpha \in W(S) \sqcup W(S)TW(S)$.  This quotient is precisely the Fr\'echet algebra  $\cF_{S}(R) \oplus (\bigoplus_{t \in T} \cF_{S}(R)\hten_{R,\pi}\cF_{S}(R))$, where the right-hand term is a square-zero ideal. Passing to filtered colimits, we see that the left adjoint sends a pair $(S,T)$ to $(\cF_S, \bigoplus_{t \in T}\cF_S^{e,\FEFC})$. 

This adjunction satisfies the conditions of Beck's Monadicity Theorem, and it follows immediately from the resulting coequaliser descriptions that an algebra for the monad consists of pairs $(A,M)$ where $A$ is an FEFC algebra and $M$ an $A^{e,\FEFC}$-module.
\end{proof}

\subsubsection{Derivations}

\begin{definition}
Given an  FEFC $R$-algebra  $A$ and an $A^{e,\FEFC}$-module $M$, say that a map $f \co A \to M$ is an FEFC derivation if  $(\id,f) \co A \to A \oplus M$ 
is an FEFC algebra homomorphism, for the FEFC structure given by the proof of Lemma \ref{bimodlemma}. 
\end{definition}

Following the approach of \cite[\S 2]{Q}, we then say:
\begin{definition}\label{anOmegadef}
 Given an FEFC $R$-algebra  $A$, define  $\Omega^1_A$ to be the $A^{e,\FEFC}$-module  representing  the functor $M \mapsto \Hom_{\FEFC(R) }(A, A \oplus M)\by_{\Hom_{\FEFC(R)}(A,A)}\{\id\}$ of closed FEFC derivations from $A$ to $M$ of degree $0$. 
\end{definition}

 \begin{lemma}
For any FEFC algebra $A$, 
there is a natural surjective morphism $A^{e,\FEFC} \to A$ of $A^{e,\FEFC}$-modules whose kernel is isomorphic to $\Omega^1_A$.
 \end{lemma}
\begin{proof}
We have an FEFC derivation $A \to A^{e,\FEFC}$ given by $a \mapsto 1\ten a - a\ten 1$, and hence an  $A^{e,\FEFC}$-module homomorphism $\alpha \co \Omega^1_A \to A^{e,\FEFC}$. We also have a left $A$-module homomorphism $u \co A \to A^{e,\FEFC}$ given by $a \mapsto a\ten 1$.

 When $A= \cF_S$ for $S$ finite, the $A^{e,\FEFC}$-module structure on $A$ is just the $\cF_S\hten \cF_S^{\op}$-module structure induced by completing the bimodule structure, and the morphism $\mu \co \cF_S\hten \cF_S\to \cF_S$ is just induced by multiplication; this satisfies  $\rho \circ u=\id$  and $\rho \circ \alpha =0$. %%latter by universality, as induced derivation is $ a \mapsto 1da - ad1$.
 For the universal derivation $d \co A \to \Omega^1_A$, we also have a left $ \cF_S$-module homomorphism $ \rho \co \cF_S\hten \cF_S  \to \Omega^1_{\cF_S}$ given by completing the map $a\ten b \mapsto adb$, satisfying $\rho \circ \alpha =\id$ and $\rho \circ u =0$. Finally, note that $(\alpha \circ \rho)(a\ten b) = a\ten b - ab \ten 1 = (\id - u \circ \mu)(a \ten b)$. For $A= \cF_S$, we thus have a split exact sequence
 \[
  \xymatrix@1{ 0 \ar[r] & \Omega^1_A \ar@<.5ex>[r]^-{\alpha} & A^{e,\FEFC}  \ar@<.5ex>[r]^-{\mu} \ar@<.5ex>[l]^-{\rho} & A \ar[r] \ar@<.5ex>[l]^-{u} & 0.}
 \]

Since these constructions are all functorial with respect to FEFC homomorphisms, passing to colimits induces such a split exact sequence for all FEFC algebras $A$. 
\end{proof}

\subsection{FEFC-Differential graded algebras}\label{dgFEFCsn}

% \subsection{Definitions and basic properties}

\subsubsection{Graded FEFC algebras}

\begin{definition}
 Given a finite graded set $S$ (i.e. a finite set $S$ equipped with a degree function $\deg \co S \to \Z$), extend the degree function to the set $W(S)$ of words in $S$ by letting the degree of a word be the sum of the degrees of its elements. Write the resulting decompositions by degree as $S= \coprod_{i \in \Z} S_i$ and $W(S)= \coprod_{i \in \Z} W(S)_i$. %$S_i := \{s \in S~:~ \deg(s)=(i)$ and $W(S)_i:= \{\deg^{-1}(i)$.  
\end{definition}

\begin{definition}\label{freegradedandef}%%13/8: just look at this: I had the definitin right back then and forgot it, invalidating some counterexamples?
 Given a commutative Banach algebra $R$ and a finite graded set $S$, define the  graded algebra $\cF_S(R)$ of free entire functions  by letting $\cF_S(R)_i $  be the subspace of $R^{W(S)_i}$ given by
 \[
  \cF_S(R)_i = \{\sum_{\alpha \in W(S)_i} \lambda_{\alpha}x^{\alpha}~:~  \sum_{\alpha}\|\lambda_{\alpha}\|\rho^{|\alpha|} <\infty ~\forall \rho \in \R_{>0}\},
 \]
where $|\alpha|$ denotes the length of a word $\alpha$. Multiplication $\cF_S(R)_i \by \cF_S(R)_j \to \cF_S(R)_{i+j} $   is then  defined by concatenation of words. When the base ring $R$ is clear, we simply write $\cF_S$ for $\cF_S(R)$.
\end{definition}
For instance, this means that if $S$ has $n$ elements of degree $0$ and one of degree $1$, then $\cF_S(R)_i \cong  \underbrace{\cF_n(R)\hten_{R,\pi} \cF_n(R) \hten_{R,\pi}\ldots \hten_{R,\pi} \cF_n(R) }_{i+1} $. This stands in contrast to the commutative situation,  where the contribution of non-zero generators is purely algebraic if generators are all in non-negative degrees,  with the associated commutative EFC algebras given by $\cF_S(R)^{\comm}\cong  \cF_{S_0}(R)^{\comm}[S_+]$, but see Remark \ref{altcharrmk} below.%, where $\cF_{S_0}(R)^{\comm}$ is the ring  $\sO(R^{S_0})$ of analytic functions on $R^{S_0}$, given by replacing $W(S_0)$ with $\N_0^{S_0}$ in the definition of $\cF_{S_0}(R)$.

Observe that each $\cF_S(R)_i $ is naturally a pro-Banach space, making $\cF_S(R)$ an inverse limit of  graded Banach algebras, where we adopt the convention that a graded Banach algebra $B$ consists of Banach spaces $B_i$ with  a continuous associative multiplication $B_i \by B_j \to B_{i+j}$ and unit $1 \in B_0$ --- beware that $\bigoplus_i B_i$ will not itself usually be a Banach space.
Continuous graded algebra homomorphisms  from $\cF_T(R)$ to $\cF_S(R)$ are then just given by $\Hom(\cF_T(R),\cF_S(R))\cong \prod_{t \in T} \cF_S(R)_{\deg t}$, determined by the value on generators $T$.

\begin{definition}\label{gradedFEFCdef}
Define  a graded $R$-algebra $B$ with free entire functional calculus (or graded FEFC $R$-algebra for short)  to consist of  sets $\{B_i\}_{i \in \Z}$  equipped, for every  function $f \in \cF_S(R)_j$, with an operation $\Phi_f \co \prod_{s \in S} B_{\deg s} \to B_j$. These operations are required to be compatible in the sense that given graded sets $T(s)$, and elements $g_s \in \cF_{T(s)}(R)_{\deg s}$ for each $s \in S$, we must have 
\[
 \Phi_{f \circ (g_s)_{s \in S}}= \Phi_f \circ (\Phi_{g_s})_{s \in S} \co \prod_{s \in S} (B_{\deg s})^{T(s)} \to B_j.  
\]
\end{definition}
Thus a graded FEFC $R$-algebra is a graded  $R$-algebra $B$ with a systematic and consistent way of evaluating expressions of the form $\sum_{\alpha \in W(S)_j } \lambda_{\alpha} b^{\alpha}$ in $B_j$ whenever the coefficients $ \lambda_{\alpha}\in R$ satisfy $\lim_{|\alpha| \to \infty } |\lambda_{\alpha}|^{1/|\alpha|}= 0$. 

Another perspective is to take the full subcategory of graded Fr\'echet algebras on objects $\cF_S(R)$, and to  say that
a graded FEFC $R$-algebra $B$ is a product-preserving functor from the opposite category  to sets, given by $ \cF_S(R) \mapsto \prod_{s \in S} B_{\deg s}$.

The following is an almost tautological consequence of the definition of $\cF_S$ as a graded completion with respect to a system of $\ell^1$ norms:
\begin{lemma}\label{extendmaplemma2}
Given a finite graded set $S$, 
 an $R$-linear map $\theta$ from the free graded associative algebra $R\<S\>$ to a graded Fr\'echet $R$-module $M$ extends to an $R$-linear map on $\cF_S$ if and only if for every $n \in \Z$ and every continuous seminorm $\|-\|$ in a system jointly defining the topology on $M$, there exist $C, r \in \R_{>0}$ such that 
 \[
  \|\theta(x_{i_1} \cdots x_{i_m})\| \le C r^m 
 \]
for every non-commutative monomial $x_{i_1} \cdots x_{i_m}$ in $R\<S\>_n$.
 \end{lemma}

\subsubsection{dg FEFC algebras}
 
Now, we need to extend the definition to give differential graded (dg) FEFC $R$-algebras.
Given a graded set $S$, let the graded set $S \sqcup dS$ be a disjoint union of two copies of $S$, with grading $\deg_{S \sqcup dS}(s)=\deg_S(s)$ and $\deg_{S \sqcup dS}(ds)=\deg_S(s)-1$. We then have:

\begin{lemma}
The graded algebra $\cF_{S\sqcup  dS}(R)$ naturally has the structure of a dg LDMC Fr\'echet algebra, with derivation $\delta$ determined by $\delta x_s:= x_{ds}$ and $\delta x_{ds}=0$.
\end{lemma}
\begin{proof}
Although the submultiplicative norms $\|-\|_{\rho}$ define the topology,  $\delta$   is unbounded in these norms. We thus introduce an alternative system $\|-\|_{\rho}^{\sim}$ of $\ell^1$ norms, given by
 \[
 \|a\|_{\rho}^{\sim}:= \|a\|_{\rho} + \|\delta a \|_{\rho}.
 \]
 Then  $\|\delta a\|_{\rho}^{\sim} \le\|a\|_{\rho}^{\sim} $ since $\delta^2=0$, while 
 \[
 \|ab\|_{\rho}^{\sim}\le  \|a\|_{\rho} \| b\|_{\rho} + \|\delta a\|_{\rho}  \|b \|_{\rho} +   \|a\|_{\rho} \| \delta b\|_{\rho} \le \|a\|_{\rho}^{\sim}\|b\|_{\rho}^{\sim}
\]
and $\|1\|_{\rho}^{\sim}= \|1\|_{\rho}=1$.

It remains only to show that the systems $\{\|-\|_{\rho}\}_{\rho}$ and $\{\|-\|_{\rho}\}_{\rho}^{\sim}$ define the same topology, which in particular implies that $\delta$ converges on all elements in  $\cF_{S\sqcup  dS}(R)$. Since $\delta$ 
sends a word of length $n$ to a sum (with Koszul signs) of up to $n$ words of the same length, we have $\|\delta \alpha\|_{\rho} \le |\alpha|\rho^{|\alpha|}$. Because $2^n \ge n+1$ for all non-negative integers $n$, we thus have 
$
 \|-\|_{\rho}^{\sim} \le  \|-\|_{2\rho}
$.
This gives the required equivalence, since clearly $\|-\|_{\rho}^{\sim} \ge \|-\|_{\rho}$ automatically.
\end{proof}

As dg Fr\'echet algebras, we then have $\Hom(\cF_{S\sqcup  dS}(R),\cF_{T\sqcup  dT}(R) )\cong \prod_{s \in S} \cF_{T\sqcup  dT}(R)_{\deg s}$.

\begin{definition}\label{dgFEFCdef}
Define  a differential graded $R$-algebra $B$ with free entire functional calculus (or dg FEFC $R$-algebra for short)  to be a product-preserving contravariant functor $ \cF_{S \sqcup dS} (R) \mapsto \prod_{s \in S} B_{\deg s}$ on the category of dg Fr\'echet algebras  $\cF_{S\sqcup  dS}(R)$ for finite sets $S$.
\end{definition}
Thus a dg FEFC $R$-algebra is a graded FEFC  $R$-algebra $B$ equipped with a square-zero derivation $\delta \co B_i \to B_{i-1}$ satisfying $\delta \Phi_f(b_1, \ldots, b_n)= \Phi_{df}(b_1,\ldots, b_n, \delta b_1, \ldots, \delta b_n)$, for $f \in \cF_S(R)$ and hence $df \in \cF_{S \sqcup dS} (R)$.

From a categorical perspective, the forgetful functor $B \mapsto \coprod_i B_i$ from dg FEFC $R$-algebras to graded  sets has a left adjoint, 
 which sends a graded set $S$ to the dg  FEFC $R$-algebra
\[
\cF_{S \sqcup dS}(R):=  \LLim_{\substack{T \subset S}\\ \text{finite}} \cF_{T \sqcup dT}(R),
\]
and  dg FEFC $R$-algebras are algebras for the resulting monad structure on the functor $S \mapsto \cF_{S\sqcup dS}(R)$.

\begin{lemma}\label{FEFCmodel}
There is a cofibrantly generated model structure on the category $dg_+\FEFC(R)$ of non-negatively graded dg FEFC $R$-algebras  in which  a morphism $A_{\bt} \to B_{\bt}$ is
\begin{enumerate}
 \item  a weak equivalence if it is a quasi-isomorphism on the underlying chain complexes;
\item a fibration if it is surjective in strictly positive chain degrees.
\end{enumerate}
\end{lemma}
\begin{proof}
 The  forgetful functor from $dg_+\FEFC(R)$ to   non-negatively graded sets factors through non-negatively graded  chain complexes, as does its left adjoint.  Since the resulting monad preserves filtered colimits, we may apply \cite[Theorem 11.3.2]{Hirschhorn} to transfer the projective model structure from chain complexes to $dg_+\FEFC(R)$. 
 
 The only condition to check is that the left adjoint  to the forgetful functor sends transfinite compositions of pushouts of generating trivial cofibrations to weak equivalences. Since the class of quasi-isomorphisms is closed under filtered colimits, it suffices just to check this for pushouts, and we can assume that the domain $A$ is finitely presented. 
 
 For $\delta x =y$, this means that we wish to show that $A \to A \coprod \cF_{x,y}=:B$ is a quasi-isomorphism. We can define a graded derivation $\gamma$ of homological  degree $1$ on $B$ by $\gamma(A)=0$, $\gamma(x)=0$ and $\gamma(y)=x$, and then the commutator $[\delta,\gamma]$ is a derivation killing $A$ and fixing $x$ and $y$, so on monomials in $A\<x,y\>$ it acts as multiplication by the total degree in $x$ and $y$. Dividing $\gamma$ by that degree termwise gives us an $R$-linear map $h \co A\<x,y\> \to B_{[1]}$ with $[\delta, h]$ equal to the identity on $A$ and zero on $(x,y)$. 
 
 It thus  suffices to show that $h$ extends from $A\<x,y\>$ to a map on the whole of $B$ with the same property. When $A=\cF_S$ for a finite set $S$, we have $B=\cF_{S\sqcup \{x, y\}}$, and $h$ extends to the completion because $\|h(b)\|_r \le \|\gamma(b)\|_r$ for all $b \in A\<x,y\>$ and $r>0$ (in fact $\|h(b)\|_r \le \|b\|_r$). In general, $A$ is of the form $\cF_S/\sum_{j=1}^m\iota_{c_j}(\cF_S\hten \cF_S)$, and then $B \cong   \cF_{S\sqcup \{x, y\}}/\sum_{j=1}^m\iota_{c_j}(\cF_{S\sqcup \{x, y\}}\hten \cF_{S\sqcup \{x, y\}})$. Now, $\gamma$ gives rise to a derivation $\gamma \ten \id +\id\ten \gamma$ on $\cF_{S\sqcup \{x, y\}}\hten \cF_{S\sqcup \{x, y\}}$, and again dividing termwise by total degree in $x$ and $y$ gives a linear endomorphism $H$. Since $\gamma \circ \iota_c = \iota_c\circ (\gamma \ten \id + \id \ten \gamma) $ for all $c \in \cF_S$, we also have  $h \circ \iota_c = \iota_c \circ H$, so $h$ preserves FEFC ideals and hence descends to $A$. 
\end{proof}

\begin{remark}\label{bimodrmk}
 The constructions and results of \S \ref{bimodsn} all extend verbatim to dg FEFC algebras, and also to the stacky dg FEFC algebras we will encounter in \S \ref{stackyFEFCsn}, simply by incorporating gradings  throughout. We will use the same notation for the corresponding dg constructions without further comment.
\end{remark}

\begin{remark}[Alternative characterisation of dg FEFC algebras]\label{altcharrmk}
 If we fix an FEFC algebra $A$, then there is a fairly algebraic characterisation of objects $B \in dg_+\FEFC(R)$ with $B_0=A$. Explicitly, $B$ is a unital associative dg algebra (with $B_0=A$) in the multicategory $\C$ whose objects are  $A^{e,\FEFC}$-modules and whose operations are
 \[
  \C(M_1, \ldots, M_j;N):= \Hom_{A^{e,\FEFC}}((M_1\ten \ldots \ten M_j)\ten_{(A^{e,\FEFC})^{\ten j}}A^{\ten^{FEFC}( j+1)}, N),
 \]
where the $(A^{e,\FEFC})^{\ten j}$-module structure on $A^{\ten^{FEFC}( j+1)}$ is universally determined by the property that  $(a_1\ten b_1) \ten \ldots \ten (a_j\ten b_j)$ sends  $c_1\ten \ldots \ten c_{j+1}$ to $(c_1a_1)\ten (b_1c_2a_2) \ten \ldots \ten (b_{j-1}c_ja_j)\ten (b_jc_{j+1})$, and the $A^{e,\FEFC}$-module structure on the tensor product is given by the outer multiplications on $A^{\ten^{FEFC}( j+1)}$.
\end{remark}
%% can rewrite using hom-tensor (but it has a cyclic offset). 

\subsubsection{Comparison with simplicial FEFC algebras}

\begin{proposition}\label{simpFEFCmodelprop}
%%followng commented out as it duplicates {FEFCmodelprop}
% There is a cofibrantly generated model structure (which we call the standard model structure) on the category $dg_+\FEFC_{R}$  in which  a morphism $A_{\bt} \to B_{\bt}$ is  
% \begin{enumerate}
%  \item  a weak equivalence if it is a quasi-isomorphism;
% \item a fibration if it is surjective in strictly positive chain degrees.
% \end{enumerate}

There is a model structure  for simplicial FEFC $R$-algebras in which weak equivalences are $\pi_*$-isomorphisms and fibrations are Kan fibrations.  Dold--Kan normalisation combined with a completion of the Eilenberg--Zilber shuffle product %\cite[Definitions 8.3.6 and 8.5.4]{W} 
gives a right Quillen equivalence $N$ from simplicial FEFC $R$-algebras to dg FEFC algebras concentrated in non-negative chain degrees. 
\end{proposition}
\begin{proof}
 %The model structure on simplicial FEFC algebras follows by applying  \cite[Theorem 11.3.2]{Hirschhorn} to the forgetful functor from simplicial FEFC $R$-algebras to simplicial $R$-modules. There is a  left adjoint $\cF$, which in particular sends $R.S$ to the simplicial algebra $n \mapsto \cF_{S_n}$ for any simplicial set $S$. Since the forgetful functor preserves filtered colimits, in order to apply the theorem it suffices to show that $\cF$ sends any generating trivial cofibration $ 0 \to M$ to a morphism whose pushouts are all weak equivalences. The contracting simplicial homotopy on $M$ then  induces  the structure of a deformation retract on the simplicial FEFC homomorphism $A \to A \coprod \cF(M)$.
 
 The model structure on simplicial algebras is given by \cite[Theorem II.4]{QHA}.
 
In order to compare the two categories, we can consider the corresponding monads on the category of simplicial $R$-modules. If we  write $\cF$ for the left adjoints to both the forgetful functor from non-negatively graded dg FEFC algebras  to non-negatively graded chain complexes of $R$-modules and from simplicial FEFC algebras to simplicial $R$-modules, and  $N^{-1}$ denotes denormalisation, then our desired functor is given by a natural transformation $N^{-1}\cF(NV) \to \cF(V)$ of monads, or equivalently $\cF(NV) \to N\cF(V)$. The corresponding construction for associative algebras is given by the iterated Eilenberg--Zilber shuffle product $\nabla \co T(NV) \to NT(V)$, where $T(U) = \bigoplus_{n \ge 0} U^{\ten n}$, with the tensor product taken according to the usual graded convention on the left, and levelwise for the simplicial vector spaces on the right. 

Since the functors $\cF$ are both defined as completions of the  functors $T$, existence of our functor $N$ from simplicial FEFC $R$-algebras to dg FEFC algebras just amounts to showing that the shuffle map in each chain degree is bounded with respect to the defining systems of norms. In chain degree $n$, with $v_i \in N_{n_i}V$, the element $\nabla(v_1\ten \ldots \ten v_k)$ is an alternating sum over $(n_1, \ldots, n_k)$-shuffles of elements in $NT(V)_n$ given by applying iterated degeneracies to the $v_i$. Since there are $k^n$  such permutations, it follows that $\| \nabla(v_1\ten \ldots v_k)\|_{\rho}\le k^n \|(v_1\ten \ldots v_k\|_{\rho}$ in chain degree $n$, for all of our defining norms $\rho$ on finite-dimensional subspaces of $NV$, extended to $V$ by setting all degeneracies to have norm $1$. 

Since we can always choose an exponential function (in $k$) to outgrow the polynomial $k^n$, it follows by Lemma \ref{extendmaplemma2} 
that $\nabla$ defines a natural transformation $\cF N \to N\cF$, and hence $N^{-1}\cF N \to \cF$. This map is automatically a natural transformation of monads, since $N^{-1}T N \to T$ is so, forcing the morphisms to agree on a dense subset. 

It remains to show that our functor $N$ is a right Quillen equivalence. It is clearly right Quillen and reflects weak equivalences, so we need only show that for cofibrant dg FEFC algebras $A$, the unit $A \to NN^*A$ of the adjunction is a quasi-isomorphism. Since cofibrant objects are retracts of quasi-free  algebras, we reduce first to the quasi-free case and then to the free case with $\delta =0$ via the exhaustive increasing filtration  by quasi-free subalgebras on generators of bounded degrees. Since $N^*\cF(U)=\cF(N^{-1}U)$, it thus suffices to show that the completed shuffle map $ \cF(U) \to N\cF(N^{-1}U)$ is a quasi-isomorphism. Because both functors commute with filtered colimits, we can reduce to the case where $U$ is finite-dimensional. 

For associative, rather than FEFC, algebras, the homotopy inverse is given by the iterated  Alexander--Whitney cup product $\cup$, which is left inverse to $\nabla$. The cup product in degree $n$ is simply given by iterated face maps, with a term for each partition $n_1+\ldots + n_k=n$, %%write it as a series of $n$ $1$s, then insert $k-1$ dividers, so $\binom{k+n-1}{k-1}$ terms, so $\frac{(n+1) \ldots (n+k-1)}{(k-1)!}$. Is that bounded in $n$? I doubt it. remember that anythin subexponential in $k$ will do us. ratio $\frac{k+n}{k}>1$, so those terms are increasing in $k$, but ratios manageable (since decreasing).
so $\|v_1 \cup \ldots \cup v_k\|_{\rho} \le \binom{k+n-1}{k-1}\|v_1\ten \ldots \ten v_k\|_{\rho}$ in degree $k$. Since $\binom{k+n-1}{k-1} %=\prod_{i=1}^{k-1} \frac{n+i}{i}
\le (n+1)^{k-1}$, Lemma \ref{extendmaplemma1}
  ensures this extends to give a  map $\cF(N^{-1}U)_k \to N^{-1}\cF(U)_k$, and these maps combine to give a simplicial map since they are compatible on the dense subspace $T(N^{-1}U)$.
  
It remains to show that the homotopy $\Phi \co \nabla \circ \cup \simeq \id $ on $T(N^{-1}U)$ satisfies the conditions of  Lemma \ref{extendmaplemma1}. Writing $\nabla^k, \cup^k$ for the components of $\nabla, \cup$ in arity $k$, we can express them as the composites
\begin{align*}
(NV)^{\ten k} \xra{ \id_{NV} \ten \nabla^{k-1}}  NV \ten N(V^{\ten k-1}) \xra{\nabla_{V,V^{\ten k-1}}} N(V^{\ten k})\\%\nabla \co T(NV) \to NT(V)
N(V^{\ten k}) \xra{\cup_{V,V^{\ten k-1}}} NV\ten N(V^{\ten k-1}) \xra{\id \ten \cup^{k-1}} (NV)^{\ten  k}.   
\end{align*}
Thus $\Phi^k = \nabla_{V,V^{\ten k-1}} \circ (\id \ten \Phi^{k-1}) \circ \cup_{V,V^{\ten k-1}} + \Phi_{V,V^{\ten k-1}}$. The proof of \cite[Theorem 2.1a]{EilenbergMacLaneHpin2} gives an explicit recursive formula for $\Phi_{V,W}$ in terms of simplicial operations, independent of inputs. Thus 
if $f(k,n)$ denotes the maximum value of  $\|\Phi^k(\alpha)  \|_{\rho}/ \|\alpha)  \|_{\rho}$ in simplicial degrees $\le n$, we  have an inequality of the form $f(k,n)\le a(n)f(k-1,n) + b(n)$ (independent of $\rho$), with $f(1,n)=0$, 
% where $a(n)= 2^n(n+1)$ and   $b(n)= f(2,n)$, 
so $f(k,n)\le \frac{(a(n)^{k-1}-1)b(n)}{a(n)-1}$ and hence grows at most exponentially in $k$ when $n$ is fixed,  as required.  
\end{proof}

\subsubsection{Simplicial model structures}

The model  category of simplicial FEFC algebras from Proposition \ref{simpFEFCmodelprop} has an obvious simplicial model category structure. Given a simplicial set $K$ and a simplicial FEFC algebra $A$, we can define $A \boxempty K$ and $A^K$ by
\[
 (A \boxempty K)_n := \coprod_{x \in K_n} A_n, \quad \quad (A^K)_n:= \Hom_{s\Set}(K \by \Delta^n,A),
\]
where $\coprod$ is coproduct in the category of FEFC algebras. 

Moreover the model category of non-negatively graded dg FEFC algebras from Lemma \ref{FEFCmodel} has a natural simplicial model structure, which we now describe. This is in contrast to the EFC setting, where the analogous construction fails from a lack of commutativity.

\begin{lemma}
For $K$ a simplicial set with normalised chain complex $\bar{C}(K)$, and $\cF$  the left adjoint to the forgetful functor from non-negatively graded dg FEFC algebras to chain complexes, there is a natural continuous chain map
\[
\uline{\Delta}\co \cF(V) \ten \bar{C}(K) \to \cF(V\ten \bar{C}(K))
\]
given by combining the Alexander--Whitney maps
\[
V^{\ten n} \ten \bar{C}(K) \xra{\id \ten \Delta_n} V^{\ten n} \ten \bar{C}(K)^{\ten n}. 
\]
\end{lemma}
\begin{proof}
 The maps given automatically define a chain map on the tensor algebra $T(V)$ and commute with filtered colimits in $V$ and $K$. Reducing to the case where $K$ is finite and $V$ finite-dimensional,  it  suffices  to check the conditions of Lemma \ref{extendmaplemma2}. If we set elements of $K$ to have norm $1$, and have a multiplicative norm $\|-\|$ on $T(V)$, then for $v_i \in V$ and $x \in K_d$ we have
\[
 \|\uline{\Delta}(v_1 \ten \ldots \ten v_m\ten x)\|= \|\sum_{d_1+\ldots+ d_m=d} \pm (v_1 \ten (\pd_{d_1+1})^{d-d_1}x)\ten \ldots \ten (v_m \ten (\pd_0)^{d-d_m}x)\|. 
\]
 Since there are $\binom{d+m-1}{d}$ such partitions, this gives $\|\uline{\Delta}(v_1 \ten \ldots \ten v_m\ten x)\|\le \binom{d+m-1}{d}\|v_1 \ten \ldots \ten v_m\ten x\| $. Because $\binom{d+m-1}{d}$ is a polynomial (of degree $d$) in $m$, it grows sub-exponentially so is asymptotically bounded by any larger multiplicative norm. 
 \end{proof}

 \begin{definition}\label{boxKdef}
  Given a simplicial set $K$, define  functors $- \boxempty K$ and $(-)^K$  on the category of dg FEFC algebras as follows. 
  
  Given a dg FEFC algebra of the form $\cF(V)$, we set $\cF(V)\boxempty K:= \cF(V\ten \bar{C}(K))$. Giving a dg FEFC morphism $\cF(U) \to \cF(V) $ is equivalent to giving a chain map  $f \co U \to \cF(V)$, and we then define $f \boxempty K \co \cF(U) \boxempty K \to \cF(V) \boxempty K$ to be the dg FEFC morphism corresponding to the chain map given by the composite
  \[
   U \ten \bar{C}(K) \xra{f\ten \id} \cF(V) \ten \bar{C}(K) \xra{\uline{\Delta}} \cF(V \ten \bar{C}(K)).
  \]

  This defines a functor  on the full subcategory of dg FEFC algebras in the image of $\cF$, which we extend to all dg FEFC algebras by Kan extension. The functor $(-)^K$ is its right adjoint, given by setting
 \[
  \Hom_{dg_+\FEFC}(\cF_{S \sqcup dS}, A^K):= \Hom_{dg_+\FEFC}(\cF(R.(S \sqcup dS)\ten\bar{C}(K)),A).
 \]
   \end{definition}

   In particular, note that the chain complex underlying $A^K$ is the good truncation of the dg $\Hom$-complex $\HHom_R(\bar{C}(K),A)$. 
  
 \begin{lemma}\label{simpmodellemma}
  The functors of Definition \ref{boxKdef} make the model structure on  $dg_+\FEFC(R)$ from Lemma \ref{FEFCmodel} a simplicial model structure.
   \end{lemma}
\begin{proof}

It suffices to show that if  $q \co A \to B$ is a fibration in $dg_+\FEFC(R) $ and $i \co K \to L$ a cofibration of simplicial sets, then the map $A^L\to A^K\by_{B^K}B^L$ is a fibration which is trivial whenever either $i$ or $q$ is so. It suffices to check this on the underlying non-negatively graded chain complexes, since (trivial) fibrations are defined via the forgetful functor. Since $A^K= \tau_{\ge 0}\HHom_R(\bar{C}(K),A)$ by construction, the desired properties now follow immediately from the simplicial model structure on non-negatively graded chain complexes. 
% 
% By condition SM7a as in \cite[Corollary 3.13]{sht}, it suffices to show that if  $q \co A \to B$ is a fibration, then $A^{\Delta^n} \to B^{\Delta^n}\by_{B^{\pd \Delta^n}}A^{\pd \Delta^n}$ is a fibration which is trivial whenever $q$ is, and that the two maps $A^{\Delta^1}\to A\by_BB^{\Delta^1}$ are trivial fibrations.
% 
% Since fibrations are maps which are surjections in strictly positive degrees, they correspond to having  the right lifting property with respect to the map $R \to  \cF_{\Z_{>0} \sqcup d\Z_{>0}}$, where the variable $n \in \Z_{>0}$ is placed in degree $n$ (and hence $dn$ in degree $n-1$). Similarly, trivial fibrations are precisely maps with the right lifting property with respect to the inclusion $\cF_{d\Z_{>0}} \to \cF_{\Z_{>0} \sqcup d\Z_{>0}} $.
% 
% If thus suffices to show that the maps 
% \begin{align*}
% \cF(R.(\Z_{>0} \sqcup d\Z_{>0})\ten \bar{C}(\pd \Delta^n)) \to \cF(R.(\Z_{>0} \sqcup d\Z_{>0})\ten \bar{C}(\Delta^n))  \quad \text{and} \\
% \cF(   (R.(d\Z_{>0})\ten \bar{C}(\Delta^1))\oplus R.(\Z_{>0})) \to  \cF(R.(\Z_{>0} \sqcup d\Z_{>0})\ten \bar{C}(\Delta^1))
% \end{align*}
% are trivial cofibrations, and that the maps
% \[
% \cF( (R.(\Z_{>0})\ten \bar{C}(\pd \Delta^n))\oplus (R.( d\Z_{>0})\ten \bar{C}( \Delta^n))  ) \to \cF(R.(\Z_{>0} \sqcup d\Z_{>0})\ten \bar{C}(\Delta^n))  
% \]
% are cofibrations. Since the functor $\cF$ is left Quillen, these properties all follow immediately.
\end{proof}

 \subsubsection{Analytification}\label{analytificationsn2}

Given a functor $F$ on $dg_+\Alg(R)$, we can simply define a form of analytification functor $F^{\FEFC}$ on  $dg_+\FEFC(R)$ by composition with the forgetful functor (as in Proposition \ref{analytifyNCprop} below), giving an obvious non-commutative analogue of \S \ref{analytificnsn}.
 
 In practical terms, the dg analytic theory and analytification seem to work best in the commutative setting by reducing to the full $\infty$-subcategory of localised dg dagger affinoid algebras. However, there are no clear analogues of the affinoid theory in the non-commutative setting as discussed in Remark \ref{NCSteinrmk}, so we mainly consider the whole of  $dg_+\FEFC(R)$ (or equivalently $dg_+\hat{\Tc}\Alg_{R}$, via Corollary \ref{keyBanFEFCcor}) for analytifications. 
 
 However, this comes at the expense of having  slightly unusual analytifications when our functors are not strongly quasi-compact, which might be  poorly suited to some non-Archimedean applications.
%  Our approach using functional calculus is chosen to be as close as possible to the algebraic theory, which might turn out to be  more restrictive in the non-commutative than the commutative setting. 
 The proximity of functional calculus to the algebraic theory does however mean that this approach is very well suited to constructing analytic maps between functors of algebraic origin, such as the Riemann--Hilbert map of  \S \ref{RHsn}. 
 
 For comparisons and other constructions not of algebraic origin, we typically work with categories of dg Fr\'echet or dg nuclear Fr\'echet algebras, as giving the richest source of functors preserving abstract quasi-isomorphisms.

\section{Complete topological dg  algebras}\label{DGLMDCsn}

In practice, functors defined directly on categories of FEFC-algebras tend only to arise from functors on algebras with no topological structure. There are many more natural functors on LDMC topological DGAs, and Corollary \ref{keyBanFEFCcor} will show that these functors  induce functors on FEFC-DGAs whenever they send abstract quasi-isomorphisms to weak equivalences. That is a strong constraint which seems rare, but there are very many natural examples of such functors on the subcategories of Fr\'echet and nuclear Fr\'echet dg algebras (\S \ref{moduliexsn}). Those categories do not apparently become $\infty$-subcategories of the $\infty$-category of FEFC-DGAs on localising at abstract quasi-isomorphisms, so much of this section is concerned with   establishing various partial comparison results and identifying classes of Fr\'echet DGAs whose homotopy theory is largely controlled by the FEFC structure. 

\subsection{Basic comparisons}\label{FEFCcompsn}

The proof of Lemma \ref{EFCforgetlemma} adapts to the free entire functional calculus to give the following, for $dg_+\hat{\Tc}\Alg_{R}$ as in Definition \ref{Tcalgdef}:
\begin{lemma}\label{FEFCforgetlemma}
 There is a natural forgetful functor $U \co dg_+\hat{\Tc}\Alg_{R} \to  dg_+\FEFC_{R}$, sending a topological dg  algebra $A_{\bt}$ to a dg FEFC algebra with the same underlying dg algebra. 
\end{lemma}

We then have:

\begin{proposition}\label{keyBanFEFCprop}
 The forgetful functor $U \co \pro(dg_+\hat{\Tc}\Alg_{R}) \to  dg_+\FEFC_{R}$ has a left adjoint, and the unit of the adjunction gives isomorphisms on cofibrant objects in the category $dg_+\EFC\Alg_{R}$.  %%%could welll be worth observing that this also works for quotient by closed dg ideal. However, counit $FUB \to B$ less likely to be equivce, e.g. for $B=R \oplus V\eps$, $UB$ knows nothing about the norm, with $FUB$ having system of all possible seminorms on $V$. 
\end{proposition}
\begin{proof}
This follows in much the same way as Proposition \ref{keyBanEFCprop}. The left adjoint $F$ is given by completion with respect to the system of submultiplicative seminorms. On objects $C$ which are quasi-free in the sense that the  underlying graded FEFC-algebra is freely generated, say by graded generators $S$, we then have
\begin{align*}
 (FC)_i \cong\Lim_{f\co S \to \R_{>0}}\ell^1\{  \alpha/f(\alpha) ~:~ \alpha \in W(S)_i\}%\\
 ~\cong~ \LLim_{\substack{T\subset S \\ \text{finite}}} 
% \Lim_{f\co T \to \R_{>0}}\ell^1\{  \alpha/f(\alpha) ~:~ \alpha \in W(T)_i\}\\
 \Lim_{r>0}\ell^1\{  \alpha/r^{|\alpha|} ~:~ \alpha \in W(T)_i\}
 ~\cong~ C_i.
\end{align*}
Here, $f$ is extended multiplicatively from $S$ to $W(S)$,  and the reduction to finite $T$ is by  
arguing as in the proof of Proposition \ref{keyBanEFCprop}. In general, for  quotients $C/I$ of quasi-free objects by FEFC ideals,  $F(C/I)$ is the quotient of $FC$ by the closure of $I$.

Since every cofibrant object is a retract of a quasi-free object, it follows that the unit $C \to UFC$ is an isomorphism whenever $C$ is cofibrant.
\end{proof}

Replacing Propositions \ref{stdmodelprop} and \ref{keyBanEFCprop} with Lemma \ref{FEFCmodel} and Proposition \ref{keyBanFEFCprop} in the proof of Corollary \ref{keyBanEFCcor} then gives the following,
referring to  morphisms $A \to B$ for which the maps  $\H_*A \to \H_*B$ are isomorphisms of abstract vector spaces as abstract quasi-isomorphisms:
\begin{corollary}\label{keyBanFEFCcor}
The  forgetful functor $U \co dg_+\hat{\Tc}\Alg_{R} \to  dg_+\FEFC_{R}$ induces an equivalence of the $\infty$-categories given by simplicial localisation at abstract quasi-isomorphisms. The same is true if we replace $dg_+\hat{\Tc}\Alg_{R}$ with any full subcategory containing all objects of the form $(\cF_S,\delta)$ for non-negatively graded sets $S$. 
\end{corollary}

\subsection{Functors on  subcategories}\label{subcatfunsn}

Defining quasi-isomorphism-preserving functors on  $dg_+\hat{\Tc}\Alg_{R}$ tends to be difficult unless they forget the topological structure. However, there are many such functors on the subcategory of dg nuclear Fr\'echet algebras, and we now introduce subcategories of that subcategory whose homotopy theory is governed by the FEFC structure.

\begin{definition}
Define  $dg_+\hat{\Tc}\Alg_{R}^{\mathrm{afp,cof}} \subset  dg_+\hat{\Tc}\Alg_{R}$ and $dg_+\FEFC_{R}^{\mathrm{afp,cof}} \subset dg_+\FEFC_{R}$ to be the full subcategories consisting of all retracts of quasi-free objects $(\cF_S,\delta) $ which are levelwise finitely generated in the sense that the graded set $S$ has finitely many elements in each degree.

Define $dg_+\hat{\Tc}\Alg_{R}^{\mathrm{fp,cof}}$ and  $dg_+\FEFC_{R}^{\mathrm{fp,cof}}$ similarly, but for finite sets $S$.
\end{definition}

\begin{corollary}\label{boxDeltacor}
 The functors $(-)\boxempty \Delta^n$ of Definition \ref{boxKdef} naturally lift to endofunctors of $dg_+\FEFC_{R}^{\mathrm{afp,cof}}$ and  $dg_+\FEFC_{R}^{\mathrm{fp,cof}}$.
\end{corollary}
\begin{proof}
By  Proposition \ref{keyBanFEFCprop}, we have equivalences $U \co dg_+\hat{\Tc}\Alg_{R}^{\mathrm{afp,cof}} \to  dg_+\FEFC_{R}^{\mathrm{afp,cof}}$ and $U \co dg_+\hat{\Tc}\Alg_{R}^{\mathrm{fp,cof}} \to  dg_+\FEFC_{R}^{\mathrm{fp,cof}}$ of categories, since all objects concerned are cofibrant.

Given a simplicial set $K$, the functor $(-)\boxempty K$ of Definition \ref{boxKdef} preserves quasi-free objects and their retracts. If $K$ and $S$ are finite (resp. levelwise finite), then $(\cF_S,\delta)\boxempty K $ is also  finitely generated  (resp. levelwise finitely generated), by the product of $S$ with the set of  non-degenerate elements of $K$. 

In particular, this means that the functors $(-)\boxempty \Delta^n$ are endofunctors of $dg_+\FEFC_{R}^{\mathrm{afp,cof}}$ and  $dg_+\FEFC_{R}^{\mathrm{fp,cof}}$, and hence of $dg_+\hat{\Tc}\Alg_{R}^{\mathrm{afp,cof}}$ and $dg_+\hat{\Tc}\Alg_{R}^{\mathrm{fp,cof}}$.
\end{proof}

\begin{proposition}\label{afpTcFEFCprop}
Given an object $A$ of $dg_+\hat{\Tc}\Alg_{R}^{\mathrm{afp,cof}}$ and any full subcategory $\cD$ of $dg_+\hat{\Tc}\Alg_{R} $ containing the objects $A \boxempty \Delta^n$, for all objects $B$ of $\cD$ we have a natural equivalence
\[
 \oR\Map_{\cD}(A,B) \to \oR\Map_{dg_+\FEFC_{R}}(A,B)
\]
of mapping spaces in the respective $\infty$-categories given by simplicial localisation at abstract quasi-isomorphisms.
\end{proposition}
\begin{proof}
Since all dg FEFC algebras are fibrant, a model for the simplicial mapping space $\oR\Map_{dg_+\FEFC_{R}}(A,B)$ is given by $n \mapsto \Hom_{dg_+\FEFC_{R}}(A \boxempty \Delta^n, B)$. In particular, this preserves weak equivalences in $B$. By adjunction, this is isomorphic to the simplicial set $ n \mapsto \Hom_{\cD}(A \boxempty \Delta^n, B)$, because $\cD$ is a full subcategory of $dg_+\hat{\Tc}\Alg_{R} $. The morphisms $A \boxempty \Delta^n \to A $ are all quasi-isomorphisms between objects of $\cD$, and this construction preserves quasi-isomorphisms in $B$, so a standard argument (e.g. \cite[Lemma \ref{PTLag-locmaplemma}]{PTLag}) implies that the same simplicial set must also be a model for $\oR\Map_{\cD}(A,B)$, as required.
\end{proof}

In particular, Proposition \ref{afpTcFEFCprop} and Corollary \ref{boxDeltacor} give the following fully faithful $\infty$-functors:   
\begin{corollary}\label{afpTcFEFCcor}
The forgetful functors  $U \co dg_+\hat{\Tc}\Alg_{R}^{\mathrm{afp,cof}} \to  dg_+\FEFC_{R}$ and  $U \co dg_+\hat{\Tc}\Alg_{R}^{\mathrm{fp,cof}} \to  dg_+\FEFC_{R}$ induce  fully faithful $\infty$-functors on the $\infty$-categories given by simplicial localisation at quasi-isomorphisms.

  The same is true if we replace $dg_+\hat{\Tc}\Alg_{R}^{\mathrm{afp,cof}}$ (resp. $dg_+\hat{\Tc}\Alg_{R}^{\mathrm{fp,cof}}$) with any larger subcategory of $dg_+\hat{\Tc}\Alg_{R} $  in which every object admits a quasi-isomorphism from an object of $dg_+\hat{\Tc}\Alg_{R}^{\mathrm{afp,cof}}$ (resp. $dg_+\hat{\Tc}\Alg_{R}^{\mathrm{fp,cof}}$).
\end{corollary}

\begin{definition}
Say that a functor from any subcategory of  $dg_+\hat{\Tc}\Alg_{R}^{\mathrm{afp,cof}}$ or $dg_+\FEFC_{R}$ to a relative category (in particular any model category) is w.e.-preserving if it maps abstract  quasi-isomorphisms to weak equivalences. 
\end{definition}

\begin{definition}
 Say that a w.e.-preserving  functor $F$ from  $dg_+\FEFC_{R}$  to a relative category $\C$ is l.f.p. if it preserves filtered homotopy colimits, i.e. if $F(\LLim_{\alpha} B_{\alpha}) \simeq \ho\LLim_{\alpha}F(B_{\alpha})$ for all filtered systems $\{B_{\alpha}\}_{\alpha}$.
\end{definition}
In particular, note that the mapping space functor $\oR\Map(A,-)$ is l.f.p. whenever $A$ is cofibrant and finitely generated.

The following applies to any model category $\C$, with our most common example being that of simplicial sets.
\begin{corollary}\label{Alfpcor}
 For a homotopy cocomplete  relative category $\C$, 
the $\infty$-category of l.f.p.  w.e.-preserving  functors  from  $dg_+\FEFC_{R}$  to %simplicial sets 
$\C$ 
is equivalent  to   the $\infty$-category of functors $dg_+\hat{\Tc}\Alg_{R}^{\mathrm{fp,cof}} \to \C$ sending $U$-quasi-isomorphisms to weak equivalences, via pre-composition $U_*$ with the forgetful functor $U \co dg_+\hat{\Tc}\Alg_{R}^{\mathrm{fp,cof}} \to dg_+\FEFC_{R} $.

Via this comparison, the functor $U_*$ on the $\infty$-category of all w.e.-preserving  functors  on  $dg_+\FEFC_{R}$ gives a derived right adjoint to the inclusion functor from the full $\infty$-subcategory of l.f.p. functors.  
\end{corollary}
\begin{proof}
 As in Lemma \ref{indFP}, $dg_+\FEFC_{R}$ is the ind-category of its finitely presented objects, and the same is true for the full subcategory of cofibrant objects. Since the functors $(-)\boxempty \Delta^n$ preserve finitely presented cofibrant objects, the simplicial set $n \mapsto \Hom_{dg_+\FEFC_{R}}(A\boxempty \Delta^n,-)$ thus preserves filtered colimits for  any $A \in dg_+\FEFC_{R}^{\mathrm{fp,cof}}$. Applying Corollary \ref{afpTcFEFCcor},  
 the simplicial localisation of  $dg_+\FEFC_{R}$ is therefore equivalent to the ind-category of the simplicial localisation of $dg_+\hat{\Tc}\Alg_{R}^{\mathrm{fp,cof}}$.
 
 The derived left adjoint $\oL U^*$ to $U_*$ is a form of Kan extension, and the previous paragraph implies that its  essential image consists precisely of of l.f.p. functors. Since $U$ is fully faithful, the unit of the adjunction $\oL U^* \dashv U_*$ is a natural equivalence, as can be seen by looking at representable functors then passing to colimits, since if a functor on $dg_+\FEFC_{R}$ is represented by an object $A$, then its image under $\oL U^*$ is represented by $UA$.
%  
% %  It remains only to show that $\oL U^*$ is fully faithful, or equivalently that the unit of the adjunction $\oL U^* \dashv U_*$ is a natural equivalence. If a functor on $dg_+\FEFC_{R}$ is represented by an object $A$, then its image under $\oL U^*$ is represented by $UA$, so full faithfulness of $U$ means that the unit is an adjunction in this case. The general case then follows by passing to homotopy colimits, since they are preserved by both $U_*$ and $\oL U^*$.
%  
%  
% %  Both statements now follow immediately, with the second using the observation that if $F,G$ are $\infty$ functors on an ind-category of an $\infty$-category $\cA$, with $F$ preserving homotopy colimits, then $\oR\Map(F,G) \simeq \oR\Map(F|_{\cA},G|_{\cA})$.
\end{proof}

On larger subcategories, we  have the following weaker comparison, which ensures that many categories of interest  are adequate for the purpose of detecting FEFC maps out of l.f.p. functors. Writing $ [\cD,\C]$ for  the $\infty$-category  of w.e.-preserving  functors   $\cD \to \C$, we have: 
\begin{corollary} \label{Dlfpcor}
Take a homotopy cocomplete relative category $\C$ and any full subcategory $\cD$ of $dg_+\hat{\Tc}\Alg_{R} $ containing $dg_+\hat{\Tc}\Alg_{R}^{\mathrm{fp,cof}}$. 

Given an l.f.p w.e.-preserving  functor $F \co dg_+\FEFC_{R} \to \C$ and a w.e.-preserving functor $G \co \cD \to \C$, the natural map
\[
 \oR\Map_{[\cD,\C]}(U_{\cD,*}F,G) \to \oR\Map_{[dg_+\hat{\Tc}\Alg_{R}^{\mathrm{fp,cof}},\C]}(U_*F, i_*G)
\]
is an equivalence, for the inclusion $i \co  dg_+\hat{\Tc}\Alg_{R}^{\mathrm{fp,cof}} \into \cD$,  the forgetful functor $U_{\cD} \co  \cD \to  dg_+\FEFC_{R}$, and $i_*, U_{\cD,*}$ the respective functors given by pre-composition.

Moreover, if  $G$ extends to FEFC-algebras in the sense that  $G\simeq U_{\cD,*}G'$ for some w.e.-preserving  functor $G' \co dg_+\FEFC_{R} \to \C$, then the natural map
\[
 \oR\Map_{[dg_+\FEFC_{R},\C]}(F,G') \to    \oR\Map_{[\cD,\C]}(U_{\cD,*}F,G)  
\]
is also an equivalence. 

%In particular, we jhve a fully faithful functor   $U_{\cD,*} \co [dg_+\FEFC_{R},\C]_{lfp}  \to [\cD,\C]$ from the $\infty$-category of l.f.p w.e.-preserving  functors   $dg_+\FEFC_{R} \to \C$  to  the $\infty$-category $[\cD,\C]$ of w.e.-preserving  functors   $\cD \to \C$.
 \end{corollary}
\begin{proof}
Write $\cA:= dg_+\hat{\Tc}\Alg_{R}^{\mathrm{fp,cof}}$  and $\cF:=dg_+\FEFC_{R}$ , so $U=U_{\cA}= U_{\cD} \circ i$. Taking objects $A, C \in \cA$ and $B \in \cD$, Proposition \ref{afpTcFEFCprop} gives us natural equivalences 
$\oR\Map_{\cA}(A,C) \to \oR \Map_{\cF}(UA,UC)$ and $\oR\Map_{\cD}(iA,B) \to \oR \Map_{\cF}(UA,U_{\cD}B)$. If $\bar{F} \in [\cA,s\Set]$ is represented by $A$, then $\oL i^*\bar{F}$ and $\oL U^*\bar{F}$ are represented by $iA$ and $UA$ respectively, for the derived left adjoints $\oL i^* \dashv i_I$ and $\oL U^* \dashv U_*$. The equivalences then translate to saying that the maps 
\[
\bar{F} \to  U_*\oL U^*\bar{F}  \quad\text{and} \quad    \oL i^* \bar{F} \to U_{\cD,*}\oL U^* \bar{F}
\]
are natural equivalences for $\bar{F}$ representable. Since these functors all preserve homotopy colimits, the corresponding statements then extend to all $\bar{F} \in [\cA,\C]$.

Now, Corollary \ref{Alfpcor} implies that any   l.f.p w.e.-preserving  functor $F \co dg_+\FEFC_{R} \to \C$ is equivalent to $\oL U^* U_*F$, so setting $\bar{F}:=U_*F$ we obtain an equivalence
\[
  \oL i^*  U_* F \to U_{\cD,*}F,
\]
which is equivalent to the first desired statement. For the second, apply $ \oL U_{\cD}^*  $ to deduce that the first map in the sequence
\[
\oL U^* U_*F  \to \oL U_{\cD}^*U_{\cD,*}F \to F
\]
is an equivalence; since the composite is also an equivalence, the second must be too, which is equivalent to the second desired statement.
\end{proof}

% \begin{remark}\label{univlfprmk}
%  Corollary \ref{Dlfpcor} implies that the functor $U_{\cD,*} \co [dg_+\FEFC_{R},\C]_{\mathrm{lfp}} \to [\cD,\C]$ from the full subcategory of l.f.p. functors  has a derived left adjoint $\oL U^* \circ i_*$, sending a functor $G$ on $\cD$ to the universal l.f.p. FEFC  functor under it. Explicitly, we can write any cofibrant  dg FEFC algebra $A$ as a nested union $\LLim_{\alpha} UA(\alpha)$ of objects $A(\alpha) \in   dg_+\hat{\Tc}\Alg_{R}^{\mathrm{fp,cof}}$, and then $\oL U^* i_* G(A) \simeq \ho\LLim_{\alpha} G(A(\alpha))$; composition with a functorial cofibrant replacmeent functor extends this to the whole of $dg_+\FEFC_{R}$. However, this construction seems to have limited utility since in general it will not preserve any limit-preserving properties $G$ may enjoy.
% \end{remark}

The following lemma can be used to send  a w.e.-preserving functor on $\cD$ to the universal l.f.p. FEFC  functor under it. The construction seems to have limited utility since in general it will not preserve any limit-preserving properties the original functor may enjoy, but it will feature in the comparisons of \S \ref{denormsn}.

\begin{lemma}\label{iotaKanextnlemma}
Take a homotopy cocomplete relative category $\C$, a w.e.-preserving functor $F \co dg_+\hat{\Tc}\Alg_{R}^{\mathrm{fp,cof}} \to \C$, and any full subcategory $\cD$ of $dg_+\hat{\Tc}\Alg_{R} $ containing $dg_+\hat{\Tc}\Alg_{R}^{\mathrm{fp,cof}}$. 

The left Kan extension $\iota_!F$ of $F$ along the $\infty$-localisation at quasi-isomorphisms of the functor $\iota \co dg_+\hat{\Tc}\Alg_{R}^{\mathrm{fp,cof}} \to dg_+\hat{\Tc}\Alg_{R}$ is then  given by
\[
 \iota_!F(B):= \LLim_{\alpha} F(Q_{\alpha}B),
\]
for any cofibrant replacement functor $Q=\{Q_{\alpha}\}_{\alpha}$ from $dg_+\FEFC_{R}$ to $dg_+\FEFC_{R}^{\mathrm{cof}} \simeq \ind(dg_+\hat{\Tc}\Alg_{R}^{\mathrm{fp,cof}})$.
\end{lemma}
\begin{proof}
 Since the description given clearly commutes with arbitrary homotopy colimits in $F$, it suffices to prove this when $F$ is a simplicial set-valued functor represented by some $A \in dg_+\hat{\Tc}\Alg_{R}^{\mathrm{fp,cof}}$, but then
 \begin{align*}
  \LLim_{\alpha} F(Q_{\alpha}B)= &\LLim_{\alpha}\oR\map_{dg_+\FEFC_{R}}(A, Q_{\alpha}B) \simeq  \oR\map_{dg_+\FEFC_{R}}(A, \LLim_{\alpha}Q_{\alpha}B)\\
  &\simeq \oR\map_{dg_+\FEFC_{R}}(A, B)= F(B).\qedhere
 \end{align*}
\end{proof}

\subsection{Nuclear Fr\'echet algebras}\label{nucfrechetsn}

In the following lemma, nuclear modules are taken in the sense of \cite[Definition 1.2]{freitag}. 

\begin{lemma}\label{nucfrechet}
 For a commutative  nuclear Fr\'echet algebra $R$, the chain complex underlying any object of $dg_+\hat{\Tc}\Alg_{R}^{\mathrm{afp,cof}}$ is a complex of nuclear Fr\'echet $R$-modules. 
\end{lemma}
\begin{proof}
When the set $S$ has finitely many elements in each degree, elements of $\cF_S$ are generated by words involving only finitely many variables in $S$ (those of the same degree and lower), so  $\cF_S$ is defined in each degree as a K\"othe sequence space, making it nuclear Fr\'echet. Any retract of a nuclear Fr\'echet space  is a subspace, hence also nuclear Fr\'echet.
\end{proof}

Combining Lemmas \ref{nucfrechet} and \ref{nucexactlemma}, we have:
\begin{corollary}\label{exactFEFCtencor}
Under the hypotheses of Lemma \ref{nucfrechet} and for any Fr\'echet space $U$, the completed tensor product $U\hten - $ defines a functor on $dg_+\hat{\Tc}\Alg_{R}^{\mathrm{afp,cof}}$ which preserves quasi-isomorphisms.
 \end{corollary}

\begin{definition}
 Given bounded below chain complexes $U_{\bt},V_{\bt}$ of Fr\'echet $\bK$-spaces, %with either all $U_n$ nuclear or all $V_n$ nuclear, 
 we define the complex $U\hten_{\pi} V$ of Fr\'echet spaces to be the total complex of the double complex $(i,j) \mapsto U_i \hten_{\pi} V_j$. When either all the spaces $U_n$  or all the spaces $V_n$ are nuclear, we simply denote this by $U\hten V$.
\end{definition}

\begin{definition}\label{hatedef}
 Given a dg %nuclear 
 Fr\'echet algebra $A$, write  $A^{\hat{e}}:=A\hten_{\pi} A^{\op}$ and let $\hat{\Omega}(A)$ be the $ A^{\hat{e}}$-module given by the kernel of the multiplication map $A^{\hat{e}}\to A$.
 
%  Given a morphism $C \to A$ of dg %nuclear 
%  Fr\'echet algebras, write $ \hat{\Omega}(A/C):= \cone(\hat{\Omega}(C) \hten_{C^{\hat{e}},\pi}A^{\hat{e}} \to  \hat{\Omega}(A))$.
 %%that's problematic: no guarantee that $\hat{\Omega}(C)$ is projective, so this isn't necessarily the derived cotangent complex. On the other hand, the natural underived definition is to take quotient. Only usd to end of {resnprop}
 
 % do we want to say $\oL\Omega(A/C):= \ker(A\hten_C^{\oL}A \to A)$ instead?? Need both nuclear if so.
\end{definition}

The following is an immediate consequence of Lemma \ref{nucexactlemma} and of total complexes preserving bounded below chain isomorphisms \cite[\S 5.6]{W}:

\begin{lemma}\label{OmeganucQIMlemma}
 If $A \to B$ is a morphism of dg nuclear Fr\'echet $\bK$-algebras in $dg_+\hat{\Tc}\Alg_{\bK}$ which is a quasi-isomorphism in the sense that it induces isomorphisms on abstract homology groups, then the morphisms $A^{\hat{e}} \to B^{\hat{e}}$ and $\hat{\Omega}(A) \to \hat{\Omega}(B)$ are also quasi-isomorphisms.
\end{lemma}

We also have the  following relative version of tensor products and Lemma \ref{nucexactlemma}, fixing for now a complete valued field $\bK$.

\begin{definition}\label{htenAdef}
 Take $A \in dg_+\hat{\Tc}\Alg_{\bK}$ with the $\bK$-vector spaces $A_n$ all nuclear Fr\'echet, together with bounded below chain complexes of Fr\'echet spaces $M,N$ with a complete right (resp. left)  $A$-module structure $M\hten A \to M$ (resp. $A\hten N \to N$). Assume moreover that   either all the $M_i$ or all the $N_i$ are nuclear $\bK$-spaces. 
 
 We then define $M\hten^{\oL}_AN$ to be the total complex of the bar construction
 \[
  \Tot( M\hten N \la M\hten A \hten N \la M \hten A \hten A \hten N \la \ldots),
 \]
 with differential given by completing the morphism  
 \[
 x_0\ten x_1\ten \ldots \ten x_i \ten x_{i+1} \mapsto \sum_{j=0}^i (-1)^j x_0 \ten \ldots \ten x_jx_{j+1} \ten \ldots \ten x_{i+1},
 \]
 noting that this complex is well defined because in each iterated tensor product there is only one potentially non-nuclear term.
 \end{definition}

\begin{lemma}\label{nucexactlemmarel} 
For $A,M,N$ as in Definition \ref{htenAdef}, the bifunctor
\[
 (M,N) \mapsto M\hten^{\oL}_AN
\]
preserves abstract  quasi-isomorphisms in $M$ and in $N$.

Moreover, for any $\bK$-Fr\'echet space $V$, with either $V$ or $N$ (resp. $M$) levelwise nuclear, the natural map $(V\hten A)\hten^{\oL}_AN \to V\hten N$ (resp. $M\hten^{\oL}_A(A\hten V) \to M\hten V$) is an abstract  quasi-isomorphism.
\end{lemma}
\begin{proof}
 By Lemma \ref{nucexactlemma}, the bifunctors $(M,N) \mapsto M\hten A^{\hten i}\hten N$ all preserve quasi-isomorphisms, since total complexes of bounded below chain multicomplexes also do so (iterating \cite[\S 5.6]{W}). Taking the the total complex of the bar double complex  with $M\hten A^{\hten i}\hten N$ in degree $i$ then gives the first statement.
 
For the second statement, the cone of the morphism $(V\hten A)\hten^{\oL}_AN \to V\hten N$ has a contracting homotopy given in degree $i$ by completing the morphism $(v\ten a_0) \ten a_1 \ten \ldots \ten a_i \ten n \mapsto (v\ten 1) \ten a_0 \ten a_1  \ten \ldots \ten a_i \ten n$.
\end{proof}

\begin{lemma}\label{Omeganuclemma}
 Given $F \in dg_+\hat{\Tc}\Alg_{R}^{\mathrm{afp,cof}}$ and a closed $2$-sided dg ideal $I \subset F$ with  quotient $A$, the complex 
 $ \oL\hat{\Omega}(A/F)_{[1]}:= \ker(\hat{\Omega}(F) \hten_{F^{\hat{e}}}A^{\hat{e}} \to  \hat{\Omega}(A))$
 %$ \hat{\Omega}(A/F)_{[1]}$ 
 is naturally abstractly quasi-isomorphic to %the total complex of the double complex
 the cone of the multiplication map
 \[
   I\hten^{\oL}_FI \to I.
 \]
% \[
%   I \la I\hten I \la I\hten F\hten I \la \ldots \la I\hten F^{\hten n-1}\hten I \la \ldots,
%  \]
% where the final map in the complex is multiplication and the higher terms are the bar construction.
 \end{lemma}
\begin{proof}
%  If we write $M\hten_F^{\oL}N$ for the complex 
%  \[
%  M\hten N \la M\hten F\hten N \la \ldots \la M\hten F^{\hten n} \hten N
% \]
% given by the bar construction, then Lemma \ref{nucexactlemma} implies taht  the biunctor $\hten_F^{\oL}$ preserves quasi-isomorphisms of bounded below chain complexes of $F$-modules in nuclear Fr\'echet spaces. Arguing as in \cite[Lemma 1.11]{freitag}, the natural map $M\hten_F^{\oL}N \to M\hten_FN$ is a quasi-isomorphism whenever $N$ is free as a left $A$-module or  $M$ is free as a right $A$-module.

We have a short exact sequence $0 \to \hat{\Omega}(F) \to F\hten F \to F \to 0$, and applying $A\hten_F-$ then $-\hten_F^{\oL}A$ thus gives a quasi-isomorphism from $A\hten_F\hat{\Omega}(F)\hten_FA $ to the cocone of $A\hten A \to A\hten_F^{\oL}A$ 
by Lemma \ref{nucexactlemmarel}, since $F$ and $\hat{\Omega}(F)$ are both free as right $F$-modules.
Thus $\oL \hat{\Omega}(A/F)$ is quasi-isomorphic to the cocone of $ A\hten_F^{\oL}A \to A$ 
since the $A\hten A$ terms cancel.

Now $A$ is quasi-isomorphic to the cone of $I \to F$, so substituting in the first copy of $A$ gives a  quasi-isomorphism  with $(I\hten_F^{\oL}A)_{[-1]}$, which in turn is quasi-isomorphic to 
\[
 \cone(I\hten_F^{\oL}I \to I\hten_F^{\oL}F)_{[-1]} \simeq \cone(I\hten_F^{\oL}I \to I)_{[-1]}.
\]
 \end{proof}

We can now identify a large class of dg nuclear Fr\'echet algebras satisfying the conditions of Corollary \ref{afpTcFEFCcor}.

\begin{proposition}\label{resnprop} 
Take  a complete valued field $\bK$, and assume $A$ is a topological dg $\bK$-algebra satisfying the following conditions:
 \begin{enumerate}
  \item $A$ can be written as a quotient of a levelwise finitely generated quasi-free  algebra $(\cF_S,\delta) \in dg_+\hat{\Tc}\Alg_{R}^{\mathrm{afp,cof}} $ by a  topologically closed $2$-sided  dg ideal $I$, with $\H_0I$  finitely generated as an $\cF_{S_0}\hten\cF_{S_0}$-module, and
  \item $A$ is almost perfect as an $A^{\hat{e}}$-module.
 \end{enumerate}
Then there exists an object $\tilde{A} \in dg_+\hat{\Tc}\Alg_{R}^{\mathrm{afp,cof}}$ and a  quasi-isomorphism $\tilde{A} \to A$ of topological $\bK$-algebras, with  $\tilde{A}^{e,\FEFC} \simeq A^{\hat{e}}$ and the FEFC cotangent module $\Omega^1_{\tilde{A}}$ quasi-isomorphic to  $\hat{\Omega}(A)$.
\end{proposition}
\begin{proof} 
 We will construct $\tilde{A}$ as the colimit of a system $(\cF_S,\delta) =F^{(0)} \to F^{(1)} \to F^{(2)} \to \ldots $ equipped with maps to $A$, where $F^{(n+1)}$ is constructed from $F^{(n)}$ by adding generators in degree $n+1$. The inductive hypothesis is that for the closed $2$-sided  dg ideal $I^{(n)}:= \ker(F^{(n)} \to A)$, we have $\H_{<n}I^{(n)}=0$. 
%  Equivalently, that condition says taht the map $F^{(n)} \to A$ gives an isomorphism on $\H_{<n}$ and is surjective on $\H_n$.
 
 Since the map $\tau_{\ge n}I^{(n)} \to I^{(n)}$ from the good truncation is a quasi-isomorphism of complexes of nuclear Fr\'echet spaces, repeated application of Lemma \ref{nucexactlemma} implies that $\H_{<2n}(I^{(n)}\hten I^{(n)})=0$ and $\H_{<n}(I^{(n)}\hten F^{(n)})=0$. Similarly, the quasi-isomorphism $\cone(I^{(n)} \to F^{(n)}) \to A$ gives a quasi-isomorphism  $\cone(I^{(n)} \to F^{(n)})^{\hten 2} \to A\hten A$. In particular, the inductive hypothesis implies that the maps $ F^{(n)} \to A$ and $(F^{(n)})^{\hat{e}} \to A^{\hat{e}}$ give isomorphisms on $\H_{<n}$, %and is surjective on $\H_n$
 so setting $\tilde{A}:= \LLim_n F^{(n)}$ gives an object with the desired properties.
 
We now proceed with the induction. The condition $\H_{<0}I^{(0)}=0$ is vacuous, and since $F^{(0)} \to A$ is surjective, the same is true of $\hat{\Omega}(F^{(0)}) \hten_{(F^{(0)})^{\hat{e}}}A^{\hat{e}} \to  \hat{\Omega}(A)$, so $\H_0\oL\hat{\Omega}(A/F^{(0)})=0$, 
meaning the conditions are satisfied for $n=0$.
 
Since $\H_0I^{(0)}$ is levelwise finitely generated by hypothesis, we may lift a finite set of generators to  $a_1, \ldots, a_m \in  I^{(0)}_0$, and form $F^{(1)}$ by freely adding variables $x_1, \ldots, x_m$ in degree $1$ and setting $\delta x_i =a_i$, then define the homomorphism from $F^{(1)}$ to $A$ by $x_i \mapsto 0$. By construction, we thus have $\H_{0}I^{(1)}=0$.
 
Now assume that we have constructed $F^{(n)} \onto A$ with  $\H_{<n}I^{(n)}=0$, for some $n \ge 1$. That 
condition  automatically implies $\H_{<2n}(I^{(n)}\hten^{\oL}_{F^{(n)}}I^{(n)})=0$, so the description of Lemma \ref{Omeganuclemma} gives  $\H_{\le n}\oL\hat{\Omega}(A/F^{(n)}) =0$. 
Since $A$ is almost perfect, so is  $\hat{\Omega}(A)=\ker(A^{\hat{e}} \to A)$. Since $F^{(n)}$ is levelwise finitely generated and quasi-free, the $A^{\hat{e}}$-module  $\oL\hat{\Omega}(A/F^{(n)})$ 
is also almost perfect; in particular, its lowest homology group $\H_{n+1}\oL\hat{\Omega}(A/F^{(n)})$ 
is finitely generated. % since the morphism $\oL\hat{\Omega}(A/F^{(n)}) \to  \H_{n+1}\oL\hat{\Omega}(A/F^{(n)})$ in the derived category must factor through a finitely generated submodule.
The exact sequence 
\[
 \to  \H_{n}I^{(n)} \to \H_{n+1}\oL\hat{\Omega}(A/F^{(n)}) \to   \H_{n-1}( I^{(n)}\hten^{\oL}_FI^{(n)})
\]
coming from  Lemma \ref{Omeganuclemma} allows us to lift a finite generating set to a set of elements of $\z_nI^{(n)}$. We then form $F^{(n+1)}$ from $F^{(n)}$ by freely adding a corresponding set of variables in degree $n+1$ by the same procedure as the $n=0$ case. 

By construction, we then have  $\H_{n+1}\oL\hat{\Omega}(A/F^{(n+1)})=0$ and automatically have $\H_{<n-1}I^{(n+1)}=0$, so $\H_{<2n-2}(I^{(n+1)}\hten^{\oL}_{F^{(n+1)}}I^{(n+1)})=0$. The exact sequence
 \[
  \H_{n}( I^{(n+1)}\hten^{\oL}_{F^{(n+1)}}I^{(n+1)}) \to  \H_{n}I^{(n+1)} \to \H_{n+1}\oL\hat{\Omega}(A/F^{(n+1)}) 
 \] 
 from Lemma \ref{Omeganuclemma} then
implies that  $\H_{n}I^{(n+1)}=0$, since $n+1\ge 2$. 
\end{proof}

\begin{example}[Finite-dimensional algebras]\label{fdex}
Every finite-dimensional $\bK$-algebra $A$ satisfies the conditions of Proposition \ref{resnprop}, so admits a levelwise finitely generated FEFC resolution compatible with the respective topologies. 

First note that the $A$-bimodule $A$ admits a projective resolution by the bar construction. Since that is freely generated in degree $n$ by the finite-dimensional vector space $A^{\ten n}$, it follows that $A$ is almost perfect. 

It now suffices to show that $A$ can we written as a quotient of some $\cF_n$ by a closed ideal which is finitely generated as an $\cF_n\hten \cF_n^{\op}$-module. Endow $A$ with a submultiplicative norm and choose a basis $e_0=1, e_1, \ldots, e_n$ of elements of norm $1$. We then have a  surjective continuous homomorphism $\phi \co \cF_n \to A$ given on generators by $\phi(x_i)=e_i$. The kernel $I$ consists of series $\sum_{\alpha} \lambda_{\alpha} x^{\alpha}$ such that $\lambda_i   + \sum_{|\alpha|>1}  \lambda_{\alpha} \phi_i(x^{\alpha})=0$ for all $i>0$ and $\lambda_{\emptyset}   + \sum_{|\alpha|>1}  \lambda_{\alpha} \phi_0(x^{\alpha})=0 $, where $\phi= \sum_{i=0}^n \phi_i e_i$. 

This is automatically a closed ideal, and it remains to show that it is finitely generated as  an $\cF_n\hten \cF_n^{\op}$-module. 
It is even finitely generated as a right $\cF_n$-module, with  generating set is given  by the elements $r_{ij}:= x_ix_j - \phi(x_ix_j)_0 - \sum_{k=1}^n \phi(x_ix_j)_k x_k$. For formal variables $t_{ij}$, we can define a $\bK$-linear map $\sigma$ from generators of $\cF_n$ to $\bigoplus_{i,j=1}^n t_{ij}\cF_n$ inductively on word length  by setting $\sigma(1)=\sigma(x_i)=0$ and
\[
\sigma(x_ix_jx^{\alpha}) :=  t_{ij} x^{\alpha} + \phi(x_ix_j)_0 \sigma(x^{\alpha}) +  \sum_{k=1}^n \phi(x_ix_j)_k \sigma(x_kx^{\alpha}).
\]
If $f(m)_r$ is the maximum value of $\|\sigma(x^{\alpha})\|_r$ for $\|\alpha\|=m$, this gives an expression of the form $f(m+2) \le r^m + \lambda f(m+1) + \mu f(m)$ for non-negative constants $\lambda, \mu$, with $f(0)=f(1)=0$.  For any $s \ge 3  \max\{\lambda, \sqrt{\mu},r\}$, it then follows by induction that $f(m+2) \le s^m$ for all $m$, so $\sigma$ extends to a continuous map on $\cF_n$  by Lemma \ref{extendmaplemma1}. When restricted to $I$, it gives a section of the map $t_{ij} \mapsto r_{ij}$, so the latter generate $I$.   
\end{example}

\begin{example}\label{fdex2}
 More generally, any levelwise finite-dimensional dg  $\bK$-algebra $A$ (concentrated in non-negative chain degrees) satisfies the conditions of Proposition \ref{resnprop}. The bar construction of the $A$-bimodule $A$ is now a double complex, and taking its total complex gives a cofibrant replacement which is finitely generated in each degree, so almost perfect. 
 
 Picking a basis levelwise for $A$, we can write it as a quotient of an quasi-free algebra $(\cF_S,\delta)$ with $S$ levelwise finite, with kernel $I$. From Example \ref{fdex} we know that $I_0$ is finitely generated as an $\cF_{S_0}\hten \cF_{S_0}^{\op}$-module, so $\H_0I$ must also be so. 
\end{example}

%%13/12/23: qn motivated by below: can we set up an ind-category where ind-QIM is the right notion, and get that as full subcategory of FEFC? Dagger affinoids the motivating example. Do we just accept the fudge of local existence instead?

\begin{example}[Free EFC and Stein algebras]\label{EFCexNC}
The kernel of the canonical surjective map $\pi \co \cF_n \to \cO^{\an}(\bA^n)$ from the free FEFC algebra to the free EFC algebra is canonically closed, and is generated as an $\cF_n\hten \cF_n^{\op}$-module by the commutators $[x_i,x_j]$. %First note that a section $s$ of $1$ is given by writing monomials in lexicographic order.  
To see this, observe that we can apply a form of division algorithm to render non-commutative monomials of degree $m$  in lexicographic order by subtracting the image under $\iota_{[x_i,x_j]} \co \cF_n\hten \cF_n \to \cF_n$ (as in Definition \ref{iotadef})
of a monomial of degree $m-2$ every time we have to pass $x_i$ through $x_j$ for $i>j$. Since there are at most $m(m-1)/2$ such steps, this defines a $\bK$-linear  map $\phi \co \bK\<x_1, \ldots, x_n\> \to (\cF_n\hten \cF_n)^{n(n-1)/2}$ with $\|\phi(x^{\alpha})\|_r \le |\alpha|(|\alpha|-1) r^{|\alpha|}/2$. Because polynomials grow sub-exponentially, Lemma \ref{extendmaplemma1} thus extends $\phi$ to a map on $\cF_n$, where it gives a section of $(\iota_{[x_i,x_j]})_{i,j}$ when restricted to $\ker \pi$, from which  the required finite generation follows. 

The $\cO^{\an}(\bA^n)\hten\cO^{\an}(\bA^n)^{\op}$-module $\cO^{\an}(\bA^n)$ has a projective resolution by the Koszul complex associated to the sequence $(x_i \ten 1 - 1 \ten x_i)_i$, so is perfect. Since $\ker \pi$ is finitely generated as an $\cF_n\hten \cF_n^{\op}$-module, the conditions of Proposition \ref{resnprop} are satisfied, so $\cO^{\an}(\bA^n)$  admits a levelwise finitely generated FEFC resolution compatible with the respective topologies. 

Much the same argument (incorporating $x^2$ terms for variables $x$ of odd degree) applies to levelwise finitely generated quasi-free dg EFC algebras, i.e. dg Stein algebras of the form $ (\cO^{\an}(\bA^n)[M], \delta)$ for levelwise finitely generated graded modules $M$ concentrated in strictly positive degrees. 
\end{example}

\begin{example}[Quantum affine space]\label{qEFCex}
As a non-commutative variant of Example \ref{EFCexNC}, we can consider the  algebra $\cO^{\an}_{\mathbf{q}}(\bA^n)$   of holomorphic functions on quantum affine $n$-space  from  \cite[Example 7.5]{pirkovskiiHFG}, where $\mathbf{q}= \{q_ij\}$ is a set of constants $q_{ij} \in \bK$ with $q_{ii}=1$ and $q_{ji}=q_{ij}^{-1}$, and $\cO^{\an}_{\mathbf{q}}(\bA^n)$ the quotient of $\cF_n$ by the closed ideal generated by the quantum commutators $c_{ij}:=x_ix_j-q_{ij}x_jx_i$. For simplicity, we impose the further restriction that $|q_{ij}|\ge 1$ for $i \le j$.

%As a non-commutative variant of Example \ref{EFCexNC}, we can consider the  algebra $\cO^{\an}_{q}(\bA^n)$   of holomorphic functions on quantum affine $n$-space  from  \cite[Definition 3.6]{pirkovskiiNC}, where  $q \in \bK$ is a constant with $|q|\ge 1$  and $\cO^{\an}_{q}(\bA^n)$ the quotient of $\cF_n$ by the closed ideal generated by the quantum commutators $c_{ij}:=x_ix_j-qx_jx_i$ for $i<j$. 

Since $x_jx_i = q_{ji}x_ix_j -c_{ij}$ for $i<j$,  we can adapt the  division algorithm of Example \ref{EFCexNC} by subtracting the image of a monomial under $\iota_{[x_i,x_j]} \co \cF_n\hten \cF_n \to \cF_n$ for $i<j$ whenever we have to pass $x_i$ through $x_j$. There are the same number of steps as before, but now the norm of the remainder is multiplied by $q_{ji}$ at each step. Since $|q_{ij}|\le 1$,   the required map $\phi_{\mathbf{q}} \co \cF_n \to (\cF_n\hten \cF_n)^{n(n-1)/2}$ converges and gives a section of  $(\iota_{[x_i,x_j]})_{i,j}$ when restricted to $\ker (\cF_n \to   \cO^{\an}_{\mathbf{q}}(\bA^n))$.
 
 The $\cO^{\an}_{\mathbf{q}}(\bA^n)\hten \cO^{\an}_{\mathbf{q}}(\bA^n)^{\op}$-module $\cO^{\an}_{\mathbf{q}}(\bA^n)$ has a free resolution given by modifying the Koszul complex. Explicitly, we have a basis $\{e_{i_1,\ldots,i_r}~:~ 1 \le i_1<i_2< \ldots <i_r\le n\}$  in degree $r$, with 
 \[
 \delta e_{i_1,\ldots,i_r}= \sum_{l=1}^r (-1)^{l-1} \left( (\prod_{m=1}^{l-1}q_{i_m,i_l}) x_{i_l}e_{i_1, \ldots, \hat{i}_l,\ldots, i_r}  -(\prod_{m=l+1}^{r}q_{i_l,i_m})  e_{i_1, \ldots, \hat{i}_l,\ldots, i_r}x_{i_l}\right),
 \]
 where $\hat{i}_l$ indicates deletion of that index.
 %%justification is to look at two variable case, with $xy=qyx$. Then $e_{12} \mapsto xe_2 -qe_2x -qye_1 +e_1y$, and $\delta$ sends that to $xy|1-x|y -qy|x+q|yx -qyx|1 + qy|x +x|y-1|xy = xy|1 - ?- ? + 1|xy - xy|1 + ? + ? -1|xy =0$, where $?$ denotes terms which cancel immediately.
  Since $\ker (\cF_n \to \cO^{\an}_{\mathbf{q}}(\bA^n))$ is finitely generated as an $\cF_n\hten \cF_n^{\op}$-module, the conditions of Proposition \ref{resnprop} are satisfied, so $\cO^{\an}_{\mathbf{q}}(\bA^n)$  admits a levelwise finitely generated FEFC resolution compatible with the respective topologies. 
\end{example}

\begin{example}[dg Stein algebras] \label{dgsteinexNC}
 Not all Stein spaces admit  global resolutions by levelwise finitely generated quasi-free dg EFC algebras, but Stein manifolds do so by iterating the argument of \cite[Remark \ref{DStein-embeddingrmk}]{DStein}, as do global complete intersections in such by taking Koszul complexes.  

 More generally, we can start by observing that for any dg dagger affinoid $(Z^0, \sO_Z)$, the $\Gamma(Z^0 \by Z^0, \sO_{Z \by Z})$-module $\Gamma(Z^0, \sO_Z)$, given by pushforward along the diagonal, is almost perfect, where $\sO_{Z \by Z} = \pr_1^{-1}\sO_Z\ten_{\pr_1^{-1}\sO_{Z^0}} \sO_{Z^0 \by Z^0}\ten_{\pr_2^{-1}\sO_{Z^0}}\pr_2^{-1}\sO_Z)$. This follows because $\Gamma(Z^0 \by Z^0, \sO_{Z^0 \by Z^0})$ is Noetherian, being dagger affinoid, while the modules $\Gamma(Z^0\by Z^0, \sO_{Z \by Z, n})$ and $\Gamma(Z^0, \sO_{Z, n})$ are coherent, hence finitely generated. A levelwise finitely generated free resolution can thus be constructed inductively, adding generators to give isomorphisms on $\H_{<n}$ and a surjection on $\H_n$. 
 
For any dg Stein space $X=(X^0, \sO_X)$ and any Stein subdomain $U \Subset X^0$ (as in Example \ref{dgsteinex}), the topological dg  algebra $A:= \Gamma(U,\sO_X)$ thus has the property that $A$ is almost perfect as an $A^{\hat{e}}$-module, since $A^{\hat{e}} \cong \Gamma(U\by U, \sO_{X \by X})$ and by hypothesis we have a compact Stein or dagger affinoid $Z$ with $U \subset Z \subset X^0$, so we can apply flat base change along $\Gamma(Z \by Z, \sO_{X^0 \by X^0}) \to \Gamma(U \by U, \sO_{X^0 \by X^0})$  to adapt the results of the previous paragraph.

If in addition $U$ is a globally finitely generated Stein space, then $\cO^{\an}(U)$ and hence $\Gamma(U, \H_0\sO_X)$ are quotients of some $\cO^{\an}(\bA^n)$ by finitely generated ideals, so Example \ref{EFCexNC} implies that they are quotients of $\cF_n$ by closed $2$-sided ideals which are finitely generated as $\cF_n^{\hat{e}}$-modules. Thus $\Gamma(U, \sO_X)$ satisfies the conditions of Proposition \ref{resnprop}, meaning any dg Stein space admits an open cover by objects satisfying those conditions.
\end{example}

%%some offcuts removed here

%%one to think about is $A\{x\} \onto B\{x\}$, where $A\{x\}$ is weighted $\ell^1$/$r^n$ sum of $A^{\hten n+1}$s. If $I= \ker (A \to B)$, can look at $(A^{\hat{e}})^m \onto I$, and then $G_n:=\bigoplus_{i=0}^n A^{\hten i} \hten  (A^{\hat{e}})^m\hten A^{\hten n-i} \onto \ker(A^{\hten n+1} \to B^{\hten n+1})$. Since surjective, quotient by kernel and use Banach--Schnauder to bound pre-image, so we can lift anything in $\ker(A\{x\} \onto B\{x\})$ to an elt of   $\ell^1$/$r^n$ sum of $G_n$s. Then observe $A\{x\} \hten(A\{x\}^{\hat{e}})\hten A\{x\}$ srjects onto that.

\begin{examples}\label{sq0ctrex}
To see the necessity of the finiteness conditions in  Proposition \ref{resnprop}, consider $B:= \bK \oplus V$, for $V$ a Fr\'echet space equipped with zero multiplication. Then $B^{e,\FEFC} \cong B \ten B$, the abstract tensor product, and this is quasi-isomorphic to  $\tilde{B}^{e,\FEFC}$ for any cofibrant replacement, as can be seen by writing the FEFC algebra underlying $B$ as the filtered colimit $\LLim_U(\bK \oplus U)$ of FEFC algebras, where $U$ ranges over all finite-dimensional subspaces of $V$. On the other hand, $B^{\hat{e}} \cong B\hten_{\pi} B$, so the constructions only agree when $V$ is finite-dimensional. 

However, in contrast to the situation for EFC and Stein algebras, if $A$ satisfies the conditions of Proposition \ref{resnprop} and $M$ is a finitely presented $A^{\hat{e}}$-module, then even the trivial square-zero extension $A \oplus M$ of $A$ seldom satisfies the conditions of that proposition. 
 
If $M$ has $n$ generators as an $A^{\hat{e}}$-module and $A$ is a quotient of $\cF_m$, then $A \oplus M$ can be written as a quotient of $\cF_{m+n}$, but the kernel is not usually finitely generated. The difficulty arises because the Fr\'echet algebra freely generated by $A$ and $M$ is  a completion of $\bigoplus_{i \ge 0} M^{\hten_A i}$, so we need to quotient by the closed $2$-sided ideal generated by the  $A^{\hat{e}}$-module $M\hten_AM$, which is seldom finitely generated.
%  but the defining ideal must contain all expressions involving  $\{x_{m+1}, \ldots , x_{m+n}\}$

Finite generation of $M\hten_AM$  is a property  not obviously closed under taking direct sums, so seems unlikely to give a well-behaved category of modules. Finiteness of $M$ as a left $A$-module or as a right $A$-module is sufficient to satisfy the condition that $M\hten_AM$ be finitely generated as an $A^{\hat{e}}$-module, which is why we do not encounter this phenomenon in the commutative setting. A much larger compliant  example is given when $A= B\hten B$ by taking $M= B\hten B \hten B$, with left action on the first two factors and right action on the last two, since then $M\hten_AM =  B\hten B \hten B\hten B =A\hten A$. 
\end{examples}

%%%Important to note that we still won't get analogue of dagger affinoids working, because differentials aren't strict, so we don't have $\H_*(A\hten V) \cong (\H_*A)\hten V$ unless we can also impose Noetherianity.

\begin{remark}[Non-commutative Stein spaces and dagger affinoids]\label{NCSteinrmk}
It seems plausible that in many cases the conditions of Proposition \ref{resnprop} should be closed under coproducts in the category of LDMC dg  algebras, the issue being closure of the $\cF_{m+n}\hten \cF_{m+n}^{\op}$-submodule of $\cF_{m+n}$ generated by closed $2$-sided ideals $I \subset \cF_m$ and $J \subset \cF_n$. 

The free non-commutative polydisks $\cF(\bD^n_R)$ of \cite[\S 7.2]{pirkovskiiHFG} are examples of this construction, being defined as  coproducts of the Stein algebras of holomorphic functions on open discs. If they  satisfy the conditions of Proposition \ref{resnprop},  their FEFC cotangent complexes will simply be free modules of rank $n$. Beware that they are not given by the na\"ive definition in terms of the norms $\|-\|_{\rho}$ of \S \ref{FEFCalgsn}  for $\rho >R$,  which do not produce a nuclear space. 

The latter phenomenon leads to apparently intractable issues in trying to define non-commutative dagger algebras meaningfully. We might like submultiplicative seminorms on $\cF_n$ to be analogous to  compact subspaces, as happens for $\cO^{\an}(\bA^n)$. However, the dense  morphism between the completions of $\cF_n$ with respect to norms $\|-\|_{\rho}$ and $\|-\|_{\rho'}$ is only nuclear for $\rho'>n\rho$, because the number of non-commutative monomials of degree $d$ is $n^d$; the corresponding expression in the commutative case is the polynomial $\binom{n+d-1}{d}$ in $d$, which is why we have nuclear morphisms whenever $\rho'>\rho$ in that case. Similar phenomena arise for any submultiplicative norms on $\cF_n$, in contrast to $\cO^{\an}(\bA^n)$ where every submultiplicative norm is a limit of relatively nuclear  submultiplicative norms. Consequently, we have a dearth of nuclear DF completions of $\cF_n$, which is what we would expect  non-commutative dagger algebras to be. Na\"ive constructions such as $\LLim_{R>R_0} \cF(\bD^n_R)$  are only ind-Fr\'echet, not DF or ind-Banach. 

There is the further setback that Noetherianity is a very demanding condition in the non-commutative setting, yet is key to the
dagger affinoid theory and  analytification described in \S \ref{affinoidsn}, ensuring for instance that differentials are strict; it is also key in GAGA results (cf. Remark \ref{esssurjrmk}). We might tender a definition of non-commutative dagger affinoids as  topological Noetherian  algebras arising as  nuclear DF completions of $\cF_n$, i.e. as direct limits  of systems of  completions with respect to submultiplicative seminorms  for which the transition maps are all nuclear.  
Such algebras would have to exhibit growth similar to commutative algebras, and 
it is hard to envisage many such objects which are close enough to being cofibrant to have  projective cotangent complexes, the only obvious examples being  finite-dimensional semisimple algebras and their tensor products  with algebras of functions on one-dimensional smooth dagger affinoid spaces. 
\end{remark}

\subsection{Homogeneity and cotangent complexes}

\begin{definition}\label{dgfrdef}
Given a commutative Fr\'echet algebra $R$, write  $dg_+\cF r\Alg(R) \subset dg_+\hat{\Tc}\Alg_{R}$ and  $dg_+\cF r\CAlg(R) \subset dg_+\hat{\Tc}\CAlg_{R}$  for the full subcategories of objects $A_{\bt}$ for which the spaces $A_n$ are all Fr\'echet. When $R$ is nuclear, write  $dg_+\cN\cF r\Alg(R) \subset dg_+\cF r\Alg(R)$ and $dg_+\cN\cF r\CAlg(R) \subset dg_+\cF r\CAlg(R)$ for the full subcategories of objects $A_{\bt}$ for which the $R$-modules $A_n$ are all nuclear in the sense of \cite[Definition 1.2]{freitag}.

For a commutative algebra $R$, write $dg_+\Alg(R)$ for the category of dg associative $R$-algebras concentrated in non-negative chain degrees, and  $dg_+\CAlg(R)$ for the subcategory of graded-commutative algebras.

%% need to add DG^+dg_+ versions too. [Not with current layout.] 
\end{definition}
% 
% \begin{definition}\label{dgalgdef}
% For a  commutative algebra $R$, we write $dg_+\Alg(R)$ for the category of associative unital $R$-algebras $A_{\bt}= (\ldots \xra{\delta} A_1 \xra{\delta} A_0)  $ in  non-negatively graded chain complexes. We write $dg_+\CAlg(R)$ for the full subcategory consisting of objects $A$ on which the multiplication is graded-commutative.  
% \end{definition}

\begin{definition}\label{sq0def}
 Say that a morphism $A \to B$ in any of the categories  in Definition \ref{dgfrdef} is a square-zero extension if the maps $A_n \to B_n$ are all surjections and the product of any elements in the kernel is $0$. 
\end{definition}

%%note that extension of nuclear Fr\'echet is automatically nuclear

\begin{definition}\label{hhgsdef}
We say that a %%w.e.-preserving 
functor $F$ from any of the categories above 
% a category with finite limits %%square-zero extensions don't make sense in general 
to a model category 
is homogeneous if  for all square-zero extensions $A \to B$ and all maps $C \to B$, the natural map
$$
F(A\by_BC) \to F(A)\by^h_{F(B)}F(C)
$$
to the homotopy fibre product is a weak  equivalence.
\end{definition}

\begin{lemma}\label{obslemma}
 Given a w.e.-preserving homogeneous functor on any of the categories above and a square-zero extension $e\co A \to B$ with kernel $I$, there is a natural map $o_e \co F(B) \to F(B \oplus I_{[-1]})$ in the homotopy category over $F(B)$, with
 \[
  F(A) \simeq F(B)\by^h_{o_e, F(B \oplus I_{[-1]}),F(i_0)}F(B),
 \]
where $i_0(b):=(b,0) \in B \oplus I_{[-1]}$.
\end{lemma}
\begin{proof}
 This is a standard obstruction theory argument. We let $\tilde{B}:=\cone(I \to A)$ equipped with the obvious multiplication,  noting that the natural map $\tilde{B} \to B$ is an abstract quasi-isomorphism. The obstruction map $o_e$ is then induced by the morphism $\tilde{B} \to 
B \oplus I_{[-1]}$. Since that map is surjective with square-zero kernel $I$, homogeneity gives
\[
 F(A) \simeq F(\tilde{B})\by^h_{o_e, F(B \oplus I_{[-1]}),F(i_0)}F(B),
\]
since $A \cong \tilde{B}\by_{B \oplus I_{[-1]}}B$. Because $F$ is w.e.-preserving, we then complete the prof by identifying $F(B)$ with $F(\tilde{B})$.
 \end{proof}

% \subsection{Tangent and cotangent complexes}
 
 \begin{definition}\label{Tdef}
Given a w.e.-preserving 
homogeneous  functor $F$ from any of the categories in Definition \ref{dgfrdef}
% 
% $dg_+\Alg(\bK)$, $dg_+\CAlg(\bK)$, $dg_+\hat{\Tc}\Alg_{R}$,  $dg_+\hat{\Tc}\CAlg_{R}$, $dg_+\cF r\Alg(\bK)$, $dg_+\cF r\CAlg(\bK)$, $dg_+\cN\cF r\Alg(\bK)$ or $dg_+\cN\cF r\CAlg(\bK)$ 
to simplicial sets,  an object $A$ of the category   and a point $x \in F(A)$, we define the tangent functor $T_xF$
by
$$
T_xF(M):= F(A\oplus M)\by^h_{F(A)}\{x\}.
$$

Here, $A \oplus M$ is a Beck module, i.e.  an abelian group object in the slice category over $A$. Explicitly, this corresponds to $M$ in each of the cases above being a non-negatively graded chain complex in the relevant category of (topological) vector spaces, equipped with zero multiplication and being either an $A$-module (for  $dg_+\CAlg(\bK)$,  $dg_+\hat{\Tc}\CAlg_{R}$,  or $dg_+\cF r\CAlg(\bK)$) or $A$-bimodule (for  $dg_+\Alg(\bK)$,  $dg_+\hat{\Tc}\Alg_{R}$,  or $dg_+\cF r\Alg(\bK)$).

% As for instance in \cite[Lemma \ref{drep-adf}]{drep}, the space $T_xF(M[1])$ deloops $T_xF(M)$, so we may  define tangent cohomology groups by $\DD^{n-i}_x(F,M):= \pi_i (F(A\oplus M[n])\by^h_{F(A)}\{x\})$. 
\end{definition}

 \begin{definition}\label{Fcotdef}
 We say that a w.e.-preserving homogeneous  functor  $F$ from $dg_+\cF r\Alg(\bK)$ or $dg_+\cN\cF r\Alg(\bK)$ (resp. $dg_+\cF r\CAlg(\bK)$ or $dg_+\cN\cF r\CAlg(\bK)$) to simplicial sets has a perfect cotangent complex $\bL^{F,x}$ at $x$ if the functor  $T_x(F)$ is homotopically representable by a perfect $A^{\hat{e}}$-module (resp. $A$-module) $\bL^{F,x}$ in chain complexes, in the sense that the simplicial mapping space 
 \[
N^{-1}\tau_{\ge 0} \oR\HHom_{A^{\hat{e}}}(\bL^{F,x},-) \quad\text{(resp.}\quad N^{-1}\tau_{\ge 0} \oR\HHom_{A}(\bL^{F,x},-) \text{ )}
 \]
is weakly equivalent to $T_xF$, as a functor on non-negatively graded Fr\'echet or nuclear Fr\'echet $A$-bimodules (resp. Fr\'echet or nuclear Fr\'echet  $A$-modules) in chain complexes. Here $\oR\HHom$ is taken in the abstract category of modules in complexes, forgetting the topological structure on the target, and $A^{\hat{e}}= A\hten_{\pi}A^{\op}$ as in Definition \ref{hatedef}.

We then write $\bT_xF(M):=  \oR\HHom_{A^{\hat{e}}}(\bL^{F,x},M)$, the non-truncated tangent complex; this can equivalently be obtained as the spectrum $\{T_xF(M_{[-n]})\}_n$, identifying $ T_xF(M_{[-n]}) $ with the loop space of $ T_xF(M_{[-n-1]}) $ by homogeneity.

We make the same definitions for functors on $dg_+\Alg(\bK)$ and $dg_+\CAlg(\bK)$, except that in the former case the complex is a module over the algebraic tensor product $A^e:= A\ten A^{\op}$ and defined accordingly.

%%% add almost perfect as well?
 \end{definition}

 The following lemma can be interpreted as saying that for a functor on dg nuclear Fr\'echet algebras, the points with perfect cotangent complexes form an open subfunctor.
 \begin{lemma}\label{opencotlemma}
 Take a w.e.-preserving homogeneous  functor  $F \co dg_+\cN\cF r\Alg(\bK)\to s\Set$, a square-zero extension $A \to B$ in $dg_+\cN\cF r\Alg(\bK)$ and a point $x \in F(A)$. If $F$ has a perfect cotangent complex at the image $\bar{x} \in F(B)$ of $x$, then it has a perfect cotangent complex at $x$.
 \end{lemma}
\begin{proof}
It suffices to show that $ \bT_xF(A^{\hat{e}})$ is perfect as an $A^{\hat{e}}$-module, with the natural map $\bT_xF(A^{\hat{e}})\ten^{\oL}_{A^{\hat{e}}}M \to  \bT_xF(M)$ a quasi-isomorphism for all $A^{\hat{e}}$-modules $M$, since then the cotangent complex must be given by the dual of $ \bT_xF(A^{\hat{e}})$.

We have a decreasing filtration on $M$ by closed submodules, with $\Fil^0M=M$, $\Fil^1M= \ker(M \to B\hten_AM\hten_AB)$, $\Fil^2M= \ker(M \to (M\hten_AB) \by (B\hten_AM))$ and $\Fil^3M=0$. Letting $I=\ker(A \to B)$, Lemma \ref{nucexactlemma1} %%worth noting that nuclearity needed for exactness in the middle. $\hten_{\pi}$ always exact on right/left and $\hten_{\vareps}$ on left.
 implies that $\Fil^2A^{\hat{e}}\cong I \hten I$, with $\gr_{\Fil}^1 A^{\hat{e}} \cong (I\hten B) \by (B \hten I)$ and $\gr _{\Fil}^0 A^{\hat{e}} \cong B^{\hat{e}}$. 
%  In particular, $I(\Fil^jM) \subset \Fil^{j+1}M$ and $(\Fil^jM)I \subset \Fil^{j+1}M$. %%%need to prove $I\Fil^1\subset \Fil^2$. 
 
Because $\Fil^1A^{\hat{e}}$ acts trivially on $\gr_{\Fil}^jM$, the latter is naturally a  $B^{\hat{e}}$-module, so $ \bT_xF(\gr_{\Fil}^jM) \simeq \bT_{\bar{x}}F(\gr_{\Fil}^jM)$. Since $F$ has a perfect cotangent complex at $\bar{x}$, we thus have $\bT_xF(\gr_{\Fil}^jM) \simeq  \bT_{\bar{x}}F(B^{\hat{e}})\ten^{\oL}_{B^{\hat{e}}}\gr_{\Fil}^jM$. 

In particular, this implies that  $\bT_xF(\gr_{\Fil}^jA^{\hat{e}}) \simeq  \bT_{\bar{x}}F(B^{\hat{e}})\ten^{\oL}_{B^{\hat{e}}}\gr_{\Fil}^jA^{\hat{e}}$, so 
\[
 (\bT_xF(\gr_{\Fil}^*A^{\hat{e}})\ten^{\oL}_{\gr_{\Fil}^*A^{\hat{e}}}\gr_{\Fil}^*M \simeq \bT_{\bar{x}}F(B^{\hat{e}})\ten^{\oL}_{B^{\hat{e}}}\gr_{\Fil}^*M \simeq \bT_xF(\gr_{\Fil}^*M).
\]

Now, since $F$ is w.e.-preserving, we have  convergent spectral sequences
 \begin{align*}
  \pi_i\bT_xF(\gr_{\Fil}^jM) &\abuts \pi_{i}\bT_xF(M),\\
 \pi_i((\bT_xF(\gr_{\Fil}^*A^{\hat{e}})\ten^{\oL}_{\gr_{\Fil}^*A^{\hat{e}}}\gr_{\Fil}^*M)^j &\abuts \pi_i (\bT_xF(A^{\hat{e}})\ten^{\oL}_{A^{\hat{e}}}M),
 \end{align*}
the latter induced by the Moore spectral sequence. The equivalences above thus induce an equivalence   $\bT_xF(A^{\hat{e}})\ten^{\oL}_{A^{\hat{e}}}M \to  \bT_xF(M)$.

It only remains to show that $ \bT_xF(A^{\hat{e}})$ is perfect as an $A^{\hat{e}}$-module, but this follows by a similar spectral sequence argument applied to the map to its double dual, since $ \gr_{\Fil}^*\bT_xF(A^{\hat{e}}) \simeq \bT_{\bar{x}}F(B^{\hat{e}})\ten^{\oL}_{B^{\hat{e}}}\gr_{\Fil}^*A^{\hat{e}}$ is perfect as a $\gr_{\Fil}^*A^{\hat{e}}$-module.
  \end{proof}

\subsection{Examples of derived moduli functors}\label{moduliexsn} %%%or say moduli functors

% \begin{definition}\label{dgfrdef} %%moved earlier
% Given a commutative Fr\'echet algebra $R$, write  $dg_+\cF r\Alg(R) \subset dg_+\hat{\Tc}\Alg_{R}$ and  $dg_+\cF r\CAlg(R) \subset dg_+\hat{\Tc}\CAlg_{R}$  for the full subcategories of objects $A_{\bt}$ for which the spaces $A_n$ are all Fr\'echet. When $R$ is nuclear, write  $dg_+\cN\cF r\Alg(R) \subset dg_+\cF r\Alg(R)$ and $dg_+\cN\cF r\CAlg(R) \subset dg_+\cF r\CAlg(R)$ for the full subcategories of objects $A_{\bt}$ for which the $R$-modules $A_n$ are all nuclear in the sense of \cite[Definition 1.2]{freitag}.
%  
% \end{definition}

%%22/5/24: note that we don't need all the nuclear Fr\'echet results for Poincar\'e lemma. I'd like to say we have map on analytifications from algebraic de Rham stack to algebraic Betti stack, just going via $\ten \to \hten_{\pi}$, since Poincar\'e works on any mc convex. Mebbe just claim the result, say to appear in a manuscript in preparation, details available on request.

%%I think it's best to put analytification in as a subsection here, with Poincar\'e lemma, referring to {analyticRH}. the map itseltf is simple enough, sending $(\sE, \nabla)$ to $\Omega^{\bt}_{X^{\an}}(\sE^{\an}, \nabla^{\an})$. It's worth introducing and naming the analytic de Rham functor first. 

%%then also want to think about things like Dolbeault--Poincar\'e, which is where more genl stuff below will come in.

The following  construction gives  a large class of examples of non-commutative derived analytic moduli functors:
\begin{example}\label{CXex}
 Given a category ${\bI}$,  a presheaf $\sC$ of commutative Fr\'echet $\bK$-algebras  on ${\bI}$ for $\bK$ a complete valued field, and a functor $F$ from the category $dg_+\Alg(\bK)$ (resp. $dg_+\CAlg(\bK)$) to simplicial sets, taking homotopy limits over ${\bI}$ gives us a functor
 \[
A \mapsto \ho\Lim_{\bI^{\op}} F(\sC\hten_{\eps} A)
 \]
from $dg_+\cF r\Alg(\bK)$ (resp. $dg_+\cF r\CAlg(\bK)$), in the notation of Definition \ref{dgfrdef}, to simplicial sets, defined up to natural weak equivalence, where $\hten_{\eps}$ is applied separately in each chain degree. 

Our motivating example of this form is when ${\bI}$ is the  affine pro-\'etale site of a scheme $X$ and $\sC$ is the sheaf $\uline{\bK}_X$ of \cite[Example 4.2.10]{BhattScholzeProEtale}, which sends $U$ to the Banach $\bK$-algebra of continuous $\bK$-valued functions on the pro-finite set  $\pi_0U$  of components of  the quasi-compact quasi-separated scheme $U$. By \cite[Theorem 44.1]{treves}, we then have a natural isomorphism $\uline{\bK}_X \hten_{\eps} A \cong \uline{A}_X$ for $A \in dg_+\cF r\Alg(\bK)$ since the spaces $A_n$ are complete. The resulting functor can be thought of as moduli of $F$-valued sheaves on the pro-\'etale site of $X$, so can be used to study $\ell$-adic local systems and Galois representations along similar lines to the commutative moduli functors of  \cite{PTLag}.
\end{example}
% 
%%moved back to tgt sn
% \begin{definition}
% Say that a functor from any of the categories in Definition \ref{dgfrdef} to a model category is w.e.-preserving if it maps quasi-isomorphisms to weak equivalences. 
% \end{definition}
% 
% \begin{definition}
%  Say that a morphism $A \to B$ in any of the categories  in Definition \ref{dgfrdef} is a square-zero extension if the maps $A_n \to B_n$ are all surjections and the product of any elements in the kernel is $0$. 
% \end{definition}
% 
% %%%note that extension of nuclear Fr\'echet is automatically nuclear
% 
% \begin{definition}\label{hhgsdef}
% We say that a %%w.e.-preserving 
% functor $F$ from any of the categories above 
% % a category with finite limits %%square-zero extensions don't make sense in general 
% to a model category 
% is homogeneous if  for all square-zero extensions $A \to B$ and all maps $C \to B$, the natural map
% $$
% F(A\by_BC) \to F(A)\by^h_{F(B)}F(C)
% $$
% to the homotopy fibre product is a weak  equivalence.
% \end{definition}

\begin{lemma}\label{CXlemma1}
 In the setting of Example \ref{CXex}, when the functor $F$ is w.e.-preserving or homogeneous, the same is true of the functor $\ho\Lim_{\bI^{\op}} F(\sC\hten -)$ on $dg_+\cN\cF r\Alg(\bK)$ (resp. $dg_+\cN\cF r\CAlg(\bK)$). If moreover $\sC$ is nuclear, then the functor  $\ho\Lim_{\bI^{\op}} F(\sC\hten -)$ on $dg_+\cF r\Alg(\bK)$ (resp. $dg_+\cF r\CAlg(\bK)$) is w.e.-preserving or homogeneous respectively.
 \end{lemma}
\begin{proof}
We give the  proof of the non-commutative statements; the commutative arguments are identical.

 If $A \to B$ is a quasi-isomorphism in  $dg_+\cN\cF r\Alg(\bK)$, then  $\sC\hten A \to \sC\hten B$ is also a quasi-isomorphism, by  Lemma \ref{nucexactlemma}, so $F(\sC\hten A) \to F(\sC\hten B)$ is a weak equivalence whenever $F$ is w.e.-preserving. 
 
 If $A \to B$ is a square-zero extension in $dg_+\cN\cF r\Alg(\bK)$ with kernel $I$, then $\sC\hten A \to \sC\hten B$ is also a square-zero extension, with kernel $\sC\hten I$,  by applying  Lemma \ref{nucexactlemma} to the complexes $I_n \to A_n \to B_n$. For any morphism 
$C \to B$ in  $dg_+\cN\cF r\Alg(\bK)$, we then have a small extension $A\by_BC \to C$ in $dg_+\cN\cF r\Alg(\bK)$ with kernel $I$ (noting that products and closed subspaces of nuclear spaces are nuclear), so $\sC\hten (A\by_BC) \to \sC\hten C$ is also a square-zero extension, with kernel $\sC\hten I$. Thus $\sC\hten (A\by_BC) \cong (\sC\hten A)\by_{(\sC\hten B)}(\sC\hten C)$, so $F(\sC\hten (A\by_BC)) \simeq  F(\sC\hten A)\by^h_{F(\sC\hten B)}F(\sC\hten C)$ whenever $F$ is homogeneous.

When $\sC$ is nuclear, the same arguments hold with $dg_+\cF r\Alg(\bK)$ replacing $dg_+\cN\cF r\Alg(\bK)$ throughout.
\end{proof}

\begin{example}\label{CXex2} 
 Given a category ${\bI}$,  a presheaf $\sC^{\bt}$ of dg algebras in Fr\'echet $\bK$-spaces on ${\bI}$, concentrated in non-negative cochain degrees, for $\bK$ a complete valued field, and a functor $F$ from the category $dg_+\Alg(\bK)$ to simplicial sets, 
 applying the construction $D_*$ of Definition \ref{Dlowerdef} below %\cite[Definition \ref{smallet-Dlowerdef}]{smallet} %%that's now premature, but I think we let it slide
 gives us
 a functor
 \[
A \mapsto \ho\Lim_{\bI^{\op}} D_*F(\sC^{\bt}\hten_{\eps} A) = \ho\Lim_{n \in \Delta}\ho\Lim_{\bI^{\op}} F((D^n\sC^{\bt})\hten_{\eps} A)
 \]
from $dg_+\cF r\Alg(\bK)$ to simplicial sets, defined up to natural weak equivalence, where $\hten_{\eps}$ is applied separately in each  degree and $D$ is the cosimplicial denormalisation functor. 

We can regard this as a special case of Example \ref{CXex}, since we have taken the homotopy limit of a diagram indexed by ${\bI}^{\op} \by \Delta$.

Our motivating example of this form is when $X$ is a manifold, ${\bI}$ a countable basis of contractible open submanifolds,  and $\sC^{\bt}_{\bI}$ is the sheaf $\sA^{\bt}_X$ of smooth  differential forms, equipped with its de Rham differential.  The Fréchet algebra $\sA^0_X$ is nuclear, and by  \cite[Theorem 44.1]{treves}, $\sA^0_X \hten B$ is isomorphic to the complex of smooth $B$-valued functions on $X$ for $B \in dg_+\cF r\Alg(\bK)$,  since the spaces $B_n$ are complete. As an immediate consequence,  $\sA^n_X \hten B$ is isomorphic to the space of smooth $B$-valued $n$-forms on $X$. 

When $F$ parametrises perfect complexes, this construction parametrises $\sA^0_X \hten B$-modules with flat $\infty$-connection. Specifically, writing $\cD_{dg}(A)$ for the dg category of all cofibrant right $A$-modules in complexes, the proof of Lemma \ref{doublecomplexlemma} gives an equivalence between $D_*\cD_{dg}(\sA^{\bt}_X \hten B)$ and the dg category of $\sA^{\bt}_X \hten B$-modules $M^{\bt}_{\bt}$ in double complexes which are  homotopy Cartesian in the sense that the maps $M^0_{\bt}\ten_{ \sA^0_X \hten B}(\sA^n_X \hten B) \to M^n_{\bt}$ are weak equivalences for all $n$.

% denormalisation gives a functor from $\sA^{\bt}_X \hten B$-modules $M^{\bt}_{\bt}$ in double complexes, concentrated in non-negative cochain degrees, to $(D\sA^{\bt}_X) \hten B$-modules in cosimplicial chain complexes. This preserves levelwise weak equivalence in the chain direction, and also preserves the property of being homotopy-Cartesian, which on the left means that the maps $M^0_{\bt}\ten_{ \sA^0_X \hten B}(\sA^n_X \hten B) \to M^n_{\bt}$ are weak equivalences. Working up the tower of brutal cotruncations of $\sA^{\bt}_X \hten B$ in the cochain direction then establishes an equivlaence between the respective $\infty$-categories of homotopy-Cartesian objects.
\end{example}

Lemma \ref{CXlemma1} combined with exactness of $D$ then immediately gives:
\begin{lemma}\label{CXlemma2}
 In the setting of Example \ref{CXex2}, when the functor $F$ is w.e.-preserving or homogeneous, the same is true of the functor $\ho\Lim_{\bI^{\op}} D_*F(\sC\hten -)$ on $dg_+\cN\cF r\Alg(\bK)$. If moreover the spaces $\sC^n$ are nuclear, then the functor  $\ho\Lim_{\bI^{\op}}D_* F(\sC\hten -)$ on $dg_+\cF r\Alg(\bK)$ is w.e.-preserving or homogeneous respectively.
 \end{lemma}

 \subsubsection{Existence of perfect tangent complexes}
 
The infinitesimal behaviour of  functors with perfect tangent complexes is algebraic in nature, and we now establish fairly general conditions under which this property holds. These rely on countable exactness properties of completed tensor products which are not shared with algebraic tensor products.
 
 \begin{lemma}\label{RGammalemma}
Take a countable category ${\bI}$, a  $dg_+\cF r\CAlg(\bK)$ (resp. $dg_+\cN\cF r\CAlg(\bK)$)
-valued presheaf $\sC^{\bt}$ on ${\bI}$, concentrated in non-negative cochain degrees,  $A \in dg_+\cN\cF r\Alg(\bK)$
and a locally perfect  %homotopy-Cartesian 
right $\sC\hten A$-module $\sT$ in complexes
  with $\ho\Lim_{\bI^{\op}} \sT$ perfect as a right $A$-module.
  
  Then for any left $A$-module $M$ in $dg_+\cN\cF r_{\bK}$ (resp. $dg_+\cF r_{\bK} $), the morphism
  \[
   (\ho\Lim_{\bI^{\op}} \sT)\ten^{\oL}_AM \to \ho\Lim_{\bI^{\op}} (\sT\ten^{\oL}_{(\sC\hten A)}(\sC\hten M) ) 
  \]
of derived algebraic tensor products is a quasi-isomorphism.
 \end{lemma}
\begin{proof}
Since ${\bI}$ is countable (i.e. has countably many objects and morphisms), its nerve $B{\bI}$ has countably many elements in each degree, and we may calculate the homotopy limit as the product total complex associated to the semi-cosimplicial chain complex
\[
 \prod_{i_0 \in {\bI}} \sT(i_0) \implies \prod_{i_0 \to i_1} \sT(i_1) \Rrightarrow \ldots \ldots \prod_{i_0 \to i_1 \to \ldots \to i_n} \sT(i_n) \ldots, 
\]
where the products are indexed by elements of  strings of  morphisms in ${\bI}$.

Since each $\sT(i)$ is assumed perfect as an $\sC(i)\hten A$-module, for each string $v \in B_m{\bI}$, with final vertex $\bar{v}$ (so $\bar{v}= i_m$ when $v = (i_0 \to \ldots \to i_m)$) we may inductively choose a quasi-isomorphic finitely generated cofibrant right $ \sC(\bar{v}) \hten A$-module $\tilde{T}(v)$, equipped with compatible face maps $\pd^i \co \tilde{T}(\pd_iv) \to \tilde{T}(v)$. These fit together to give a 
 semi-cosimplicial chain complex
\[
 \tilde{T}^0 \implies \tilde{T}^1  \Rrightarrow \ldots \ldots ~\tilde{T}^n \ldots,
\]
with $\tilde{T}^m:=  \prod_{v \in B_m{\bI} } \tilde{T}(v)$. % for   $ C^m :=\prod_{i_0 \to i_1 \to \ldots \to i_m} \sC(i_m)$-module.

Moreover, since the homotopy limit is assumed perfect as a right $A$-module, we may choose a finitely generated cofibrant replacement $T'$ and a quasi-isomorphism
$
 \phi \co T' \to \Tot^{\Pi} \tilde{T}^{\bt}
$
of right $A$-modules.
The countability and finiteness hypotheses allow us to regard these as complexes of Fr\'echet right $A$-modules, and  since $T'$ is finitely generated, $\phi$ is automatically continuous. 
 
 The morphism $T'\hten^{\oL}_AM  \to \Tot^{\Pi} (\tilde{T}^{\bt})\hten^{\oL}_AM $ on derived completed tensor products induced by $\phi$ is thus a quasi-isomorphism by Lemma \ref{nucexactlemmarel}, noting that arbitrary products of nuclear spaces are nuclear by \cite[Proposition 50.1]{treves}. Since completed tensor products commute with products and our boundedness hypotheses mean we can express the bar construction  $\hten^{\oL}$ as a product total complex, the morphism $ \Tot^{\Pi}(\tilde{T}^{\bt})\hten^{\oL}_AM \to \Tot^{\Pi}(\tilde{T}^{\bt}\hten^{\oL}_AM)$ is an isomorphism. The finiteness and cofibrancy hypotheses ensure that the natural maps $\tilde{T}(v)\hten^{\oL}_AM \to \tilde{T}(v)\hten_AM$ and $T'\hten^{\oL}_AM \to T'\hten_AM $ are quasi-isomorphisms (again by  Lemma \ref{nucexactlemmarel}), so it only remains to observe that $T'\hten_AM\cong T'\ten_AM$ and 
 \[
 \tilde{T}^m\hten_AM \cong  \prod_{v \in B_m{\bI}} (  \tilde{T}(v) \hten_AM) \cong  \prod_{v \in B_m{\bI}} (  \tilde{T}(v) \ten_{(\sC(\bar{v})\hten A)}(\sC(\bar{v})\hten M)   )
\]
by finiteness. 
 \end{proof}

\begin{proposition} \label{RGammacotprop}
Take a countable category ${\bI}$, 
a presheaf $\sC$ of commutative Fr\'echet $\bK$-algebras  on ${\bI}$ for $\bK$ a complete valued field, and a simplicial set-valued w.e.-preserving homogeneous functor  $F$ on the category $dg_+\Alg(\bK)$ (resp. $dg_+\CAlg(\bK)$).

Take $A \in dg_+\cF r\Alg(\bK)$ (resp. $A \in dg_+\cF r\CAlg(\bK)$), with either $A$ or $\sC$ nuclear, and take a point
$\phi \in \ho\Lim_{\bI^{\op}} F(\sC\hten A)$ such that the functor $F$ has a perfect cotangent complex $\bL^{F,\sC(i)\hten A,\phi}(i)$ at $\phi_i \in F(\sC(i)\hten A)$ for  all   $i \in {\bI}$.

 If the complex   $\ho\Lim_{\bI^{\op}} \bT_{\phi}F(\sC\hten A^{\hat{e}})$  % $\oR\HHom_{(\sC\hten A)^e}(\bL^{F,\sC\hten A,\phi},\sC\hten A^{\hat{e}})$ 
%(resp. $\oR\HHom_{\sC\hten A}(\bL^{F,\sC\hten A,\phi}, \sC\hten A)$ 
(resp.  $\ho\Lim_{\bI^{\op}} \bT_{\phi}F(\sC\hten A)$)
is perfect as a right $  A^{\hat{e}}$-module (resp. $A$-module), then the functor  $\ho\Lim_{\bI^{\op}} F(\sC\hten-)$ on $A \in dg_+\cN\cF r\Alg(\bK)$ (resp. $A \in dg_+\cN\cF r\CAlg(\bK)$), or on $A \in dg_+\cF r\Alg(\bK)$ (resp. $A \in dg_+\cF r\CAlg(\bK)$) if  $\sC$ is nuclear, has a perfect cotangent complex at $\phi$ in the sense of Definition \ref{Fcotdef}.
 \end{proposition}
\begin{proof}
We will prove this in the non-commutative case. The commutative case is entirely similar, except that $A^{\hat{e}}$ and $(\sC\hten A)^{e}$ should be replaced with $A$ and $\sC\hten A $, respectively.

We seek a perfect left $  A^{\hat{e}}$-module $L$ with the property that
 \[
  \oR\HHom_{A^{\hat{e}}}(L, M) \simeq \ho\Lim_{\bI^{\op}} \bT_{\phi}F(\sC\hten M)   % \oR\HHom_{(\sC\hten A)^e}(\bL^{F,\sC\hten A,\phi},\sC\hten M) 
 \]
for all non-negatively graded nuclear Fr\'echet (or just Fr\'echet if $\sC$ is nuclear) $A^{\hat{e}}$-modules $M$ in chain complexes.
If $L$ exists, it must be the $(A^{\hat{e}})^{\op}$-linear dual of the perfect complex %$\oR\HHom_{(\sC\hten A)^{e}}(\bL^{F,\sC,\phi}, \sC\hten A^{\hat{e}})$,
$\ho\Lim_{\bI^{\op}} \bT_{\phi}F(\sC\hten A^{\hat{e}})$,
so it suffices to show that the natural morphism
% \[
%  \oR\HHom_{(\sC\hten A)^e}(\bL^{F,\sC,\phi},\sC\hten A^{\hat{e}})\ten^{\oL}_{A^{\hat{e}}}M \to \oR\HHom_{(\sC\hten A)^e}(\bL^{F,\sC\hten A,\phi},\sC\hten M)
% \]
\[
 (\ho\Lim_{\bI^{\op}} \bT_{\phi}F(\sC\hten A^{\hat{e}}))\ten^{\oL}_{A^{\hat{e}}}M \to \ho\Lim_{\bI^{\op}} \bT_{\phi}F(\sC\hten M)
\]
from the derived algebraic tensor product is a quasi-isomorphism.

Since $\bL^{F,\sC\hten A,\phi}$ is assumed to be locally perfect as a $(\sC(i)\hten A)^e$-module, it follows that $\bT_{\phi}F(\sC\hten A^{\hat{e}})  $ is locally perfect as a right $\sC\hten A^{\hat{e}} $-module,  and moreover that  the natural morphism $\bT_{\phi}F(\sC\hten A^{\hat{e}})\ten^{\oL}_{ \sC\hten A^{\hat{e}}} (\sC\hten M) \to \bT_{\phi}F(\sC\hten M) $ is a quasi-isomorphism.

The question thus reduces to showing that the natural morphism 
\[
 (\ho\Lim_{\bI^{\op}} \bT_{\phi}F(\sC\hten A^{\hat{e}}))\ten^{\oL}_{A^{\hat{e}}}M \to \ho\Lim_{\bI^{\op}} (\bT_{\phi}F(\sC\hten A^{\hat{e}} ) \ten^{\oL}_{ \sC\hten A^{\hat{e}}} (\sC\hten M) )
\]
 is a quasi-isomorphism, which it is by Lemma \ref{RGammalemma} since ${\bI}$ is countable.
\end{proof}

 \begin{examples}\label{ctbleex}
 
  \begin{enumerate}
   \item If $X$ is a second-countable topological space with countable base $\cU$, and $\sC_X$ a sheaf of Fr\'echet algebras on $X$,  we can take $\bI$ to be the poset $\cU$. The natural map 
   \[
  \oR\Gamma(X, \bT_{\phi}F(\sC_X\hten M)) \to     \oR\Gamma(\cU, \bT_{\phi}F(\sC_X|_{\cU}\hten M))
   \]
is then a quasi-isomorphism for all  nuclear Fr\'echet $A^{\hat{e}}$-modules $M$ in complexes, where we can drop the nuclear condition if $\sC_X$ is nuclear. Proposition \ref{RGammacotprop} can thus be applied in such situations.
 
 Examples of this form are given by taking $\sC_X$ to be the sheaf of smooth functions on a real manifold  and by the sheaf of holomorphic functions on a complex manifold; in both examples $\sC_X$ is nuclear, and $\sC_X\hten B$ is the sheaf of $B$-valued smooth (resp. holomorphic functions), reasoning as in Example \ref{CXex2}.
   
   \item  As a variant of the previous example, if  $\sC_X^{\bt}$ is a dg algebra of sheaves of Fr\'echet spaces on $X$, concentrated in non-negative cochain degrees, then as in Example \ref{CXex2} we can work with the presheaf $D\sC_X^{\bt}$ defined on the category $\cU \by \Delta^{\op}$. Proposition \ref{RGammacotprop} can thus be applied to the functor  $\oR\Gamma(X, D_*F(\sC_X\hten -))$ in such situations. 
   
   Examples of this form are given by taking $\sC_X^{\bt}$ to be the sheaf of smooth or holomorphic differential forms on a real or complex manifold, with $\sC_X^{\bt}\hten B$ then the sheaf of $B$-valued forms.
   
   \item By \cite[Tag 09A0]{stacks-project}, any quasi-compact quasi-separated scheme $X$ admits a pro-\'etale simplicial hypercover $K$ by weakly contractible affine schemes. Then for any sheaf $\sC_X$ of Fr\'echet algebras  on the pro-\'etale site of $X$, the natural map $\oR\Gamma(X_{\pro\et}, F(\sC_X\hten -))\to \oR\Gamma(\Delta^{\op}, F((K^*\sC_X)\hten - )$ is an equivalence, where $(K^*\sC_X)_n := \Gamma(K_n, \sC_X)$.  Proposition \ref{RGammacotprop} can thus be applied to the functor  $\oR\Gamma(X_{\pro\et}, F(\sC_X\hten -))$ in such situations.
   
   As in Example \ref{CXex}, our motivating example of this form is when $\sC_X$ is the sheaf $\uline{\bK}_X$ of continuous $\bK$-valued functions on the pro-finite sets of components, in which case $\uline{\bK}_X\hten B\cong \uline{B}_X$, the sheaf of $B$-valued functions, when $B$ is nuclear Fr\'echet. When $F$ parametrises vector bundles or perfect complexes, $\oR\Gamma(X_{\pro\et}, F(\uline{\bK}_X\hten -))(B)$ is then the space of locally free or perfect $\uline{B}_X$-modules.
      
  \end{enumerate}

   \end{examples}

 %%%12/3/24 here. Logically at this point we should have consequences of having perfect cotangent complex. I'd like to say $\infty$-category of pro-nilpotent extensions of some fixed nuclear Fr\'echet algebra is dependent only on underlying $\C^{\infty}$ thing.

\section{Stacky dg FEFC  and topological algebras}\label{stackyFEFCsn} %%%prob merge this with previous section.

When it comes to the construction of shifted double Poisson structures on prestacks, we will need to use stacky DGAs, so we now develop  the relevant theory, some of which is also used for the hyperconnection constructions of  \S \ref{RHsn}. 

\subsection{Stacky dg algebras and their modules}\label{stackysn}

 Following \cite[\S \ref{NCstacks-stackysn}]{NCstacks}, consider the category  $DG^+dg_+\Alg(R)$  consisting of associative $R$-algebras $A^{\bt}_{\bt}$ in  cochain chain complexes (i.e. double complexes) concentrated in non-negative bidegrees. We will refer to objects of this category as stacky $R$-DGAs.
It has a cofibrantly generated model structure in which fibrations are surjective in non-negative chain degrees and weak equivalences are levelwise quasi-isomorphisms in the chain direction.

The following is our main source of functors on stacky DGAs, adapting \cite[Definition \ref{smallet2-Dlowerdef}]{smallet2}: %%%deliberately back here for $D_*\per$ results

\begin{definition}\label{Dlowerdef0}
 Given a  functor $F \co dg_+\Alg_{R} \to \C$ to a model category $\C$, we define  $D_*F$ on  $DG^+dg_+\Alg_{R}$ as the homotopy limit
\[
 D_*F(B):= \ho\Lim_{n \in \Delta} F(D^nB),
\]
for the cosimplicial denormalisation functor $D \co  DG^+dg_+\Alg_{R} \to dg_+\Alg_{R}^{\Delta}$. 
\end{definition}

We now collect some properties relating to modules over a stacky $R$-DGAA. These are largely independent of the rest of the section, but will be needed in \S \ref{RHsn}.

 \subsubsection{Modules}\label{stackymodsn}

%  \subsubsection{Absolute derived categories and compact objects}
 
 As a preliminary to considering perfect complexes with flat hyperconnections, we introduce some dg categories.
 
\begin{definition}
 Given a differential graded associative algebra  (DGAA) $A$, write $\per_{dg}(A)$ for the dg category of cofibrant  perfect complexes of right $A$-modules for the projective model structure; for alternative characterisations, see \cite[\S 4.6]{kellerModelDGCat}.
 
%  We also write  $\cD_{dg}(A)$ for the dg category of all cofibrant right $A$-modules in complexes. When $A$ is concentrated in non-negative chain degrees (equivalently, non-positive cochain degrees), we write and $\cD_{dg}^-(A)$ for the dg category of cofibrant complexes of right $A$-modules which are bounded below as chain complexes (equivalently, bounded above as cochain complexes).
\end{definition}

\begin{definition}
 Given $B \in DG^+dg_+\Alg_R$, define $DG^+dg\Mod(B)$ to be the category of right $B$-modules in cochain chain complexes concentrated in non-negative cochain degrees. 
 
 Given a cosimplicial dg $R$-algebra $C \in dg_+\Alg_R^{\Delta}$, define $dg\Mod^{\Delta}(C)$ to be the category of right $C$-modules in cosimplicial chain complexes.
 
We  equip both of these categories with the projective model structure, in which a morphism $M \to N$ is
 \begin{itemize}
  \item a weak equivalence if the maps $\H_iM^r \to \H_iN^r$ are isomorphisms for all $i,r$, and
  \item a fibration if the maps $M_i^r \to N_i^r$ are surjective for all $i>0$ and all $r$.
 \end{itemize}
\end{definition}

The following arises as an application of strictification as in \cite[Theorem 18.6]{HirschowitzSimpson}:
\begin{lemma}\label{stackystrictnlemma}
For $B \in DG^+dg_+\Alg_R$, the dg category $(D_*\per_{dg})(B)$ is quasi-equivalent to the dg category of  cofibrant objects $M$  of $dg\Mod^{\Delta}(DB)$ which: 
\begin{enumerate}
 \item are homotopy Cartesian in the sense that the morphisms $\pd^i \co M^n\ten_{B^n, \pd^i}B^{n+1} \to M^{n+1}$ are all quasi-isomorphisms, and
 \item have $M^0$  perfect as a $B^0$-module in complexes.
\end{enumerate}
\end{lemma}

We can then reduce to a description intrinsic to the stacky DGA:
\begin{lemma}\label{doublecomplexlemma}
For $B \in DG^+dg_+\Alg_R$, the dg category $(D_*\per_{dg})(B)$ is quasi-equivalent to the dg category of cofibrant objects $M$ of $DG^+dg\Mod(B)$   which:
\begin{enumerate}
 \item are homotopy Cartesian in the sense that the morphisms $ M^0\ten_{B^0}B^{n} \to M^{n}$ are all quasi-isomorphisms, and
 \item have $M^0$  perfect as a $B^0$-module in complexes.
\end{enumerate}
\end{lemma}
\begin{proof}
Cosimplicial Dold--Kan denormalisation gives  a right Quillen functor $D \co DG^+dg_+\Alg_R \to dg\Mod^{\Delta}(DB)$, with a left adjoint which we denote $D_{\Mod}^*$. For any cosimplicial dg $R$-module $L$, we necessarily have $D_{\Mod}^*(L\ten_RDB) \cong (N_cL)\ten_RB$, where $N_c$ denotes the cosimplicial normalisation functor. 

In particular, if $N_cL^i$ is acyclic for all $i>0$, then $D_{\Mod}^*(L\ten_RDB)$ is   homotopy Cartesian and the unit $L\ten_RDB \to DD_{\Mod}^*(L\ten_RDB)$ is a levelwise quasi-isomorphism since both sides are quasi-isomorphic to $L^0\ten_RD^nB$ in cochain degree $n$. On passing to transfinite pushouts and retracts, it follows that the unit of the adjunction is a levelwise quasi-isomorphism for all homotopy Cartesian cofibrant $DB$ modules, and that  $D_{\Mod}^*$ sends such objects to homotopy Cartesian $B$-modules. 

Thus $\oL D_{\Mod}^*$ gives a full and faithful dg functor on the respective dg subcategories of homotopy Cartesian objects, so it only remains to show essential surjectivity, since
Lemma  \ref{stackystrictnlemma} then completes the comparison. For any homotopy Cartesian $B$-module $N$, we have a quasi-isomorphism $N^0\ten_{B^0}^{\oL}B \to N^{\#}$ on turning off the cochain differential, so the almost cosimplicial module $(DN)^{\#}$ (see e.g. \cite[\S \ref{NCstacks-modDGAAsn}]{NCstacks}) underlying $DN$ admits a levelwise quasi-isomorphism $N^0\ten_{B^0}^{\oL}(DB)^{\#} \to (DN)^{\#}$, and in particular $DN$ is homotopy Cartesian. As above, the unit of the adjunction then gives a  levelwise quasi-isomorphism $DN \to D\oL D_{\Mod}^*DN$, and hence the co-unit $\oL D_{\Mod}^*DN \to N$ must be a levelwise quasi-isomorphism, since $D$ reflects such.
\end{proof}

\begin{definition}
 Given $B \in DG^+dg_+\Alg_R$, define the dg category of  perfect $B^0$-complexes with flat $B$-hyperconnection as follows.
  \begin{enumerate}
  \item Objects are pairs $(E,\nabla)$ for  $E$ a perfect cofibrant $B^0$-module in complexes, and  $\nabla \co E_{\#} \to \prod_{i>0} E\ten_{B^0}B^i_{\#[i-1]}$ a $(B^0,\pd)$-connection (i.e. $\nabla(eb)=\nabla(e)b \pm e\ten \pd b$ for $e \in E$, $b \in B^0$), satisfying the flatness condition $(\delta + \nabla)^2=0 \co E_{\#} \to \prod_{i>0} E\ten_{B^0}B^i_{\#[i-2]}$.
  
  \item The complex of morphisms from $(E,\nabla)$ to $(E',\nabla')$ is given by $\prod_i\HHom_{B^0}(E,E'\ten_{B^0}B^i)_{[i]}$ with differential $f \mapsto (\delta +\nabla')\circ f \mp f \circ (\delta +\nabla)$ and the obvious composition.
  \end{enumerate}
\end{definition}

The following allows us to strictify further and to work with finitely presented models.  
\begin{lemma}\label{hyperconnlemma}
 For $B \in DG^+dg_+\Alg_R$, the following dg categories are quasi-equivalent:
 \begin{enumerate}
  \item\label{hyperconnlemma:first} $(D_*\per_{dg})(B)$,
  \item\label{hyperconnlemma:second} the dg category of  complete exhaustively filtered $(\Tot^{\Pi}B, \Tot^{\Pi}\sigma^{\ge *}B)$-modules $(M,F)$ in complexes, with $\gr_FM$ cofibrant and perfect as a $B^{\#}$-module and the maps $(\gr_F^0M)\ten_{B^0}B^n \to \gr_F^nM$ all being quasi-isomorphisms,
  \item\label{hyperconnlemma:third} the dg category of  perfect $B^0$-complexes with flat $B$-hyperconnection.
   \end{enumerate}

 Moreover, if $(E,\nabla)$ is a perfect $B^0$-complex with flat $B$-hyperconnection and $E'$ a cofibrant $B^0$-module quasi-isomorphic to $E$, there exists a flat $B$-hyperconnection $\nabla'$ on $E'$ such that $(E,\nabla)$ is homotopy equivalent to $(E',\nabla')$ in the dg category.
 \end{lemma}
\begin{proof}
By \cite[Lemma 1.5]{DQLag}, the  $\infty$-category  given by $DG^+dg\Mod(B)$ localised at levelwise quasi-isomorphisms is equivalent to the $\infty$-category $\C(B)$ of complete exhaustively filtered $(\Tot^{\Pi}B, \Tot^{\Pi}\sigma^{\ge *}B)$-modules  in chain complexes, localised at filtered quasi-isomorphisms. Via the Rees construction $\xi$, the $\infty$-category of exhaustively filtered $(\Tot^{\Pi}B, \Tot^{\Pi}\sigma^{\ge *}B)$-modules is equivalent to the category of  $\bG_m$-equivariant $C$-modules, for  $C:=\bigoplus_i t^{-i} \Tot^{\Pi}\sigma^{\ge i}B$ with $t$ of $\bG_m$-weight $1$. 

When the filtration $F$ on a module  $M$ is complete, we can  then use an obstruction theory argument as in \cite[\S 14.3.2]{mhs2}, since then $\xi(M,F)$ will be a $\bG_m$-equivariant limit of the modules $\xi(M,F)/t^k$. In particular, $\HHom_{C}(\xi(N,F),\xi(M,F))\cong \Lim_k \HHom_{C}(\xi(N,F),\xi(M,F)/t^k)$, with exact sequences
\begin{align*}
0 \to &\HHom_{C/t}^{\bG_m}(\gr_FN,t^k\gr_FM)    \\
&\to \HHom_{C}^{\bG_m}(\xi(N,F),\xi(M,F)/t^{k+1}) \to \HHom_{C}^{\bG_m}(\xi(N,F),\xi(M,F)/t^k)\to 0
\end{align*}
whenever $\gr_FN$ (i.e. $\xi(N,F)/t$) is cofibrant as a module over $C/t \cong \bigoplus_i B^i_{[i]}$. Thus the complexes $\HHom_{C}$ are quasi-isomorphism invariant for objects with this cofibrancy condition, so calculate $\Hom$-complexes in the $\infty$-category. This gives an equivalence between (\ref{hyperconnlemma:first}) and (\ref{hyperconnlemma:second}) above, via Lemma \ref{doublecomplexlemma}.

Now, note that quasi-isomorphism invariance also applies to deformations of objects (again as in \cite[\S 14.3.2]{mhs2}), and that the quasi-Cartesian condition only restricts $\gr_FM \cong\xi(M,F)/t$. Thus for any $(M,F)$ as in (\ref{hyperconnlemma:second}) and any quasi-isomorphism $E \to  \gr_FM$ of cofibrant graded $B^{\#}$ modules, there exists a $\bG_m$-equivariantly $t$-adically complete flat $C$-module $\tilde{E}$ with $\tilde{E}/t \cong E$  and a quasi-isomorphism $\tilde{E} \to \xi(M,F)$.  Applying this to the case $E= (\gr_F^0M)\ten_{B^0}t^{-\#}B^{\#}$ for $M$ quasi-Cartesian gives $\tilde{E}\cong (E\ten_{B^0}C, \delta + \sum_{i>0}\nabla_i)$, yielding the description of objects  in  (\ref{hyperconnlemma:third}) and establishing the final property.
\end{proof}

\subsection{Stacky dg FEFC algebras}

Consider the category  $DG^+dg_+\FEFC(R)$ of stacky dg FEFC $R$-algebras, generalising \S \ref{dgFEFCsn} to have bigradings and commuting horizontal and vertical  differentials $\delta, \pd$, where the notation $DG^+dg_+$ signifies that we are writing the objects as cochain (superscript) complexes of chain (subscript) complexes, both concentrated in their respective non-negative degrees. These provide an analytic analogue of \cite[\S \ref{NCstacks-stackysn}]{NCstacks}, with every stacky dg FEFC algebra having an underlying double complex. The free objects are given by $ \cF_{S \sqcup \delta S \sqcup \pd S \sqcup \pd \delta S} (R)$ for bigraded sets $S$ (modifying Definition \ref{freegradedandef} by the introduction of a second grading), with $\pd \co A^i_j\to A^{i+1}_j$ and $\delta \co A^i_j \to A^i_{j-1}$.

Combining the arguments of \cite[Lemma \ref{NCstacks-bidgaamodel}]{NCstacks} and Lemma \ref{FEFCmodel} gives: 
\begin{lemma}\label{biFEFCmodel}
There is a cofibrantly generated model structure on $DG^+dg_+\FEFC(R)$ in which  weak equivalences are levelwise quasi-isomorphisms and  fibrations are surjections in strictly positive chain degrees.
\end{lemma}

Since the monad $S \mapsto \cF_{\delta S \sqcup \pd S \sqcup \pd \delta S}(R)$ preserves filtered colimits, we immediately have: 
\begin{lemma}\label{stackyindFPdg}
 The category of stacky dg FEFC $R$-algebras is equivalent to the category of ind-objects of the category of finitely presented stacky dg FEFC $R$-algebras.
\end{lemma}

\subsection{Stacky topological dg algebras}
% \subsection{Stacky dg algebras}

% When it comes to the construction of shifted double Poisson structures on prestacks, we will need to use stacky DGAs, so we now develop some of the relevant theory.  %%%that moved earlier

\begin{definition}
 Define a unital associative  stacky differential graded  complete  LDMC  ${R}$-algebra $A_{\bt}^{\bt}$ to be a cochain chain complex of complete locally convex topological ${R}$-vector spaces, concentrated in non-negative bidegrees, with the differentials $\delta\co A_n^i \to A_{n-1}^i$ and $\pd \co A_n^i \to A_n^{i+1}$ being continuous operators, together with a unit $1 \in \z_0\z^0A$ and a continuous ${R}$-bilinear multiplication $A_m^i \by A_n^j \to A_{m+n}^{i+j}$, such that the topology is given by a system of submultiplicative  seminorms $\|-\|$ with $\|1\|=1$,  $\|\delta a\|\le \|a\|$ and $\|\pd a \| \le \|a\|$. We require that the multiplication be
 associative and that $\delta$ and $\pd$ be  derivations with respect to it.  
% We say that $A_{\bt}^{\bt}$ is commutative if the multiplication is graded-commutative with repsect to the total grading. 
\end{definition}

\begin{definition}
Write $DG^+dg_+\hat{\Tc}\Alg_{R}$ for the category of  unital associative stacky differential graded complete  LDMC ${R}$-algebras concentrated  in non-negative homological degrees, in which morphisms are bidifferential bigraded ${R}$-algebra morphisms $A_{\bt}^{\bt} \to B_{\bt}^{\bt}$ which are continuous in each degree.

 Write   $dg_+\cF r\Alg(R)$ (resp. $dg_+\cN\cF r\Alg(R)$) for the full subcategories of  $DG^+dg_+\hat{\Tc}\Alg_{R}$ consisting of objects for which the $R$-modules $A^i_n$ are all Fr\'echet (resp. nuclear Fr\'echet).
\end{definition}

By completing with respect to each submultiplicative seminorm, we can identify $DG^+dg_+\hat{\Tc}\Alg_{R}$ with the full subcategories of the pro-category $\pro(DG^+dg_+\Ban\Alg_{R}) $ consisting of filtered systems in which all transition morphisms are dense.

We have the following immediate analogue of Lemma \ref{FEFCforgetlemma}:
\begin{lemma}\label{stackyFEFCforgetlemma}
 There is a natural forgetful functor $U \co DG^+dg_+\hat{\Tc}\Alg_{R} \to  DG^+dg_+\FEFC\Alg_{R}$, sending a topological stacky dg  algebra $A_{\bt}^{\bt}$ to a stacky dg FEFC algebra with the same underlying bidifferential bigraded algebra. 
\end{lemma}

Adapting Lemma \ref{FEFCmodel}, we have a model structure on $DG^+dg_+\FEFC_{R}$ in which fibrations are maps which are surjective in strictly positive chain degrees, and weak equivalences are levelwise quasi-isomorphisms. The unit of the adjunction in Lemma \ref{stackyFEFCforgetlemma} is an isomorphism on objects whose underlying graded FEFC algebra is freely generated, and hence on all cofibrant objects. The proof of Corollary \ref{keyBanFEFCcor} then adapts to give the following.

\begin{corollary}\label{stackykeyBanFEFCcor}
The  forgetful functor $U \co DG^+dg_+\hat{\Tc}\Alg_{R} \to  DG^+dg_+\FEFC_{R}$ induces an equivalence of the $\infty$-categories given by simplicial localisation at levelwise abstract quasi-isomorphisms, i.e.  morphisms $A \to B$ for which the maps  $\H_*A^i \to \H_*B^i$ are isomorphisms of abstract vector spaces for all $i$. The same is true if we replace $DG^+dg_+\hat{\Tc}\Alg_{R}$ with any full subcategory containing all objects of the form $(\cF_S,\delta,\pd)$ for non-negatively bigraded sets $S$. 
\end{corollary}

There is also an analogue of Corollary \ref{Dlfpcor} for stacky DGAs, with exactly the same proof, but its applicability for us is limited because our functors of interest on stacky DGAs are seldom l.f.p. Instead, we will make use of the following variant.

\begin{definition}
 Say that a w.e.-preserving functor  $F$ from  $DG^+dg_+\FEFC_{R}$  to a relative category $\C$ is d.f.p. (degreewise finitely presented) if it preserves filtered homotopy colimits of objects bounded in the cochain direction, i.e. if $F(\LLim_{\alpha} B_{\alpha}) \simeq \ho\LLim_{\alpha}F(B_{\alpha})$ for all filtered systems $\{B_{\alpha}\}_{\alpha}$ for which there exists an integer $n$  with $ B_{\alpha}^i=0$ for all $i>n$ and all $\alpha$.
\end{definition}
In particular, note that the mapping space functor $\oR\Map(A,-)$ is d.f.p. whenever $A$ is cofibrant with finitely many generators in each cochain degree.

\begin{proposition} \label{stackyDlfpprop}
Take a homotopy complete and cocomplete relative category $\C$ and any full subcategories $\cD' \subset \cD$ of $DG^+dg_+\hat{\Tc}\Alg_{R} $ containing all  finitely presented %%can do fp rather than just dfp here, because of truncation condition.
cofibrant objects  and closed under
 brutal  truncations $A \mapsto A^{\le n}$ in the cochain direction. %%$dg_+\hat{\Tc}\Alg_{R}^{\mathrm{fp,cof}}$. 

Given a d.f.p w.e.-preserving  functor $F \co DG^+dg_+\FEFC_{R} \to \C$ and a w.e.-preserving functor $G \co \cD \to \C$ with the property that $G(B) \simeq \ho\Lim_n G(B^{\le n})$ for all $B \in \cD$,
the natural map
\[
 \oR\Map_{[\cD,\C]}(U_{\cD,*}F,G) \to \oR\Map_{[\cD',\C]}(U_{\cD',*}F, \iota_*G)
\]
is an equivalence, for the inclusion $\iota \co  \cD' \into \cD$,  the forgetful functor $U_{\cD} \co  \cD \to  DG^+dg_+\FEFC_{R}$, and $\iota_*, U_{\cD,*}$ the respective functors given by pre-composition.

Moreover, if  $G$ extends to stacky  FEFC-DGAs in the sense that  $G\simeq U_{\cD,*}G'$ for some w.e.-preserving  functor $G' \co DG^+dg_+\FEFC_{R} \to \C$ with $G'(B) \simeq \ho\Lim_n G'(B^{\le n})$ for all $B$,
then the natural map
\[
 \oR\Map_{[DG^+dg_+\FEFC_{R},\C]}(F,G') \to    \oR\Map_{[\cD,\C]}(U_{\cD,*}F,G)  
\]
is also an equivalence. 

%In particular, we jhve a fully faithful functor   $U_{\cD,*} \co [dg_+\FEFC_{R},\C]_{lfp}  \to [\cD,\C]$ from the $\infty$-category of l.f.p w.e.-preserving  functors   $dg_+\FEFC_{R} \to \C$  to  the $\infty$-category $[\cD,\C]$ of w.e.-preserving  functors   $\cD \to \C$.
 \end{proposition}
\begin{proof}
 We have model structures on the full subcategories $ DG^{[0,n]}dg_+\hat{\Tc}\Alg_{R} \subset DG^+dg_+\hat{\Tc}\Alg_{R}$ and  $DG^{[0,n]}dg_+\FEFC_{R} \subset DG^+dg_+\FEFC_{R}$ on objects concentrated in cochain degrees $[0,n]$, in which morphisms are fibrations or weak equivalences whenever they are so in the larger category. The brutal cotruncation functor $\sigma^{\le n}$ is left adjoint to inclusion, forming a Quillen adjunction. If  $A$ is a  cofibrant object of one of the larger categories, with finitely many generators in each cochain degree, then $\sigma^{\le n} A$ is a cofibrant object of the smaller category, with finitely many generators. 

Writing $\cD^{[0,n]}$ for the full subcategory of $\cD$ on objects concentrated in cochain degrees $[0,n]$, the proofs of  Proposition \ref{afpTcFEFCprop} and Corollary \ref{Dlfpcor} adapt to give analogues for these categories, so that for $F_n \in [DG^{[0,n]}dg_+\FEFC_{R},\C]$ l.f.p. and $G_n \in [\cD^{[0,n]},\C]$, we have equivalences
\begin{align*}
 \oR\Map_{[\cD^{[0,n]},\C]}(U_{\cD^{[0,n]},*}F_n,G_n) &\xra{\sim} \oR\Map_{[DG^{[0,n]}dg_+\hat{\Tc}\Alg_{R}^{\mathrm{fp,cof}}  ,\C]}(F_n, i_*G_n),\\
 \oR\Map_{[DG^{[0,n]}dg_+\FEFC_{R},\C]}(F_n,G'_n)  &\xra{\sim}   \oR\Map_{[\cD^{[0,n]},\C]}(U_{\cD^{[0,n]},*}F_n,G_n), 
\end{align*}
for $i \co DG^{[0,n]}dg_+\hat{\Tc}\Alg_{R}^{\mathrm{fp,cof}} \to \cD^{[0,n]}$ and $G'_n$ a lift of $G_n$.
 
Now, since $F$ is d.f.p., it follows that $j_{n*}F$ is l.f.p.  for the inclusion $j_n \co  DG^{[0,n]}dg_+\FEFC_{R} \to DG^+dg_+\FEFC_{R}$. The condition $G'(B) \simeq \ho\Lim_n G'(B^{\le n})$ can be rephrased as saying $G'\simeq \ho\Lim_n j_n^!j_{n*}G'$ for $j_n^!$ right adjoint to $j_{n*}$, with a similar property for $G$ in terms of the inclusion $j_n \co \cD^{[0,n]} \to \cD$.

We thus have
\begin{align*}
 \oR\Map_{[\cD,\C]}(U_{\cD,*}F,G)&\simeq  \ho\Lim_n \oR\Map_{[\cD^{[0,n]},\C]}(j_{n*}U_{\cD,*}F,j_{n*}G)\\ 
&\simeq  \ho\Lim_n \oR\Map_{[DG^{[0,n]}dg_+\hat{\Tc}\Alg_{R}^{\mathrm{fp,cof}},\C]}(U_*F,i_*j_{n*}G),
 \end{align*}
 where the morphisms  $[DG^{[0,n+1]}dg_+\hat{\Tc}\Alg_{R}^{\mathrm{fp,cof}},\C] \to [DG^{[0,n]}dg_+\hat{\Tc}\Alg_{R}^{\mathrm{fp,cof}},\C]$ in the second inverse system are given by the left adjoints $(\sigma^{\le n})^*$ to the functors $(\sigma^{\le n})_*$ induced by brutal truncation  $\sigma^{\le n} \co DG^{[0,n+1]}dg_+\hat{\Tc}\Alg_{R}^{\mathrm{fp,cof}} \to DG^{[0,n]}dg_+\hat{\Tc}\Alg_{R}^{\mathrm{fp,cof}}$.
 Since this depends only on $F$ and $i_*j_{n*}G$, replacing $\cD$ with $\cD'$ gives an equivalent expression, establishing the first equivalence.
 
Likewise, we have
\begin{align*}
 \oR\Map_{[DG^+dg_+\FEFC_{R},\C]}(F,G') &\simeq  \ho\Lim_n \oR\Map_{[DG^{[0,n]}dg_+\FEFC_{R},\C]}(j_{n*}F,j_{n*}G')\\   
 &\simeq  \ho\Lim_n\oR\Map_{[\cD^{[0,n]},\C]}(U_{\cD^{[0,n]},*}j_{n*}F,j_{n*}G)  \\
  &\simeq  \ho\Lim_n\oR\Map_{[\cD^{[0,n]},\C]}(j_{n*}U_{\cD,*}F,j_{n*}G)  \\
  &\simeq  \oR\Map_{[\cD,\C]}(U_{\cD,*}F,\ho\Lim_n j_n^!j_{n*}G)\\
   &\simeq  \oR\Map_{[\cD,\C]}(U_{\cD,*}F,G). \qedhere
\end{align*}

\end{proof}

\subsection{Denormalisation}\label{denormsn}

The following lemma allows us to apply Proposition \ref{stackyDlfpprop} to functors of the form $D_*F$ from Definition \ref{Dlowerdef0}. %%note that d.f.p. only really makes sense on FEFC, not the other cases
\begin{lemma}\label{DstarFdfpprop}
 For $F \co dg_+\FEFC_{R} \to s\Set$, we have $D_*F(B) \simeq \ho\Lim_m D_*F(B^{\le m})$. 

 Moreover, if $F$ is w.e.-preserving, l.f.p. and homogeneous in the sense of Definition \ref{hhgsdef}, then $D_*F$ is d.f.p. 
 \end{lemma}
\begin{proof}
 The expression for $D^nB$ depends only on $B^{\le n}$, so $D^nB \cong D^n(B^{\le m})$ for all $m\ge n$, and hence $\ho\Lim_m F(D^n(B^{\le m})) \simeq F(D^nB)$, giving the first statement. %(The same is true for $F \co \cD \to s\Set$ and $D_*F \co \cE \to s\Set$ as in Definition \ref{Dlowerdef}.)
 
 For the second statement, we wish to show that $D_*F$ preserves filtered colimits of objects concentrated in cochain degrees $[0,m]$, which we prove by induction on $m$. For $m=0$, the statement is an immediate consequence of $F$ being l.f.p., because $D^n(B^0)=B^0$ for all $n$, so $D_*F(B^0) \simeq F(B^0)$.  
 
 For the inductive step, observe that the maps $B^{\le m+1} \to B^{\le m}$ are square-zero extensions with kernels $(B^{m+1})^{[-m-1]}$, so the maps $D^n(B^{\le m+1}) \to  D^n(B^{\le m})$ are also square-zero extensions with kernels $D^n((B^{m+1})^{[-m-1]})$. We can now apply a standard obstruction argument, observing that $B^{\le m}$ is levelwise quasi-isomorphic to the stacky FEFC-DGA $\widetilde{B^m}:= \cone((B^{m+1})^{[-m-1]}\to B^{\le m+1})$, where the cone is taken in the chain direction. We then have 
 \[
D^n( B^{\le m+1}) \cong D^n\widetilde{B^m}\by_{B^0 \oplus D^n(B^{m+1})^{[-m-1]}_{[-1]})}B^0,
 \]
 and since the map $D^n\widetilde{B^m}\to B^0 \oplus D^n(B^{m+1})^{[-m-1]}_{[-1]})$ is a square-zero extension, homogeneity gives
 \[
  D_*F(B^{\le m+1}) \simeq D_*F(B^{\le m})\by^h_{D_*F(B^0 \oplus(B^{m+1})^{[-m-1]}_{[-1]}) }F(B^0) 
 \]
Since the maps $B^0 \oplus M \to B^0$ are also all square-zero extensions, we can identify $F(B^0\oplus D^n M)$ with the $n$th term of the Dold--Kan denormalisation of the cochain complex $F(B^0 \oplus M^0) \xra{\pd} F(B^0 \oplus M^1) \xra{\pd} \ldots$ of bundles over $F(B^0)$, and substitution then identifies 
 %  Arguing as in \cite[Lemma \ref{NCstacks-DFtgtlemma}]{NCstacks}, 
 $D_*F(B^0 \oplus(B^{m+1})^{[-m-1]}_{[-1]}) $ with the $m+1$-fold loop object of the bundle $F(B^0 \oplus B^{m+1}_{[-1]}) \to F(B^0)$. Thus everything  in sight commutes with filtered colimits, as required.
 %
%%the homotopy fibre of $D_*F(B^0 \oplus M)$ over $x \in F(B^0)$ is the truncation of the product total complex of the cochain complex $T_xF(M^0) \to T_x(F,M^1) \to \ldots$ of chain complexes,   
 %
 \end{proof}

% The following is our main source of functors on stacky DGAs, adapting \cite[Definition \ref{smallet2-Dlowerdef}]{smallet}:

We now extend Definition \ref{Dlowerdef0} to our other settings. 

\begin{definition}\label{Dlowerdef}
 Assume we have
 full  subcategories $\cD \subset dg_+\hat{\Tc}\Alg_{R}$ and $\cE \subset 
DG^+dg_+\FEFC_{R}$ such that $D^nA \in \cD $ for all $n \in \N_0$ and $A \in \cE$, for the cosimplicial denormalisation functor $D \co  DG^+dg_+\hat{\Tc}\Alg_{R}   \to dg_+\hat{\Tc}\Alg_{R}^{\Delta}$.
 
Given  a  functor $F \co \cD \to \C$ to a model category $\C$, we then define  $D_*F$ on  $\cE$ as the homotopy limit
\[
 D_*F(B):= \ho\Lim_{n \in \Delta} F(D^nB).
\]
\end{definition}

% \begin{definition}\label{Dlowerdef}
%  Given a  functor $F \co dg_+\FEFC_{R} \to \C$ to a model category $\C$, we define  $D_*F$ on  $DG^+dg_+\FEFC_{R}$ as the homotopy limit
% \[
%  D_*F(B):= \ho\Lim_{n \in \Delta} F(D^nB),
% \]
% for the cosimplicial denormalisation functor $D \co  DG^+dg_+\FEFC_{R} \to dg_+\FEFC_{R}^{\Delta}$. 
% 
% Similarly, given full  subcategories $\cD \subset dg_+\hat{\Tc}\Alg_{R}$ and $\cE \subset 
% DG^+dg_+\FEFC_{R}$ such that $D^nA \in \cD $ for all $n \in \N_0$ and $A \in \cE$, we associate to any  functor $F \co \cD \to\C$ a functor $D_*F \co \cE \to \C$ by the same formula. 
% \end{definition}

  Substituting Lemma \ref{DstarFdfpprop} in Proposition \ref{stackyDlfpprop} gives:
\begin{corollary}\label{Dstarcor}
Given $G \co dg_+\FEFC_{R} \to s\Set$ and  a d.f.p w.e.-preserving  functor $F \co DG^+dg_+\FEFC_{R} \to \C$, we have a weak equivalence
\[
 \oR\Map_{[DG^+dg_+\FEFC_{R},s\Set]}(F,D_*G) \xra{\sim}    \oR\Map_{[\cC,s\Set]}(U_{\cC,*}F, U_{\cC,*}D_*G)  
\]
for all full subcategories $\cC \subset DG^+dg_+\hat{\Tc}\Alg_{R} $ containing all  finitely presented cofibrant objects and closed under
 brutal  truncation in the cochain direction. 
 
 In particular, this applies if $F\simeq D_*F'$ for a  w.e.-preserving, l.f.p., homogeneous functor $F'\co dg_+\FEFC_{R} \to s\Set$.
\end{corollary}
 
 \begin{definition}\label{inftygeomdef2}
  Define a derived  FEFC prestack $F\co  dg_+\FEFC(R) \to s\Set$ to be $n$-geometric  or $\infty$-geometric   Artin (resp.  Deligne--Mumford) if it arises as homotopy colimit of the form described in \cite[\S \ref{NCstacks-derivedgeomprestacksn}]{NCstacks}, substituting  $dg_+\FEFC(R)$ for $dg_+\Alg(R)$ throughout, with FEFC cotangent complexes in place of NC algebraic cotangent complexes when defining submersive, \'etale and open morphisms.
 \end{definition}

The following is then an immediate consequence of the definitions:
\begin{proposition}\label{analytifyNCprop}
 Given a derived $n$-geometric or $\infty$-geometric   Artin (resp. Deligne--Mumford) NC prestack $F \co dg_+\Alg(R) \to s\Set$, the induced functor $F^{\FEFC} \co  dg_+\FEFC(R) \to s\Set$, given by composition with the forgetful functor $dg_+\FEFC(R) \to dg_+\Alg(R)$, is a  derived $n$-geometric  or $\infty$-geometric    Artin (resp. Deligne--Mumford) FEFC prestack.
\end{proposition}

By \cite[\S \ref{NCstacks-dmodperfsn}]{NCstacks}, examples of $\infty$-geometric   Artin  NC prestacks are given by derived moduli of perfect complexes over finite flat $R$-algebras.
 
The proof of \cite[Proposition \ref{NCstacks-replaceprop}]{NCstacks} adapts to give the following in our context:
 \begin{proposition}\label{replaceprop}
 Given a  derived $\infty$-geometric    Artin  FEFC prestack $F$ and any w.e.-preserving functor $G \co dg_+\FEFC(R) \to s\Set$, we have a natural weak equivalence
 \[
  \map_{[dg_+\FEFC(R)^{\wedge},s\Set]}(F,G) \to \map_{[DG^+dg^+\FEFC(R),s\Set]}(D_*F,D_*G).
 \]
\end{proposition}
 
 Consequently, Proposition \ref{replaceprop} and Corollary \ref{Dstarcor} show that to study  structures parametrised by $G$  on a derived $\infty$-geometric    Artin  FEFC prestack $F$, it suffices to restrict $F$ to a suitable subcategory. In cases where we only know how to define $G$ on such a subcategory, Proposition \ref{stackyDlfpprop} can also be seen as giving a consistency result. We now give further, but somewhat weaker, consistency results for more general classes of functors $F$ satisfying conditions which are much easier to verify.
 
 \begin{definition}\label{hetdef}
 A morphism  $A \to B$ in $DG^+dg_+%\cF r\Alg(R) 
\hat{\Tc}\Alg_{R}^{\mathrm{dfp,cof}}$
 is said to be  homotopy \'etale when the maps 
\[
  (\hat{\Omega}_{A/R}^1\ten_{A^{\hat{e}}}(B^{\hat{e}})^0)^i \to (\hat{\Omega}_{B/R}^1\ten_{B^{\hat{e}}}(B^{\hat{e}})^0)^i
\]
 are quasi-isomorphisms for all $i \gg 0$, and 
\[
\Tot \sigma^{\le q} (\hat{\Omega}_{A}^1\ten_{A^{\hat{e}}}(B^{\hat{e}})^0) \to \Tot \sigma^{\le q}(\hat{\Omega}_{B}^1\ten_{B^{\hat{e}}}(B^{\hat{e}})^0)
\]
is a quasi-isomorphism for all $q \gg 0$, where $\sigma^{\le q}$ denotes the brutal cotruncation
\[
 (\sigma^{\le q}M)^i := \begin{cases} 
                         M^i & i \le q, \\ 0 & i>q.
                        \end{cases}
\]

Define $DG^+dg_+%\cF r\Alg(R) 
\hat{\Tc}\Alg_{R}^{\mathrm{dfp,cof},\et} \subset DG^+dg_+%\cF r\Alg(R) 
\hat{\Tc}\Alg_{R}^{\mathrm{dfp,cof}}$ to be the subcategory with the same objects but only  homotopy  \'etale morphisms.
\end{definition}
 
 Note that the cofibrancy condition on $A$ and $B$ means that we do not need to derive the tensor products in Definition \ref{hetdef}, because $\hat{\Omega}^1$ is cofibrant. Also note that the finite presentation conditions mean that replacing algebraic tensor products with completed tensor products would not affect the expression.
 %%woth noting that {NCpoisson-Perfthm2} just uses that $\Perf_{\cA}$ is homogeneous.
 
 \begin{proposition}\label{cosimplicialresnprop} 
 If a w.e.-preserving homogeneous  functor $F\co dg_+\cF r\Alg(R) \to s\Set$ has a perfect cotangent complex at a point  $x \in F(A)$ for  $A\in dg_+\hat{\Tc}\Alg_{R}^{\mathrm{fp,cof}}$, 
 then the functor 
 \[
\hat{F}_x \co   B \mapsto D_*F(B)\by_{F(B^0),x^*}^h\map_{dg_+\cF r\Alg(R)}(A,B^0)
 \]
on the category $DG^+dg_+\cF r\Alg(R) $ 
is representable by a cosimplicial homotopy  \'etale diagram  
$C(-) \co \Delta \to DG^+dg_+%\cF r\Alg(R) 
\hat{\Tc}\Alg_{R}^{\mathrm{dfp,cof},\et}$. Specifically,
% in the sense that
\[
 \hat{F}_x(B) \simeq \ho\LLim_{j \in \Delta^{\op}} \oR\map_{DG^+dg_+\cF r\Alg(R)}(C(j),B),%%and can replace $\map$ with $\Hom$
\]
naturally in $B$.

The same statements hold if we replace $\cF r$ with $\cN\cF r$ throughout.
\end{proposition}
\begin{proof}
Since the cotangent complex of $F$ at $x$ exists and is perfect and homologically bounded below, we may choose a model $L\in dg\Mod(A^{\hat{e}})$ for it which is cofibrant and finitely generated. In particular, $L$ is strictly bounded below, so there exists some $m\ge 0$ such that $L_i=0$ for all $i<-m$. 

We can now adapt the proof of \cite[Proposition \ref{NCstacks-cosimplicialresnprop}]{NCstacks}, constructing 
$C(-)$ by an inductive process as a filtered colimit $\LLim_i C^{(i)}(-)$, with $C^{(i)}(-)$ representing the functor $\hat{F}_x^{(i)}\co B \mapsto D_*F(B^{\le i})\by_{F(B^0),x^*}^h\map(A,B^0)$ and $C^{(i)}(-)^n= C(-)^n$ for all $n \le i$. %%formation of $C^{(i+1)}$ from $C^{(i)}$ adds generators in cochain degrees $i+1, i+2$.

We begin by setting $C^{(0)}(-)$ to be the constant cosimplicial diagram consisting of the stacky DGA $\hat{\Omega}^{\bt}_A$, so 
\[
C^{(0)}(j)^n := \underbrace{\hat{\Omega}^1_A\ten_A \hat{\Omega}^1_A\ten_A \ldots \ten_A\hat{\Omega}^1_A}_n
\]
for all $j$, with cochain differential given by the de Rham differential.  (The hypotheses imply that $\hat{\Omega}^1_A$ is a finitely generated cofibrant $A$-module, so  $\ten_A$ is equivalent to $\hten_A$ here.)

Now assume we have constructed $C^{(n-1)}(j)$ and a universal element $u \in \ho\Lim_{j \in \Delta}\hat{F}_x^{(n-1)}(C^{(n-1)}(j))$. Applying $D_*F$ to the weak equivalence $\cone((B^n)^{[-n]} \to B^{\le n}) \to B^{<n}$ (the cone being taken in the chain direction)
and the map $\cone((B^n)^{[-n]} \to B^{\le n}) \to B^0 \oplus (B^n)^{[-n]}_{[-1]}$ gives us a homotopy fibre sequence
\[
 \hat{F}_x^{(n)}(B) \to  \hat{F}_x^{(n-1)}(B) \to \oR\map_A(L, (B^n)_{[n-1]})
\]
by homogeneity. Applying this to $B=C^{(n-1)}(j)$ and the universal element $u$ gives 
us an element
\[
 o_e(u) \in \ho\Lim_{j \in \Delta}\map_{dg\Mod(A)}(L,C^{(n-1)}(j)^n_{[n-1]}).
\]

Writing $L\ten \Delta^j:= L\ten_{\Z}C_{\bt}(\Delta^j,\Z)$ as the tensor product with simplicial chains, we 
% can 
%  represent $o_e(u)$ as a morphism $\{L \ten\Delta^j \to C^{(n-1)}(j)^n_{[n-1]}\}_{j \in \Delta}$ of cosimplicial $A$-modules in chain complexes, since $L$ is cofibrant, and hence $j\mapsto L \ten \Delta^j$ is a cofibrant resolution in the Reedy model category $dg_+\Alg(R)^{\Delta}$. 
% We then rewrite this as a morphism $ \{ \sigma_{\ge 0}(L_{[1-n]} \ten \Delta^j) \to C^{(n-1)}(j)^n\}_{j \in \Delta}   $, where $\sigma$ denotes brutal truncation.
% We 
then set $C^{(n)}(j)$ to be the pushout of the diagram
\[
  C^{(n-1)}(j) \la \Omega^{\bt}_{A\<\sigma_{\ge 0}(L_{[1-n]}\ten \Delta^j )^{[-n]}\> } \to  \Omega^{\bt}_{A\<\cone\ \sigma_{\ge 0}( L_{[1-n]}\ten \Delta^j )^{[-n]}\>}.\qedhere
\]
\end{proof}

%%In following, seems we can just take homotopy \'etale in the following, then restrict? along lines of {afpTcFEFCcor}, {stackyDlfpprop}. I think that's OK, because we're imposing a homotopy-invariant condition on morphisms, so framings are fine. I think we can argue directly with that localisation of \'et, without having to mention $\infty$-localisation.

The following is adapted from \cite[Definition \ref{NCstacks-Frigdef}]{NCstacks}:
\begin{definition}\label{Frigdef}
 Given   $B \in DG^+dg_+\hat{\Tc}\Alg_{R}^{\mathrm{dfp,cof}}$  for which the chain complexes
$
( \Omega^1_B\ten_{B^{\hat{e}}}(B^0)^{\hat{e}})^i
$
are acyclic for all $i > q$, and  a   w.e.-preserving homogeneous functor $F\co dg_+\cN\cF r(R) \to s\Set$   with a cotangent complex $L_F(B^0,x)$ at a point $x \in F(B^0)$, we say that a point $y \in D_*F(B)$ lifting $x \in F(B^0)$ is \emph{rigid} if the induced morphism
 \[
  L_F(B^0,x)\to \Tot \sigma^{\le q} \Omega^1_B\ten_{B^{\hat{e}}}(B^0)^{\hat{e}}
 \]
is a quasi-isomorphism of $B^0$-bimodules. 

We denote by $(D_*F)_{\rig}(B) \subset D_*F(B)$ the space of rigid points (a union of path components). Note that this defines a w.e.-preserving functor on $DG^+dg_+\hat{\Tc}\Alg_{R}^{\mathrm{dfp,cof},\et}$.
\end{definition}

 \begin{definition}
Given $F \co dg_+\cF r\Alg(R)\to s\Set$, define $QD_*F \co DG^+dg_+ \cF r\Alg_{R} \to s\Set$ by 
\[
 QD_*F(B):= D_*F(B)\by^h_{F(B^0)} (\iota_!F|)(B^0),
\]
for $\iota \co  dg_+\hat{\Tc}\Alg_{R}^{\mathrm{fp,cof}} \to dg_+\cF r\Alg_{R}$, %. Make the analogous definition for 
with $F|$ the restriction of $F$ to $dg_+\hat{\Tc}\Alg_{R}^{\mathrm{fp,cof}}$ and  $\iota_!$ as in Lemma \ref{iotaKanextnlemma}.
 
 %$Q^0_*F$ if we can't think of anything better. $\oL^0F$ maybe??
 \end{definition}

\begin{corollary}\label{etsitecor}
  If a w.e.-preserving homogeneous functor $F\co dg_+\cF r\Alg(R)\to s\Set$   has perfect cotangent complexes at all points  $x \in F(A)$ for all $A\in dg_+\hat{\Tc}\Alg_{R}^{\mathrm{fp,cof}}$, 
 then 
  for any w.e.-preserving functor $G \co DG^+dg_+ \cF r\Alg(R) \to s\Set$, we have  natural weak equivalences
\begin{align*}
 \map_{[DG^+dg_+ \cF r\Alg_{R},s\Set] }(QD_*F,G) &\simeq \map_{[DG^+dg_+ \hat{\Tc}\Alg_{R}^{\mathrm{dfp,cof}},s\Set]}(D_*F|,G|)\\
 \simeq \map_{[DG^+dg_+ \hat{\Tc}\Alg_{R}^{\mathrm{dfp,cof,\et}},s\Set]}((D_*F)_{\rig}, G|) 
  \end{align*}
  where $G|$ denotes the restriction of $G$ to the relevant subcategory.

  The same statements hold if we replace $\cF r$ with $\cN\cF r$ throughout. 
   \end{corollary}
\begin{proof}
We adapt the proof of \cite[Corollary \ref{NCstacks-etsitecor}]{NCstacks}.
 The resolution of Proposition \ref{cosimplicialresnprop} lies in $DG^+dg_+\hat{\Tc}\Alg_{R}^{\mathrm{dfp,cof},\et}$,
   and the canonical elements $u_j \in D_*F(C(j))$ are all rigid. That proposition gives us an expression $ \hat{F}_x \simeq \ho\LLim_{j \in \Delta} \oR \Spec C(j)$, and similarly setting 
\[
\hat{F}_{x,\rig}(B):= (D_*F)_{\rig}(B)\by^h_{F(B^0),x^*}\map(A,B^0)  
\]
gives $
 \hat{F}_{x,\rig} \simeq \ho\LLim_{j \in \Delta^{\op}} (\oR \Spec C(j))_{\rig}$,
 for  $(\oR \Spec C(j))_{\rig}  \co DG^+dg_+\hat{\Tc}\Alg_{R}^{\mathrm{dfp,cof},\et}  \to s\Set$ the right-derived  functor of $\Hom_{DG^+dg_+\hat{\Tc}\Alg_{R}^{\mathrm{dfp,cof},\et}}(C(j),-)$, which can be calculated explicitly via a cosimplicial frame since levelwise quasi-isomorphisms are homotopy \'etale.

 Now, for the inclusion functor $\theta \co DG^+dg_+\hat{\Tc}\Alg_{R}^{\mathrm{dfp,cof},\et} \to DG^+dg_+\hat{\Tc}\Alg_{R}^{\mathrm{dfp,cof}}$, the restriction $\theta^* \co [DG^+dg_+\hat{\Tc}\Alg_{R}^{\mathrm{dfp,cof}},s\Set] \to [DG^+dg_+\hat{\Tc}\Alg_{R}^{\mathrm{dfp,cof},\et},s\Set]$ has a derived left adjoint $\theta_!$ which sends $(\oR \Spec C)_{\rig}$ to $\oR\Spec C$, so 
 \[
 \theta_!\hat{F}_{x,\rig} \simeq \hat{F}_x|.
 \]
 Similarly, for the inclusion functor $\rho \co  DG^+dg_+ \hat{\Tc}\Alg_{R}^{\mathrm{dfp,cof}} \to DG^+dg_+ \cF r\Alg_{R}$, we have
 \[
  \rho_!\hat{F}_x| \simeq \hat{F}_x.
  \]
 
 Since homotopy colimits in $s\Set$ are universal and derived left adjoints commute with homotopy colimits, we have
  \begin{align*}
  \holim_{\substack{\longrightarrow \\ (x,A)}} \hat{F}_x(B) 
   &\simeq D_*F(B)\by_{F(B^0)}^h\holim_{\substack{\longrightarrow \\ (x,A)}}\map_{dg_+\cF r\Alg_R}(A,B^0)\\
 &\simeq D_*F(B)\by_{F(B^0)}^h(\iota_!F)(B^0)\simeq QD_*F(B),
% &\simeq D_*F(B) 
 \end{align*}
 for all $B  \in DG^+dg_+\cF r \Alg_R$, where the homotopy colimit is taken over $(x,A) \in (dg_+\hat{\Tc}\Alg_{R}^{\mathrm{fp,cof}})^{\op}\da F
 $. % and $\iota_!$ as in Lemma \ref{iotaKanextnlemma}. 
 
For $B \in DG^+dg_+\hat{\Tc}\Alg_{R}^{\mathrm{dfp,cof}}$, we have $QD_*F(B) \simeq D_*F(B)$ because $B^0 \in dg_+\hat{\Tc}\Alg_{R}^{\mathrm{fp,cof}}$. Since $\rho_!$ commutes with homotopy colimits, we thus have
\[
 QD_*F \simeq \holim_{\substack{\longrightarrow \\ (x,A)}} \hat{F}_x \simeq \rho_!\holim_{\substack{\longrightarrow \\ (x,A)}} \hat{F}_x| \simeq \rho_!D_*F|, 
\]
which is equivalent to the first desired equivalence.

  Similarly,
for  $B \in DG^+dg_+\hat{\Tc}\Alg_{R}^{\mathrm{dfp,cof}}$ we have
 $ \ho\LLim_{(x,A)} \hat{F}_{x,\rig}(B)\simeq (D_*F)_{\rig}(B)$ so
 \[
  D_*F| \simeq \holim_{\substack{\longrightarrow \\ (x,A)}} \hat{F}_x|  \simeq  \holim_{\substack{\longrightarrow \\ (x,A)}}\theta_!\hat{F}_{x,\rig}  \simeq  \theta_!(D_*F)_{\rig}, 
\]
 which is equivalent to the second desired equivalence. 
\end{proof}

% probably not much use for this:
%  Although $QD_*F$ will not usually inherit homogeneity from $F$, we automatically have the following:
%  \begin{lemma}
%   If a w.e.-preserving functor $F\co dg_+\cF r\Alg(R)\to s\Set$ is homogeneous, then for any diagram $A \to B \la C$ in  $DG^+dg_+ \cF r\Alg(R)$ for which $A \to B$ is a square-zero extension with $A^0 \cong B^0$, the natural map 
%   \[
%   QD_*F(A\by_BC) \to  QD_*F(A)\by_{QD_*F(B)}QD_*F(C)
%   \]
% is a weak equivalence.
%   
%  The same statement holds if we replace $\cF r$ with $\cN\cF r$ throughout. 
%  \end{lemma}

%%% 7/1/25: how about a section here using {ctbleex} and {CXlemma1}, giving something like {twistorthm} for these things to summarise the nice functors we get for $\ell$-adic local systems.
 
 %%%%want to mebbe combine with Dlfpcor? Need to summarise the most useful results here. {Dstarcor} seems so. {univlfprmk} highly relevant here, now cut, with {iotaKanextnlemma} moved to its place as largely duplicate.
 
 %%%Should we instead delay the big theorem to bisymplectic/Poisson section? I think the idea is that for pre-bisymplectic we could work on dgas, but need stacky dgas for bisymplectic or for double Poisson. Also problem that both are defined via $\iota_!$
 
 \section{Derived moduli of hyperconnections and the Riemann--Hilbert map} \label{RHsn} 

 \subsection{Non-commutative algebraic Betti, de Rham and Dolbeault functors}
 
 The following definitions slightly generalise Betti and de Rham moduli stacks by allowing perfect complexes and non-commutative inputs.

  \begin{definition}\label{Bettidef}
 For  $B  \in dg_+\Alg_{\Z}$ and any topological space $X$, define  the $B$-linear Betti  dg category of $X$ to be
  \[
    \per_{dg}^{M,\oB}(B):=\oR\Gamma(M,  \per_{dg}(B)),
  \]
the dg category  of locally constant $B$-linear perfect complexes on $X$.
\end{definition}

 \begin{definition}\label{dRalgperdef}
 For $B  \in dg_+\Alg_R$ and   any smooth $R$-scheme $Y$,   define the algebraic $B$-linear de Rham  dg category of $Y$  to be 
\[
 \per_{dg}^{Y,\dR,\alg}(B):=\oR\Gamma(Y,  D_*\per_{dg}(\Omega^{\bt,\alg}_Y\ten_RB)),
\]
which as in Lemma \ref{hyperconnlemma} parametrises locally perfect $\sO_Y\ten_RB$-modules on $Y$ with flat hyperconnections.
\end{definition}

Replacing the Zariski site with the \'etale site, the same definition immediately generalises to smooth Deligne--Mumford stacks.

\begin{definition}
 For $B  \in dg_+\Alg_R$ and   any smooth $R$-scheme $Y$, define the $B$-linear algebraic Dolbeault dg category of $Y$  to be 
\[
 \per_{dg}^{Y,\Dol,\alg}(B):=\oR\Gamma(Y,  D_*\per_{dg}(\Omega^{\#,\alg}_Y\ten_RB)),
\]
which as in Lemma \ref{hyperconnlemma} parametrises locally perfect $\sO_Y\ten_RB$-modules on $Y$ with generalised Higgs forms.
    \end{definition}

 \subsection{Non-commutative analytic de Rham and Dolbeault functors}

%  \begin{definition}
%  Given a  complex manifold $X$  define $\Perf_{X, \dR}^{\hol}\co dg_+\hat{\Tc}\Alg(\Cx) \to s\Set$ to be the functor 
%  \[
% B_{\bt} \mapsto \oR\Gamma(X, D_*\Perf(\Omega^{\bt,\hol}_{X/\Cx}\hten_{\Cx}B_{\bt}))
%  \]
%  of locally perfect $\sO_X^{\hol}\hten_{\Cx}B_{\bt}$-modules on $X$ with flat $\pd$-hyperconnection.
%  \end{definition}
%  Note that because the modules of holomorphic differentials are nuclear Fr\'echet, the completed tensor product $\hten$ is unambiguous.

%  
%  \begin{definition}
%   Given a differentiable real manifold $X$, and writing $\sA^{\bt}_X$ for the complex of real $\C^{\infty}$-differential forms, we define  $\Perf_{X, \dR}^{\C^{\infty}}\co dg_+\hat{\Tc}\Alg(\R) \to s\Set$ to be the functor 
%  \[
% B_{\bt} \mapsto \oR\Gamma(X, D_*\Perf(\sA^{\bt}_{X}\hten_{\R}B_{\bt}))
%  \]
%  of locally perfect $\C^{\infty}(X,B_{\bt})$-modules on $X$ with flat $d$-hyperconnection (via the identification  $\C^{\infty}(X,B_{\bt}) \cong \sA^0_X\hten_RB_{\bt}$).
%   \end{definition}

\begin{definition}\label{OmegaBandef}
 Given a complex Banach space $V$ and a complex  manifold $Y$, we let $\Omega^{\bt}_{Y}(V)$ be the complex of sheaves of  holomorphic $V$-valued differential forms on $Y$, equipped with the de Rham differential $\pd \co \Omega^p_{Y}(V)\to \Omega^{p+1}_{Y}(V)$. We write  $\Omega^p(Y,V):= \Gamma(Y,\Omega^p_{Y}(V))$.
 
 Similarly, given a real Banach space $V$ and a real manifold $M$, we let $\sA^{\bt}_{M}(V)$ be the complex of sheaves of infinitely differentiable  $V$-valued differential forms on  $M$, equipped with the de Rham differential $d \co \sA^p_{M}(V)\to \sA^{p+1}_{M}(V)$. We write $A^{\bt}(M,V):=\Gamma(M,\sA^{\bt}_{M}(V))$.
  \end{definition}

\begin{lemma}\label{BaninjtenlemmaCx}
 For the sheaves $\Omega^p_{Y}:=\Omega^p_{Y}(\Cx)$ of  Fr\'echet spaces, we have  natural isomorphisms
 \[
  \Omega^p_{Y}(V)\cong \Omega^p_{Y}\hten_{\eps}V
 \]
of sheaves on $Y$ for any complex Banach space $V$, where $\hten_{\eps}$ is the completed injective tensor product
\end{lemma}
\begin{proof}
Since the question is local, it suffices to prove it on an open polydisc  $D^n:=\{\uline{z}~:~ \|z_i\|<1 \forall i\} \subset  \Cx^n$. 
 By  \cite[Theorem 5]{BarlettaDragomirVectorValdHoloFns}, $V$-valued holomorphic functions on $D^n$    are analytic and uniformly convergent, giving 
 \begin{align*}
  \Omega^0(D^n,V) &\cong \Lim_{r<1} \{\sum \lambda_Iz^I \in V\llb z_1, \ldots z_n\rrb ~:~ \|\lambda_I\|r^{|I|} \text{ bounded}\}\\
  &\cong \Omega^0(D^n,\Cx)\hten_{\eps}V,
 \end{align*}
and since $\Omega^p(D^n,V) \cong \bigoplus_{i_1<\ldots<i_p} \Omega^0(D^n,V)dx_{i_1}\wedge \ldots \wedge dx_{i_p}$, it immediately follows that $ \Omega^p(D^n,V) \cong\Omega^p(D^n,\Cx)\hten_{\eps}V$, as required.
  \end{proof}
 
 \begin{lemma}\label{BaninjtenlemmaR}
 For the sheaves $\sA^p_{M}:=\sA^p_{M}(\R)$ of  Fr\'echet spaces, we have  natural isomorphisms
 \[
  \sA^p_{M}(V)\cong \sA^p_{M}\hten_{\eps}V
 \]
of sheaves on $M$ for any real complete %Hausdoff 
locally convex topological vector space
space $V$, where $\hten_{\eps}$ is the completed injective tensor product
\end{lemma}
\begin{proof}
 This is \cite[Theorem 44.1]{treves}.
\end{proof}

Consequently, we make the following definition:

\begin{definition}
 Given a  complex manifold $Y$ and $B \in dg_+\hat{\Tc}\Alg(\Cx)$, write $\Omega^{\bt}_Y(B):= \Omega^{\bt}_Y\hten_{\Cx}B_{\bt}$, the sheaf (in double complexes) of $B$-valued holomorphic forms equipped with differential $\pd$.

 Given a differentiable real manifold $X$ and $B \in dg_+\hat{\Tc}\Alg(\R)$, with $\sA^{\bt}_M$  the complex of real $\C^{\infty}$-differential forms, write $\sA^{\bt}_M(B):=\sA^{\bt}_{M}\hten_{\R}B_{\bt}$, the sheaf (in double complexes) of $B$-valued smooth forms equipped with differential $d$.
\end{definition}
Note that because the modules of holomorphic and smooth differentials are nuclear Fr\'echet, the completed tensor product $\hten$ is unambiguous in both cases.

  \begin{definition}
For $B  \in dg_+\hat{\Tc}\Alg_{\R}$ and  any real $\C^{\infty}$ manifold $M$, the $B$-linear de Rham dg category of $M$  is given by 
\[
 \per_{dg}^{M,\dR,\C^{\infty}}(B):=\oR\Gamma(M,  D_*\per_{dg}(\sA^{\bt}_M(B))).
\]
\end{definition}
Note that via Lemma \ref{hyperconnlemma}, this can be identified with the dg category of perfect $\C^{\infty}(M,B)$-modules with flat $d$-hyperconnections.

\begin{definition}\label{dRanperdef}
 For $B  \in dg_+\hat{\Tc}\Alg_{\Cx}$ and   any complex manifold $Y$, the $B$-linear analytic de Rham dg category of $Y$  is given by 
\[
 \per_{dg}^{Y,\dR,\an}(B):=\oR\Gamma(Y,  D_*\per_{dg}(\Omega^{\bt,\hol}_Y(B))).
\]
    \end{definition}
Note that via Lemma \ref{hyperconnlemma}, this can be identified with the dg category of perfect $\sO_Y^{\hol}(B)$-modules with flat $\pd$-hyperconnections.

\begin{definition}
 For $B  \in dg_+\hat{\Tc}\Alg_{\Cx}$ and   any complex manifold $Y$, the $B$-linear analytic Dolbeault dg category of $Y$  is given by 
\[
 \per_{dg}^{Y,\Dol,\an}(B):=\oR\Gamma(Y,  D_*\per_{dg}(\Omega^{\#,\hol}_Y(B))).
\]

Similarly, the $B$-linear $\C^{\infty}$ Dolbeault dg category of $Y$  is given by 
\[
 \per_{dg}^{Y,\Dol,\C^{\infty}}(B):=\oR\Gamma(Y,  D_*\per_{dg}(\sA^{\#,\bt}_Y(B))).
\]
where $\sA^{\#,\bt}_Y(B)$ is the graded sheaf $\bigoplus_n \sA^n_Y(B)[-n]=\bigoplus_{p,q} \sA^{p,q}_Y(B)[-p-q] $ equipped with the differential $\bar{\pd} \co \sA^{p,q}_Y(B) \to \sA^{p,q+1}_Y(B)$.
    \end{definition}
Note that via Lemma \ref{hyperconnlemma}, the analytic Dolbeault dg  category can be identified with the dg category of perfect $\sO_Y^{\hol}(B)$-modules with generalised Higgs forms. The $\C^{\infty}$ Dolbeault dg category likewise consists of perfect $\C^{\infty}(M,B)$-modules with flat $\bar{\pd}$-hyperconnection.

  \subsection{The Poincar\'e Lemma and Riemann--Hilbert map}
  
\begin{definition}
 Given an associative algebra $(B,\delta)$ in cochain complexes, recall that the Maurer--Cartan set  is given by
 \[
  \mc(B):=\{ b \in B^1~:~ \delta b +b\cdot b =0 \in B^2\}.
 \]
This admits a gauge action by the group $(B^0)^{\by}$ of invertible elements $g$ of $B^0$, given by 
\[
g\star b := gbg^{-1} - (\delta g)g^{-1}.
 \]
\end{definition}

We now give a form of Poincar\'e Lemma in this setting: 
\begin{lemma}\label{poincarelemma}
 If $B$ is a unital associative complete multiplicatively convex differential graded  $\Cx$-algebra and $\bD^n \subset \Cx^n$ the open unit polydisc, then every Maurer-Cartan element $\omega \in \mc(\Tot(\Omega^{\bt}(\bD^n,B)))$ is gauge equivalent to an element of the subset $\mc(B)$, via an invertible element of $\Tot\Omega^{<n}(\bD^n,B)^0$. %%$\Tot$ is needed there, fo inductive reasons. 
 
 Similarly, if $B$ is a unital associative complete multiplicatively convex differential graded $\R$-algebra and $J^n \subset \R^n$ the open unit hypercube, then every Maurer-Cartan element $\omega \in \mc(\Tot(A^{\bt}(J^n,B)))$ is gauge equivalent to an element of the subset $\mc(B)$, via an invertible element of $\Tot A^{<n}(J^n,B)$.
\end{lemma}
\begin{proof}
We begin by proving the first statement.

 If $n=1$, then we can write $\omega = (\phi, \eta) \in  \Omega^0(\bD^1,B^1) \by \Omega^1(\bD^1, B^0)$, satisfying $\phi \in \mc(\Omega^0(\bD^1,B))$ and $d\omega + \delta \eta + [\omega,\eta]=0 \in \Omega^1(\bD^1,B^1)$. 
 
 We can then let $g \in \Omega^0(\bD^1,B^0)$ be the 
 iterated integral 
 \[
  g(t):=\sum_{n\ge 0} (-1)^n \int_{t_1=0}^t \eta(t_1) \wedge \int_{t_2=0}^{t_1}\eta(t_2) \wedge \ldots  \int_{t_n=0}^{t_{n-1}} \eta(t_{n-1}).
  \]
 This converges because with respect to the sup norm $\|-\|$ associated to any submultiplicative seminorm of $B$  on each closed disc within $\bD^1$, we have $|\eta(t)|\le 
 \|\eta\|$, and thus $\|g(t)\|\le \exp(\|\eta\|t)$. 

This element satisfies  $dg = -\eta \wedge g$ and $g(0)=1$; reversing the order of multiplication and the sign of $\eta$ similarly gives us an element $f \in \Omega^0(\bD^1,B^0)$ with $df= f \wedge \eta$ and $f(0)=1$. Thus $d(fg)=d(gf)=0$, so $fg$ and $gf$ are constant, hence both $1$, implying $g$ is indeed invertible, with $f=g^{-1}$. 
 
Now, writing $g^{-1}\star_{\delta} \phi:= g^{-1}\phi g +g^{-1}\delta g$, we have
\begin{align*}
 g^{-1}\star (\phi, \eta) &= g^{-1}(\phi, \eta)g + g^{-1}(\delta+d)g\\ 
 &=(g^{-1}\phi g +g^{-1}\delta g, g^{-1}\eta g + g^{-1}dg )
 = (g^{-1}\star_{\delta} \phi, 0),
\end{align*}
and $\phi_0 := g^{-1}\star \phi$ is a Maurer--Cartan element, so $d\phi_0=0$ and hence $\phi_0$ is constant, implying $\phi_0 \in \mc(B)$.

To prove the case for general $n$, we can apply the step above inductively to the inclusions
\[
 \Tot\Omega^{\bt}(\bD^{n-m-1},B) \into \Tot\Omega^{\bt}(\bD^1,\Tot\Omega^{\bt}(\bD^{n-m-1},B)) \cong  \Tot\Omega^{\bt}(\bD^{n-m},B)
\]
of pro-dg Banach algebras, giving the required result.

\smallskip
The  second statement now proceeds almost identically. The only difference is that for the existence of $g$, we need to check convergence of all the derivatives $D^{(n)}(g)$ (and not just $g$ itself) on compact subspaces. Expressions for bounds on $|D^{(n)}g(x)|$, in terms of $|g(x)|$ and $\|D^{(i)}\eta\|$ for $i<n$,  follow inductively from the formula $dg = -\eta \wedge g$.
\end{proof}

\begin{proposition}\label{poincareprop}
For $B  \in dg_+\hat{\Tc}\Alg_{\Cx}$ and 
$\bD^n \subset \Cx^n$ the open unit polydisc,
 the canonical dg functor
 \begin{align*}
\eta \co \per_{dg}(B) &\to (D_*\per_{dg})(\Omega^{\bt}(\bD^n,B)) %\\
% M &\mapsto M\ten_{ B}\Omega^{\bt}(\bD^n,B)
\end{align*}
is a quasi-equivalence.
 
 For $B \in dg_+\hat{\Tc}\Alg_{\R}$ and $J^n \subset \R^n$ the open  hypercube $(-1,1)^n$, the canonical dg functor 
 \begin{align*}
\eta \co  \per_{dg}(B) &\to (D_*\per_{dg})(A^{\bt}(J^n,B)) %\\
%  M &\mapsto  M\ten_{B} A^{\bt}(J^n,B) 
 \end{align*}
is a quasi-equivalence.
  \end{proposition} 
\begin{proof}
Lemma \ref{hyperconnlemma} allows us to replace $D_*\per_{dg}$ with the relevant category of flat hyperconnections.

For cofibrant $B$-modules $M,N$, these dg functors can thus be realised on morphisms as  the natural inclusion maps 
\begin{align*}
 \HHom_{B}(M,N)&\to \HHom_{B}(M, \Tot^{\Pi}(\Omega^{\bt}(\bD^n,B)\ten_{B}N))\\
 \HHom_{B}(M,N)&\to \HHom_{B}(M, \Tot^{\Pi} (A^{\bt}(J^n,B)\ten_{B}N)),
 \end{align*}
so full faithfulness is equivalent to showing that the inclusions $B \to \Tot \Omega^{\bt}(\bD^n,B)$ and $B \to  \Tot A^{\bt}(J^n,B)$ are quasi-isomorphisms (boundedness in the cochain direction makes $\Tot^{\Pi}$ and $\Tot$ equivalent here).

To establish this, it suffices to show that the inclusions $\Tot \Omega^{\bt}(\bD^{n-1},B) \to \Tot \Omega^{\bt}(\bD^n,B) $ and $\Tot A^{\bt}(J^{n-1},B) \to \Tot A^{\bt}(J^n,B)$, coming from the projection maps $\bD^n \to \bD^{n-1}$ and $J^n \to J^{n-1}$,  are quasi-isomorphisms. This follows because the integration operators $\int_{z_n=0}^{z}$ give a contracting homotopy; these are well-defined because they define  morphisms $\Omega^p(\bD^n,\Cx) \to \Omega^{p-1}(\bD^n,\Cx) $ and $A^p(J^n,\R)\to A^{p-1}(J^n,\R) $ of Fr\'echet spaces, so extend to the completed injective tensor product with $B$.

It remains to show that the dg functors are essentially surjective. Writing $C^{\bt}$ for $\Omega^{\bt}(\bD^n,B)$ (resp. $A^{\bt}(J^n,B)$), we wish to show that every  flat $C$-hyperconnection $(E,\nabla)$, with $E$ perfect as a $C^0$-module, lies in the essential image of the dg functor. Since $E$ is perfect, there exists a quasi-isomorphic $C^0$-module $E'$ which is freely generated as a graded $C^0_{\#}$-module by a levelwise finite graded generating set $S$. %(and in fact homotopy idempotent in a finitely generated free module). 
By the final part of Lemma \ref{hyperconnlemma}, there then exists a flat $C$-connection $\nabla'$ on $E'$ such that $(E',\nabla')$ is homotopy equivalent to $(E,\nabla)$.

%%%%below just gives free case, not homotopy idempotent in it.

If we fix the graded set $S$ of generators, then $(E',\nabla')$ such  is determined by first specifying the differential $\delta$ making the module $E'_{\#}$  into a chain complex, then giving the hyperconnection $\nabla \co E_{\#} \to \Tot^{\Pi} (E_{\#}\ten_{C^0}C^{>0}_{\#})$  satisfying $(\delta \pm \nabla)^2=0$.

Writing $V:= S.\bK$ with trivial differential and noting that $\EEnd_{\bK}(V) $ is pro-finite-dimensional,
the set of such structures is isomorphic to
\[
 \mc(\HHom_{\bK}(V, \Tot^{\Pi}(V\ten_{\bK}C )) \cong \mc(\EEnd_{\bK}(V)\hten_{\bK}C), 
\]
% for the obvious DGAA structure on the space of homomorphisms,
since Maurer--Cartan elements $\omega= \sum_{i=0}^n \omega^i$ correspond to the data of a closed differential $\id_V \ten \delta_{C^0} +\omega^0$ on $V\ten_{\bK}C^0$ (giving the complex $E'$) together with a flat $C$-hyperconnection  $\nabla' = (d+\omega^1)+\omega^2+\dots +\omega^n$ on $E'$.

We can now appeal to Lemma \ref{poincarelemma}, applied to the pro-dg Banach algebra $ B \hten_{\bK} \EEnd_{\bK}(V) \cong \HHom(V,V\ten_{\bK}B)$. This provides a gauge equivalence $g$ from $\omega$ to a Maurer-Cartan element $\phi_0 \in \mc(B)$. 
In other words, the gauge equivalence provides an isomorphism
\begin{align*}
g \co (E',\nabla') \to \eta(V\ten_{\bK}B, \id\ten \delta_B + \phi_0)
\end{align*}
in the category of flat $C$-hyperconnections.
  \end{proof}

The following $\C^{\infty}$ and analytic Riemann--Hilbert correspondences with LMC dg algebra coefficients now follow immediately from Proposition \ref{poincareprop} by taking derived global sections.
\begin{corollary}\label{poincarecor}
For  $M$ a real $\C^{\infty}$ manifold, there are quasi-equivalences
\[
\per_{dg}^{M,\oB}(B)\to \per_{dg}^{M,\dR,\C^{\infty}}(B)
\]
of dg categories, natural in $B \in dg_+\hat{\Tc}\Alg_{\R}$, from
the Betti moduli functor  to the $\C^{\infty}$ de Rham moduli functor.

Similarly, for $Y$ a complex manifold, there are weak equivalences
\[
\per_{dg}^{Y,\oB}(B)\to \per_{dg}^{Y,\dR,\an}(B) \to \per_{dg}^{Y,\dR,\C^{\infty}}(B)
\]
of dg categories, natural in $B \in dg_+\hat{\Tc}\Alg_{\Cx}$, from
the Betti moduli functor  to both analytic and $\C^{\infty}$ de Rham moduli functors.
\end{corollary}

\begin{corollary}\label{dRcotperfcor}
 For  $M$ a compact oriented $d$-dimensional real $\C^{\infty}$ manifold,  the simplicial set-valued functor $\Perf^{M,\dR,\C^{\infty}}$ given by the nerve of the core of  $\per_{dg}^{M,\dR,\C^{\infty}}$  has perfect cotangent complexes at all points.
\end{corollary}
\begin{proof}
 By Corollary \ref{poincarecor}, we can replace $\per_{dg}^{M,\dR,\C^{\infty}}$ with the analytification of $\per_{dg}^{M,\oB}$. A point of $ \per_{dg}^{M,\oB}(B)$ corresponds to an $\infty$-local system $\vv$ of perfect $B$-complexes on $M$. For any $B^e$-module $N$, endomorphisms of $\vv\ten_B^{\oL}(B \oplus N)$ in the fibre of $ \per_{dg}^{M,\oB}(B \oplus N) \to  \per_{dg}^{M,\oB}(B)$ over $\vv$ are given by 
 \[
  \oR\Gamma(M, \oR\sHom_B(\vv, \vv\ten_B^{\oL}N)) \simeq \oR\Gamma(M, \vv\ten_R\oR\sHom_B(\vv,B))\ten_{B^e}N.      
\]

Writing $\vv^*:=\oR\sHom_B(\vv,B))$,  Poincar\'e duality gives a quasi-isomorphism between this and 
\[
\oR\HHom_{B^e}(   \oR\Gamma(M, \vv\ten_{\R}\vv^*)[d], N).
 \]
Since the  functor $ \Perf^{M,\oB}$ is homogeneous, its tangent functor is determined by the complexes of morphisms, so this implies its (algebraic) cotangent complex at $\vv$ is quasi-isomorphic to $\oR\Gamma(M, \vv\ten_{\R}\vv^*)[d-2]$, which is perfect. At the corresponding point $(\sE, \nabla)$ of $\Perf^{M,\dR,\C^{\infty}}$, the (analytic) cotangent complex is thus the shift by $[d-2]$ of the perfect $B^{\hat{e}}$-complex
\[
 \oR\Gamma(M, \vv\ten_{\R}\vv^*)\ten^{\oL}_{B^e} B^{\hat{e}} \simeq  \oR\Gamma(Y, (\oR\sHom_{ \sA^0_M(B)}(\sE, \sE\hten_{\sA^0_M}\sA^{\bt}_M(B)), \nabla^{\ad})). \qedhere %%that's OK because RHS is $B^{\hat{e}}$-module, and it's the RH action we're using.
\]
\end{proof}

 \begin{remark}\label{cfPortarmk}
  The analytic de Rham comparison statement in Corollary \ref{poincarecor}, at least when restricted to commutative DGAs, is very similar in content to the results of \cite{portaDerivedRH}. The latter is a statement about a functor on derived Stein spaces, and could be interpreted in our terms as coming from weak equivalences
  \[
  \per_{dg}(B) \to D_*\per_{dg}( \Omega^{\bt}_Y \odot B)
%\Perf_{Y, \oB}(B) \to \Perf_{Y,\dR,\an}  
\]
for $B \in \dg_+\EFC\Alg_{\Cx}$, where $\odot$ is the coproduct for EFC algebras. 

However, the corresponding statement for the $\C^{\infty}$ de Rham functor seems unlikely, since $A^0(M)$ is far from being a finitely generated free EFC algebra.  

Another reason for our use of topological algebras instead of  (F)EFC base change to set up the moduli functors is for their behaviour on  non-commutative algebras,  since \emph{a priori} (F)EFC base change has poor exactness properties. 
 \end{remark}

 \begin{remark}\label{deRhamEFCrmk}
 An immediate consequence of Corollary \ref{poincarecor} is that the de Rham moduli functors send abstract quasi-isomorphisms of LMC dg algebras to quasi-equivalences, so descend to functors on FEFC-DGAs by  Proposition \ref{keyBanFEFCprop}, or if restricted to commutative algebras,  to EFC-DGAs by   Proposition \ref{keyBanEFCprop}. However, those propositions are overkill in this case because we already have an equivalent (Betti) functor on abstract DGAs.
 \end{remark}

\begin{definition}
 For $B  \in dg_+\Alg_R$ and   any smooth $R$-scheme $Y$, the $B$-linear algebraic de Rham dg category of $Y$  is given by 
\[
 \per_{dg}^{Y,\dR,\alg}(B):=\oR\Gamma(Y,  D_*\per_{dg}(\Omega^{\bt,\alg}_Y\ten_RB)).
\]
    \end{definition}
 
 The following shows that a Riemann--Hilbert map exists in this generality as an analytic (FEFC) morphism from the algebraic de Rham functor to the Betti functor. Here, we are following the notation of \S \ref{analytificationsn2} in writing $F^{\FEFC}$ for the composition of $F$ with the forgetful functor from FEFC-DGAs to DGAs. 
 \begin{corollary}\label{RHcoralg}
  For any smooth complex variety $Y$, we have a natural map
  \[
   (\per_{dg}^{Y,\dR,\alg})^{\FEFC} \to (\per_{dg}^{Y(\Cx)_{\an},\oB})^{\FEFC}
  \]
of functors  on  $dg_+\FEFC(\Cx)$ taking values in the $\infty$-category of dg categories. 
   \end{corollary}
\begin{proof}
For the morphism $h \co Y_{\an} \to Y_{\Zar}$ from the analytic to the Zariski  site, we have a  dg functor
 \[
  h^{-1} D_*\per_{dg}(\Omega^{\bt,\alg}_{Y}\ten_{\Cx}B)) \to D_*\per_{dg}(\Omega^{\bt,\hol}_{Y^{\an}}(B))
\]
of hypersheaves on $Y_{\an}$,
natural in $B \in  dg_+\hat{\Tc}\Alg_{\Cx}$. On taking derived global sections, this gives
\[
 \per_{dg}^{Y,\dR,\alg}(B) \to \per_{dg}^{Y,\dR,\an}(B),
\]
and hence $\per_{dg}^{Y,\dR,\alg}(B) \to \per_{dg}^{Y(\Cx)_{\an},\oB}(B)$ 
 by inverting the second equivalence of Corollary \ref{poincarecor}.

 Applying Corollary \ref{keyBanFEFCcor} then gives rise to a natural map of functors  on  $dg_+\FEFC(\Cx)$.

 \end{proof}

\begin{remark}\label{genmodulirmk}
 In choosing to work with perfect complexes, we have taken almost the largest well-behaved stacks available. 
 %Any open substack will inherit $\infty$-geometricity when $Y$ is proper, so 
 It is straightforward to deduce statements about moduli of local systems, or of complexes with  amplitude $[a,b]$, and to introduce constraints in terms of ranks or Euler characteristics. 
 
 To see that equivalences such as Proposition \ref{poincareprop} restrict to equivalences of these examples, note that everything is functorial in the complex manifold $\Cx$ or real manifold $M$, and this includes inclusions of points. Since we can characterise a local system (resp. vector bundle) as a constructible complex (resp. perfect complex) whose pullback to any point is concentrated in degree $0$, Proposition \ref{poincareprop} restricts to give a correspondence between local systems and vector bundles with connection. There are similar statements for any constraint which can be defined pointwise.
 
 Also note that when we restrict to  commutative algebras, tensor products give us a symmetric monoidal structure on our dg categories, and the quasi-equivalences of dg categories are  monoidal. This means that the equivalences of Corollary \ref{poincarecor}  automatically extend to moduli of any objects defined in terms of the monoidal structure. One example is moduli of $\cP$-algebras in these dg categories, for dg operads $\cP$. Another is  moduli of $G$-representations in perfect complexes, for an affine  algebraic group $G$,  since  we can regard $G$-representations as $O(G)$-comodules in perfect complexes, for the Hopf algebra $O(G)$ of algebraic functions on $G$. 
 \end{remark}

 \subsection{The Dolbeault--Grothendieck Lemma}

We write $\bD^n \subset \Cx^n$ for the open unit polydisc, and more generally $\bD^n_r \subset \Cx^n $ for the open polydisc of radius $r$.
We now have the following adaptation of the Dolbeault--Grothendieck (or 
$\bar{\pd}$-Poincar\'e) lemma:
\begin{lemma}\label{dolbeaultlemma1}
For any complete locally convex complex topological vector  space $V$, any $p \ge 0$ and any $r<1$, the canonical commutative diagram
%\[
%\begin{CD}
%  \Omega^{p}(\bD^1,V) @>{\iota_1}>> A^{p,\bt}(\bD^1,V)\\
%  @V{\rho_{\Omega}} VV @VV{\rho_A}V \\
%   \Omega^{p}(\bD^1_r,V) @>{\iota_r}>> A^{p,\bt}(\bD^1_r,V),
%   \end{CD}
%\]
\[
 \begin{CD}
  \Omega^{p}(\bD^1,V)@>{\rho_{\Omega}}>> \Omega^{p}(\bD^1_r,V) \\
     @V{\iota_1}VV @VV{\iota_r}V \\
   A^{p,\bt}(\bD^1,V)  @>{\rho_A}>>  A^{p,\bt}(\bD^1_r,V),
   \end{CD}
\]
for inclusion maps $\iota$ and restriction maps $\rho$,
admits a morphism $\sigma \co A^{p,\bt}(\bD^1,V) \to \Omega^{p}(\bD^1_r,V)$ with $\sigma  \circ \iota_1= \rho_{\Omega}$, together with a homotopy $h$  between $\iota_r \circ \sigma $ and $\rho_A$.
\end{lemma}
\begin{proof}
It suffices to prove this for $B=\Cx$ with $\sigma $ and $h$ as morphisms of nuclear Fr\'echet spaces, since the general statement follows by taking the completed  tensor product with $V$.

If $n=1$, for $s<r<1$  we define $h \co A^{0,1}(\bD^1,\Cx) \to A^{0,0}(\bD^1_s,\Cx)$ by 
\[
  h(\alpha)(z):= \frac{1}{2\pi i}\int_{|w|\le  r}\frac{\alpha(w)\wedge dw}{z-w};
\]
as in \cite[p.5]{GriffithsHarris},
this lies in $A^{0,0}(\bD^1_s,\Cx)$  and satisfies $\bar{\pd}h(\alpha)=\alpha$. Cauchy's Integral formula \cite[p.2]{GriffithsHarris} applied to $\bar{z}$ gives 
\[
 2\pi i(h(d\bar{z})- \bar{z}) = \int_{|w|=r}\frac{\bar{w}}{w-z}dw = \int_{|w|=r}\frac{r^2}{w(w-z)}dw = 0
\]
for $z <r$, since the residues cancel. %%note that for holomorphic $f$, sa,e argt gives gives $h(fd\bar{z}) = f\bar{z} \pm (f -f(0))$.

Note that for $\alpha = fd\bar{z}$, we can rearrange $h(\alpha)$ as %$\int_{|w|\le  r} \frac{\alpha(w)\wedge dw}{z-w}$ as  
\[
 h(\alpha) = \frac{1}{2\pi i}\int_{|w|\le  r}\frac{f(w)-f(z)}{z - w} dw d\bar{w} +f(z)h(d\bar{z}). 
\]
Since $h(d\bar{z})= \bar{z}$, we can thus bound $| h(\alpha)(z)|$ by 
\[
\frac{1}{2\pi}\pi r^2\|D(f)\|_r + |f(z)\bar{z}| \le \frac{r^2}{2}\|D(f)\|_r + r\|f\|_r
\]
where $\|-\|_r$ denotes the sup norm on $|z|\le r$. 
%% $\bar{\pd}$'s automatically OK. What does $\pd$ give us? The $1/(z-w)$ becomes $-1/(z-w)^2$, and then we'd need to apply a generalisation of the CIF.
Similar considerations, differentiating the Cauchy integral formula, give bounds on the norms $\| \frac{\pd^n}{\pd z^n}h(fd\bar{z}) \|_r$ in terms of $\|D^{(i)}f\|_r$ for $i \le n$,  and the other derivatives are bounded since $ \frac{\pd^n}{\pd z^n} \frac{\pd^j}{\pd \bar{z}^j}h(fd\bar{z}) = \frac{\pd^n}{\pd z^n} \frac{\pd^{j-1}}{\pd \bar{z}^{j-1}}f$ for $j>0$,   so $h$ is a continuous map of Fr\'echet spaces.

Now, we can extend $h$ to a map $A^{1,1}(\bD^1,\Cx) \to A^{1,0}(\bD^1_r,\Cx)$ by setting $h(\alpha dz)=h(\alpha)dz$, and then the operator
\[
\sigma := \rho_A -  (h\circ\bar{\pd}+\bar{\pd} \circ h) 
\]
vanishes on $A^{p,1}$ and sends $A^{p,0}$ to $ \Omega^p$, so the required conditions are all satisfied.

For $n>1$, we have a homotopy as above for each variable $z_1, \ldots, z_n$, and we let $h$ be the sum of all these. 
\end{proof}

We now have the following non-abelian analogue of the Dolbeault--Grothendieck  lemma:
\begin{lemma}\label{dolbeaultlemma}
 If $B$ is a unital associative differential graded complete multiplicatively convex $\Cx$-algebra and $r<1$, then the image in $\mc(\Tot(A^{\#,\bt}(\bD^n_r,B))) $ of any  Maurer-Cartan element $\omega \in \mc(\Tot(A^{\#,\bt}(\bD^n,B)))$ is gauge equivalent to an element of the subset $\mc(\Tot(\Omega^{\#}(\bD^n_r,B)))$,  
  via an invertible element of $\Tot A^{<2n}(\bD^n_r,B)^0$). 
\end{lemma}
\begin{proof}
We adapt the proof of Lemma \ref{poincarelemma}.
 
If $n=1$, then we can write $\omega = \phi+\eta$, for $\phi \in  A^0(\bD^1,B^1)\oplus A^{1,0}(\bD^1,B^0)$ and $\eta \in  A^{0,1}(\bD^1,B^0) \oplus A^2(\bD^1,B^{-1})$, satisfying $\phi \in \mc(A^{\#,0}(\bD^1,B))$ and $\bar{\pd}\omega + \delta \eta + [\omega,\eta]=0 \in A^{0,1}(\bD^1,B^1)\oplus A^{1,1}(\bD^1,B^0)$. %%no $\eta^2$ term, as $A^{\#,2}=0$.
 
Now,  the proof of  Lemma \ref{poincarelemma} gives us an invertible element $g \in \Tot A^{<2}(\bD^1,B)^0$ with $dg = -\eta \wedge g$, so $\pd g = -\eta \wedge g -\bar{\pd} g$. Comparing decompositions by Hodge type, this implies that  $\pd g \in A^{1,1}(\bD^1,B)$.
Writing $g^{-1}\star_{\delta} \phi:= g^{-1}\phi g +g^{-1}\delta g$, we thus have
 \begin{align*}
 g^{-1}\star (\phi + \eta) &= g^{-1}(\phi + \eta)g + g^{-1}(\delta+\bar{\pd})g\\
 &= g^{-1}\star_{\delta} \phi - g^{-1}\pd g 
%  \in \left(A^0(\bD^1,B^1)\oplus A^{1,0}(\bD^1,B^0)\right \oplus \left(A^{0,1}(\bD^1,B^0) \oplus A^2(\bD^1,B^{-1})\right)
\end{align*}
so we have replaced $(\phi,\eta)$ with a gauge-equivalent pair $(\phi',\eta')$ with 
$\phi':= g^{-1}\star_{\delta} \phi \in  A^0(\bD^1,B^1)\oplus A^{1,0}(\bD^1,B^0)$ and $\eta' :=- g^{-1}\pd g \in A^{1,1}(\bD^1,B^{-1})$ (a more restrictive condition than $\eta$ satisfied).

 Now, for our element $\eta' \in A^{1,1}(\bD^1,B^{-1})$, the homotopy $h$ from the proof of Lemma \ref{dolbeaultlemma1} gives $h(\eta') \in A^{1,0}(\bD^1,B^{-1})$ with $\bar{\pd}h(\eta')= \eta'$, so both  $\eta'\wedge h(\eta')$ and $h(\eta')\wedge \eta'$ lie in    $A^{1,2}(\bD^1_r,B)=0$. 
Applying the gauge transformation by $1+h(\eta')$ thus sends $\phi'+\eta'$ to a Maurer--Cartan element $\phi''=  (1+h(\eta'))\star_{\delta} \phi'$ in $A^0(\bD^1_r,B^1)\oplus A^{1,0}(\bD^1_r,B^0) $, since
\[
(1+h(\eta'))\eta'(1-h(\eta')) - (1+h(\eta'))\bar{\pd} h(\eta') =  \eta' - \eta' =0.
\]

The proof now proceeds  exactly as in Lemma \ref{poincarelemma}, applying 
 the step above inductively to the inclusions
\begin{align*}
 \Tot\Omega^{\#}(\bD^{m+1},  \Tot A^{\#,\bt}(\bD^{n-m-1},B)\cong \Tot\Omega^{\#}(\bD^1,  \Tot\Omega^{\#}(\bD^{m}, \Tot A^{\#,\bt}(\bD^{n-m-1},B))\\
  \into \Tot A^{\#,\bt}(\bD^1,  \Tot\Omega^{\#}(\bD^{m}, \Tot A^{\#,\bt}(\bD^{n-m-1},B)) \cong  \Tot\Omega^{\#}(\bD^{m},  \Tot A^{\#,\bt}(\bD^{n-m}, B))
 \end{align*}
of multiplicatively convex dg  algebras, on a sequence $r=r_1<r_2\ldots < r_n <1$ of radii,   giving the required result.
\end{proof}

\begin{proposition}\label{dolbeaultprop}
For $B \in dg_+\hat{\Tc}\Alg_{\Cx}$ and 
$\bD^n_r \subset \Cx^n$ the open  polydisc of radius $r$, 
 the canonical dg functor
 \begin{align*}
\LLim_{r>0}D_*\per_{dg}(\Omega^{\#}(\bD^n_r,B)) &\to\LLim_{r>0} D_*\per_{dg}(A^{\#,\bt}(\bD^n_r,B))
\end{align*}
is a quasi-equivalence. 
   \end{proposition}
\begin{proof}
Lemma \ref{hyperconnlemma} allows us to replace $D_*\per_{dg}$ with the relevant category of flat hyperconnections. On the left, these are $\sO_{\bD^n_r}\hten B$-modules $M$ in complexes equipped with  generalised Higgs forms, i.e. chain maps $\theta \co M \to \Tot^{\Pi} \Omega^{>0}(\bD^n_r)\ten_{\sO_{\bD^n_r}}M[1]$ satisfying $\theta \wedge \theta =0$.

For cofibrant $\sO_{\bD^n_r}\hten B$-modules $M,N$ equipped with generalised Higgs forms $\theta, \phi$ respectively,
this dg functor is given on morphisms by  the natural inclusion map 
\begin{align*}
&\LLim_{0<s<r}\Tot^{\Pi} (\HHom_{\sO_{\bD^n_r}\hten B}(M,  N)\ten_{\sO_{\bD^n_r}}\Omega^{\#}(\bD^n_s), \delta \pm \theta_* \mp \phi^*)\\
&\to
\LLim_{0<s<r}\Tot^{\Pi} (\HHom_{\sO_{\bD^n_r}\hten B}(M,  N)\ten_{\sO_{\bD^n_r}}A^{\#,\bt}(\bD^n_s,B), \delta \pm \bar{\pd} \pm \theta_* \mp \phi^*),
 \end{align*}
 and we wish to show that this is a quasi-isomorphism.
 
 We can filter both sides by Hodge filtrations, setting $F^pV:= \Omega^{\ge p}V$ on the left, and $F^pV:= A^{\ge p, \#}V$ on the right. These induce  filtrations on these complexes, with associated graded piece $\gr^p$ given by
\begin{align*}
& \HHom_{\sO_{\bD^n_r}\hten B}(M,\LLim_{0<s<r}\Omega^{p}(\bD^n_s,B)\ten_{ \sO_{\bD^n_r}\hten B }N)\\
&\to  \HHom_{\sO_{\bD^n_r}\hten B}(M,\LLim_{0<s<r}A^{p,\bt}(\bD^n_s,B)\ten_{ \sO_{\bD^n_r}\hten B }N).
\end{align*}
Since  $M$ and $N$ are both perfect, by taking shifts and extensions, we may reduce to the case 
 $M=N= \sO_{\bD^n_r}\hten B$, and the  map is then a quasi-isomorphism by Lemma \ref{dolbeaultlemma1}.

It remains to show that the dg functors are essentially surjective, but this follows by exactly the same argument as in Proposition \ref{poincareprop}, substituting Lemma \ref{dolbeaultlemma} for Lemma \ref{poincarelemma}.
\end{proof}

The following Dolbeault--Grothendieck correspondence with LMC dg algebra coefficients now follows immediately from Proposition \ref{dolbeaultprop} by taking derived sections.
\begin{corollary}\label{dolbeaultcor}
For  $Y$ a complex  manifold, there are quasi-equivalences
\[
\per_{dg}^{Y,\Dol,\an}(B) \to \per_{dg}^{Y,\Dol,\C^{\infty}}(B)
\]
of dg categories, natural in $B \in dg_+\hat{\Tc}\Alg_{\Cx}$, from
the analytic Dolbeault moduli functor  to the $\C^{\infty}$ Dolbeault moduli functor.
\end{corollary}

At this stage, beware that we have no analogue of Remark \ref{deRhamEFCrmk}
 --- it is not true in general that the dg categories $\per_{dg}^{Y,\Dol,\an}(B)$ and  $\per_{dg}^{Y,\Dol,\C^{\infty}}(B) $ depend only on the dg algebra structure underlying $B$, or even the dg FEFC algebra structure. %%that's now FEFC, not EFC.

We have already seen that the de Rham functors are representable by the analytification of a derived NC prestack,
via Corollary \ref{poincarecor}. The corresponding statement does not hold for the Dolbeault functors, but the following proposition means that they have good infinitesimal behaviour in the form of tangent and obstruction theory. The statement involves a strict epimorphism $A \to B$ of complete multiplicatively convex dg algebras, which is equivalent to saying that $B\cong A/I$ for a \emph{closed} dg ideal $I$ in $A$.

 Since the spaces $\Omega^p_Y$ are all  nuclear Fr\'echet, we have the following immediate consequence of Lemma \ref{CXlemma2}:
 \begin{proposition}\label{Dolhgsprop}
  For any complex manifold $Y$, the functor $\per_{dg}^{Y,\Dol,\an}$ restricted to $dg_+\cF r\Alg(\Cx)$ is w.e.-preserving and homogeneous.
 \end{proposition}

 In particular, this gives rise to obstruction theory and tangent complexes as in Lemma \ref{obslemma} and Definition \ref{Tdef}.
 
 \begin{corollary}\label{Dolperfectcotcor}
   For  $Y$ a smooth proper complex manifold,  the simplicial set-valued functor $\Perf^{Y,\Dol,\an}$ on  $dg_+\cF r\Alg(\Cx)$ given by the nerve of the core of  $\per_{dg}^{Y,\Dol,\an}$  has perfect cotangent complex at any $B$-valued point $[(\sE,\theta)]$ where the complex
   \[
    \oR\Gamma(Y,  (\oR\sHom_{\sO_Y(B)} (\sE, \sE\hten_{\sO_Y}\Omega^{\#}_Y(B), \delta \pm [\theta,-]))
   \]
is perfect as a $B^{\hat{e}}$-module.
\end{corollary}
\begin{proof}
 Since we are in the setting of Proposition \ref{RGammacotprop}, this follows from the quasi-isomorphism
 \[
  \bT_x(\Perf^{Y,\Dol,\an}, B^{\hat{e}}) \simeq  \oR\Gamma(Y,  (\oR\sHom_{\sO_Y(B)} (\sE, \sE\hten_{\sO_Y}\Omega^{\#,\hol}_Y(B), \delta \pm [\theta,-]))[1]
 \]
of $B^{\hat{e}}$-modules, which is a consequence of  the description of morphism spaces in Lemma \ref{hyperconnlemma}. Here, in the target complex $B$  acts on the second copy of $\sE$, while $B^{\op}$ acts on $ \Omega^{\#}_Y(B)$ by left multiplication.
  \end{proof}

 \subsection{Comparisons with analytified algebraic functors}\label{analytificncomp}

We now adapt some results from \cite{GAGA} to the generality of complete multiplicatively convex algebra coefficients, although in some cases we have to use quite different proofs, as Cartan's Theorems A and B do not easily generalise.

\begin{lemma}\label{HodgeDollemma}
 If $\sE$ is a holomorphic  vector bundle on a compact K\"ahler  manifold  $Y$, 
with $\sE$ admitting a $\C^{\infty}$ hermitian metric, then for any complete locally convex complex topological vector  space $V$ the natural map
 \[
  \oR\Gamma(Y,\sE)\ten_{\Cx}UV \to \oR\Gamma(Y, \sE\hten V)
 \]
 is a quasi-isomorphism, where $UV$ denotes the abstract vector space underlying $V$.
 \end{lemma}
\begin{proof}
The homotopies from  Lemma \ref{dolbeaultlemma1} %below 
ensure that
 the morphism $\sE\hten V \to (\sA_Y^{0,\bt}\ten_{\sO_Y}\sE)\hten V$ is a quasi-isomorphism.
 The sheaves on the right are all modules over the sheaf $\sA_Y^{0}(\R)$, which has  a partition of $1$ subordinate to any cover, so they are all fine sheaves. Thus we have
 \[
  \oR\Gamma(Y, \sE\hten V) \simeq \Gamma(Y, \sE\ten_{\sO_Y}\sA^{0,\bt}_{Y}(V)). 
 \]

The complex $ \Gamma(Y, \sE\ten_{\sO_Y}\sA^{0,\bt}_{Y}(V))$ is given by applying $\hten V$ to the complex $A^{0,\bt}(Y,E):= \Gamma(Y, \sE\ten_{\sO_Y}\sA^{0,\bt}_{Y}(\Cx)$, so we can rewrite our map as
\[
 A^{0,\bt}(Y,E)\ten_{\Cx}UV \to A^{0,\bt}(Y,E)\hten V.
\]

As in \cite[pp.150--153]{GriffithsHarris}, existence of the metric means that we have direct sum decompositions
\[
 A^{0,q}(Y,E) \cong \cH^{0,q}(Y,E) \oplus \bar{\pd} A^{0,q-1}(Y,E) \oplus \bar{\pd}^* A^{0,q+1}(Y,E)
\]
of Fr\'echet spaces, with $\pd$ vanishing on $\cH$ and the map $\bar{\pd}\co \bar{\pd}^* A^{0,q}(Y,E) \to  \bar{\pd} A^{0,q-1}(X,E) $ being an isomorphism of Fr\'echet spaces. On taking the completed injective tensor product with $V$, the direct sum decompositions survive, so we have  isomorphisms
\begin{align*}
 \cH^{0,q}(Y,E)\hten_{\eps}V &\to \H^q(A^{0,\bt}(Y,E)\hten V),\\
 \cH^{0,q}(Y,E)\ten_{\Cx}UV &\to  \H^qA^{0,\bt}(Y,E)\ten_{\Cx}UV.
\end{align*}

Since $\cH^{0,q}(Y,E) $ is finite dimensional,  we have  $\cH^{0,q}(Y,E)\hten_{\eps}V \cong \cH^{0,q}(Y,E)\ten_{\Cx}UV$, giving the required quasi-isomorphism.
 \end{proof}

\begin{proposition}\label{GAGAprop} 
 For any $B \in dg_+\hat{\Tc}\Alg_{\Cx}$ and any perfect complexes $\sE,\sF$ of $\sO_{X}\ten_{\Cx}B$-modules on a complex projective scheme $X$, the natural maps
\[
 \oR\Gamma(X,  \oR\sHom_{\sO_{X}\ten_{\Cx}B}(\sE,\sF)) \to \oR\Gamma(X_{\an}, \oR\sHom_{\sO_{X}^{\hol}(B)}(  \sO_{X}^{\hol}(E), \sO_{X}^{\hol}(F))
\]
are quasi-isomorphisms, where we write $ \sO_{X}^{\hol}(F):=(h^{-1}\sF)\ten_{h^{-1}(\sO_{X}\ten_{\Cx}B)}\sO_{X}^{\hol}(B)$.
\end{proposition}
\begin{proof} 
Since $X$ is projective with some closed embedding $i \co X \into \bP^n_{\Cx}$, perfect complexes on $X$ are generated as a triangulated category by the vector bundles $\sO_X(m)$ for $m \in \Z$, so we may reduce to the case where $\sE= \sO_X(r)\ten_{\Cx}B$. We thus wish to show that  $\oR\Gamma(X, \sF(-r)) \to \oR\Gamma(X_{\an},\sO_{X}^{\hol}(F)(-r))$ is a quasi-isomorphism.

Now, $i_*\sF(-r)$ is a perfect complex on $\bP^n_{\Cx}$, with 
 \[
 i_*^{\an}\sO_{X}^{\hol}(B)\cong (h^{-1}i_*\sO_X)\ten_{h^{-1}\sO_{\bP^n}}\sO_{\bP^n}^{\hol}(B), 
\] 
essentially by definition, so
 \[
 i_*^{\an}\sO_{X}^{\hol}(F)(-r)\cong (h^{-1}i_*\sF(-r))\ten_{h^{-1}(\sO_{\bP^n}\ten_{\Cx}B)}\sO_{\bP^n}^{\hol}(B).
\] 
 
Replacing $(X,\sF(-r))$ with $(\bP^n,i_*\sF(-r))$ we thus reduce to the case $X=\bP^n$, and since perfect complexes on $\bP^n$ are generated as a triangulated category by the vector bundles $\sO_{\bP^n}(m)$ for $m \in \Z$, we may further reduce to the case where $\sF= \sO_{\bP^n}(r)\ten_{\Cx}B$.

The line bundles $\sO^{\hol}(m)$ all admit hermitian metrics, so by Lemma \ref{HodgeDollemma} the morphisms
 \[
  \oR\Gamma(\bP^n_{\Cx,\an}, \sO_{\bP^n}^{\hol}(m))\ten_{\Cx} B \to \oR\Gamma(\bP^n_{\Cx,\an},\sO_{\bP^n}^{\hol}(B)(m))
 \]
are quasi-isomorphisms.

Now, by \cite{GAGA}, the maps
\[
  \oR\Gamma(\bP^n_{\Cx}, \sO_{\bP^n}(m)) \to 
  \oR\Gamma(\bP^n_{\Cx,\an}, \sO_{\bP^n}^{\hol}(m))
\]
are quasi-isomorphisms, and the desired result follows by applying $\ten_{\Cx}B$.
\end{proof}

\begin{remark}\label{esssurjrmk}
%%note that $\prod B_i$ won't give thing we want unless we impose limit after the fact. That's now covered in \S \ref{derivedanalyticstacksn}

%%see what you can deduce from existence of a compact generator: it feels powerful. It fails local/global, though, and I think it'll just give a circular argument.

 Note that the dg functor $\per_{dg}(\sO_{X}\ten_{\Cx} B) \to \per_{dg}(\sO_{X}^{\hol}(B))$ from Proposition \ref{GAGAprop}, sending $\sF$ to  $(h^{-1}\sF)\ten_{h^{-1}(\sO_{X}\ten_{\Cx}B)}\sO_{X}^{\hol}(B)$, will  not  be essentially surjective in general,
 By \cite{GAGA}, we do have essential surjectivity when $B$ is finite-dimensional over $\Cx$, and the proof of \cite[Proposition 5.15]{HolsteinPortaDerivedAnMapping} implies essential surjectivity when $B$ is a ring of overconvergent functions from a compact Stein space to a  finite-dimensional algebra over $\Cx$.

To see that essential surjectivity does not hold in general,  take $B=\Cx^{\Z}$ with the product topology. Then  $\prod_{m \in \Z} \sO_{\bP^n}^{\hol}(m)$ is a locally  perfect complex over $\prod_{m \in \Z}\sO_{\bP^n}^{\hol}=\sO_{\bP^n}^{\hol}(\Cx^{\Z})$, but does not come from any locally  perfect complex over $ \sO_{\bP^n}\ten_{\Cx}\Cx^{\Z}$ --- in particular, $\prod_{m \in \Z} \sO_{\bP^n}(m)$ is not perfect over that ring. %%e.g. consider Chern class

The argument from \cite{GAGA} relies on Noetherianity and strong nuclearity properties which are satisfied by dagger algebras but rare for finitely generated non-commutative LMC algebras.
  \end{remark}

 \begin{corollary}\label{RHcoralgff}
  For any smooth projective complex variety $Y$, the Riemann--Hilbert map
  \[
   (\per_{dg}^{Y,\dR,\alg})^{\FEFC} \to (\per_{dg}^{Y(\Cx)_{\an},\oB})^{\FEFC}
  \]
from Corollary \ref{RHcoralg} is fully faithful as a dg functor. 
   \end{corollary}
\begin{proof}
 It suffices to show that the natural  dg functor  $\per_{dg}^{Y,\dR,\alg}(B) \to \per_{dg}^{Y,\dR,\an}(B)$ is fully faithful for all $B \in dg_+\hat{\Tc}\Alg_{\Cx}$. Given perfect complexes $\sF,\sF'$ of $\sO_{Y}\ten_{\Cx}B$-modules, with flat hyperconnections $\nabla, \nabla'$, there is an induced hyperconnection $\nabla^{\ad}$ on $\oR\sHom_{\sO_{Y}\ten_{\Cx}B}(\sF,\sF')$, and this amounts to considering derived global sections of the morphism
 \begin{align*}
  &h^{-1}(\oR\sHom_{\sO_{Y}\ten_{\Cx}B}(\sF,\sF')\ten_{\sO_Y} \Omega^{\bt}_Y, \delta \pm \nabla^{\ad})\\
  &\to (\oR\sHom_{\sO_{X}^{\hol}(B)}(  \sO_{Y}^{\hol}(F), \sO_{Y}^{\hol}(F')) \ten_{\sO_Y^{\hol}} \Omega^{\bt,\hol}_{Y^{\an}}, \delta \pm \nabla^{\ad}).  
 \end{align*}

Brutal truncation of  $\Omega^{\bt}$ induces compatible  Hodge filtrations on both sides, and it suffices to show that we have quasi-isomorphisms on the associated graded pieces, where the map is 
\begin{align*}
  &\oR\Gamma(Y,  \oR\sHom_{\sO_{Y}\ten_{\Cx}B}(\sF,\sF'\ten_{\sO_Y}\Omega^i_Y))\\ 
  &\to \oR\Gamma(Y_{\an}, \oR\sHom_{\sO_{Y}^{\hol}(B)}(  \sO_{X}^{\hol}(F), \sO_{X}^{\hol}(F')\ten_{\sO_X^{\hol}}\Omega^{i,\hol}_{Y^{\an}})).
\end{align*}
The result then follows by applying Proposition \ref{GAGAprop} to the pair $(\sF,\sF\ten_{\sO_Y}\Omega^i_Y)$.
\end{proof}

\begin{corollary}\label{RHcoralget}
 The Riemann--Hilbert map from Corollary \ref{RHcoralg} is formally \'etale in the sense that for any  surjection $A \to B$  in $dg_+\FEFC(\Cx)$ with square-zero kernel $I$, the map
 \[
%    (\per_{dg}^{Y,\dR,\alg})^{\FEFC}(A)  \to (\per_{dg}^{Y(\Cx)_{\an},\oB})^{\FEFC}(A)\by^h_{(\per_{dg}^{Y(\Cx)_{\an},\oB})^{\FEFC}(B)} (\per_{dg}^{Y,\dR,\alg})^{\FEFC}(B)
   \per_{dg}^{Y,\dR,\alg}(A)  \to \per_{dg}^{Y(\Cx)_{\an},\oB}(A)\by^h_{\per_{dg}^{Y(\Cx)_{\an},\oB}(B)} \per_{dg}^{Y,\dR,\alg}(B)
  \]
  is a quasi-equivalence.
\end{corollary}
\begin{proof}
Since both functors $F$ are homogeneous and homotopy-preserving, there is a natural equivalence $F(A) \simeq F(B)\by^h_{u,F(B \oplus I[1]),0}F(B)$ for $u \co B \to B\oplus I[1]$ the obstruction map  in $\Ho(dg_+\FEFC(\Cx))$ associated to the extension. Replacing $A$ with $A \oplus I[1]$,  we can therefore reduce to the case where $A$ is a trivial square-zero extension. Homotopy-homogeneity also implies that the homotopy fibre  $T_x(F,M[1])$  of  $F(A \oplus M[1])$ over $x \in F(A)$ deloops $T_x(F,M)$.  Corollary \ref{RHcoralgff} implies that the Riemann--Hilbert map gives  isomorphisms  on all higher homotopy groups of tangent spaces $T_x$, so by delooping gives weak equivalences on $T_x$, as required.
\end{proof}

 \begin{corollary}\label{Dolcoralg}
  For any smooth complex variety $Y$, we have a natural map
  \[
   (\per_{dg}^{Y,\Dol,\alg})^{\FEFC} \to \per_{dg}^{Y(\Cx)_{\an},\Dol}
  \]
of functors  on  $dg_+\hat{\Tc}\Alg_{\Cx}$ taking values in the $\infty$-category of dg categories. It is fully faithful and its restriction to $dg_+\cF r\Alg_{\Cx}$ is formally \'etale.
   \end{corollary}   
  \begin{proof}
   Existence follows exactly as in the proof of Corollary \ref{RHcoralg}, turning off the de Rham differential. Full faithfulness and formal \'etaleness then follow as in Corollaries \ref{RHcoralgff} and \ref{RHcoralget} in the same way, noting that restriction to Fr\'echet algebras ensures that the target is homogeneous, by Proposition \ref{Dolhgsprop}.
  \end{proof}

In particular, this implies that   $\Perf^{Y(\Cx)_{\an},\Dol}$ has perfect cotangent complexes at all points of algebraic origin.
   
 \subsection{Hodge and twistor functors}
 
 We now set about constructing derived NC enhancements of the Hodge moduli space $M_{\Hod}$ and the Deligne--Hitchin--Simpson twistor space $M_{\mathrm{DH}}$ of  \cite[\S 3]{simpsonwgt2}.

 \subsubsection{The analytic Hodge and \tps{$\C^{\infty}$}{C-infinity} pre-twistor functors}
 
 As a preliminary to the following definitions, observe that because $ \sO(\bA^1_{\Cx})^{\hol}$ is the free LMC $\Cx$-algebra on one generator,  we can identify $dg_+\hat{\Tc}\Alg_{\sO(\bA^1_{\Cx})^{\hol}}$ with the category of pairs $(B, \lambda)$ for $B \in dg_+\hat{\Tc}\Alg_{\Cx}$ and $\lambda \in Z(B_0)$ an element of the centre of $B$, corresponding to the image of the   standard co-ordinate in  $ \sO(\bA^1_{\Cx})^{\hol}$.
 
\begin{definition}\label{Hodgedefan}
 For $(B,\lambda)  \in dg_+\hat{\Tc}\Alg_{\sO(\bA^1_{\Cx})^{\hol}}$ and   any complex manifold $Y$, the $B$-linear analytic Hodge dg category of $Y$  is given by 
\[
 \per_{dg}^{Y,\Hod,\an}(B, \lambda):=\oR\Gamma(Y,  D_*\per_{dg}(\Omega^{\#,\hol}_Y(B), \lambda\pd, \delta_B)).
\]
 
 This also has a form of analytic $\bG_m$-action: for any $\mu \in Z(B_0)^{\by}$, we have an isomorphism
\[
 \per_{dg}^{Y,\Hod,\an}(B, \lambda) \to    \per_{dg}^{Y,\Hod,\an}(B, \lambda\mu)
\]
induced by the morphism $(\Omega^{\#,\hol}_{Y}(B), \lambda \pd, \delta_B)\to (\Omega^{\#,\hol}_{Y}(B), \lambda\mu \pd , \delta_B)$ given by multiplying $\Omega^p$ by $\mu^p$.
 \end{definition}
 Note that via Lemma \ref{hyperconnlemma}, the analytic Hodge dg category can be identified with the dg category of perfect $\sO_Y^{\hol}(B)$-modules with flat  $\lambda \pd$-hyperconnections. Also note that $\per_{dg}^{Y,\Hod,\an}(B, 1)= \per_{dg}^{Y,\dR,\an}(B) $ and $\per_{dg}^{Y,\Hod,\an}(B, 0)= \per_{dg}^{Y,\Dol,\an}(B) $. The $\bG_m$-action thus gives isomorphisms $\per_{dg}^{Y,\Hod,\an}(B, \lambda)\cong \per_{dg}^{Y,\dR,\an}(B) $ whenever $\lambda$ is invertible.
 
 \begin{remark}\label{LBrmkHod}
 We can generalise the construction of Definition \ref{Hodgedefan} to take inputs $(B,L, \lambda)$, for $L$ a $Z(B_0)$-module of rank $1$ and $\lambda \in L$, by letting $\per_{dg}^{Y,\Hod,\an}(B, \lambda)$ be $\oR\Gamma(Y,  D_*\per_{dg}(L^{\ten_{Z(B_0)} \#}\ten_{Z(B_0)}\Omega^{\#,\hol}_Y(B), \lambda\ten\pd, \delta))$. This construction corresponds to working over $[\bA^1/\bG_m]$ instead of $\bA^1$, encoding and sheafifying the analytic $\bG_m$-action.
  \end{remark}

 We also have a $\C^{\infty}$ version. The real commutative Fr\'echet algebra $\sO(\bA^2_{\R})^{\hol}:= (\sO(\bA^2_{\Cx})^{\hol})^{\Gal(\Cx/\R)}$ consists of holomorphic functions $f \co \Cx^2 \to \Cx$ satisfying $\overline{f(w)}=\bar{w}$, and is the free commutative LMC $\R$-algebra on two variables $u,v$ (corresponding to the two standard co-ordinates). We can thus identify $dg_+\hat{\Tc}\Alg_{\sO(\bA^2_{\R})^{\hol}}$ with the category of triples $(B, u,v)$ for $B \in dg_+\hat{\Tc}\Alg_{\R}$ and $u,v \in Z(B_0)$.
 
 \begin{definition}\label{twistordefCinftyR}
 For $(B,u,v) \in dg_+\hat{\Tc}\Alg_{\sO(\bA^2_{\R})^{\hol}}$ and   any complex manifold $Y$, the $B$-linear  $\C^{\infty}$ pre-twistor dg category of $Y$  is given by 
\[
 \per_{dg}^{Y,\mathrm{pTw},\C^{\infty}}(B, u,v):=\oR\Gamma(Y,  D_*\per_{dg}(A^{\#}_Y(B),  ud+v\dc, \delta_B)),
\]
where $\dc = i\pd - i\bar{\pd}$, the conjugate of the de Rham differential $d$ under the almost complex structure.
 
 This has a natural analytic action of the Deligne torus $S$: for any  $ (a+ib) \in Z(B_0\ten_{\R}\Cx)^{\by}$ (equivalently, $a,b \in Z(B_0)$ with  $a^2+b^2 \in Z(B_0)^{\by}$),  we have an equivalence
\[
 \per_{dg}^{Y,\mathrm{pTw},\C^{\infty}}(B,u,v) \to \per_{Dg}^{Y,\mathrm{pTw},\C^{\infty}}(B,au -bv,av+bu)
\]
induced by the morphism $(\sA^{\#}_{Y}(B), ud+v\dc , \delta_B)\to (\sA^{\#}_{Y}(B), (au -bv)d + (av+bu)\dc , \delta_B)$ given by multiplying $\sA^{p,q}$ by $(a+ib)^p(a-ib)^q$.
  \end{definition}
  
  Note that via Lemma \ref{hyperconnlemma}, the $\C^{\infty}$ pre-twistor dg category can be identified with the dg category of perfect $\C^{\infty}(Y,B)$-modules with flat $ud+v\dc$-hyperconnections. Also note that when $B$ is a $\Cx$-algebra, we have
\[  
  \per_{dg}^{Y,\mathrm{pTw},\C^{\infty}}(B, 1,0)= \per_{dg}^{Y,\dR,\C^{\infty}}(B) \quad\text{and}\quad \per_{dg}^{Y,\mathrm{pTw},\C^{\infty}}(B, \half, \frac{-i}{2})= \per_{dg}^{Y,\Dol,\C^{\infty}}(B).
  \]
  The $S$-action thus gives isomorphisms $\per_{dg}^{Y,\mathrm{pTw},\C^{\infty}}(B, u,v)\cong \per_{dg}^{Y,\dR,\C^{\infty}}(B) $ whenever $u^2+v^2$ is invertible, and $\per_{dg}^{Y,\mathrm{pTw},\C^{\infty}}(B, iv,v)\cong \per_{dg}^{Y,\Dol,\C^{\infty}}(B) $ if $v$ is invertible and $B$ is a $\Cx$-algebra.

 \begin{remark}\label{LBrmktwistor}
 We can generalise the construction of Definition \ref{twistordefCinftyR} to take inputs $(B,L, w)$, for $L$ a $Z(B_0)\ten_{\R}\Cx$-module of rank $1$ and $w \in L$. Explicitly, let $\per_{dg}^{Y,\mathrm{pTw},\an}(B,L,w)$ be given by applying $\oR\Gamma(Y,  D_*\per_{dg}-)$ to the stacky DGAA 
$(\sA_{\Tw,Y}^{\bt}(B,L,w):= (\bigoplus_{p,q} L^{\ten p}\ten \bar{L}^{\ten q} \ten \sA^{p,q}_Y(B\ten_{\R}\Cx)^{[-p-q]})^{\Gal(\Cx/\R)},w\pd+\bar{w}\bar{\pd}, \delta)$, where tensor products are taken over $Z(B_0)\ten_{\R}\Cx$, and $\bar{L}$ is the $Z(B_0)\ten_{\R}\Cx$-module given by the conjugate action on $L$, with $\bar{w} \in L$ the element corresponding to $w$.

When $L$ is the trivial module $Z_0(B)\ten_{\R}\Cx$, this recovers Definition \ref{twistordefCinftyR} with $w=u+iv$, since $w\pd + \bar{w}\bar{\pd}= ud+v\dc$.

This construction corresponds to working over $[\bA^2/S]$ instead of $\bA^2$, encoding and sheafifying the analytic $S$-action.
  \end{remark}
 
 The pre-twistor functor only tends to have finite cotangent complexes away from the origin $\{0\} \in \bA^2_{\R}$, so the twistor functor later in this section will be based on the restriction to $\bA^2_{\R}\setminus \{0\}$.
 
%  %%%%%%%%%%%%%%%%
%  
%  Can copy over some twistor stuff from below. Note differences between algebraic and analytic Rees constructions: algebraic gives exhaustive filtration, while analytic will just say that $\bigcup_i F_i$ is dense. Since we only have $A^{p,q}$ for $p,q \ge 0$, we could work with $\bar{S}:=\prod_{\Cx/\R}\Mat_1$ rather than Deligne torus $S=\prod_{\Cx/\R}\bG_m$. The point is that we'll want to work with filtrations $F_0(B\ten \Cx)\subset F_1(B\ten \Cx) \subset \ldots$, and then instead of looking at $(\bigoplus_p A^{p,*}\hten F_p(B\ten \Cx))^{\Gal(\Cx/\R)}$, we end up looking at some sort of $\hat{\bigoplus}$. 
%  
% %  Actually, this demonstrates they are the same in this case because the sum is finite, so we might as well do algebraic. We could even just restrict to filtrations with $F_{-1}B=0$ and $F_dB=B$, where $d = \dim X$. NOOOO, because $F_dB$ isn't a subalgebra.
% %  Given arb $B$, that just means replacing it with $F_d(B\ten \Cx)^{\Gal(\Cx/\R)}$ (i.e. $F_d \cap \bar{F}_d$), and forgetting $F_{<0}$.
%  
%  Key point: instead of exhaustive, we just have completion of $\bigcup_i F_i$ w.r.t. LMC seminorms. The module over $\sO^{\hol}(\bA^1)$ is just LMC completion of the Rees construction, with the mnodule structure coming auotmamtically from $z$ corresponding to $1 \in F_1$.
%  
%   %%%%%%

 Since the spaces $\Omega^p_Y$ and $A^p_Y$ are all  nuclear Fr\'echet, we have the following immediate consequences of Lemma \ref{CXlemma2}:
\begin{proposition}\label{Hodhgsprop} 
 For any complex manifold $Y$, the functor $\per_{dg}^{Y,\Hod,\an}$ restricted to $dg_+\cF r\Alg(\sO(\bA^1_{\Cx})^{\hol})$ is w.e.-preserving and homogeneous.

 Likewise, the functor $\per_{dg}^{Y,\mathrm{pTw},\C^{\infty}}$ restricted to $dg_+\cF r\Alg(\sO(\bA^2_{\R})^{\hol})$ is w.e.-preserving and homogeneous.
 \end{proposition}

 In particular, this gives rise to obstruction theory and tangent complexes as in Lemma \ref{obslemma} and Definition \ref{Tdef}. Proposition \ref{RGammacotprop} then gives the following.
 \begin{corollary}\label{twistorcotperfcor}
   For  $Y$ a smooth proper complex manifold,  the simplicial set-valued functor $\Perf^{Y,\Hod,\an}$ on $dg_+\cF r\Alg(\sO(\bA^1_{\Cx})^{\hol})$ given by the nerve of the core of  $\per_{dg}^{Y,\Hod,\an}$  has perfect cotangent complex at any $B$-valued point $[(\sE,\nabla)]$ where the complex 
   \[
    \oR\Gamma(Y,  (\oR\sHom_{\sO_Y(B)} (\sE, \sE\hten_{\sO_Y}\Omega^{\#}_Y(B), \delta \pm \nabla^{\ad}))
   \]
is perfect as a $B^{\hat{e}}$-module.

Likewise, the simplicial set-valued functor $\Perf^{Y,\mathrm{pTw},\C^{\infty}}$ on $dg_+\cF r\Alg(\sO(\bA^2_{\R})^{\hol})$
given by the nerve of the core of  $\per_{dg}^{Y,\mathrm{pTw},\C^{\infty}}$  has perfect cotangent complex at any $B$-valued point $[(\sE,\nabla)]$ where the complex 
   \[
    \oR\Gamma(Y,  (\oR\sHom_{\C^{\infty}_Y(B)} (\sE, \sE\hten_{\sA^0_Y}\sA^{\#}_Y(B), \delta \pm \nabla^{\ad}))
   \]
is perfect as a $B^{\hat{e}}$-module.
\end{corollary}

\subsubsection{The twistor Dolbeault--Grothendieck Lemma} 
 
 The $\bar{\pd}$-Poincar\'e lemma now adapts as follows:
\begin{lemma}\label{twistordolbeaultlemma}
 If $B$ is a unital associative  differential graded complete multiplicatively convex $O(\Cx)^{\hol}$-algebra and $s<r$, then for $\lambda \in O(\Cx)^{\hol}$ the canonical co-ordinate,
the image in $\mc(\Tot(A^{\#}(\bD^n_s,B)), \lambda\pd +\bar{\pd}\pm \delta ) $  of any
Maurer-Cartan element $\omega \in \mc(\Tot(A^{\#}(\bD^n_r,B)), \lambda\pd +\bar{\pd}\pm \delta )$ is gauge equivalent to an element of the subset $\mc(\Tot(\Omega^{\#}(\bD^n_s,B)), \lambda \pd \pm \delta)$,  
  via an invertible element of $\Tot A^{<2n}(\bD^n_s,B)^0$). 
  \end{lemma}
\begin{proof}
The proof of Lemma \ref{dolbeaultlemma} adapts almost verbatim, so we will just summarise the main steps. 

Inductively, we can reduce to the case $n=1$. Write $\omega = \phi+\eta$, for $\phi \in  A^0(\bD^1_r,B^1)\oplus A^{1,0}(\bD^1_r,B^0)$ and $\eta \in  A^{0,1}(\bD^1_r,B^0) \oplus A^2(\bD^1_r,B^{-1})$, with the proof of  Lemma \ref{poincarelemma} then giving us an invertible element $g \in \Tot A^{<2}(\bD^1_r,B)^0$ with $dg = - \eta \wedge g$.  

Writing $g^{-1}\star_{\delta} \phi:= g^{-1}\phi g +g^{-1}\delta g$, we   thus have
 \begin{align*}
 g^{-1}\star (\phi + \eta) &= g^{-1}(\phi + \eta)g + g^{-1}(\delta+\lambda \pd+ \bar{\pd})g\\
 &= g^{-1}\star_{\delta} \phi +g^{-1}(\lambda-1)\pd g,
\end{align*}
so we have replaced $(\phi,\eta)$ with a gauge-equivalent pair $(\phi',\eta')$ with $\eta' \in A^{1,1}(\bD^1_r,B^{-1})$ and $\phi' \in  A^0(\bD^1_r,B^1)\oplus A^{1,0}(\bD^1_r,B^0)$. 

For our element $\eta' \in A^{1,1}(\bD^1_r,B)$, the operator $h$ from  Lemma \ref{dolbeaultlemma1} then gives  an invertible element $1+h(\eta') \in 1+  A^{1,0}(\bD^1_s,B^{-1})$
which satisfies 
\[
 (1+h(\eta'))\eta'(1-h(\eta'))+   (1+h(\eta'))\bar{\pd}(1-h(\eta')) = \eta' -\eta' =0, 
\]
since $\bar{\pd}h(\eta')=\eta'$ and  
$\eta'\wedge h(\eta') = h(\eta') \wedge \eta'=0$.

Since $\pd h(\eta') \in  A^{2,0}(\bD^1_s,B^{-1})=0$,  applying the gauge transformation $1+h(\eta')$ thus sends $(\phi',\eta')$ to a Maurer--Cartan element $\phi'':=(1+h(\eta'))\star_{\delta} \phi' $ lying in $A^0(\bD^1_s,B^1)\oplus A^{1,0}(\bD^1_s,B^0) $. For $\phi''$, the Maurer--Cartan equations split to give
\[
 \delta \phi'' + \lambda \pd \phi''+ (\phi'')^2=0, \quad \bar{\pd}\phi'' =0,
\]
so $\phi'' \in \mc(\Tot(\Omega^{\#}(\bD^1_s,B)), \lambda \pd \pm \delta)$, as required.
\end{proof}

\begin{proposition}\label{twistordolbeaultprop}
For $B \in dg_+\hat{\Tc}\Alg_{O(\Cx)^{\hol} }$,
 the canonical dg functor
 \begin{align*}
\LLim_{r>0}D_*\per_{dg}(\Omega^{\#}(\bD^n_r,B), \lambda \pd, \delta) &\to \LLim_{r>0}D_*\per_{dg}(A^{\#}(\bD^n_r,B),\lambda \pd + \bar{\pd}, \delta )
%\cD^-_{dg}(B) &\to \cD^-_{dg}(\Tot\Omega^{\bt}(\bD^n,B))\\
% M &\mapsto M\ten_{ \Tot\Omega^{\#}(\bD^n_r,B)}\Tot A^{\#}(\bD^n_r,B)
\end{align*}
is a quasi-equivalence. 
   \end{proposition}
\begin{proof}
This follows exactly as in  the proof of Proposition \ref{dolbeaultprop}, substituting Lemma \ref{twistordolbeaultlemma} for Lemma \ref{poincarelemma}.
\end{proof}

\begin{corollary}\label{twistorpoincarecor}
For  $Y$ a complex  manifold, there are quasi-equivalences
\[
\per_{dg}^{Y,\Hod,\an}(B,\lambda) \to \per_{dg}^{Y,\mathrm{pTw},\C^{\infty}}(B,\frac{\lambda+1}{2}, \frac{\lambda -1}{2i}) 
\]
of simplicial sets, natural in $(B,\lambda) \in dg_+\hat{\Tc}\Alg_{O(\Cx)^{\hol}}$, from
the analytic Hodge moduli functor  to the  $\C^{\infty}$ pre-twistor moduli functor.

This equivalence exhibits a form of $\bG_m$-equivariance, with the action of $\mu \in Z(B_0)^{\by}$ on the left corresponding to the action of $ (\frac{\mu+1}{2}\ten 1) +(\frac{\mu -1}{2i}\ten i) \in Z(B_0\ten\Cx)^{\by}$ on the right.
\end{corollary}
\begin{proof}
The dg category
$ \per_{dg}^{Y,\mathrm{pTw},\C^{\infty}}(B,u,v) $
is given by taking derived global sections of  the hypersheaf 
\[
D_*\per_{dg}(\sA^{\#}_{Y}(B),   ud+v\dc, \delta )
\]
of dg categories on $Y  $, and on setting $u= \frac{\lambda+1}{2}$, $v=\frac{\lambda -1}{2i}$, this becomes
\[
D_*\per_{dg}(\sA^{\#}_{Y}(B),   \lambda \pd +\bar{\pd}, \delta )
\]
 
Since $Y$ has a  basis of open neighbourhoods biholomorphic to $\bD^n$, it follows from Proposition \ref{twistordolbeaultprop} that the dg functor
\[
D_* \per_{dg}(\Omega^{\#}_Y(B), \lambda \pd, \delta) \to D_*\per_{dg}(\sA^{\#}_{Y}( B),   \lambda \pd + \bar{\pd},  \delta )))
\]
induces a quasi-equivalence on hypersheafification, from which the required equivalence on derived sections follows. Equivariance of the $Z(B_0)^{\by}$-action follows immediately from the observation that $ (\frac{\mu+1}{2}\ten 1) +(\frac{\mu -1}{2i}\ten i)$ acts on $\sA^{p,q}$ as multiplication by $ (\frac{\mu+1}{2} + i\frac{\mu -1}{2i})^p \cdot  (\frac{\mu+1}{2} - i\frac{\mu -1}{2i})^q =\mu^p1^q=\mu^p$.
 \end{proof}

 \subsubsection{The base of the twistor functor}
 
 The space $M_{\mathrm{DH}}$ from \cite[\S 3]{simpsonwgt2} is a complex analytic space over $\bP^1(\Cx)$, equipped with an analytic $\Cx^*$ action (equivariant with respect to the obvious action on $\bP^1(\Cx)$), and an antilinear involution covering the involution $\lambda \mapsto \bar{\lambda}^{-1}$ of $\bP^1(\Cx)$.
 
 The most succinct way to phrase this is that we have a $\Gal(\Cx/\R)$-equivariant morphism  $[M_{DH}/\Cx^*] \to [\bP^1(\Cx)/\Cx^*]$ of analytic stacks, where we think of the target as the analytification of the real Artin stack  $[\bP^1_{\R}/\SO_2]$; in particular, $\Cx^*$ is here regarded as the group of complex points of the real group  scheme $\SO_2= \Spec \R[a,b]/(a^2+b^2=1)$ rather than of $\bG_{m,\R}$, and the co-ordinate $t \in \bP^1_{\R}$ corresponds to the co-ordinate $\frac{t+i}{t-i} \in \bP^1_{\Cx}$.

 On the big \'etale site $[\bP^1_{\R}/\SO_2]^{\hol}_{\Et}$ 
 of real Stein spaces %manifolds 
 $U$ % smooth 
 over 
 $[\bP^1_{\R}/\SO_2]^{\hol}$, i.e. anti-holomorphically $\Gal(\Cx/\R)$-equivariant Stein %manifolds 
 spaces
 $U(\Cx)$ 
 %smooth 
 over the analytic stack $[\bP^1(\Cx)/\Cx^*]$, we have a sheaf $\sO$ of commutative nuclear Fr\'echet algebras given by $\sO(U):= \sO^{\hol}(U(\Cx))^{\Gal(\Cx/\R)}$, i.e. holomorphic  functions $f \co U \to \Cx$ with $\overline{f(u)}=f(\bar{u})$. Our twistor functors will send hypersheaves $\sB$ of LDMC dg $\sO$-algebras over $[\bP^1_{\R}/\SO_2]^{\hol}_{\Et}$ to hypersheaves of dg categories.

\begin{remark}\label{cfmhsrmk}
 As  in \cite[\S 2]{mhs2}, the action of the Deligne torus $S=\Pi_{\Cx/\R}\bG_m$ on affine space $\bA^2_{\R}$ comes from the identification of $\bA^2_{\R}$ with the Weil restriction of scalars $\Pi_{\Cx/\R}\bA^1$. In other words, for a commutative $\Cx$-algebra $B$, the action of 
 %\[
  %S(B)= \{(a,b) \in B\by B ~:~ a^2 + b^2 \in B^{\by}}
 %\]
$S(B)$ on $\bA^2(B)$ is just given by multiplication in $B\ten \Cx$, since $(a+ib)(u+iv)= (au-bv) +i(av+bu)$.
 
In the notation of \cite[\S 2]{mhs2}, the real affine scheme  $\Pi_{\Cx/\R}\bA^1\cong \bA^2_{\R}$ is denoted by $C$, and Hodge filtrations on complexifications of real objects correspond to families over the quotient stack $[C^*/S]$, where $C^*$ is the quasi-affine scheme $C\setminus \{0\}$, while Hodge filtrations on complex objects correspond to families over $[\bA^1_{\Cx}/\bG_{m,\Cx}]$. In the analytic setting, there is a similar interpretation, but the filtrations will be dense rather than exhaustive.

On commutative $\Cx$-algebras, the map $\lambda \mapsto (\frac{\lambda+1}{2}, \frac{\lambda -1}{2i})$  from Corollary \ref{twistorpoincarecor}  is a morphism $\bA^1_{\Cx} \to C^*$, and equivariance extends this to  a finite \'etale cover $[\bA^1_{\Cx}/\bG_m] \to [C^*/S]$. In the notation of \cite[\S 2]{mhs2}, this is the $S$-equivariant quotient of the   finite \'etale cover $\widetilde{C^*} \to C^*$. Also note that, pulling out the action of $\bG_m \subset S$, we have an isomorphism $[C^*/S] \cong [\bP^1_{\R}/\SO_2]$, where $\SO_2$ is the circle group $S/\bG_m$, given by
\[
\SO_2(B)= \{(a +i b) \in (B\ten \Cx)^{\by} ~:~ (a+ib)^{-1} = a-ib\}.
\]
Families over $\bP^1_{\R}$ correspond to twistor structures, with $\SO_2$-equivariance upgrading these to Hodge structures.
\end{remark}

 \begin{lemma}\label{twistorcoordlemma}
 For a real Stein space $U$, giving a morphism $f \co U \to [\bP^1_{\R}/\SO_2]^{\hol}$ is equivalent to giving the following data:
 \begin{enumerate}
  \item a rank $1$ module $L$ over $\sO(U)\ten_{\R}\Cx$,
  \item an element $w \in L$ for which there exists some $z$ in the dual module  $L^*:=\Hom_{\sO(U)\ten_{\R}\Cx}(L,\sO(U)\ten_{\R}\Cx)$ (not considered part of the data) with $wz \in 1 + i\sO(U) \subset \sO(U)\ten_{\R}\Cx$.
 \end{enumerate}
\end{lemma}
\begin{proof}

For the Weil restriction of scalars $[C/S]\cong \Pi_{\Cx/\R}[\bA^1/\bG_m]$,  a morphism $f' \co U \to [C/S]$ corresponds to giving an algebraic $\bG_m$-torsor $P$ over $\sO(U)\ten \Cx$ together with a $\bG_m$-equivariant morphism $P \to \bA^1$. Under the correspondence between $\bG_m$-torsors and line bundles, this amounts to giving a  rank $1$ module $L$ over $\sO(U)\ten_{\R}\Cx$ and an element $w \in L$.

Under the isomorphism $[\bP^1_{\R}/\SO_2] \cong [C^*/S]$ of Remark \ref{cfmhsrmk}, the morphism $f$ thus  corresponds to giving such data $(w,L)$ together with a non-degeneracy condition. The latter is most easily understood by considering the isomorphism $[C^*/S]\by_{\Spec \R} \Spec \Cx \cong [\bA^2_{\Cx}/\bG_{m,\Cx}^2]$, with $(L,w)$ giving rise to the $\sO(U)\ten_{\R}\Cx$-valued point $(L,\bar{L},w, \bar{w})$ of $[\bA^2_{\Cx}/\bG_{m,\Cx}^2]$. The condition for this to lie in $[(\bA^2_{\Cx}\setminus \{0\})/\bG_{m,\Cx}^2]$ is that $1$ lies in the ideal $(w,\bar{w})$ in the ring $\bigoplus_{p,q \in \Z} L^{\ten p}\ten \bar{L}^{\ten q}$ of functions of the $\bG_m^2$-torsor corresponding to $(L,\bar{L})$, where all tensor products are taken over $\sO(U)\ten_{\R}\Cx$. This amounts to saying that there exist $a \in L^*$ and $b \in \bar{L}^*$ with $aw+b\bar{w}=1\in \sO(U)\ten_{\R}\Cx$. Setting $z:= (a+\bar{b})/2$ gives data of the form above, and conversely given $z$ we can set $a=z/2$, $b=\bar{z}/2$. 
\end{proof}
 
 The following lemma allows us to interpret objects over $[\bP^1_{\R}/\SO_2]$ as real structures equipped with filtrations on their complexifications.
 \begin{lemma}\label{twistorpushoutlemma}
  The commutative diagram
  \[
   \begin{CD}
    [S_{\Cx}/S]^{\hol}_{\Et} @>>> [\widetilde{C^*}/S]^{\hol}_{\Et} @. \quad\quad \text{i.e.}\quad\quad @. (\Spec \Cx)^{\hol}_{\Et} @>>> [\bA^1_{\Cx}/\bG_m]^{\hol}_{\Et} \\
    @VVV @VVV @. @VVV @VVV \\
    [S/S]^{\hol}_{\Et} @>>> [C^*/S]^{\hol}_{\Et}, @. @.                                (\Spec \R)^{\hol}_{\Et} @>>> [\bP^1_{\R}/\SO_2]^{\hol}_{\Et}
   \end{CD}
  \]
is a pushout of \'etale sites, in the sense that it induces homotopy pullbacks on $\infty$-categories of hypersheaves.
   \end{lemma}
\begin{proof}
 It suffices to show that we have a pushout diagram on passing to the associated complex analytic stacks, ignoring the $\Gal(\Cx/\R)$-action. That diagram is isomorphic to
 \[
  \begin{CD}
 [(\Cx^{*2} \sqcup \Cx^{*2})/\Cx^{*2}]_{\Et} @>>>  [((\Cx^*\by \Cx) \sqcup (\Cx \by \Cx^*))/\Cx^{*2}]_{\Et}\\
   @VVV @VVV \\
    [\Cx^{*2}/\Cx^{*2}]_{\Et} @>>> [(\Cx^2 \setminus \{0\})/\Cx^{*2}]_{\Et},
    \end{CD}
%    
  %    [(\bG_{m,\Cx}^2 \sqcup \bG_{m,\Cx}^2)/\bG_m^2]^{\hol}_{\Et} @>>>  [((\bG_{m,\Cx}\by \bA^1_{\Cx}) \sqcup (\bA^1_{\Cx} \by \bG_{m,\Cx}))/\bG_m^2]^{\hol}_{\Et}\\
%    @VVV @VVV \\
%     [\bG_{m,\Cx}^2/\bG_m^2]^{\hol}_{\Et} @>>> [(\bA^2_{\Cx} \setminus \{0\})//\bG_m^2]^{\hol}_{\Et},
%   \end{CD}
 \]
and the result follows immediately from the observation that $\Cx^*\by \Cx$ and  $\Cx \by \Cx^*$ form an open  cover of $\Cx^2 \setminus \{0\}$ with intersection $\Cx^{*2}$. 
\end{proof}

\begin{remark}
 The structures we construct here generalise Hodge filtrations and twistor analogues to non-linear objects, but beware that these alone are insufficient to give mixed Hodge or mixed twistor structures. This is not an oversight, but occurs because weight filtrations essentially capture nilpotent information, measuring how far local systems are from being semisimple, so only exist in infinitesimal neighbourhoods of most points on moduli spaces.
 
 In more detail, \cite{mhs2} constructed a canonical mixed twistor structure on the pro-algebraic homotopy types of  compact K\"ahler manifolds and of smooth complex quasi-projective varieties. In the compact case, the weight filtration just corresponds to good truncation on cohomology of semisimple local systems. In keeping with the philosophy of \cite{Poids}, this weight filtration agrees with the filtration by eigenvalues of Frobenius for \'etale pro-algebraic homotopy types of varieties in finite characteristic, as studied in \cite{weiln}. 
 
 If we did have some sort of weight filtration on our analytic moduli functor $\sM$, we would expect it to be a family over the analytification of $[\bA^1/\bG_m]$ with generic fibre $\sM$. The special fibre would then be a family over $B\bG_m^{\an}$, so take the form $[Z/\bG_m^{\an}]$ for some $\bG_m$-equivariant $Z$. The results of \cite{mhs2} would then determine the infinitesimal structure of $Z$, but they imply that the $\bG_m$-invariant locus of $Z$ contains all the semisimple local systems, yet has trivial tangent space. Since the semisimple local systems are dense in all local systems, this indicates that $Z$ cannot exist as a reasonably behaved object.
 
 It is worth noting that the moduli functor of semisimple local systems is continuous rather than analytic in nature, as in \cite{sim2}. Instead of functors on dg Banach algebras, this entailed looking at functors on $C^*$-algebras in \cite{mhsDS}, involving continuous functional calculus rather than entire functional calculus. 
\end{remark}
 
 \subsubsection{The twistor functors}
 
 We now construct our twistor functors, which will come in three flavours: a $\C^{\infty}$-functor defined directly, an equivalent analytic functor defined by Deligne gluing, and a full subfunctor of algebraic origin, also  defined by Deligne gluing.
  
 We first set up the categories on which our constructions will be defined. 
 \begin{definition}
For an anti-holomorphically $\Gal(\Cx/\R)$-equivariant complex analytic Artin stack $\fY= (\fY(\Cx), \sigma)$ and $\C$ any of the categories 
$dg_+\Alg$,
$dg_+\FEFC$, %$dg_+\EFC$, 
$dg_+\hat{\Tc}\Alg$,  %$dg_+\hat{\Tc}\CAlg$, 
$dg_+\cF r\Alg$, %$dg_+\cF r\CAlg$, 
or
$dg_+\cN\cF r\Alg$, %or $dg_+\cN\cF r\CAlg$, 
define the category $\Shf_{\Et}\C(\sO_{\fY})$ of $\C$-valued \'etale hypersheaves as follows.

Objects of $\Shf_{\Et}\C(\sO_{\fY})$ are \'etale hypersheaves $\sB$ in real chain complexes on the big \'etale site of real Stein spaces $U=(U(\Cx),\sigma)$ over $(\fY(\Cx), \sigma)$, with the following additional structure, where we write $\sO(U):= \{f \in \sO^{\hol}(U(\Cx))~:~ \overline{f(u)}= f(\sigma(u))\}$.
\begin{enumerate}
 \item For each $U$, we have $\sB(U) \in \C(\sO(U))$.
 \item For each morphism $f \co V \to  U$, the corresponding structure morphism  $f^{\sharp} \co \sB(U) \to \sB(V)$ lies in $\C(\sO(U))$, via the natural forgetful functor $f_* \co \C(\sO(V)) \to \C(\sO(U))$
induced by the  morphism $f^{\sharp} \co \sO(U) \to  \sO(U)$.
 \end{enumerate}
Note that the categories $\C(\sO(U))$ make sense in all cases, because  $\sO(U)$ is a commutative nuclear Fr\'echet LMC $\R$-algebra.
  
  %% note that for these hypersheaves, the notion of weak equivalence is abstract QIM.
  \end{definition}
%%this characterisation avoids the left adjoint $\hten_{\pi}$  which might be a  problem for (F)EFC.

\begin{definition}\label{TwCinftydef}
For any complex manifold $Y$, define the $\C^{\infty}$ twistor functor from $\Shf_{\Et} dg_+\hat{\Tc}\Alg(\sO_{[\bP^1_{\R}/\SO_2]^{\hol}})$ to the $\infty$-category of hypersheaves of dg categories on $[\bP^1_{\R}/\SO_2]^{\hol}_{\Et}$ as follows. 

By Lemma \ref{twistorcoordlemma}, objects $U \in [\bP^1_{\R}/\SO_2]^{\hol}_{\Et}$ correspond to triples $(\sO(U),L,w)$ with $\sO(U)$ a real Stein algebra, $L$ an $\sO(U)\ten_{\R}\Cx$-module of rank $1$ and $w \in L$ satisfying a weak non-degeneracy condition. For any $\sB \in \Shf_{\Et}dg_+\hat{\Tc}\Alg (\sO_{[\bP^1_{\R}/\SO_2]^{\hol}})$, we can then form the sheaf  
\[
 \sA_{\Tw,Y}^{\bt}(\sB(U),L,w):= ((\bigoplus_{p,q} L^{\ten p}\ten \bar{L}^{\ten q} \ten \sA^{p,q}_Y(\sB(U)\ten_{\R}\Cx)^{[-p-q]})^{\Gal(\Cx/\R)},w\pd+\bar{w}\bar{\pd}, \delta)
\]
of
stacky DGAAs on $Y$ as in Remark \ref{LBrmktwistor}, where all tensor products are taken over $\sO(U)\ten_{\R}\Cx$.  

We then let $\per_{dg}^{Y,\Tw,\C^{\infty}}(\sB)$ be the hypersheafification of the dg category-valued presheaf
\[
 (U,L,w) \mapsto \oR\Gamma(Y,  D_*\per_{dg}(\sA_{\Tw,Y}^{\bt}(\sB(U),L,w)))
\]
on 
$[\bP^1_{\R}/\SO_2]^{\hol}_{\Et}$.
 \end{definition}
 Thus Lemma \ref{hyperconnlemma}  implies that an object of $\per_{dg}^{Y,\Tw,\C^{\infty}}(\sB)(U,L,w)$ is locally (on 
$[\bP^1_{\R}/\SO_2]^{\hol}_{\Et}$) given by a perfect $\C^{\infty}_Y(\sB(U))$-complex  with flat  $(w\pd +\bar{w}\bar{\pd})$-hyperconnection.

 Before introducing the analytic analogue, we introduce some notation to handle Deligne gluing over a real base.
\begin{definition}
 Given a real Stein space $U$ over $[\bP^1_{\R}/\SO_2]^{\hol}$, define an \'etale cover  of $U$ by $\tilde{U}:= [\bA^1_{\Cx}/\bG_m]^{\hol}\by_{[\bP^1_{\R}/\SO_2]^{\hol}}U$, i.e. $\tilde{U}=U\by_{[C^*/S]^{\hol}}[\widetilde{C^*}/S]^{\hol}$, 
in the notation of Remark \ref{cfmhsrmk}. 

Define the open subspace $U^o$ of $U$ to be the pre-image of the generic point $\{1\} \in [\bP^1_{\R}/\SO_2]^{\hol}$, or equivalently $U^o=U\by_{[C^*/S]^{\hol}}[S/S]^{\hol}$. Let $U^o_{\Cx}$ be its complex form.  
\end{definition}
Note that Lemma \ref{twistorpushoutlemma} implies that $U$ is the \'etale pushout of the diagram $\tilde{U} \la U^o_{\Cx} \to U^o$.

Under the characterisation of Lemma \ref{twistorcoordlemma}, the morphism $U \to [\bP^1_{\R}/\SO_2]^{\hol}$ corresponds to the data of a line bundle $L$ over $\sO(U)_{\Cx}:= \sO(U)\ten_{\R}\Cx$ with a section $w$ satisfying a constraint, and then $\sO(\tilde{U})$ is the LMC completion of the $\sO(U)$-algebra 
%$(\bigoplus_{q \in \Z} \bar{L}^{\ten_{\sO(U)_{\Cx}} q})/(\bar{w}-1)$. 
$(\bigoplus_{q \in \Z} \bar{L}^{\ten q})/(\bar{w}-1)$, 
where tensor products are taken over $\sO(U)_{\Cx}$.
Similarly,  $\sO(U^o_{\Cx})$ is the LMC completion of the $\sO(U)$-algebra 
% $(\bigoplus_{p,q \in \Z} {L}^{\ten_{\sO(U)_{\Cx}} p}\ten_{\sO(U)_{\Cx}} \bar{L}^{\ten_{\sO(U)_{\Cx}} q})/(w-1,\bar{w}-1)$, 
$(\bigoplus_{p,q \in \Z} {L}^{\ten p}\ten \bar{L}^{\ten q})/(w-1,\bar{w}-1)$, 
and $\sO(U^o)= \sO(U^o_{\Cx})^{\Gal(\Cx/\R)}$.
 
\begin{definition}\label{Twandef}
 For any complex manifold $Y$, define the analytic twistor functor from $\Shf_{\Et} dg_+\hat{\Tc}\Alg(\sO_{[\bP^1_{\R}/\SO_2]^{\hol}})$ to the $\infty$-category of hypersheaves of dg categories on $[\bP^1_{\R}/\SO_2]^{\hol}_{\Et}$ as follows. 

 For any $\sB \in \Shf_{\Et}dg_+\hat{\Tc}\Alg (\sO_{[\bP^1_{\R}/\SO_2]^{\hol}})$, we can  form the sheaf  
\begin{align*}
 \Omega_{\Hod,Y}^{\bt}(\sB(\tilde{U}),L,w):= (\bigoplus_{p} L^{\ten p} \ten \Omega^{p}_Y(\sB(\tilde{U}))^{[-p]},w\pd, \delta)
\end{align*}
of
stacky DGAAs on $Y$, where all tensor products are taken over $\sO(U)\ten_{\R}\Cx$, then set 
\[
\per_{dg}^{Y,\Hod,\an}(\sB(\tilde{U}),L,w):= \oR\Gamma(Y,  D_*\per_{dg}(\Omega_{\Hod,Y}^{\bt}(\sB(\tilde{U}),L,w)))
\]
as in Remark \ref{LBrmkHod}.

We then let $\per_{dg}^{Y,\Tw,\an}(\sB)$ be the hypersheafification of the dg category-valued presheaf on 
$[\bP^1_{\R}/\SO_2]^{\hol}_{\Et}$ sending $(U,L,w)$ to the homotopy fibre product of the diagram
\[
\per_{dg}^{Y,\Hod,\an}(\sB(\tilde{U}),L,w) \to \per_{dg}^{Y,\dR,\an}(\sB(U^o_{\Cx})) \la  \per_{dg}^{Y,\oB}(\sB(U^o)),
\]
where the first map is induced by the natural  morphism $\Omega_{\Hod,Y}^{\bt}(\sB(\tilde{U}),L,w) \to \Omega^{\bt}_Y(\sB(U^o_{\Cx}))$, since $w=1$ on $U^o_{\Cx}$, and the second map is induced by the natural composite $ \sB(U^o) \to \sB(U^o_{\Cx}) \to \Omega^{\bt}_Y(\sB(U^o_{\Cx}))$.
\end{definition}
Thus Lemma \ref{hyperconnlemma} and Corollary \ref{poincarecor} imply that an object of $\per_{dg}^{Y,\Tw,\an}(\sB)(U,L,w)$ is locally (on 
$[\bP^1_{\R}/\SO_2]^{\hol}_{\Et}$) given by a perfect $\sO_Y(\sB(\tilde{U}))$-complex $\sE$ with flat $w\pd$-hyperconnection $\nabla$, a local system $\vv$ of perfect $\sB(U^o)$-modules on $Y$, and a quasi-isomorphism between the de Rham complex of $\sE\hten_{\sB(\tilde{U})}\sB(U^o_{\Cx})$ and  $\vv\ten_{ \sB(U^o)}\sB(U^o_{\Cx})$.

We also have an analogue for data of algebraic origin, defined on FEFC-DGAs rather than LDMC topological DGAAs.
\begin{definition}\label{alganTwdef}
 For any smooth complex algebraic variety $X$, define the algebro-analytic twistor functor from $\Shf_{\Et} dg_+\FEFC\Alg(\sO_{[\bP^1_{\R}/\SO_2]^{\hol}})$ to the $\infty$-category of hypersheaves of dg categories on $[\bP^1_{\R}/\SO_2]^{\hol}_{\Et}$ as follows. 

 For any $\sB \in \Shf_{\Et}dg_+\Alg(\sO_{[\bP^1_{\R}/\SO_2]^{\hol}})$, we can  form the sheaf  
\begin{align*}
 \Omega_{\Hod,X}^{\bt,\alg}(\sB(\tilde{U}),L,w):= (\bigoplus_{p} L^{\ten p} \ten \Omega^{p,\alg}_X\ten_{\Cx}\sB(\tilde{U})^{[-p]},w\pd, \delta)
\end{align*}
of
stacky DGAAs on $X$, where all unspecified tensor products are taken over $\sO(U)\ten_{\R}\Cx$, then set 
\[
\per_{dg}^{X,\Hod,\alg}(\sB(\tilde{U}),L,w):= \oR\Gamma(X,  D_*\per_{dg}(\Omega_{\Hod,X}^{\bt,\alg}(\sB(\tilde{U}),L,w)))
\]
similarly to Remark \ref{LBrmkHod} (adapted to algebraic forms).

For $\sB \in \Shf_{\Et}dg_+\FEFC(\sO_{[\bP^1_{\R}/\SO_2]^{\hol}})$,
we then let $\per_{dg}^{X,\Tw,\alg-\an}(\sB)$ be the hypersheafification of the dg category-valued presheaf on 
$[\bP^1_{\R}/\SO_2]^{\hol}_{\Et}$ sending $(U,L,w)$ to the homotopy fibre product of the diagram
\[
\per_{dg}^{X,\Hod,\alg}(\sB(\tilde{U}),L,w) \to \per_{dg}^{X(\Cx)^{\an},\oB}(\sB(U^o_{\Cx}))  \la  \per_{dg}^{X(\Cx)^{\an},\oB}(\sB(U^o)),
\]
where the first map is induced by combining the Riemann--Hilbert map of Corollary \ref{RHcoralg} with the natural  morphism $\Omega_{\Hod,X}^{\bt,\alg}(\sB(\tilde{U}),L,w) \to \Omega^{\bt,\an}_X\ten_{\Cx}\sB(U^o_{\Cx})$ given by multiplying $\Omega^p$ by $w^{-p}$.
\end{definition}
Thus Lemma \ref{hyperconnlemma}  implies that an object of $\per_{dg}^{X,\Tw,\alg-\an}(\sB)(U,L,w)$ is locally (on 
$[\bP^1_{\R}/\SO_2]^{\hol}_{\Et}$) given by a perfect $\sO_X\ten_{\Cx}\sB(\tilde{U})$-complex $\sE$ with flat $w\pd$-hyperconnection $\nabla$, a local system $\vv$ of perfect $\sB(U^o)$-modules on $X(\Cx)_{\an}$, and a quasi-isomorphism between the analytic de Rham complex of $\sE^{\hol}\hten_{\sB(\tilde{U})}\sB(U^o_{\Cx})$ and  $\vv\ten_{ \sB(U^o)}\sB(U^o_{\Cx})$. 

\begin{theorem}\label{twistorthm}\
\begin{enumerate}[itemsep=5pt, parsep=0pt]
 \item\label{twistorthm1} For any complex manifold $Y$, the $\C^{\infty}$ and analytic twistor functors 
 $\per_{dg}^{Y,\Tw,\C^{\infty}}$ and $\per_{dg}^{Y,\Tw,\an}$ are quasi-equivalent.
%  as functors %from
%  on
%  $\Shf_{\Et} dg_+\hat{\Tc}\Alg(\sO_{[\bP^1_{\R}/\SO_2]^{\hol}})$. 
 % to the $\infty$-category of hypersheaves of dg categories on $[\bP^1_{\R}/\SO_2]^{\hol}_{\Et}$

\item\label{twistorthm2} On restriction to  $\Shf_{\Et} dg_+\cF r\Alg(\sO_{[\bP^1_{\R}/\SO_2]^{\hol}})$, the  $\C^{\infty}$ and analytic twistor functors become w.e.-preserving and homogeneous.

The simplicial set-valued functor $\Perf^{Y,\Tw,\C^{\infty}}$ given by the nerve of the core of  $\per_{dg}^{Y,\Tw,\C^{\infty}}$ then has perfect cotangent complex at a $\sB(U)$-valued point $[(\sE,\nabla)]$ if and only if the complex 
   \[
    \oR\Gamma(Y,  (\oR\sHom_{\C^{\infty}_Y(\sB(U))} (\sE, \sE\hten_{\C^{\infty}_Y}\sA^{\#}_Y(\sB(U)), \delta \pm \nabla^{\ad}))
   \]
is perfect as a $\sB(U)^{\hat{e}}$-module.

On restriction to $\Shf_{\Et} dg_+\cN\cF r\Alg(\sO_{[\bP^1_{\R}/\SO_2]^{\hol}})$, the inclusion map from the  full dg subcategory of points with perfect cotangent complexes is formally \'etale, giving a homogeneous subfunctor with the same cotangent complexes.

\item\label{twistorthm3} For any smooth complex algebraic variety $X$, the algebro-analytic twistor functor $\per_{dg}^{X,\Tw,\alg-\an}$ 
(on hypersheaves of FEFC-DGAs)
% $\Shf_{\Et} dg_+\FEFC\Alg(\sO_{[\bP^1_{\R}/\SO_2]^{\hol}})$ 
is w.e.-preserving and homogeneous. When $X$ is projective, the %algebro-analytic twistor functor $\Perf^{X,\Tw, \alg-\an}$ 
simplicial set-valued functor $\Perf^{X,\Tw, \alg-\an}$ 
given by the nerve of its core has perfect cotangent complexes at all points.

\item\label{twistorthm4} When $X$ is projective, the natural transformation $\per_{dg}^{X,\Tw,\alg-\an} \to \per_{dg}^{X(\Cx)_{\an},\Tw,\an}$ of functors on $\Shf_{\Et} dg_+\hat{\Tc}\Alg(\sO_{[\bP^1_{\R}/\SO_2]^{\hol}})$ is fully faithful. Its restriction to $\Shf_{\Et} dg_+\cF r\Alg(\sO_{[\bP^1_{\R}/\SO_2]^{\hol}})$ is formally \'etale.
 \end{enumerate}
\end{theorem}
 \begin{proof}
 For (\ref{twistorthm1}),
since $\per_{dg}^{Y,\Tw,\C^{\infty}}(\sB)$ is a hypersheaf by definition, Lemma \ref{twistorpushoutlemma} gives a quasi-equivalence
  from 
$  \per_{dg}^{Y,\Tw,\C^{\infty}}(\sB)(U,L,w)$ to the homotopy pullback of 
  \[
     \per_{dg}^{Y,\Tw,\C^{\infty}}(\sB)(\tilde{U},L,w)\to \per_{dg}^{Y,\Tw,\C^{\infty}}(\sB)(U^o_{\Cx},L,w) \la \per_{dg}^{Y,\Tw,\C^{\infty}}(\sB)(U^o,L,w).
  \]
 Inclusion of analytic forms and of constants in $\C^{\infty}$ forms then give us a map to this from the homotopy fibre product of
 \[
\per_{dg}^{Y,\Hod,\an}(\sB(\tilde{U}),L,w) \to \per_{dg}^{Y,\dR,\an}(\sB(U^o_{\Cx})) \la  \per_{dg}^{Y,\oB}(\sB(U^o)),
  \]
  once we note that multiplying $\sA^{p,q}$ by $w^p\bar{w}^q$ gives an isomorphism $\sA^{\bt}_Y(\sB(U^o)) \cong \sA_{\Tw,Y}^{\bt}(\sB(U^o),L,w)$. We thus have a morphism $\per_{dg}^{Y,\Tw,\an}(\sB) \to \per_{dg}^{Y,\Tw,\C^{\infty}}(\sB)$ in the  $\infty$-category of hypersheaves of dg categories.
  
  To show that this is a quasi-equivalence, it suffices to address the terms in the fibre product separately. 
  The dg functors $ \per_{dg}^{Y,\dR,\an}(\sB(U^o_{\Cx})) \to \per_{dg}^{Y,\dR\C^{\infty}}(\sB(U^o_{\Cx}))$ and $ \per_{dg}^{Y,\oB}(\sB(U^o)) \to \per_{dg}^{Y,\dR,\C^{\infty}}(\sB(U^o))$  are quasi-equivalences by Corollary \ref{poincarecor}. That the dg functor $\per_{dg}^{Y,\Hod,\an}(\sB(\tilde{U}),L,w) \to \per_{dg}^{Y,\mathrm{pTw},\C^{\infty}}(\sB(\tilde{U}),L,w)$ is a quasi-equivalence follows by the argument of Corollary \ref{twistorpoincarecor}; we can either adapt the proof by incorporating twists by powers of $L$, or invoke the corollary directly, noting that   since we only need the statement locally, it suffices to replace $\tilde{U}$ with its $\Cx^*$-torsor determined by $L$, where the line bundle becomes trivial.
  
 For (\ref{twistorthm2}), that these functors are w.e.-preserving and homogeneous on Fr\'echet input follows from Lemma \ref{CXlemma1} as in Examples \ref{CXex2}. The condition for the cotangent complex to be perfect follows from Proposition \ref{RGammacotprop},  exactly as in Corollary \ref{twistorcotperfcor}. That inclusion of points with perfect cotangent complexes is a formally \'etale map then follows from Lemma \ref{opencotlemma}.
  
For (\ref{twistorthm3}), observe that the algebro-analytic twistor functor is defined as a homotopy limit of w.e.-preserving homogeneous functors, so inherits those properties. When $X$ is projective, the nerves of the cores of those constituent functors have perfect cotangent complexes, so  $\Perf^{X,\Tw, \alg-\an}$ inherits the same property since the Riemann--Hilbert map is formally \'etale by Corollary \ref{RHcoralget}.
  
For (\ref{twistorthm4}), the analytic Riemann--Hilbert correspondence  of Corollary \ref{poincarecor} gives us a local quasi-equivalence between $\per_{dg}^{X,\Tw,\alg-\an}(\sB(U),L,w)$ and the homotopy fibre product of the diagram
\[
 \per_{dg}^{X(\Cx)_{\an},\Tw,\an}(\sB(U))\to \per_{dg}^{X(\Cx)_{\an},\Hod,\an}(\sB(\tilde{U}),L,w) \la \per_{dg}^{X,\Hod,\alg}(\sB(\tilde{U}),L,w),
\]
so it suffices to proves the corresponding results for the dg functor $\per_{dg}^{X,\Hod,\alg}(\sB(\tilde{U}),L,w) \to \per_{dg}^{X(\Cx)_{\an},\Hod,\an}(\sB(\tilde{U}),L,w)$. These follow exactly as in Corollary \ref{Dolcoralg}, replacing Proposition \ref{Dolhgsprop} with Proposition \ref{Hodhgsprop}.
 \end{proof}

\section{Shifted bisymplectic and double Poisson structures}\label{poissonsn}

All of constructions and results from \cite{NCpoisson} have analytic analogues, in the same way that  \cite{DQDG} extends the results of \cite{poisson} to commutative analytic settings.

\subsection{Analytic shifted bisymplectic structures}

% \subsubsection{Definitions for affine objects}

The only subtleties in adapting the definitions and results of \cite[\S \ref{NCpoisson-spsn}]{NCpoisson} to the analytic setting now come from the tensor algebra constructions forming $\Omega^{\bt}_A$, and in taking cyclic and commutative quotients. 

\begin{definition}\label{anDRdef}
Given a finitely presented cofibrant  dg FEFC $R$-algebra (resp. stacky dg FEFC $R$-algebra), we define $\Omega^p_A$ to be the complex (resp. double complex) given by regarding $A$ as a (bi)graded nuclear Fr\'echet algebra by Proposition \ref{keyBanFEFCprop} and Lemma \ref{nucfrechet}, then taking the completed tensor product $\hat{\Omega}^p_A$ given by $\underbrace{\hat{\Omega}^1_A\hten_{A} \ldots \hat{\Omega}^1_A\hten_{A} \ldots \hat{\Omega}^1_A  }_p$ and forgetting the topology. Here $\hat{\Omega}^1_A=\ker(A\hten A \to A)$ as in Definition \ref{hatedef}, with underlying complex corresponding to  $\Omega^1_A$ from Definition \ref{anOmegadef} as in Lemma \ref{Omeganuclemma}.

We then extend the  definition to arbitrary cofibrant (stacky)  dg FEFC $R$-algebras by passing to filtered colimits via Lemma \ref{stackyindFPdg}.

The de Rham differential $d$ then gives us a cochain complex $\Omega^{\bt}_A$  in chain complexes (resp. double complexes), given by  $   A \to  \Omega^1_A \to \Omega^{2}_A \to \ldots $. We then define $\DR(A)$ to be the product total complex of this double (resp. triple) complex, with a  Hodge filtration $F$ on  $\DR(A)$ given  by setting    $F^p\DR(A) \subset \DR(A)$  to consist of terms $\Omega^i_{A}$ with $i \ge p$.
\end{definition}

\begin{remark}\label{anDRequiv}
There are other equivalent characterisations:
\begin{enumerate}
\item If the underlying graded (resp. bigraded) FEFC algebra of $A$ is free on generators $S$, then as in  the construction of \cite[Definition \ref{dgFEFCdef}]{NCstacks} we can form  %a bigraded (resp. trigraded) set $S \sqcup dS$ and hence 
a free bigraded (resp. trigraded) FEFC-algebra  $\cF_{S \sqcup dS}(R)$, and this will be $\Omega^{\bt}_A$ once we restore the differentials. This construction then extends to all  graded (resp. bigraded) FEFC algebras by Kan extension, giving a purely algebraic characterisation. 

\item From this description, it follows that $\Omega^{\bt}_A$ has a bidifferential (resp. tridifferential) bigraded (resp. trigraded) FEFC-algebra structure, with the property that a map from $\Omega^{\bt}_A$ to another such algebra $B$ is equivalent to a map from $A$ to the  differential (resp. bidifferential) graded (resp. bigraded) FEFC-algebra $B^0$.

\item To relate  $\Omega^{\bt}_A$ to the tensor algebra construction of \cite[Definition \ref{NCpoisson-DRdef}]{NCpoisson} in the algebraic setting,   we introduce an analogue of the construction $[A^{\ten (p+1)}]$ from \cite[Proposition 4.1.2]{vdBerghDoublePoisson}
% used in \cite[Definition \ref{NCpoisson-polcdef}]{NCpoisson} 
by writing    $[A^{\hten_{\pi} (p+1)}]$   for the $(A^{e,\FEFC})^{\ten p}$-module given by a $p$-fold completed projective tensor product $A^{\hten_{R,\pi} (p+1)}$, with the $i$th copy of $A \ten 1$ acting on the $i$th copy of $A$ on the right  and the $i$th copy of $1 \ten A$ acting on the $(i+1)$th copy of $A$ on the left. Then we have
\[
 \Omega^p_A\cong (\Omega^1_A)^{\ten p}\ten_{(A^{e,\FEFC})^{\ten p} }[A^{\hten_{\pi} (p+1)}].
\]
Note that when $A$ is finitely presented cofibrant, it is nuclear, so we can write $[A^{\hten (p+1)}]$ instead of $[A^{\hten_{\pi} (p+1)}]$; the general construction for cofibrant $A$ is then given by passing to filtered colimits.
\end{enumerate}
\end{remark}

We now need to construct the cyclic de Rham complex, which will entail taking a quotient by the closure of the commutator. Since the construction of  commutative truncations requires similar considerations, we deal with both constructions simultaneously.

\begin{definition}\label{ancommdef}
Given a retract of a free multigraded FEFC $R$-algebra $A$, we define the multigraded $R$-module $A_{\cyc}$ and the commutative multigraded EFC algebra $A^{\comm}$ as follows. 

If $A$ is finitely presented,  we use the multigraded analogue of  Lemma \ref{Omeganuclemma} to regard $A$ as a multigraded nuclear Fr\'echet algebra, and then we let $A_{\cyc}$ be the quotient of $A$ by the closure $\overline{[A,A]}$ of the  subspace of commutators, while $A^{\comm}$ is the quotient by the closure of the $2$-sided ideal generated by $[A,A]$.  We then extend these definitions to the general case by passing to filtered colimits.
 \end{definition}

\begin{remark}\label{ancycequiv}
 In the case where the FEFC algebra is $\cF_S(R)$,  freely   generated by a set $S$ concentrated in degree $0$,  consider the quotients $ \N_0^{S}= \coprod_{n\ge 0} \{\alpha \in W(S)~:~ |\alpha|=n\}/S_n $ and $ W(S)_{\cyc} :=\coprod_{n\ge 0} \{\alpha \in W(S)~:~ |\alpha|=n\}/C_n $ of the set $W(S)$ of words in $S$. Then  $ \cF_S(R)^{\comm}$ and $\cF_S(R)_{\cyc}$ are just given by replacing $W(S)$ with $\N_0^{S} $ and $W(S)_{\cyc}  $. There are similar expressions for graded sets, but with Koszul signs appearing in the actions of the symmetric and cyclic groups potentially identifying some elements of $W(S)$ with elements of $-W(S)$. 

These  descriptions can then be used to extend the definitions of $A_{\cyc}$ and $A^{\comm}$ to arbitrary multigraded FEFC $R$-algebras $A$ by  Kan extension. This gives an alternative characterisation of $A_{\cyc}$ (resp. $A^{\comm}$)  as the quotient of $A$ by the subspace generated by elements in the image of all maps  $\Phi_f$ for all finite multigraded sets $S$ and all $f \in \ker(\cF_S(R) \to \cF_S(R)_{\cyc})$ (resp.  $f \in \ker(\cF_S(R) \to \cF_S(R)^{\comm})$).
\end{remark}

\begin{definition}\label{anDRcycdef}
Given a cofibrant (stacky)  dg FEFC $R$-algebra $A$, define $\DR_{\cyc}(A)$ to be the product total complex of the double (or triple, when $A$ is stacky) complex $(\Omega^{\bt}_A)_{\cyc}$, with the  Hodge filtration $F$ on  $\DR_{\cyc}(A)$ given  by setting    $F^p\DR_{cyc}(A) \subset \DR_{cyc}(A)$  to be the image  of terms $\Omega^i_{A}$ with $i \ge p$.
\end{definition}

\begin{remark}\label{anDRcycequiv}
Note that for $p>0$,  $(\Omega^p_A)_{\cyc}$ is the complex underlying  $(\hat{\Omega}^p_A\hten_{(A^{\hat{e}})}A)/C_p$ when $A$ is finitely presented. %%NB we can't forgo completions, since we're sedning $A^{\hat{e}}\hten M$ to $A \hten M$.

Consider the analogue 
of the construction $\{A^{\ten p}\}$ %used in \cite[Definition \ref{NCpoisson-bipoldef}]{NCpoisson} 
 from \cite[Proposition 4.1.2]{vdBerghDoublePoisson}
by writing $\{A^{\hten_{\pi} p}\}$ for $A^{\hten_{\pi}p}$ given the $(A^{\hat{e}})^{\hten p}$-bimodule structure for which the $i$th copy of $A\ten 1$ acts on the right on the $i$th copy of $A$, and  the $i$th copy of $1\ten A$ acts on the left on the $(i+1 \mod p)$th copy of $A$. Here, $\hten_{\pi}$ is just $\hten$ when $A$ is finitely presented cofibrant, and the expression for arbitrary cofibrant $A$ is given by passing to filtered colimits.

Following Remark \ref{anDRequiv}, we then have
\[
 (\hat{\Omega}^p_A)_{\cyc}\cong (\hat{\Omega}^1_A)^{\hten p}\hten_{(A^{\hat{e}})^{\hten p} }\{A^{\hten p}\}/C_p
\]
and hence
\[
 (\Omega^p_A)_{\cyc}\cong (\Omega^1_A)^{\ten p}\ten_{(A^{\hat{e}})^{\ten p} }\{A^{\hten p}\}/C_p.
\]
\end{remark}

The definitions and properties of $n$-shifted bisymplectic structures from \cite[\S \ref{NCpoisson-spsn}]{NCpoisson} now all carry over unchanged.

The following adapts \cite[Definitions \ref{NCpoisson-presymplecticdef} and \ref{NCpoisson-PreSpdef}]{NCpoisson}.
\begin{definition}\label{presymplecticdef}
Define an $n$-shifted pre-bisymplectic structure $\omega$ on a cofibrant (stacky) FEFC $R$-DGA to be a cocycle
\[
 \omega \in \z^{n+2}F^2\DR_{\cyc}(A).
\]

Then define the space $\PreBiSp(A,n)= \Lim_{p\ge 2}\PreBiSp(A,n)/F^p $ of $n$-shifted pre-bisymplectic structures on $A/R$ to be the simplicial set given in degree $k$ by setting
 \[
(\PreBiSp(A,n)/F^p)_k:= \z^{n+2}((F^2\DR_{\cyc}(A/R)/F^p)\ten_{\Q} \Omega^{\bt}(\Delta^k)),
\]
where 
\[
\Omega^{\bt}(\Delta^k)=\Q[t_0, t_1, \ldots, t_k,\delta t_0, \delta t_1, \ldots, \delta t_k ]/(\sum t_i -1, \sum \delta t_i)
\]
is the commutative dg algebra of de Rham polynomial forms on the $k$-simplex, with the $t_i$ of degree $0$.
\end{definition}
Note that  $\PreBiSp(A,n)/F^p$  is canonically weakly  equivalent to the Dold--Kan denormalisation of the complex $\tau^{\le 0}(F^2\DR_{\cyc}(A)[n+2]/F^p)$, where $\tau$ denotes good truncation, and similarly for the limit  $ \PreBiSp(A,n)$. However,  the definition in terms of de Rham polynomial forms  simplifies the comparison with double Poisson structures.

The following adapts \cite[Definitions \ref{NCpoisson-bisymplecticdef} and \ref{NCpoisson-bisymplecticdefstacky}]{NCpoisson}.
\begin{definition}\label{bisymplecticdef}
When $A$ is a cofibrant FEFC-DGA, say that $\omega$ is bisymplectic if $\Omega^1_A$ is perfect as an $A^{e,\FEFC}$-module and the component $\omega_2 \in \z^n\Omega^2_{A,\cyc}$ induces a quasi-isomorphism
\[
 \omega_2^{\sharp} \co \HHom_{A^{e,\FEFC}}(\Omega^1_A,A^{e,\FEFC}) \to (\Omega^1_{A})_{[-n]}
\]
by contraction $\phi \mapsto i_{\phi}(\omega_2)$, for $i_{\phi} \co \Omega^*_A \to \Omega^*_A\hten_{\pi} \Omega^*_A$ 
the double derivation determined by the properties that $i_{\phi}(a)=0$ and $i_{\phi}(da)= \phi(a)$, for $a \in A$.

When $A$ is a stacky FEFC-DGA, say that $\omega$ is bisymplectic if $\Tot^{\Pi} (\Omega^1_{A}\ten_{A^{e,\FEFC}}(A^0)^{e})$ is perfect as an $(A^0)^{e}$-module and $\omega_2$
induces a quasi-isomorphism
\[
 \omega_2^{\sharp} \co \hatHHom_{A^{e,\FEFC}}(\Omega^1_A,(A^0)^{e,\FEFC}) \to \Tot^{\Pi}(\Omega^1_{A}\ten_{A^{e}}(A^0)^{e})_{[-n]},
\]
where  $\hatHHom= \widehat{\Tot}\cHHom$, for $\cHom$ the internal $\Hom$ in double complexes and $\widehat{\Tot}$ the sum-product total complex (sometimes known as the Tate realisation).

Let $\BiSp(A,n) \subset \PreBiSp(A,n)$ consist of the bisymplectic structures --- this is a union of path-components.
\end{definition}

There are also similar definitions for analytic bi-Lagrangian structures defined in exactly the same way, which we will not spell out.

\begin{remark}\label{bispcommrmk}
 As in \cite[Lemma \ref{NCpoisson-commSplemma}]{NCpoisson}, it follows quickly from the definitions that  the assignment $A \mapsto A^{\comm}$ gives a map from the space of $n$-shifted analytic bisymplectic structures on $A$ to the space of $n$-shifted analytic symplectic structures   on the EFC-DGA $A^{\comm}$, in the sense of \cite{DQDG,DStein}.
\end{remark}

\begin{lemma}\label{bispwepreservinglemma}
 The construction $\PreBiSp(-,n)$ defines a w.e.-preserving functor on the category %$DG^+dg_+\FEFC(R)^{\mathrm{cof}}$, 
 of cofibrant stacky FEFC-DGAs, and
$\BiSp(-,n)$ defines a w.e.-preserving functor on its subcategory of  homotopy  \'etale morphisms (Definition \ref{hetdef}). 

Moreover, if a homotopy \'etale morphism $A \to B$ admits an $A^{e,\FEFC}$-linear retraction, then  $\PreBiSp(A,n) \to \PreBiSp(B,n)$ and $\BiSp(A,n) \to \BiSp(B,n)$ are weak equivalences.
\end{lemma}
\begin{proof}
 Functoriality of pre-bisymplectic structures is immediate, and functoriality of bisymplectic structures with respect to homotopy \'etale morphisms follows because they do not affect the non-degeneracy condition. 
 
To show that these functors are w.e.-preserving, it suffices to prove that $A \mapsto \Omega^p_{A,\cyc}$ is so. Using the final description  %$(\Omega^p_A)_{\cyc}\cong (\Omega^1_A)^{\ten p}\ten_{(A^{\hat{e}})^{\ten p} }\{A^{\hten p}\}/C_p$
 of Remark \ref{anDRcycequiv}, this follows from homotopy invariance of the cofibrant $A^{e,\FEFC}$-module $\Omega^1_A$ and of the $(A^{\hat{e}})^{\ten p}$-module  $\{A^{\hten p}\}$, since taking $C_p$-coinvariants is exact in characteristic $0$.

The final statement, which has no commutative analogue, follows by the method of \cite[Lemma \ref{NCpoisson-calcTOmegaetlemma}]{NCpoisson}, since the retraction allows us to write $B/A$ as an $A^{e,\FEFC}$-linear direct summand of $\Omega^1_{B/A}:=\coker(\Omega^1_A\ten_{A^{e,\FEFC}}B^{e,\FEFC} \to \Omega^1_B)$.
 \end{proof}

\subsection{Analytic shifted double Poisson structures}

We now  introduce modifications of \cite[Definitions \ref{NCpoisson-bipoldef} and \ref{NCpoisson-bipolcycdef}]{NCpoisson} for application to this analytic context. 

\begin{definition}\label{anbipoldef}
Define the complex of $n$-shifted non-commutative multiderivations (or polyvectors) on a cofibrant stacky dg FEFC $R$-algebra $A \in DG^+dg_+\FEFC(R)$ by
\[
F^i \widehat{\Pol}^{nc}(A,n):= \prod_{p \ge i} \hatHHom_{(A^{e,\FEFC})^{\ten p}}(((\Omega^1_{A})_{[-n-1]})^{\ten p},[A^{\hten_{\pi} (p+1)}]), 
\]
for $i \ge 0$ and  for the $(A^{e,\FEFC})^{\ten_R p}$-module $[A^{\hten_{\pi} (p+1)}]$ introduced in Remark \ref{anDRequiv}.

This has a filtration-preserving associative multiplication given by the same formulae as in \cite[Definition \ref{NCpoisson-poldef}]{NCpoisson}. %%can't even use filtered colimit arguments now, because $\Omega$ won't tend to be perfect on the nose.
\end{definition}

\begin{definition}\label{anbipolcycdef}
We  define the filtered cochain complex of $n$-shifted cyclic multiderivations  on   a cofibrant stacky dg FEFC $R$-algebra  $A \in DG^+dg_+\FEFC(R)$ by
\[
 F^i\widehat{\Pol}^{nc}_{\cyc}(A,n):= \prod_{p \ge i} (\hatHHom_{(A^{e,\FEFC})^{\ten p}}(((\Omega_A^1)_{[-n-1]})^{\ten p}, \{A^{\hten_{\pi} p}\})^{C_p},
\]
for $i \ge 1$, for the $(A^{e,\FEFC})^{\ten p}$-module $\{A^{\hten_{\pi} p}\}$ introduced in Remark \ref{anDRcycequiv}.

We then set $F^0\widehat{\Pol}^{nc}_{\cyc}(A,n):= A_{\cyc} \oplus F^1\widehat{\Pol}^{nc}_{\cyc}(A,n)$.
\end{definition}

 Adapting \cite[Proposition 4.1.1]{vdBerghDoublePoisson},   the trace isomorphism between cyclic invariants and coinvariants gives a natural filtered map
\[
\tr \co \widehat{\Pol}^{nc}(A,n)/[\widehat{\Pol}^{nc}(A,n),\widehat{\Pol}^{nc}(A,n)]\to  \widehat{\Pol}^{nc}_{\cyc}(A,n) 
\]
from the quotient by the commutator of the associative multiplication. Unlike the algebraic case, this will not be a filtered quasi-isomorphism when $\Omega^1_A$ is perfect and cofibrant as an $A$-bimodule, essentially because the commutator subspace does not take the analytic structure into account; when $A$ finitely generated, this can be rectified by replacing the commutator with its closure.

The definitions of \cite[\S\S \ref{NCpoisson-poisssn} and \ref{NCpoisson-ArtinPoisssn}]{NCpoisson} now adapt as follows.

\begin{proposition}\label{bracketprop}
There is a natural  bracket $\{-,-\}$ making  $\widehat{\Pol}^{nc}_{\cyc}(A,n)^{[n+1]}$ into a differential graded Lie algebra (DGLA) over $R$, satisfying $\{F^i,F^j\} \subset F^{i+j-1}$. 

There  is also an
$R$-bilinear map 
\[
 \{-,-\}^{\smile} \co \widehat{\Pol}^{nc}_{\cyc}(A,n) \by \widehat{\Pol}^{nc}(A,n) \to \widehat{\Pol}^{nc}(A,n)^{[-n-1]} 
\]
which lifts $\{-,-\}$ in the sense that $\tr(\{\pi, \alpha\}^{\smile})=\{\pi, \tr\alpha\}$. This also satisfies   $\{F^i,F^j\}^{\smile} \subset F^{i+j-1}$, and is a derivation in  its second argument. Given   a  $p$-bracket $\pi$ and an element $a \in A$, the $(p-1)$-derivation $\{\pi, a\}^{\smile}$ is given by   $\pm\pi(-,-, \ldots, -,a)$.
\end{proposition}
\begin{proof}
 This follows almost exactly as in the proof of \cite[Proposition \ref{NCpoisson-bracketprop}]{NCpoisson}, except that $\{-,-\}^{\smile}$
 is induced by
a double bracket 
\[
 \widehat{\Pol}^{nc}_{\cyc}(A,n) \by \widehat{\Pol}^{nc}(A,n)\to \widehat{\Pol}^{nc}(A,n)\hten_{\pi} \widehat{\Pol}^{nc}(A,n)^{[-n-1]}
 \]
 taking values in a completed topological (rather than algebraic) tensor product.
 \end{proof}

 \begin{definition}\label{mcPLdef}
 Given a   DGLA $(L, \{-,-\})$, define the the Maurer--Cartan set by 
\[
\mc(L):= \{\omega \in  L^{1}\ \,|\, \delta\omega + \half\{\omega,\omega\}=0 \in  \bigoplus_n L^{2}\}.
\]

Following \cite{hinstack}, define the Maurer--Cartan space $\mmc(L)$ (a simplicial set) of a nilpotent  DGLA $L$ by
\[
 \mmc(L)_k:= \mc(L\ten_{\Q} \Omega^{\bt}(\Delta^k)),
\]
for the differential graded commutative algebras $ \Omega^{\bt}(\Delta^n)$ of de Rham polynomial forms on the $k$-simplex, 
as in Definition \ref{presymplecticdef}.
% as in Definition \ref{PreSpdef}.
% % where 
% % \[
% % \Omega^{\bt}(\Delta^n)=\Q[t_0, t_1, \ldots, t_n,\delta t_0, \delta t_1, \ldots, \delta t_n ]/(\sum t_i -1, \sum \delta t_i)
% % \]
% % is the commutative dg algebra of de Rham polynomial forms on the $n$-simplex, with the $t_i$ of degree $0$.
\end{definition}

\begin{definition}\label{poissdef}
Define an $R$-linear analytic $n$-shifted double Poisson structure on a cofibrant stacky dg FEFC $R$-algebra  $A$ to be an element of
\[
 \mc(F^2 \widehat{\Pol}^{nc}_{\cyc}(A,n)[n+1]), 
\]
 and the space $\cD\cP(A,n)$ of analytic $R$-linear  $n$-shifted double Poisson structures on $A$ to be given by the simplicial 
set
\[
 \cD\cP(A,n):= \Lim_i \mmc(F^2 \widehat{\Pol}^{nc}_{\cyc}(A,n)[n+1]/F^{i+2}).
\]
\end{definition}

\begin{definition} \label{binondegdef}
If %$A^e \simeq A^{\oL,e}$ (e.g. for $A$ cofibrant) and  %%already assumed
$\Tot^{\Pi} (\Omega^1_{A/R}\ten_{A^{e,\FEFC}}(A^0)^{e,\FEFC}) $ is perfect as an $(A^0)^{e,\FEFC}$-module, and
there exists $N$ for which the chain complexes $(\Omega^1_{A/R}\ten_{A^{e,\FEFC}}(A^0)^{e,\FEFC})^i $ are acyclic for all $i >N$,
we say that an  $n$-shifted double Poisson structure $\pi = \sum_{i \ge 2}\pi_i $ on $A$ is non-degenerate if
contraction with  $\pi_{2}$  induces a quasi-isomorphism
\[
\pi_{2}^{\flat}\co \Tot^{\Pi} (\Omega^1_{A/R}\ten_{A^{e,\FEFC}}(A^0)^{e,\FEFC})_{[-n]} \to \hatHHom_{A^{e,\FEFC}}(\Omega^1_A, (A^0)^{e,\FEFC}).
\]

We then define $\cD\cP(A,n)^{\nondeg}\subset \cD\cP(A,n)$ to consist of the non-degenerate elements --- this is a union of path-components.
\end{definition}

There are also similar definitions for analytic double co-isotropic structures defined in exactly the same way, which we will not spell out.

\begin{remark}\label{poissoncommrmk}
 Observe that for the abelianisation functor $A \mapsto A^{\comm}$ from DGAAs to CDGAs of Definition \ref{ancommdef}, we have a natural map  $F^i\widehat{\Pol}^{nc}_{\cyc}(A,n) \to F^i\widehat{\Pol}(A^{\comm},n)$, for the complex $\widehat{\Pol}$ of polyvectors from \cite{DQDG}. This map is compatible with the Lie bracket, so gives a natural map $\cD\cP(A,n) \to  \cP(A^{\comm},n)$ from the space of $n$-shifted double Poisson structures on $A$ to the space of $n$-shifted  Poisson structures on $A^{\comm}$.
\end{remark}

\begin{lemma}\label{DPwepreservinglemma}
 The construction $\cD\cP(-,n)$ and its subfunctor $\cD\cP(-,n)^{\nondeg}$ extend to  $\infty$-functors on the $\infty$-category %$DG^+dg_+\FEFC(R)^{\mathrm{cof}}$, 
 of  homotopy  \'etale morphisms (Definition \ref{hetdef})
 of cofibrant stacky FEFC-DGAs Localised at levelwise quasi-isomorphisms.
 
 Moreover, if a homotopy \'etale morphism $A \to B$ admits an $A^{e,\FEFC}$-linear retraction, then  $\cD\cP(A,n) \to \cD\cP(B,n)$ and $\cD\cP(A,n)^{\nondeg} \to \cD\cP(B,n)^{\nondeg}$ are weak equivalences.
\end{lemma}
\begin{proof}
 This follows from the same arguments as \cite[\S \ref{NCpoisson-Artindiagramsn}]{NCpoisson}, replacing $A^e$ with $A^{e,\FEFC}$ and $\{A^{\ten p}\}$ with $\{A^{\hten_{\pi} p}\}$ throughout.
\end{proof}

\begin{proposition}\label{equivpropaff}
  There is a canonical equivalence
 \[
  \cD\cP(A,n)^{\nondeg} \simeq \BiSp(A,n)
 \]
between the spaces of analytic non-degenerate $n$-shifted double Poisson structures and of $n$-shifted bisymplectic structures 
on a cofibrant stacky FEFC-DGA $A$, compatible with homotopy \'etale functoriality.
 \end{proposition}
\begin{proof}
 This follows by a construction almost identical to that used in proving the analogous algebraic statement \cite[Corollary \ref{NCpoisson-compatcor2}]{NCpoisson}. The only subtlety lies  in the construction of a compatibility map $\mu$. Since $\Omega^{\bt}_A$ is no longer generated as an associative algebra by $A$ and $\Omega^1_A$, we have to check that the formulae of \cite[Definition \ref{NCpoisson-mudef}]{NCpoisson} extend uniquely to give maps on the infinite sums featuring the spaces $\Omega^p_A$ in the analytic context.  To do this, we can use the expression
\[
  \Omega^p_A\cong (\Omega^1_A)^{\ten p}\ten_{(A^{e,\FEFC})^{\ten p} }[A^{\hten_{\pi} (p+1)}]
 \]
from Remark \ref{anDRequiv}, and extend the construction of $\mu(-,\pi)$ 
to $\Omega^p_A$ by making use of the various natural $[A^{\hten_{\pi} (p+1)}] $-module structures on $[A^{\hten_{\pi} (n+1)}] $ for $n \ge p$.
 \end{proof}

%%%pread to here

\subsection{Structures on prestacks}

We can now define shifted bisymplectic or double Poisson structures on prestacks as mapping spaces from the prestack to the relevant construction. Since these structures are only functorial with respect to homotopy \'etale morphisms, while pre-bisymplectic structures are fully functorial, for consistency we will restrict attention to prestacks for which the conclusion of Corollary \ref{etsitecor} holds. 

\begin{definition}\label{PreBiSphgsdef}
Given a w.e.-preserving functor $F \co DG^+dg_+\hat{\Tc}\Alg_{R}^{\mathrm{dfp,cof}} \to s\Set$ satisfying $F \simeq \theta_!F_{\rig}$ for $\theta \co DG^+dg_+\hat{\Tc}\Alg_{R}^{\mathrm{dfp,cof},\et} \to DG^+dg_+\hat{\Tc}\Alg_{R}^{\mathrm{dfp,cof}}$,
 we define 
 the spaces $\PreBiSp(F,n)$ and $\BiSp(F,n)$ of analytic $n$-shifted pre-bisymplectic and bisymplectic structures on $F$ by
\begin{align*}
 \PreBiSp(F,n)&:= \map_{ [DG^+dg_+\hat{\Tc}\Alg_{R}^{\mathrm{dfp,cof}},s\Set]}(F, \PreBiSp(-,n)),\\
 \BiSp(F,n)&:= \map_{ [DG^+dg_+\hat{\Tc}\Alg_{R}^{\mathrm{dfp,cof,\et}},s\Set] }(F_{\rig}, \BiSp(-,n)),
 \end{align*}
 and the spaces   $\cD\cP(F,n)$ and $\cD\cP(F,n)^{\nondeg}$ of analytic  $n$-shifted double Poisson and non-degenerate double Poisson structures on $F$ by
\begin{align*}
 \cD\cP(F,n)&:= \map_{ [DG^+dg_+\hat{\Tc}\Alg_{R}^{\mathrm{dfp,cof,\et}},s\Set] }(F_{\rig}, \cD\cP(-,n))\\
 \cD\cP(F,n)^{\nondeg}&:= \map_{ [DG^+dg_+\hat{\Tc}\Alg_{R}^{\mathrm{dfp,cof,\et}},s\Set] }(F_{\rig}, \cD\cP(-,n)^{\an}).
 \end{align*}
 \end{definition}

 As an immediate consequence of Proposition \ref{equivpropaff}, we then have:
 \begin{proposition}\label{equivpropprestack}
  There is a canonical equivalence
 \[
  \cD\cP(F,n)^{\nondeg} \simeq \BiSp(F,n).
 \]
 \end{proposition}
 
 \begin{remarks}
The condition $F \simeq \theta_!F_{\rig}$ ensures that  $\PreBiSp(F,n)\simeq \map_{ [DG^+dg_+\hat{\Tc}\Alg_{R}^{\mathrm{dfp,cof,\et}},s\Set]}(F_{\rig}, \PreBiSp(-,n))$, consistently with the others.
 
 \smallskip
The proof of Corollary \ref{etsitecor} shows that for any w.e.-preserving homogeneous functor $F\co dg_+\cN\cF r\Alg(R)\to s\Set$   with perfect cotangent complexes at all points in $F(A)$ for all $A\in dg_+\hat{\Tc}\Alg_{R}^{\mathrm{fp,cof}}$, the functor $D_*F| \co  DG^+dg_+\hat{\Tc}\Alg_{R}^{\mathrm{dfp,cof}} \to s\Set$ satisfies the condition $D_*F| \simeq \theta_!(D_*F|)_{\rig}$. Most of our examples of interest will arise in this form.
 
 \smallskip
 The setting of Definition \ref{PreBiSphgsdef} is fairly minimal in that we have restricted attention to functors defined on degreewise finitely presented cofibrant objects. We could equivalently enlarge the category to include all levelwise quasi-isomorphic objects, with  Proposition \ref{resnprop} (or rather its stacky analogue) and subsequent examples providing a rich class of such objects. Also, functors of interest defined on larger categories tend to be locally of finite presentation, so Proposition \ref{stackyDlfpprop} establishes consistency with other potential definitions in those cases.
 
%  Finally, Remarks \ref{bispcommrmk} and \ref{poissoncommrmk} imply that a shifted (pre-)bisymplectic structure on $F$ gives rise to a shifted (pre-)symplectic structure on the restriction $F^{\comm}$ of $F$ to nah, not quite true as those are never NC cofibrant. 
 \end{remarks}

\begin{example}[Integration and (pre-)CY structures]\label{IntCYex}
As a consequence of Lemmas \ref{bispwepreservinglemma} and \ref{DPwepreservinglemma}, homotopy \'etale morphisms with retractions induce equivalences on the spaces of shifted analytic pre-bisymplectic, bisymplectic and  double Poisson structures. In particular, if the functor $F$ can be resolved by a simplicial derived analytic NC affine $X_{\bt}$ satisfying the FEFC analogues of the derived NC Artin hypergroupoid conditions from  \cite[Definition \ref{NCstacks-npreldef2}]{NCstacks}, then there is an associated stacky FEFC-DGA $D^*O(X)$ and we have $\BiSp(F,n) \simeq \BiSp(D^*O(X),n)$ and  $\cD\cP(F,n) \simeq \BiSp(D^*O(X),n)$, reasoning as in \cite[Corollaries \ref{NCpoisson-integratecorBiSp} and \ref{NCpoisson-integratecorDP}]{NCpoisson}.

For example, for a prestack $Y$ represented by an FEFC-DGA $A$, we can consider the quotient prestack $[Y/\bG_m]$  universal under the functor $B \mapsto [\Hom(A,B)/B_0^{\by}]$, where the group $B_0^{\by}$ of multiplicative units in $B_0$ acts by conjugation on the set of FEFC homomorphisms. The \v Cech nerve $X_{\bt}$ of $Y \to [Y/\bG_m]$ gives a suitable resolution, and then $D^*O(X)$ is the stacky FEFC-DGA freely generated over $A$ by a variable $s \in D^*O(X)^1_0$ with $\pd s =s^2$ and $\pd a = [s,a]$ for all $a \in A$. 

Comparison with \cite[\S\S \ref{NCpoisson-CYsn},\ref{NCpoisson-preCYsn}]{NCpoisson} characterises shifted analytic bisymplectic and double Poisson structures on $[Y/\bG_m]$ as analytic analogues of Calabi--Yau and pre-Calabi--Yau structures on $A$, respectively.
\end{example}

\begin{remark}\label{bispcommrmk2}
Following on from Remark \ref{bispcommrmk}, if $F^{\comm}$ denotes the restriction of $F$ to EFC-DGAs or to commutative topological DGAs, then   we have a natural map  $\BiSp(F,n) \to \Sp(F^{\comm, \sharp},n)$ to the space of shifted  symplectic structures, where $(-)^{\sharp}$ denotes \'etale hypersheafification. 

We can however say much more. The proofs of \cite[Propositions \ref{NCpoisson-weilArtinprop} and \ref{NCpoisson-weilhgsprop}]{NCpoisson} give maps $\BiSp(F,n) \to \BiSp(\Pi_{S/R}F,n)$ for any finite flat Frobenius $R$-algebra $S$, where $\Pi_{S/R}$ is the Weil restriction of scalars functor $(\Pi_{S/R}F)(B):=F(B\ten_RS)$. In particular, for matrix rings this gives natural maps $\BiSp(F,n) \to \Sp( (\Pi_{\Mat_r}F)^{\comm ,\sharp},n)$, so an analytic  shifted bisymplectic structure gives analytic shifted symplectic structures on all of the representation spaces. 

More is true: for any flat Lagrangian subalgebra $T$ of $S$, the same proofs give a map $\BiSp(F,n) \to \BiLag(\Pi_{S/R}F,\Pi_{T/R}F; n)$ to the space of analytic shifted bi-Lagrangian structures. Applied to the ring $T(m,n)$ of block upper triangular matrices, regarded as Lagrangian in the ring $\Mat_{m+n}\by \Mat_m \by \Mat_n$ with its trace $(\tr, -\tr,-\tr)$, an $n$-shifted analytic bisymplectic structure on $F$ thus gives rise to an $n$-shifted analytic Lagrangian correspondence 
\[
(\Pi_{\Mat_{m+n}}F)^{\comm ,\sharp} \la (\Pi_{T(m,n)} F)^{\comm ,\sharp} \to (\Pi_{\Mat_m} F)^{\comm ,\sharp} \by (\Pi_{\Mat_n} F)^{\comm ,\sharp}
\]
on the derived analytic stacks featuring in the construction of Hall algebras (cf. \cite[\S \ref{NCstacks-flagHallsn}]{NCstacks}).
\end{remark}

\subsection{Examples}\label{bispexsn}

Throughout this section, we will write $\Perf^?$ for the simplicial set given by the nerve of the core of the dg category $\per_{dg}^?$. this has the properties that its points are the objects of the dg category, with the fundamental group at $a$ consisting of invertible elements in $\H_0\EEnd(a)$ and the $i$th homotopy group being $\H_{i-1}\EEnd(a)$ for $i>1$.  

\subsubsection{Analytification}\label{analytificnsn3}

As in \ref{analytificationsn2}, we can use
the obvious forgetful functors $dg_+\FEFC(R) \to dg_+\Alg(R)$ and $DG^+dg_+\FEFC(R)\to DG^+dg_+\Alg(R)$,  which both have left adjoints $(-)^{\FEFC}$, to associate to an analytic derived NC prestack $F^{\an} \co dg_+\FEFC(R) \to s\Set$ to any derived NC prestack $F$. Since cotangent complexes transform as $ \bL_{F^{\an},x} \simeq \bL_{F,x}\ten^{\oL}_{B^e}B^{e,\an}$, with \'etaleness and submersiveness being preserved by analytification, we automatically get natural maps
\begin{align*}
 \PreBiSp(F,n) &\to \PreBiSp(D_*F^{\FEFC},n)\\
 \BiSp(F,n) &\to \BiSp(D_*F^{\FEFC},n)\\
 \cD\cP(F,n) &\to \cD\cP(D_*F^{\FEFC},n)\\
  \cD\cP(F,n)^{\nondeg} &\to \cD\cP(D_*F^{\FEFC},n)^{\nondeg},
\end{align*}
where the structures on the left are defined in \cite{NCpoisson}.

In particular, this can be applied to the Betti moduli functor  $\Perf^{M,\oB}$ (after Definition \ref{Bettidef}) when $M$ is a compact oriented $d$-dimensional real manifold, with the $(2-d)$-shifted bisymplectic structure of \cite[Remark \ref{NCpoisson-locsysrmk}]{NCpoisson} coming from Poincar\'e duality giving rise to an analytic $(2-d)$-shifted bisymplectic structure.

Similarly, if $Y$ is a complex projective variety of dimension $d$, the algebraic de Rham functor $\Perf^{X,\dR,\alg}$ (after Definition \ref{dRalgperdef}), the algebraic Dolbeault functor $\Perf^{X,\Dol,\alg}$ and the algebraic Hodge functor  $\Perf^{X,\Hod,\alg}$ (after Definition \ref{alganTwdef}) carry $(2-2d)$-shifted bisymplectic structures coming from Grothendieck--Verdier duality, so their analytifications carry analytic bisymplectic structures.

\subsubsection{Constructions from dg categories}

For dg categories equipped with suitable classes in either cyclic homology or cohomology, \cite[Corollaries \ref{NCpoisson-Perfpoissoncor} and \ref{NCpoisson-Morpoissoncor}]{NCpoisson} establishes shifted double Poisson structures on moduli of perfect complexes and of derived Morita morphisms (a common generalisation is also possible fixing two dg categories $\cA, \C$ and parametrising Morita morphisms $\cA \to \C\ten B$ for varying $B$).  For categories enriched in complexes of nuclear Fr\'echet spaces, entirely analogous constructions of analytic shifted double Poisson structures are possible, using homology or cohomology theories defined using completed topological tensor products.

\subsubsection{Transgression structures on moduli functors}

We now consider moduli functors of the form introduced in \S \ref{moduliexsn}, giving rise to a large class of shifted bisymplectic structures which are not of algebraic origin, by an analytic  transgression process. This is certainly not the most general form of transgression possible, since there are also analytic analogues of the mapping stack results of \cite[Remark \ref {NCpoisson-weilhgsprop2}]{NCpoisson}, but it covers our motivating examples.

Specifically, as in that section we will take a countable category ${\bI}$ and 
 a presheaf $\sC$ of commutative Fr\'echet $\bK$-algebras  on ${\bI}$ for $\bK$ a complete valued field. We will moreover assume that $\oR\Lim_{\bI^{\op}}\sC$ is a perfect complex over $\bK$, equipped with a $\bK$-linear map $\tau$ to $\bK[-d]$.  In particular, note that for any $M \in dg_+\cN\cF r_{\bK}$, Lemma \ref{RGammalemma} then gives a quasi-isomorphism $ (\oR\Lim_{\bI^{\op}}\sC)\ten M \to \oR\Lim_{\bI^{\op}}(\sC\hten M)$, and hence a map $\tau_M \co \oR\Lim_{\bI^{\op}}(\sC\hten_{\bK}M) \to M[-d]$ in the $\bK$-linear derived category.
 
\begin{proposition}\label{Cbispprop}
In the setting above, take a derived NC Artin prestack $F \co  dg_+\Alg(\bK) \to s\Set$ equipped with an $n$-shifted bisymplectic structure $\omega$.
 
 Then %for the functor  $F_{\sC}:= \ho\Lim_{\bI^{\op}} F(\sC\hten-) \co dg_+\cN\cF r\Alg(\bK) \to s\Set$, 
 the pair $(\omega,\tau)$ give rise to an analytic  $(n-d)$-shifted pre-bisymplectic structure $\omega_{\tau}$ on $D_*F_{\sC}$, for the functor
 \[
F_{\sC}:=\ho\Lim_{\bI^{\op}} F(\sC\hten-) \co  DG^+dg_+\hat{\Tc}\Alg_{R}^{\mathrm{dfp,cof}} \to s\Set.
 \]

The space of points $\phi \in F_{\sC}(A)$ at which the pairing
\[
 (\ho\Lim_{\bI^{\op}} \bT_{\phi}F(\sC\hten A^{\hat{e}}))^{\ten_{\bK} 2} \xra{\lrcorner \omega_{2,\sC\hten A}} \ho\Lim_{\bI^{\op}} (\sC\hten A^{\hat{e}})[n] \xra{\tau_{A^{\hat{e}}}} A^{\hat{e}}[n-d]
\]
is perfect forms a formally \'etale subfunctor $F_{\sC}^{nd} \into F_{\sC}$, and  the restriction of $\omega_{\tau}$ to $D_*F_{\sC}^{nd}$ is bisymplectic.
\end{proposition}
\begin{proof}
First, note that Lemma \ref{CXlemma2} ensures that $F_{\sC}$ is w.e.-preserving and homogeneous, so $D_*F_{\sC}$ satisfies the conditions of Definition \ref{PreBiSphgsdef} by Corollary \ref{etsitecor}.

Elements $\phi \in D_*F_{\sC}(A)$ correspond to elements of $\ho\Lim_{\bI^{\op}}D_*F(\sC\hten A) $, for $D_*$ as in  \cite[Definition \ref{NCstacks-Dlowerdef}]{NCstacks}, so $\phi$ gives rise to a map
\[
\oR\Gamma(F, \oL F^2\DR_{\cyc}^{\alg}(\sO))\to  \ho\Lim_{\bI^{\op}} \oL F^2\DR_{\cyc}^{\alg}(\sC\hten A).
\]
Composing this with 
 the map
\begin{align*}
 \ho\Lim_{\bI^{\op}} \oL F^2\DR_{\cyc}^{\alg}(\sC\hten A) \to  &\ho\Lim_{\bI^{\op}}(\sC\hten F^2\DR_{\cyc}(A)) \\
 &\xla{\sim}  (\ho\Lim_{\bI^{\op}}\sC)\ten F^2\DR_{\cyc}(A)) \xra{\tau\ten \id} F^2\DR_{\cyc}(A))[-d]
 \end{align*}
gives our map from the space of $n$-shifted  bisymplectic structures on $F$ to the space of  analytic  $(n-d)$-shifted pre-bisymplectic structures on $D_*(F_{\sC})$.

 It thus follows immediately from  homogeneity of $F_{\sC}$ that $F_{\sC}^{nd} \subset F_{\sC}$ is a functor, and Lemma \ref{opencotlemma} implies that the morphism is formally \'etale. Moreover, Proposition \ref{RGammacotprop} implies that the cotangent complex of $F_{\sC}$ at $\phi \in F_{\sC}^{nd}(A)$ is the  $A^{\hat{e}}$-linear dual of  $(\ho\Lim_{\bI^{\op}} \bT_{\phi}F(\sC\hten A^{\hat{e}}))$.  
 
 Since $F_{\sC}^{nd} \to F_{\sC}$ is formally \'etale, we also have $(D_*F_{\sC}^{nd})(B)\simeq D_*F_{\sC}(B)\by^h_{F_{\sC}(B^0)}F_{\sC}^{nd}(B^0)$ for all $B \in DG^+dg_+\hat{\Tc}\Alg_{R}^{\mathrm{dfp,cof}}$. For $\phi \in (D_*F_{\sC}^{nd})_{\rig}(B)$ with induced point $\phi^0 \in F_{\sC}^{nd}(B^0)$, rigidity gives quasi-isomorphisms
 \begin{align*}
   \Tot^{\Pi}(\Omega^1_{B}\ten_{B^{e}}(B^0)^{e}) &\simeq \oL^{F_{\sC},\phi^0}\\
   \widehat{\Tot}\cHHom_{B^{e,\FEFC}}(\Omega^1_B,(B^0)^{e,\FEFC}) &\simeq  \bT_{\phi^0}(F_{\sC}, (B^0)^{e,\FEFC}),
 \end{align*}
so the non-degeneracy condition of Definition \ref{bisymplecticdef} is satisfied, making $\omega_{\tau}$ is a  bisymplectic structure on $D_*F_{\sC}^{nd}$.
\end{proof}

\subsubsection{Moduli of pro-\'etale local systems}

We can apply Proposition \ref{Cbispprop} to Examples \ref{CXex} to give purely analytic structures analogous to those on moduli of local systems from \S \ref{analytificnsn3}.

Letting $\bK$ be the non-Archimedean field  $\Ql$, we can take  ${\bI}$ to be the  affine pro-\'etale site of  a smooth proper scheme $X$ of dimension $m$ over a separably closed field $k$ prime to $\ell$, with $\sC$ the sheaf $\uline{\Ql}_X$ of \cite[Example 4.2.10]{BhattScholzeProEtale}.  Fixing an isomorphism $\Zl(m) \cong \Zl$ for the Tate twist over $k$, we then have a trace map
\[
 \oR\Gamma (X_{\pro\et},\uline{\Ql}_X)\to \Ql
\]
as in \cite[Examples \ref{PTLag-traceex}]{PTLag} forming part of a six functors formalism for constructible complexes, with $\uline{\Ql}_X\hten V \cong \uline{V}_X$ for any nuclear Fr\'echet space $V$.

Then Proposition \ref{Cbispprop} equips the moduli prestack $D_*\Perf_{\uline{\Ql}_X}$ of perfect $\Ql$-complexes on the pro-\'etale site of $X$ with an analytic  $(2-2m)$-shifted analytic pre-bisymplectic structure. 

The bisymplectic subfunctor $\Perf_{\uline{\Ql}_X}^{nd}(A) \subset\Perf_{\uline{\Ql}_X}(A) $ consists of  $\uline{A}_X$-complexes $\sE$ at which the $A\hten A$-bilinear pairing on $\oR\Gamma(X_{\pro\et}, \sHom_{\uline{A}_X}(\sE, \sE\hten \uline{A}_X))$  is perfect. In particular, this includes all constructible complexes.

\subsubsection{Moduli of hyperconnections}

Proposition \ref{Cbispprop} also applies to Examples \ref{CXex2}, giving structures on the various analytic de Rham, Dolbeault, Hodge and twistor moduli functors.  Comparison of  traces also ensures that the Riemann--Hilbert map is bisymplectic (i.e. preserves the bisymplectic structures), which in particular gives a bisymplectic structure on the algebro-analytic twistor moduli functor. Specifically:

\medskip
If $M$ is a compact oriented $d$-dimensional real manifold, then the trace $\int_M \co A^{\bt}(M) \to \R[-d]$ gives a $(2-d)$-shifted analytic pre-bisymplectic structure on the $\C^{\infty}$ de Rham prestack $D_*\Perf^{M,\dR,\C^{\infty}}$, and compatibility of traces ensures that the $\C^{\infty}$ Riemann--Hilbert  equivalence from the analytification of the Betti prestack $\Perf^{M,\oB}$ to $\Perf^{M,\dR,\C^{\infty}} $  (induced by Corollary \ref{poincarecor}) is bisymplectic, %preserves the respective pre-bisymplectic structures, 
so this structure is in fact bisymplectic rather than just pre-bisymplectic.

\medskip
Similarly, if $Y$ is a compact complex manifold of dimension $m$,  the $(2-2m)$-shifted analytic bisymplectic structure  on the analytic de Rham prestack  $D_*\Perf^{Y,\dR,\an}$ induced by the equivalences 
\[
 (\Perf^{Y,\oB})^{\FEFC} \to \Perf^{Y,\dR,\an} \to \Perf^{Y,\dR,\C^{\infty}} 
\]
of Corollary \ref{poincarecor} comes from the trace $\int_Y \co \oR\Gamma(Y, \Omega^{\bt}_Y) \to \Cx[-2m]$. 

\medskip
Likewise, $\int_Y$ gives traces  $\oR\Gamma(Y, \Omega^{\#}_Y) \to \Cx[-2m]$ and   $\oR\Gamma(Y,(\Omega^{\#}_Y\hten \sO(\bA^1_{\Cx})^{\hol}, \lambda \pd)) \to \sO(\bA^1_{\Cx})^{\hol} $, inducing $(2-2m)$-shifted analytic pre-bisymplectic structures  on the Dolbeault and Hodge prestacks $D_*\Perf^{Y,\Dol,\an}$ and $D_*\Perf^{Y,\Hod,\an}$. 

\medskip
For a smooth projective $m$-dimensional variety $X$,  compatibility of traces implies that the formally \'etale  map from the analytification of the algebraic de Rham prestack $\Perf^{X,\dR,\alg}$ to  $\Perf^{X(\Cx),\dR,\an}$ is bisymplectic.
%preserves the  $(2-2m)$-shifted bisymplectic structures. 
That in turn implies that the Riemann--Hilbert map $(\Perf^{X,\dR,\alg})^{\FEFC} \to (\Perf^{X(\Cx),\oB})^{\FEFC} $ of Corollary \ref{RHcoralg} is also bisymplectic. Similarly, the formally \'etale maps $(\Perf^{X,\Dol,\alg})^{\FEFC} \to \Perf^{X(\Cx),\Dol,\an}$ and $(\Perf^{X,\Hod,\alg})^{\FEFC} \to \Perf^{X(\Cx),\Hod,\an}$, from the analytifications of the algebraic Dolbeault and Hodge prestacks to their analytic counterparts, are bisymplectic. 

\medskip
The trace $\int_Y$ also gives a chain map $(A^{\#}(Y, \sO(\bA^2_{\R})^{\hol}),ud+v\dc)  \to \sO(\bA^2_{\R})^{\hol}[-2m]$ and hence a $(2-2m)$-shifted analytic pre-bisymplectic structure on the $\C^{\infty}$ pre-twistor prestack $D_*\Perf^{Y,\mathrm{pTw},\C^{\infty}}$ (see Definition \ref{twistordefCinftyR}), which is far from being bisymplectic at $(u,v)=0$. Since this pre-bisymplectic structure is not invariant under the action of the Deligne torus, it  does not quite give rise to a $(2-2m)$-shifted analytic pre-bisymplectic structure on the $\C^{\infty}$ twistor prestack  $D_*\Perf^{Y,\Tw,\C^{\infty}}$ over $[\bP^1_{\R}/\SO_2]^{\hol}_{\Et}$  (see Definition \ref{TwCinftydef}). %because of the twisting by line bundles: 
The trace map takes the form $A_{\Tw,Y}^{\bt}(\sB(U),L,w) \to L^{\ten m}\ten \bar{L}^{\ten m} $, so in the definitions we have to replace $F^2\DR$ with $(L\ten\bar{L}) ^{\ten m}\ten F^2\DR$ (an $m$-fold Tate twist), or equivalently $\sO_{\bP^1}(-m)\ten_{\sO_{\bP^1}} F^2\DR$, giving an $\sO_{\bP^1}(-m)[2-2m]$-shifted analytic pre-bisymplectic structure, defined by analogy with the symplectic structures shifted by line bundles from \cite{BouazizGrojnowski}.

\medskip
A similar phenomenon arises for the twistor functors constructed by Deligne gluing. The isomorphism $(\Omega^{\#}_Y\hten \sO(\bG_{m,\Cx})^{\hol}, w \pd) \to \Omega^{\bt}_Y\hten \sO(\bG_{m,\Cx})^{\hol} $ from Definition \ref{Twandef}
is given in degree $p$ by multiplication by $w^{-p}$, so the natural equivalence $\Perf^{Y,\Hod,\an}|_{\Cx^*} \to \Perf^{Y,\dR,\an} \by \Cx^*$ is not quite bisymplectic. Instead, we have $\omega_{\tau,\dR} \simeq w^{-m}\omega_{\tau,\Hod}$, which corresponds to the gluing data for generators of $\sO_{\bP^1}(m)$ under the equivalences of Remark  \ref{cfmhsrmk}, so under the Deligne gluing construction of  Definition \ref{Twandef} it
also induces an $\sO_{\bP^1}(-m)[2-2m]$-shifted analytic pre-bisymplectic structure on the analytic twistor prestack  $D_*\Perf^{Y,\Tw,\an}$ over $[\bP^1_{\R}/\SO_2]^{\hol}_{\Et}$. A similar check shows that this agrees with the $\C^{\infty}$ construction under the equivalence of Theorem \ref{twistorthm}. 

\medskip
Likewise, Deligne gluing gives an $\sO_{\bP^1}(-m)[2-2m]$-shifted analytic bisymplectic structure on the algebro-analytic twistor prestack  $D_*\Perf^{Y,\Tw,\alg-\an}$ over $[\bP^1_{\R}/\SO_2]^{\hol}_{\Et}$ (an open subfunctor of the analytic twistor prestack --- see Definition \ref{alganTwdef}). In this case, we can see that it is bisymplectic (i.e. non-degenerate) because its component parts are so.

\bibliographystyle{alphanum}
\bibliography{references.bib}

\end{document}